\documentclass[11pt, oneside]{article}   	% use "amsart" instead of "article" for AMSLaTeX format
\usepackage{geometry}                		% See geometry.pdf to learn the layout options. There are lots.
\usepackage[parfill]{parskip}    			% Activate to begin paragraphs with an empty line rather than an indent
\usepackage{graphicx}				% Use pdf, png, jpg, or epsÂ§ with pdflatex; use eps in DVI mode
\usepackage{amssymb}
\usepackage{mathtools}
\usepackage{enumerate}
\usepackage{tikz}
\usepackage{amsmath, amsthm}
\usepackage{bm}

\usepackage[utf8]{inputenc}
\usepackage[english]{babel}
\usepackage[scr]{rsfso}

\usepackage[bookmarks=false,hyperindex,breaklinks]{hyperref}

\usepackage{comment}

\usepackage{subeqnarray}

\usepackage{enumitem}

\usepackage{arydshln}

\setlist[description]{font=\normalfont\itshape\textbullet\space}

\numberwithin{equation}{section}

\usepackage{color}
\def\blue{\textcolor{blue}}
\def\red{\textcolor{red}}
\def\green{\textcolor{green}}

\usetikzlibrary{arrows,positioning,graphs}

\newcommand*{\Scale}[2][4]{\scalebox{#1}{$#2$}}

\newtheorem{thm}{Theorem}[section]
\newtheorem{lem}[thm]{Lemma}
\newtheorem{cor}[thm]{Corollary}

\newtheorem{conj}[thm]{Conjecture}
\newtheorem{prob}[thm]{Open Problem}

\newtheoremstyle{dotless}{}{}{\itshape}{}{\bfseries}{}{ }{}

\theoremstyle{dotless}
\newtheorem{thmdot}[thm]{Theorem}

\theoremstyle{definition}

\newtheorem{rem}{Remark}[section]
\newcommand{\N}{\mathbb{N}}
\newcommand{\Z}{\mathbb{Z}}

\newcommand{\wt}{{\rm wt}}

\newcommand{\proofof}[1]{\bigskip\noindent{\sc Proof of #1.\ }}

\newcommand{\myendremark}{ $\blacksquare$ \bigskip}

\newcommand{\sfa}{{{\sf a}}}
\newcommand{\sfb}{{{\sf b}}}
\newcommand{\sfc}{{{\sf c}}}
\newcommand{\sfd}{{{\sf d}}}
\newcommand{\sfe}{{{\sf e}}}
\newcommand{\sff}{{{\sf f}}}

\newcommand{\bsfa}{{\mbox{\textsf{\textbf{a}}}}}
\newcommand{\bsfb}{{\mbox{\textsf{\textbf{b}}}}}
\newcommand{\bsfc}{{\mbox{\textsf{\textbf{c}}}}}
\newcommand{\bsfd}{{\mbox{\textsf{\textbf{d}}}}}
\newcommand{\bsfe}{{\mbox{\textsf{\textbf{e}}}}}
\newcommand{\bsff}{{\mbox{\textsf{\textbf{f}}}}}

\newcommand{\scra}{{\mathcal{A}}}
\newcommand{\scrb}{{\mathcal{B}}}
\newcommand{\scrc}{{\mathcal{C}}}
\newcommand{\bfscra}{{\bm{\mathcal{A}}}}
\newcommand{\bfscrb}{{\bm{\mathcal{B}}}}
\newcommand{\bfscrc}{{\bm{\mathcal{C}}}}

\newcommand{\scrd}{{\mathcal{D}}}

\newcommand{\scrm}{{\mathcal{M}}}

\newcommand{\scrs}{{\mathcal{S}}}
\newcommand{\bzero}{{\bm{0}}}

\newcommand{\rbf}{v}

\newcommand{\dperm}{{\mathfrak{D}}}
\newcommand{\dcycle}{{\mathfrak{DC}}}

\newcommand{\restrict}{\upharpoonright}
\newcommand{\sinv}{\sigma^{-1}}

\def\omegahat{{\widehat{\omega}}}

\newcommand{\eqdef}{\stackrel{\rm def}{=}}

\newcommand{\textbfit}[1]{\textbf{\textit{#1}}}

\newcommand{\be}{\begin{equation}}
\newcommand{\ee}{\end{equation}}

\newcommand{\fS}{\mathfrak{S}}
\newcommand{\Sym}{\mathfrak{S}}

\newcommand{\lev}{{\rm lev}}

\newcommand{\cyc}{{\rm cyc}}

\newcommand{\earec}{{\rm earec}}

\newcommand{\erec}{{\rm erec}}

\newcommand{\nrar}{{\rm nrar}}
\newcommand{\ereccval}{{\rm ereccval}}
\newcommand{\ereccdrise}{{\rm ereccdrise}}
\newcommand{\eareccpeak}{{\rm eareccpeak}}
\newcommand{\eareccdfall}{{\rm eareccdfall}}
\newcommand{\eareccval}{{\rm eareccval}}
\newcommand{\ereccpeak}{{\rm ereccpeak}}
\newcommand{\rar}{{\rm rar}}
\newcommand{\evenrar}{{\rm evenrar}}
\newcommand{\oddrar}{{\rm oddrar}}
\newcommand{\nrcpeak}{{\rm nrcpeak}}
\newcommand{\nrcval}{{\rm nrcval}}
\newcommand{\nrcdrise}{{\rm nrcdrise}}
\newcommand{\nrcdfall}{{\rm nrcdfall}}
\newcommand{\nrfix}{{\rm nrfix}}
\newcommand{\Evenfix}{{\rm Evenfix}}

\newcommand{\Oddfix}{{\rm Oddfix}}

\newcommand{\evennrfix}{{\rm evennrfix}}
\newcommand{\oddnrfix}{{\rm oddnrfix}}
\newcommand{\Cpeak}{{\rm Cpeak}}
\newcommand{\cpeak}{{\rm cpeak}}
\newcommand{\Cval}{{\rm Cval}}
\newcommand{\cval}{{\rm cval}}

\newcommand{\Cdrise}{{\rm Cdrise}}
\newcommand{\cdrise}{{\rm cdrise}}
\newcommand{\Cdfall}{{\rm Cdfall}}
\newcommand{\cdfall}{{\rm cdfall}}

\newcommand{\minval}{{\rm minval}}
\newcommand{\nminval}{{\rm nminval}}

\newcommand{\Fix}{{\rm Fix}}
\newcommand{\fix}{{\rm fix}}

\newcommand{\zo}{z_{\rm o}}
\newcommand{\ze}{z_{\rm e}}
\newcommand{\wo}{w_{\rm o}}
\newcommand{\we}{w_{\rm e}}
\newcommand{\so}{s_{\rm o}}
\newcommand{\se}{s_{\rm e}}

\newcommand{\Rec}{{\rm Rec}}
\newcommand{\rec}{{\rm rec}}
\newcommand{\Arec}{{\rm Arec}}
\newcommand{\arec}{{\rm arec}}

\newcommand{\ucross}{{\rm ucross}}
\newcommand{\ucrosscval}{{\rm ucrosscval}}
\newcommand{\ucrosscpeak}{{\rm ucrosscpeak}}
\newcommand{\ucrosscdrise}{{\rm ucrosscdrise}}
\newcommand{\lcross}{{\rm lcross}}
\newcommand{\lcrosscpeak}{{\rm lcrosscpeak}}
\newcommand{\lcrosscval}{{\rm lcrosscval}}
\newcommand{\lcrosscdfall}{{\rm lcrosscdfall}}
\newcommand{\unest}{{\rm unest}}
\newcommand{\unestcval}{{\rm unestcval}}
\newcommand{\unestcpeak}{{\rm unestcpeak}}
\newcommand{\unestcdrise}{{\rm unestcdrise}}
\newcommand{\lnest}{{\rm lnest}}
\newcommand{\lnestcpeak}{{\rm lnestcpeak}}
\newcommand{\lnestcval}{{\rm lnestcval}}
\newcommand{\lnestcdfall}{{\rm lnestcdfall}}
\newcommand{\ujoin}{{\rm ujoin}}
\newcommand{\ljoin}{{\rm ljoin}}
\newcommand{\psnest}{{\rm psnest}}
\newcommand{\upsnest}{{\rm upsnest}}
\newcommand{\lpsnest}{{\rm lpsnest}}
\newcommand{\epsnest}{{\rm epsnest}}
\newcommand{\opsnest}{{\rm opsnest}}

\newcommand{\vtilde}{{\widetilde{v}}}
\newcommand{\ytilde}{{\widetilde{y}}}

\newcommand{\laguerre}[1]{\left. L\right|_{#1}}
\newcommand{\laguerrep}[1]{\left. L'\right|_{#1}}

\title{Continued fractions using a Laguerre digraph interpretation
of the Foata--Zeilberger bijection and its variants}
\author{Bishal Deb${}^{1,2}$
\\[3mm]
     \hspace*{-1.3cm}
      \normalsize
	   ${}^1$Department of Mathematics,
           University College London, 
           London WC1E 6BT,
           United Kingdom\\[2mm] 
\hspace*{-1.15cm}
      \normalsize
	   ${}^2$Sorbonne Universit\'e and Universit\'e Paris Cit\'e, CNRS,
         Laboratoire de Probabilit\'es, \\
     \hspace*{-6cm}
      \normalsize
     Statistique et Mod\'elisation, 75005 Paris, France \\[2mm]
 {\texttt{bishal@gonitsora.com}}}
\date{April 27, 2023\\
Revised on September 27, 2024\vspace{3mm} 
}
\begin{document}

\maketitle\thispagestyle{empty}
%\vspace{3mm}

\begin{abstract}
In the combinatorial theory of continued fractions, 
the Foata--Zeilberger bijection and its variants have been extensively used
to derive various continued fractions enumerating several 
(sometimes infinitely many) simultaneous statistics 
on permutations (combinatorial model for factorials) 
and D-permutations 
(combinatorial model for Genocchi and median Genocchi numbers).
A Laguerre digraph is a digraph in which each vertex has 
in- and out-degrees $0$ or $1$.
In this paper, we interpret the Foata--Zeilberger bijection
in terms of Laguerre digraphs,
which enables us to count cycles in permutations.
Using this interpretation, we obtain
Jacobi-type continued fractions for multivariate polynomials 
enumerating permutations,
and also Thron-type and Stieltjes-type 
continued fractions for multivariate polynomials enumerating D-permutations,
in both cases including the counting of cycles.
This enables us to prove some conjectured continued fractions 
due to Sokal and Zeng (2022 Advances in Applied Mathematics)
in the case of permutations,
and Randrianarivony and Zeng (1996 Electronic Journal of Combinatorics)
and Deb and Sokal (2024 Advances in Applied Mathematics)
in the case of D-permutations.
\end{abstract}

%\bigskip

%\bigskip
\medskip
\noindent
{\bf Mathematics Subject Classification (MSC 2020) codes:}
05A19 (Primary);
05A05, 05A10, 05A15, 05A30, 11B68, 30B70 (Secondary).

\clearpage

\tableofcontents

\clearpage

\section{Introduction}

\subsection{Foreword}

This paper will introduce new results to 
the combinatorial theory of continued fractions
for multivariate polynomials generalising
the following three sequences of integers:
factorials $(n!)_{n\geq 0}$, the Genocchi numbers \cite[A110501]{OEIS}
\begin{equation}
	 (g_n)_{n \ge 0}
   \;=\;
   1, 1, 3, 17, 155, 2073, 38227, 929569, 28820619, 1109652905,
   %% 51943281731,
   \ldots
 \label{eq.genocchi}
\end{equation}
and the median Genocchi numbers \cite[A005439]{OEIS}
\begin{equation}
	(h_n)_{n \ge 0}
   \;=\;
   1, 1, 2, 8, 56, 608, 9440, 198272, 5410688, 186043904,
   %% 7867739648,
   \ldots
   \;.
\end{equation}
We shall use permutations for studying factorials, and a subclass of permutations 
called {\em D-permutations} (first introduced in \cite{Lazar_22,Lazar_20})
for studying the Genocchi and median Genocchi numbers.

We shall consider continued fractions of Stieltjes-type (S-fraction),
\be
   \sum_{n=0}^\infty a_n t^n
   \;=\;
   \cfrac{1}{1 - \cfrac{\alpha_1 t}{1 - \cfrac{\alpha_2 t}{1 - \cdots}}}
   \label{def.Stype}
   \;\;,
\ee
%where $a_n$ are its coefficients when expanded as a formal power series.
as well as Jacobi-type and Thron-type (defined in \eqref{def.Jtype},\/\eqref{def.Ttype}).
The ordinary generating functions of our integer sequences 
have S-fractions with coefficients
$\alpha_{2k-1} = \alpha_{2k} = k$ for factorials \cite[section~21]{Euler_1760},
$\alpha_{2k-1} = k^2$ and $\alpha_{2k} = k(k+1)$ for the Genocchi numbers
\cite[eq.~(7.5)]{Viennot_81}
\cite[p.~V-9]{Viennot_83}
\cite[eqns.~(1.4) and (3.9)]{Dumont_94b},
and $\alpha_{2k-1} = \alpha_{2k} = k^2$ for the Genocchi medians
\cite[eq.~(9.7)]{Viennot_81}
\cite[p.~V-15]{Viennot_83}
\cite[eqns.~(1.5) and (3.8)]{Dumont_94b}.

%\begin{enumerate}
%	\item \begin{equation}
%			\sum_{n=0}^\infty n! \, t^n \;=\;
%   \cfrac{1}{1 - \cfrac{t}{1 - \cfrac{ t}{1 - \cfrac{2 t}{1- \cfrac{2 t}{1-\cdots}}}}}
%		\label{eq.factorial.Sfrac}
%                \end{equation}
%
%	\item \begin{equation}
%			\sum_{n=0}^\infty g_n \, t^n   \;=\;
%   \cfrac{1}{1 - \cfrac{1 \cdot 1 t}{1 - \cfrac{1 \cdot 2t}{1 - \cfrac{2 \cdot 2t}{1- \cfrac{2 \cdot 3 t}{1-\cdots}}}}}
%		\label{eq.genocchi.Sfrac}
%	\end{equation}
%	
%	\item \begin{equation}
%		\sum_{n=0}^\infty h_n \, t^n   \;=\;
%   \cfrac{1}{1 - \cfrac{1 t}{1 - \cfrac{1 t}{1 - \cfrac{4t}{1- \cfrac{4 t}{1-\cdots}}}}}
%  \label{eq.mediangenocchi.Sfrac}
%	\end{equation}
%\end{enumerate}

A systematic study of some combinatorial families whose
associated S-fraction coefficients $(\alpha_n)_{n\geq 0}$ grow linearly in $n$
was carried out by Sokal and Zeng in \cite{Sokal-Zeng_masterpoly}.
They introduced various ``master polynomials'' enumerating
permutations, set partitions and perfect matchings
with respect to a large (sometimes infinite) number of simultaneous statistics.
A similar study was carried out for D-permutations and its subclasses
by Deb and Sokal in \cite{Deb-Sokal}:
here the associated T-fraction coefficients $(\alpha_n)_{n\geq 0}$
for these families grow quadratically in $n$.
The continued fractions in \cite{Sokal-Zeng_masterpoly,Deb-Sokal}
were classified as ``first'' or ``second''
depending on whether they did not or did involve the count of cycles.
Both Sokal--Zeng and Deb--Sokal were able to prove  
``second'' continued fractions but by using two specialisations.
They conjectured continued fractions with only one specialisation
%However, they conjectured results with only one specialisation
(\cite[Conjecture~2.3]{Sokal-Zeng_masterpoly} and \cite[Conjecture~4.1]{Deb-Sokal}),
but a proof was lacking.
%but were unable to prove them.
Here we prove these conjectures.
We will also prove a conjectured continued fraction of
Randrianarivony and Zeng from 1996
\cite[Conjecture~12]{Randrianarivony_96b} 
for D-o-semiderangements\footnote{In their paper
\cite{Randrianarivony_96b},
Randrianarivony and Zeng call these Genocchi permutations.
We will explain our nomenclature in subsection~\ref{subsec.intro.Dperm.def}.}
 (a subclass of D-permutations).

Our proofs bring a surprising twist to this story.
A common feature in the work of Sokal--Zeng \cite{Sokal-Zeng_masterpoly}
and Deb--Sokal \cite{Deb-Sokal} is that
the proofs of the first and second continued fractions involve
two different bijections:
the first continued fractions were proved using bijections  
motivated from the Foata--Zeilberger bijection \cite{Foata_90},
whereas the second continued fractions used the Biane bijection \cite{Biane_93}
or a Biane-like bijection.
However, we will prove these conjectured second continued fractions
by precisely the same bijections that were used to prove 
the first bijections in these papers.
We will show, perhaps surprisingly,
that these variants of the Foata--Zeilberger bijection 
can be used to obtain the counting of cycles.

Let us mention the historical context for our bijections.
The Foata--Zeilberger bijection \cite{Foata_90} is a bijection
between permutations and labelled Motzkin paths
that has been very successfully
employed to obtain continued fractions involving polynomial coefficients
counting various permutation statistics (see for example
\cite{Randrianarivony_98b, Corteel_07, Blitvic_21, Sokal-Zeng_masterpoly}).
In a similar essence to the Foata--Zeilberger bijection,
Randrianarivony \cite{Randrianarivony_97}
introduced a bijection between D-o-semiderangements
and labelled Dyck paths to obtain continued fractions
counting various statistics on D-o-semiderangements.
Motivated by Randrianarivony's bijection, Deb and Sokal \cite{Deb-Sokal}
introduced two new bijections involving all D-permutations,
one of which extends Randrianarivony's bijection.

The fundamental idea in this paper is that
we interpret the intermediate steps in these existing bijections
in a new light in terms of {\em Laguerre digraphs}.
A Laguerre digraph of size $n$ is a directed graph
where each vertex has a distinct label from the label set $[n]$
and has indegree $0$ or $1$ and outdegree $0$ or $1$.\footnote{
Foata and Strehl \cite{Foata_84}
introduced an equivalent class of combinatorial objects
called Laguerre configurations
as a combinatorial interpretation of the Laguerre polynomials.
Laguerre digraphs in the form that we use in this paper
were first defined in
\cite{Sokal_2022}. Also see \cite{Deb-Dyachenko_laguerre}.}
Thus, the connected components in a Laguerre digraph are
either directed paths or directed cycles.
A path with one vertex and no edges will be called an isolated vertex,
and a cycle with one vertex (and one edge) will be called a loop.

The Sokal--Zeng conjecture
\cite[Conjecture~2.3]{Sokal-Zeng_masterpoly}
is a multivariate continued fraction
containing 8 variables along with a one-parameter family of
infinitely many variables
(the latter associated to fixed points)
counting various simultaneous statistics for permutations.
The Deb--Sokal conjecture \cite[Conjecture~4.1]{Deb-Sokal}
is a multivariate continued fraction with 12 variables
counting similar simultaneous statistics for D-permutations.
The Randrianarivony--Zeng conjecture from 1996
\cite[Conjecture~12]{Randrianarivony_96b}
is a 4-variable continued fraction
for D-o-semiderangements.
In the same spirit as \cite{Elvey-Price-Sokal_wardpoly, Sokal-Zeng_masterpoly, Deb-Sokal},
we will generalise these conjectured continued fractions and
use our proofs to churn out
continued fractions containing an infinite number of variables.

The rest of the introduction is organised as follows:
We begin by explaining briefly the types of continued fractions that
will be employed (Section~\ref{subsec.contfrac})
and then introduce the required statistics (Section~\ref{subsec.intro.stats}).
We state the conjecture for permutations 
\cite[Conjecture~2.3]{Sokal-Zeng_masterpoly} in
Section~\ref{subsec.intro.perms}.
We then define Genocchi, median Genocchi numbers and D-permutations in Section~\ref{subsec.intro.Dperm.def},
and state the associated conjectures 
(\cite[Conjecture~12]{Randrianarivony_96b} and \cite[Conjecture~4.1]{Deb-Sokal})
in Section~\ref{subsec.intro.Dperm.conj}.
Then, in Section \ref{subsec.intro.proof},
we summarise our main ideas by
providing an overview of the Foata--Zeilberger bijection
and our interpretation of this bijection using Laguerre digraphs.
The outline of the rest of the paper is mentioned in Section \ref{subsec.intro.outline}.

Throughout this paper, including the rest of this introduction, 
we shall use two running examples.
The first is the permutation
\be
\sigma = 9\,3\,7\,4\,6\,11\,5\,8\,10\,1\,2
           = (1,9,10)\,(2,3,7,5,6,11)\,(4)\,(8) \in \Sym_{11};
\ee
the second is the permutation
\begin{eqnarray}
\sigma & = & 7\, 1\, 9\, 2\, 5\, 4\, 8\, 6\, 10\, 3\, 11\, 12\, 14\, 13\,
        \nonumber\\
       & = & (1,7,8,6,4,2)\,(3,9,10)\,(5)\,(11)\,(12)\,(13,14) \in \Sym_{14}.
\label{eq.running.example.2}
\end{eqnarray}
We will see later in Section~\ref{subsec.intro.Dperm.def} that our second example is a D-permutation.

\begin{sloppy}
\subsection{Classical continued fractions: S-fractions, J-fractions
   and T-fractions}
   \label{subsec.contfrac}
\end{sloppy}

If $(a_n)_{n \ge 0}$ is a sequence of combinatorial numbers or polynomials
with $a_0 = 1$, it is often fruitful to seek to express its
ordinary generating function as a continued fraction.
The most commonly studied types of continued fractions
are Stieltjes-type (S-fractions),
\be
   \sum_{n=0}^\infty a_n t^n
   \;=\;
   \cfrac{1}{1 - \cfrac{\alpha_1 t}{1 - \cfrac{\alpha_2 t}{1 - \cdots}}}
   \;\;,
\ee
and Jacobi-type (J-fractions),
\be
   \sum_{n=0}^\infty a_n t^n
   \;=\;
   \cfrac{1}{1 - \gamma_0 t - \cfrac{\beta_1 t^2}{1 - \gamma_1 t - \cfrac{\beta_2 t^2}{1 - \cdots}}}
   \label{def.Jtype}
   \;\;.
\ee
A less commonly studied type of continued fraction is 
the Thron-type (T-fraction):
\be
   \sum_{n=0}^\infty a_n t^n
   \;=\;
   \cfrac{1}{1 - \delta_1 t - \cfrac{\alpha_1 t}{1 - \delta_2 t - \cfrac{\alpha_2 t}{1 - \cdots}}}
   \label{def.Ttype}
   \;\;.
\ee
(Both sides of all these expressions are to be interpreted as
formal power series in the indeterminate $t$.)
Flajolet \cite{Flajolet_80} showed that
any S-fraction (resp.\ J-fraction)
can be interpreted combinatorially as a generating function
for Dyck (resp.\ Motzkin) paths with suitable weights for each rise and fall
(resp.\ each rise, fall and level step).
More recently, several authors
\cite{Fusy_15,Oste_15,Josuat-Verges_18,Elvey-Price-Sokal_wardpoly}
have found a similar combinatorial interpretation
of the general T-fraction:
namely, as a generating function for Schr\"oder paths with suitable
weights for each rise, fall and long level step.
These interpretations will be reviewed in
Section~\ref{subsec.prelimproofs.1} below.

\subsection{Permutation statistics: record and cycle classification}
\label{subsec.intro.stats}

We will follow the terminology in \cite{Deb-Sokal}.

We begin by introducing some statistics which have been called 
the \textbfit{record-and-cycle classification} for permutations
and D-permutations in \cite{Sokal-Zeng_masterpoly}
and \cite[Section~2.7]{Deb-Sokal} respectively.
These statistics will play a central role
in what follows.

Given a permutation $\sigma \in \fS_N$, an index $i \in [N]$ is called
\begin{itemize}
   \item {\em cycle peak}\/ (cpeak) if $\sigma^{-1}(i) < i > \sigma(i)$;
   \item {\em cycle valley}\/ (cval) if $\sigma^{-1}(i) > i < \sigma(i)$;
   \item {\em cycle double rise}\/ (cdrise) if $\sigma^{-1}(i) < i < \sigma(i)$;
   \item {\em cycle double fall}\/ (cdfall) if $\sigma^{-1}(i) > i > \sigma(i)$;
   \item {\em fixed point}\/ (fix) if $\sigma^{-1}(i) = i = \sigma(i)$.
\end{itemize}
Clearly every index $i$ belongs to exactly one of these five types;
we refer to this classification as the \textbfit{cycle classification}.

On the other hand, an index $i \in [N]$ is called a
\begin{itemize}
   \item {\em record}\/ (rec) (or {\em left-to-right maximum}\/)
         if $\sigma(j) < \sigma(i)$ for all $j < i$
      [note in particular that the indices 1 and $\sigma^{-1}(N)$
       are always records];
   \item {\em antirecord}\/ (arec) (or {\em right-to-left minimum}\/)
         if $\sigma(j) > \sigma(i)$ for all $j > i$
      [note in particular that the indices $N$ and $\sigma^{-1}(1)$
       are always antirecords];
   \item {\em exclusive record}\/ (erec) if it is a record and not also
         an antirecord;
   \item {\em exclusive antirecord}\/ (earec) if it is an antirecord
         and not also a record;
   \item {\em record-antirecord}\/ (rar) (or {\em pivot}\/)
      if it is both a record and an antirecord;
   \item {\em neither-record-antirecord}\/ (nrar) if it is neither a record
      nor an antirecord.
\end{itemize}
Every index $i$ thus belongs to exactly one of the latter four types;
we refer to this classification as the \textbfit{record classification}.

Furthermore, one can apply the record and cycle classifications simultaneously,
to obtain 10 (not~20) disjoint categories: 
%\begin{itemize}
%   \item ereccval:  exclusive records that are also cycle valleys;
%   \item ereccdrise:  exclusive records that are also cycle double rises;
%   \item eareccpeak:  exclusive antirecords that are also cycle peaks;
%   \item eareccdfall:  exclusive antirecords that are also cycle double falls;
%   \item rar:  record-antirecords (these are always fixed points);
%   \item nrcpeak:  neither-record-antirecords that are also cycle peaks;
%   \item nrcval:  neither-record-antirecords that are also cycle valleys;
%   \item nrcdrise:  neither-record-antirecords that are also cycle double rises;
%   \item nrcdfall:  neither-record-antirecords that are also cycle double falls;
%   \item nrfix:  neither-record-antirecords that are also fixed points.
%\end{itemize}
\medskip
\begin{center}
\begin{tabular}{c|c|c|c|c|c|}
                & cpeak & cval & cdrise & cdfall & fix \\
        \hline
        erec &   & ereccval & ereccdrise & &\\
        earec & eareccpeak & &  & eareccdfall &\\
        rar & & & & & rar \\
        nrar & nrcpeak & nrcval & nrcdrise & nrcdfall & nrfix\\
        \hline
\end{tabular}
\end{center}
\medskip
Clearly every index $i$ belongs to exactly one of these 10~types;
we call this the \textbfit{record-and-cycle classification}.

A variant of this classification involving
record and antirecord values rather than indices
was used in \cite[Section~3.5]{Deb-Sokal}.
A value $i \in [N]$ is called a
\begin{itemize}
   \item {\em record value}\/ ($\rec'$) (or {\em left-to-right maximum value}\/)
         if $\sigma(j) < i$ for all $j < \sinv(i)$
	[note in particular that the values $\sigma(1)$ and $N$
       are always record values];
   \item {\em antirecord value}\/ ($\arec'$) (or {\em right-to-left minimum value}\/)
         if $\sigma(j) > i$ for all $j > \sinv(i)$
	[note in particular that the values $\sigma(N)$ and $1$
       are always antirecord values].
\end{itemize}
Thus, $i$ is a record value (antirecord value resp.) if and only if $\sinv(i)$ is a record (antirecord).

\begin{sloppy}
We also analogously define {\em exclusive record value} ($\erec'$),
{\em exclusive antirecord value} ($\earec'$),
{\em record-antirecord value} ($\rar'$) (or {\em pivot value}),
{\em neither-record-antirecord value}\\ 
($\nrar'$).
Every index $i$ thus belongs to exactly one of these four types;
we refer to this classification as the \textbfit{variant record classification}.
\end{sloppy}

We can similarly introduce the \textbfit{variant record-and-cycle classification}
consisting of the following 10~disjoint categories:
%\begin{itemize}
%   \item $\eareccval'$:  exclusive record values that are also cycle valleys;
%   \item $\ereccdrise'$:  exclusive record values that are also cycle double rises;
%   \item $\ereccpeak'$:  exclusive antirecord values that are also cycle peaks;
%   \item $\eareccdfall'$:  exclusive antirecord values that are also cycle double falls;
%   \item $\rar'$:  record-antirecord values (these are always fixed points);
%   \item $\nrcpeak'$:  neither-record-antirecord values that are also cycle peaks;
%   \item $\nrcval'$:  neither-record-antirecord values that are also cycle valleys;
%   \item $\nrcdrise'$:  neither-record-antirecord values that are also cycle double rises;
%   \item $\nrcdfall'$:  neither-record-antirecord values that are also cycle double falls;
%   \item $\nrfix'$:  neither-record-antirecord values that are also fixed points.
%\end{itemize}
\medskip
\begin{center}
\begin{tabular}{c|c|c|c|c|c|}
                & cpeak & cval & cdrise & cdfall & fix \\
        \hline
        $\erec'$ & $\ereccpeak'$  & & $\ereccdrise'$ & &\\
        $\earec'$ & & $\eareccval'$ &  & $\eareccdfall'$ &\\
        $\rar'$ & & & & & $\rar'$ \\
        $\nrar'$ & $\nrcpeak'$ & $\nrcval'$ & $\nrcdrise'$ & $\nrcdfall'$ & $\nrfix'$\\
        \hline
\end{tabular}
\end{center}
\medskip
Notice that in record-and-cycle classification,
cycle valleys (cycle peaks resp.) can be exclusive records (exclusive anti-records) 
whereas in the variant record-and-cycle classification,
cycle valleys (cycle peaks) can now be exclusive anti-records (exclusive records).

While working with D-permutations 
(defined in Section~\ref{subsec.intro.Dperm.def}), 
we will refine the fixed points according to their parity:
%For fixed points, we will sometimes record the parity of $i$,
%by distinguishing even and odd fixed points:
%
\be
   \begin{aligned}
   & \bullet\; \hbox{{\em even fixed point}\/ (evenfix):  $\sigma^{-1}(i) = i = \sigma(i)$ is even}  \\[2mm]
   & \bullet\; \hbox{{\em odd fixed point}\/ (oddfix):  $\sigma^{-1}(i) = i = \sigma(i)$ is odd}   \hspace*{3.7cm} \\
   \end{aligned}
 \label{eq.parities.2}
\ee

We therefore refine the record-and-cycle classification
by distinguishing even and odd fixed points:
\begin{itemize}
   \item evenrar:  even record-antirecords (these are always fixed points);
   \item oddrar:  odd record-antirecords (these are always fixed points);
   \item evennrfix:  even neither-record-antirecords that are also fixed points;
   \item oddnrfix:  odd neither-record-antirecords that are also fixed points.
\end{itemize}
%When $\sigma$ is a D-permutation 
%(defined in Section~\ref{subsec.intro.Dperm.def}),
We will use the \textbfit{parity-refined record-and-cycle classification},
in which each index $i$ belongs to exactly one of 12 types.
More precisely, we will see (by using Equation~\eqref{eq.parities.1})
that each even index $i$ belongs to exactly one of the 6 types
\begin{quote}
   eareccpeak, nrcpeak, eareccdfall, nrcdfall, evenrar, evennrfix,
\end{quote}
while each odd index $i$ belongs to exactly one of the 6 types
\begin{quote}
   ereccval, nrcval, ereccdrise, nrcdrise, oddrar, oddnrfix.
\end{quote}
We can also similarly introduce 
\textbfit{variant parity-refined record-and-cycle classification}
for D-permutations.
We will see (again by using Equation~\eqref{eq.parities.1})
that each even index $i$ belongs to exactly one of the 6 types
\begin{quote}
   $\ereccpeak'$, $\nrcpeak'$, $\eareccdfall'$, $\nrcdfall'$, $\evenrar'$, $\evennrfix'$,
\end{quote}
while each odd index $i$ belongs to exactly one of the 6 types
\begin{quote}
   $\eareccval'$, $\nrcval'$, $\ereccdrise'$, $\nrcdrise'$, $\oddrar'$, $\oddnrfix'$.
\end{quote}

Additionally, we define the {\em pseudo-nestings} of a fixed point \cite[eq.~(2.51)]{Deb-Sokal} by
\begin{equation}
	\psnest(i,\sigma) \;\eqdef\; \#\{j<i\colon\: \sigma(j)>i\} \;=\; \#\{j>i\colon\: \sigma(j)<i \}.
\label{def.level}
\end{equation}
This quantity was called level of $i$
and was denoted as $\lev(i,\sigma)$ in \cite[eq.~(2.20)]{Sokal-Zeng_masterpoly}. 
In this paper, we prefer to use $\psnest$.

Also, notice that each non-singleton cycle consists of exactly one
minimum element, which must be a cycle valley,
and one maximum element, which must be cycle peak.
With this observation, the following four statistics were introduced in 
\cite[Section~4.1.3]{Deb-Sokal}.
\begin{itemize}
\item cycle valley minimum (minval): cycle valley that is the minimum in its cycle;
\item cycle peak maximum (maxpeak): cycle peak that is the maximum in its cycle;
\item cycle valley non-mimimum (nminval): cycle valley that is not the minimum in its cycle;
\item cycle peak non-maximum (nmaxpeak): cycle peak that is not the maximum in its cycle.
\end{itemize}

Finally, whenever we use the name of a statistic
but with its first letter in capital,
we will refer to the set of elements that belong to that statistic
(in case that makes sense).
For example, we use $\Cval$ to denote the set of all cycle valleys,
or $\Evenfix$ to denote the set of even fixed points.

We now state the cycle and record classifications for our two running examples
which will be of use throughout this introduction and the rest of this paper.

\subsubsection{Running example 1}

We consider our first running example
%$\sigma = 9\,3\,7\,4\,6\,11\,2\,8\,10\,1\,5\\
%           = (1,9,10)\,(2,3,7)\,(4)\,(5,6,11)\,(8) \in \Sym_{11}.$
in its cycle notation, 
$\sigma = (1,9,10)\,(2,3,7,5,6,11)\\
\,(4)\,(8) \in \Sym_{11}.$
%The excedance classification of $\sigma$
%partitions the index set $[11] \eqdef \{1,\ldots,11\}$ as follows:
%\be
%{\rm Exc} \;=\; \{1, 2, 3, 5, 6, 9\} ,
%\qquad {\rm Aexc} \;=\; \{7, 10, 11\} ,
%\qquad \Fix \;=\; \{4, 8\} \;.
%\label{eq.example.1.excedance.classification}
%\ee
%Thus, $\exc(\sigma) = 6$, $\aexc(\sigma) = 3$
%and $\fix(\sigma) = 2$.
The cycle classification of $\sigma$
partitions the index set $[11] \eqdef \{1,\ldots,11\}$ as follows:
%Next, we write out the cycle classification of $\sigma$:
\begin{subeqnarray}
        \Cpeak(\sigma) \; = \; \{7, 10, 11 \}  &\qquad&
        \Cval(\sigma) \; = \; \{1,2,5\} \\
        \Cdrise(\sigma) \; = \; \{ 3,6,9\} &\qquad&
        \Cdfall(\sigma) \; = \; \emptyset\\
        \Fix(\sigma) \; = \; \{4, 8\} &&
\label{eq.example.1.cycle.classification}
\end{subeqnarray}
Thus, the statistics $\cpeak, \cval, \cdrise, \cdfall$ and $\fix$ are
simply the cardinalities of these sets, respectively.

For the record classification, we write $\sigma$ as a word,
i.e., $\sigma = 9\,\, 3\,\, 7\,\, 4\,\, 6\,\, 11\,\, 5\,\, 8\,\, 10\,\, 1\,\, 2$.
The~record and antirecord positions are therefore
$\Rec(\sigma) = \{1, 6\}$ and $\Arec(\sigma) = \{10, 11\}$.
(Also, notice that the record and antirecord values are 
$\Rec'(\sigma) = \{9, 11\}$ and $\Arec'(\sigma) = \{1, 2\}$.)
The full record classification is
\begin{subeqnarray}
        {\rm Erec}(\sigma) \;=\; \{1,6\} &&
        {\rm Earec}(\sigma) \;=\; \{10,11\} \\
        {\rm Rar}(\sigma) \;=\; \emptyset &&
        {\rm Nrar}(\sigma) \;=\; \{2,3,4,5,7,8,9\}
\label{eq.example.1.record.classification}
\end{subeqnarray}

Finally, the record-and-cycle classification gives us
\begin{subeqnarray}
        {\rm Eareccpeak}(\sigma) \;=\; \{10, 11\} &&
        {\rm Nrcpeak}(\sigma) \;=\; \{7\} \\
        {\rm Ereccval}(\sigma) \;=\; \{1\} &&
        {\rm Nrcval}(\sigma) \;=\; \{2,5\} \\
        {\rm Erecdrise}(\sigma) \;=\; \{6\} &&
        {\rm Nrcdrise}(\sigma) \;=\; \{3,9\} \\
        {\rm Earecdfall}(\sigma) \;=\; \emptyset &&
        {\rm Nrcdfall}(\sigma) \;=\; \emptyset \\
        {\rm Rar}(\sigma) \;=\; \emptyset &&
        {\rm Nrfix}(\sigma) \; = \; \{4, 8\}
\label{eq.example.1.record.and.cycle.classification}
\end{subeqnarray}

We leave the variant record classification and 
the variant record-and-cycle classification
as an exercise for the reader.

%{\bf CHECK THIS!!!!!}

\subsubsection{Running example 2}

We now consider our second running example
%\begin{eqnarray}
%\sigma & = & 7\, 1\, 9\, 2\, 5\, 4\, 8\, 6\, 10\, 3\, 11\, 12\, 14\, 13\,
%        \nonumber\\
%       & = & (1,7,8,6,4,2)\,(3,9,10)\,(5)\,(11)\,(12)\,(13,14) \in \Sym_{14}.
%\end{eqnarray}
in its cycle notation, \\
$\sigma = (1,7,8,6,4,2)\,(3,9,10)\,(5)\,(11)\,(12)\,(13,14) \in \Sym_{14}.$
%The excedance classification of $\sigma$
%partitions the index set $[14] \eqdef \{1,\ldots,14\}$ as follows:
%\be
%{\rm Exc} \;=\; \{1, 3, 7, 9, 13\} ,
%\qquad {\rm Aexc} \;=\; \{2, 4, 6, 8, 10, 14\} ,
%\qquad \Fix \;=\; \{5,11,12\} \;.
%\label{eq.example.2.excedance.classification}
%\ee
%Thus, $\exc(\sigma) = 5$, $\aexc(\sigma) = 6$
%and $\fix(\sigma) = 3$.
The cycle classification of $\sigma$
partitions the index set $[14] \eqdef \{1,\ldots,14\}$ as follows:
%Next, we write out the cycle classification of $\sigma$:
\begin{subeqnarray}
        \Cpeak(\sigma) \; = \; \{8, 10, 14 \}  &\qquad&
        \Cval(\sigma) \; = \; \{1, 3, 13\} \\
        \Cdrise(\sigma) \; = \; \{7, 9 \} &\qquad&
        \Cdfall(\sigma) \; = \; \{2, 4, 6\}\\
        \Fix(\sigma) \; = \; \{5,11,12\} &&
\label{eq.example.2.cycle.classification}
\end{subeqnarray}
Thus, the statistics $\cpeak, \cval, \cdrise, \cdfall$ and $\fix$ are
simply the cardinalities of these sets.

For the record classification, we write $\sigma$ as a word,
$\sigma = 7\: 1\: 9\: 2\: 5\: 4\: 8\: 6\: 10\: 3\: 11\: 12\: 14\: 13$.
The record and antirecord positions are therefore
$\Rec(\sigma) = \{1, 3, 9, 11, 12, 13\}$
and $\Arec(\sigma) = \{2,4, 10, 11, 12, 14\}$.
The full record classification is
\begin{subeqnarray}
        {\rm Erec}(\sigma) \;=\; \{1,3,9,13\} &&
        {\rm Earec}(\sigma) \;=\; \{2,4,10,14\} \\
        {\rm Rar}(\sigma) \;=\; \{11,12\} &&
        {\rm Nrar}(\sigma) \;=\; \{5,6,7,8\}
\label{eq.example.2.record.classification}
\end{subeqnarray}

Finally, the record-and-cycle classification gives us
\begin{subeqnarray}
        {\rm Eareccpeak}(\sigma) \;=\; \{10,14\} &&
        {\rm Nrcpeak}(\sigma) \;=\; \{8\} \\
        {\rm Ereccval}(\sigma) \;=\; \{1,3,13\} &&
        {\rm Nrcval}(\sigma) \;=\; \emptyset \\
        {\rm Erecdrise}(\sigma) \;=\; \{9\} &&
        {\rm Nrcdrise}(\sigma) \;=\; \{7\} \\
        {\rm Earecdfall}(\sigma) \;=\; \{2,4\} &&
        {\rm Nrcdfall}(\sigma) \;=\; \{6\} \\
        {\rm Rar}(\sigma) \;=\; \{11,12\} &&
        {\rm Nrfix}(\sigma) \; = \; \{5\}
\label{eq.example.2.record.and.cycle.classification}
\end{subeqnarray}

We leave the variant record classification and 
the variant record-and-cycle classification
as an exercise for the reader.

%{\bf CHECK THIS!!!!!}

\subsection{Permutations: Statement of conjecture}
\label{subsec.intro.perms}

The polynomial $\widehat{Q}_n$
was defined in \cite[Equation~(2.29)]{Sokal-Zeng_masterpoly}
\begin{eqnarray}
	& &
	\widehat{Q}_n(x_1,x_2, y_1, y_2,u_1, u_2, v_1, v_2, \mathbf{w}, \lambda)
        \;=\;
	\nonumber\\[4mm]
	 & &\qquad\qquad 
	 \sum_{\sigma \in \fS_n}
        x_1^{\eareccpeak(\sigma)} x_2^{\eareccdfall(\sigma)}
        y_1^{\ereccval(\sigma)} y_2^{\ereccdrise(\sigma)}
	\:\times
	\qquad\qquad
	\nonumber\\[-1mm]
   & & \qquad\qquad\qquad\:
   u_1^{\nrcpeak(\sigma)} u_2^{\nrcdfall(\sigma)}
   v_1^{\nrcval(\sigma)} v_2^{\nrcdrise(\sigma)}
        \mathbf{w}^{\boldsymbol{\fix}(\sigma)} \lambda^{\cyc(\sigma)}
\label{def.poly.conjecture}
\end{eqnarray}
where $\mathbf{w}^{\boldsymbol{\fix}(\sigma)}$ as defined in
\cite[Equation~(2.22)]{Sokal-Zeng_masterpoly} is
\begin{equation}
        \mathbf{w}^{\boldsymbol{\fix}(\sigma)} \; =\;  \prod_{i\in \Fix} w_{\psnest(i,\sigma)}.
\end{equation}

Sokal and Zeng stated the following conjecture in their paper:
\begin{conj}[{{\cite[Conjecture~2.3]{Sokal-Zeng_masterpoly}}}]
The ordinary generating function of the polynomials $\widehat{Q}_n$
specialised to $v_1 = y_1$ has the J-type continued fraction
\begin{eqnarray}
        \sum_{n=0}^\infty \widehat{Q}_n(x_1,x_2,y_1,y_2,u_1,u_2,y_1,v_2,\mathbf{w},\lambda)
        \: t^n
        \;=\; \qquad\qquad\qquad\qquad\qquad\qquad\qquad\qquad
        \nonumber \\
	\quad  \Scale[0.9]{\cfrac{1}{1 -  \lambda w_0 t - \cfrac{\lambda x_1 y_1 t^2}{1 - (x_2 \!+\! y_2 \!+\! \lambda w_1)t -  \cfrac{(\lambda \!+\! 1)(x_1 \!+\! u_1) y_1 t^2 }{1 -(x_2 \!+\! y_2 \!+\! u_2 \!+\! v_2 \!+\! \lambda w_2)t - \cfrac{(\lambda \!+\! 2)(x_1 \!+\! 2u_1)y_1 t^2}{1 - \cdots}}}}}
        \nonumber \\[1mm]
    \label{eq.conj.conj.SZ}
        \end{eqnarray}
with coefficients
\begin{subeqnarray}
    \gamma_{0} & = & \lambda w_0
         \\[1mm]
    \gamma_{n}   & = & [x_2 \!+\! (n-1)u_2 ] + [y_2 \!+\! (n-1)v_2 ] + \lambda w_n   \qquad\hbox{for $n \ge 1$}
         \\[1mm]
        \beta_n  & = &   (\lambda+n-1)[x_1 \!+\! (n-1)u_1 ] y_1
  \label{def.weights.conj.conj.SZ}
 \end{subeqnarray}
\label{conj.conj.SZ}
\end{conj}

Sokal and Zeng \cite[Theorem~2.4]{Sokal-Zeng_masterpoly}
proved this continued fraction subject to the further specialisation $v_2 = y_2$
using the Biane bijection. Here we will prove the full conjecture
by using the Foata--Zeilberger bijection, suitably reinterpreted.

In Section~\ref{sec.permutations}, 
we will see that this conjecture is a special case of a 
more general J-fraction
involving five families of infinitely many indeterminates
and one additional variable.
We will prove these results in Section~\ref{sec.permutations.proofs}.

\subsection{Genocchi, median Genocchi numbers, and D-permutations}
\label{subsec.intro.Dperm.def}

We will follow the terminology in \cite{Deb-Sokal}.

The Genocchi numbers \cite[A110501]{OEIS}
\be
   (g_n)_{n \ge 0}
   \;=\;
   1, 1, 3, 17, 155, 2073, 38227, 929569, 28820619, 1109652905,
   %% 51943281731,
   \ldots
\ee
%% $(g_n)_{n \ge 0}$ [cf.\ \reff{eq.genocchi}]
are odd positive integers \cite{Lucas_1877,Barsky_81,Han_18}
\cite[pp.~217--218]{Foata_08}
defined by the exponential generating function
\be
   t \, \tan(t/2)
   \;=\;
   \sum_{n=0}^\infty g_n \, {t^{2n+2} \over (2n+2)!}
   \;.
\ee
%They are therefore rescaled versions of the augmented tangent numbers:
%\be
%   g_n  \;=\; 4^{-n} \, (n+1) \, E_{2n+1}
%        \;=\; 2^{-(2n+1)} \, E^\sharp_{2n+1}
%   \;.
% \label{eq.gn.augmented_tangent}
%\ee

The median Genocchi numbers (or Genocchi medians for short) 
\cite[A005439]{OEIS}
\begin{equation}
        (h_n)_{n \ge 0}
   \;=\;
   1, 1, 2, 8, 56, 608, 9440, 198272, 5410688, 186043904,
   %% 7867739648,
   \ldots
   \;.
\end{equation}
are defined by \cite[p.~63]{Han_99b}
\be
   h_n  \;=\;  \sum_{i=0}^{n-1} (-1)^i \, \binom{n}{2i+1} \, g_{n-1-i}
   \;.
 \label{eq.hn.binomgn}
\ee
See \cite[Sections~2.5,~2.6]{Deb-Sokal} and references therein 
for continued fractions associated to the Genocchi and median Genocchi numbers.

The median Genocchi numbers enumerate a class of permutations called D-permutations (short for Dumont-like permutations),
they were introduced by Lazar and Wachs in \cite{Lazar_22,Lazar_20}.
A permutation of $[2n]$ is called a D-permutation in case
$2k-1 \le \sigma(2k-1)$ and $2k \ge \sigma(2k)$ for all $k$,
i.e., it contains no even excedances and no odd anti-excedances.
Let us say also that a permutation is an
{\em e-semiderangement}\/ (resp.\ {\em o-semiderangement}\/)
in case it contains no even (resp.~odd) fixed points;
it is a {\em derangement}\/ in case it contains no fixed points at all.
A D-permutation that is also an
e-semiderangement (resp.\ o-semiderangement, derangement)
will be called a \textbfit{D-e-semiderangement}
(resp.\ \textbfit{D-o-semiderangement}, \textbfit{D-derangement}).
A D-permutation that contains exactly one cycle is called a \textbfit{D-cycle}. Notice that a D-cycle is also a D-derangement.
Let $\dperm_{2n}$
(resp.~$\dperm^{\rm e}_{2n}, \dperm^{\rm o}_{2n}, \dperm^{\rm eo}_{2n}, \dcycle_{2n}$)
denote the set of all D-permutations
(resp.\ D-e-semiderangements, D-o-semiderangements, D-derangements, D-cycles) of $[2n]$.
For instance,
\begin{subeqnarray}
   \dperm_2  & = & \{ 12 ,\, 21^{\rm eo} \}  \\[1mm]
   \dperm_4  & = & \{ 1234 ,\, 1243 ,\, 2134 ,\, 2143^{\rm eo} ,\,
                     3142^{\rm eo} ,\, 3241^{\rm o} ,\, 4132^{\rm e} ,\, 4231 \}\\[1mm]
\dcycle_2 & = & \{21\}\\[1mm]
\dcycle_4 & = & \{3142\}
%% For a D-permutation of [4], \sigma(2) must be 1 or 2,
%%    and \sigma(3) must be 3 or 4; for each such combination,
%%    there are then two choices for the other two entries.
\end{subeqnarray}
where $^{\rm e}$ denotes e-semiderangements that are not derangements,
$^{\rm o}$ denotes o-semi\-de\-range\-ments that are not derangements,
and $^{\rm eo}$ denotes derangements.
Additionally, our second running example~\eqref{eq.running.example.2} 
is also an example of a D-permutation 
where $n=7$.

It is known \cite{Dumont_74,Dumont_94,Lazar_22,Lazar_20, Deb-Sokal}
that
\begin{subeqnarray}
   |\dperm_{2n}|    & = &  h_{n+1}  \\[0.5mm]
   |\dperm^{\rm e}_{2n}|  \;=\; |\dperm^{\rm o}_{2n}|  & = &  g_n \\[1mm]
   |\dperm^{\rm eo}_{2n}|  & = &  h_n\\[0.5mm]
    |\dcycle_{2n}| & = & g_{n-1}
\end{subeqnarray}

For a D-permutation $\sigma$,
the cycle classification of a non-fixed-point index $i$
is equivalent to recording the parities of $\sinv(i)$ and $i$ 
\cite[eq.~(2.47)]{Deb-Sokal}:

\be
   \begin{aligned}
   & \bullet\; \hbox{{\em cycle peak}\/: $\sigma^{-1}(i) < i > \sigma(i)$
            $\implies$ $\sinv(i)$ odd, $i$ even} \\[2mm]
   & \bullet\; \hbox{{\em cycle valley}\/: $\sigma^{-1}(i) > i < \sigma(i)$
            $\implies$ $\sinv(i)$ even, $i$ odd} \\[2mm]
   & \bullet\; \hbox{{\em cycle double rise}\/: $\sigma^{-1}(i) < i < \sigma(i)$
            $\implies$ $\sinv(i)$ odd, $i$ odd} \\[2mm]
   & \bullet\; \hbox{{\em cycle double fall}\/:  $\sigma^{-1}(i) > i > \sigma(i)$
            $\implies$ $\sinv(i)$ even, $i$ even} \hspace*{1.7cm} \\
   \end{aligned}
 \label{eq.parities.1}
\ee

For a fixed point $i$, we will need to explicitly record the parity of $i$.
Thus, using Equation~\eqref{eq.parities.1},
it is clear that each even index $i$ belongs to exactly one of the following three types: 
cpeak, cdfall, evenfix;
and each odd index $i$ belongs to exactly one of the following three types: 
cval, cdrise, oddfix.
Combining this with the record-classification we get the
\emph{parity-refined record-and-cycle classification} for D-permutations:
we see that each even index $i$ belongs to exactly one of the 6 types
\begin{quote}
   eareccpeak, nrcpeak, eareccdfall, nrcdfall, evenrar, evennrfix,
\end{quote}
while each odd index $i$ belongs to exactly one of the 6 types
\begin{quote}
   ereccval, nrcval, ereccdrise, nrcdrise, oddrar, oddnrfix.
\end{quote}
as was claimed in Section~\ref{subsec.intro.stats}.
Similarly, we also get the  
\emph{variant parity-refined record-and-cycle classification}
for D-permutations:
we see that each even index $i$ belongs to exactly one of the 6 types
\begin{quote}
   $\ereccpeak'$, $\nrcpeak'$, $\eareccdfall'$, $\nrcdfall'$, $\evenrar'$, $\evennrfix'$,
\end{quote}
while each odd index $i$ belongs to exactly one of the 6 types
\begin{quote}
   $\eareccval'$, $\nrcval'$, $\ereccdrise'$, $\nrcdrise'$, $\oddrar'$, $\oddnrfix'$.
\end{quote}

\subsection{D-Permutations: Statements of conjectures}
\label{subsec.intro.Dperm.conj}

%We are now ready to prove \cite[Conjecture~12]{Randrianarivony_96b}
%by comparing Theorem~\ref{thm.DS.pqgen.variant} and
%\cite[Theorem~3.13]{Deb-Sokal}.

In \cite{Randrianarivony_96b},
Randrianarivony and Zeng introduced two sequences of polynomials
for D-o-semiderangements 
\cite[eq.~(3.3)]{Randrianarivony_96b}
\be
   R_n(x,y,\bar{x},\bar{y})
   \;=\;
   \sum_{\sigma \in \dperm^{\rm o}_{2n}}
      x^{{\rm lema}(\sigma)}
      y^{{\rm romi}(\sigma)}
      \bar{x}^{{\rm fix}(\sigma)}
      \bar{y}^{{\rm remi}(\sigma)}
\label{eq.def.Rn.RZ}
\ee
and \cite[p.~9]{Randrianarivony_96b}
\be
   G_n(x,y,\bar{x},\bar{y})
   \;=\;
   \sum_{\sigma \in \dperm^{\rm o}_{2n}}
      x^{{\rm comi}(\sigma)}
      y^{{\rm lema}(\sigma)}
      \bar{x}^{{\rm cemi}(\sigma)}
      \bar{y}^{{\rm remi}(\sigma)}
\label{eq.def.Gn.RZ}
\ee
where the statistics ${\rm lema}, {\rm romi}, {\rm remi}, 
{\rm comi}, {\rm cemi}$ are defined as follows:
\begin{itemize}
\item ${\rm lema}$ -- left-to-right maxima whose value is even,
\item ${\rm romi}$ -- right-to-left minima whose value is odd,
\item ${\rm remi}$ -- right-to-left minima whose value is even,
\item ${\rm comi}$ -- odd cycle minima,
\item ${\rm cemi}$ -- even cycle minima;
\end{itemize}
for a permutation $\sigma$, ${\rm lema}(\sigma)$ denotes the number of 
left-to-right maxima (i.e. record) whose value $\sigma(i)$ is even, etc.
See \cite[p.~2]{Randrianarivony_96b} for a full description of these statistics.

In their paper, Randrianarivony and Zeng stated the following conjecture which we shall prove:

\begin{conj}[{{\cite[Conjecture~12]{Randrianarivony_96b}}}]
For $n\geq 1$ we have $R_n(x,y,\bar{x},\bar{y})\;=\;G_n(x,y,\bar{x},\bar{y})$.
\label{conj.RZ}
\end{conj}
Using \cite[Proposition~10]{Randrianarivony_96b}, 
Conjecture~\ref{conj.RZ} can be equivalently stated as

\addtocounter{thm}{-1}
\begin{conj}\hspace*{-3mm}${}^{\bf\prime}$
The ordinary generating function of the polynomials 
$G_n(x,y,\bar{x},\bar{y})$ defined in \eqref{eq.def.Rn.RZ}
has the S-type continued fraction
\begin{equation}
	1 + \sum_{n=1}^{\infty} G_n(x,y,\bar{x},\bar{y}) t^n \; = \; 
	\cfrac{1}{1-\cfrac{xy t}{1-\cfrac{1(\bar{x} +\bar{y})t}{1-\cfrac{(x+1)(y+1)t}{1-\cfrac{2(\bar{x} + \bar{y} + 1) t}{1-\cfrac{(x+2)(y+2)t}{1-\cfrac{3(\bar{x}+\bar{y}+2)t}{\cdots}}}}}}}
\end{equation}
\end{conj}
\addtocounter{thm}{0}

It is worthwhile to translate the statistics of Randrianarivony and Zeng
to the statistics we introduced in Section~\ref{subsec.intro.stats}.
This was already done for the statistics involved in the
polynomials $R_n(x,y,\bar{x},\bar{y})$
in \cite[Remark p.~36]{Deb-Sokal};
the following statistics are identical for D-o-semiderangements:
\begin{itemize}
        \item ${\rm lema} = \ereccpeak'$
        \item ${\rm remi} = \eareccdfall'$
        \item ${\rm romi} = \eareccval'$
        \item $\fix = \evennrfix$.
\end{itemize}
It remains to translate the statistics ${\rm cemi}$ and ${\rm comi}$.

Notice that the smallest element $i$ of a cycle with at least two elements
must be a cycle valley, and hence must be odd.
As D-o-semiderangements do not have any odd fixed points,
the cycle minima for fixed points are necessarily even.
On the other hand, the smallest element $i$ of a cycle
with at least two elements must be a cycle valley, and hence must be odd.
Thus, $i$ is an even cycle minima if and only if it is an even fixed point,
and $i$ is an odd cycle minima if and only if it is the minimum valley of cycle
with at least two elements. Thus, we have shown that
\begin{itemize}
        \item ${\rm cemi} = \evennrfix$
        \item ${\rm comi} = \minval$.
\end{itemize}
This shows that
\begin{equation}
G_n(x,y,\bar{x},\bar{y})
   \;=\;
   \sum_{\sigma \in \dperm^{\rm o}_{2n}}
	x^{\minval(\sigma)}
      y^{{\ereccpeak'}(\sigma)}
      \bar{x}^{{\evennrfix}(\sigma)}
      \bar{y}^{{\eareccdfall'}(\sigma)}.
\label{eq.Gn.stats}
\end{equation}

We will also look at the following sequence of polynomials
introduced by Deb and Sokal for all D-permutations \cite[Equation~(4.2)]{Deb-Sokal}:
\begin{eqnarray}
   & &
   \Scale[0.97]{
   \widehat{P}_n(x_1,x_2,y_1,y_2,u_1,u_2,v_1,v_2,\we,\wo,\ze,\zo,\lambda)}
   \;=\;
       \nonumber \\[4mm]
   & & \qquad\qquad
   \sum_{\sigma \in \dperm_{2n}}
   x_1^{\eareccpeak(\sigma)} x_2^{\eareccdfall(\sigma)}
   y_1^{\ereccval(\sigma)} y_2^{\ereccdrise(\sigma)}
   \:\times
       \qquad\qquad
       \nonumber \\[-1mm]
   & & \qquad\qquad\qquad\:
   u_1^{\nrcpeak(\sigma)} u_2^{\nrcdfall(\sigma)}
   v_1^{\nrcval(\sigma)} v_2^{\nrcdrise(\sigma)}
   \:\times
       \qquad\qquad
       \nonumber \\[3mm]
   & & \qquad\qquad\qquad\:
   \we^{\evennrfix(\sigma)} \wo^{\oddnrfix(\sigma)}
   \ze^{\evenrar(\sigma)} \zo^{\oddrar(\sigma)}
   \, \lambda^{\cyc(\sigma)}\;.
 \label{def.Pnhat.dperm}
\end{eqnarray}
We will prove the following conjectured Thron-type continued fraction
involving the polynomials \eqref{def.Pnhat.dperm}:

\begin{conj}[{{\cite[Conjecture~4.1]{Deb-Sokal}}}]
The ordinary generating function of the polynomials \eqref{def.Pnhat.dperm}
specialised to $v_1 = y_1$ has the T-type continued fraction
\begin{eqnarray}
   & & \hspace*{-12mm}
   \sum_{n=0}^\infty
   \widehat{P}_n(x_1,x_2,y_1,y_2,u_1,u_2,y_1,v_2,\we,\wo,\ze,\zo,\lambda) \: t^n
   \;=\;
       \nonumber \\
   & &
   \cfrac{1}{1 - \lambda^2 \ze \zo  \,t - \cfrac{\lambda x_1 y_1  \,t}{1 -  \cfrac{(x_2\!+\!\lambda\we)(y_2\!+\!\lambda\wo) \,t}{1 - \cfrac{(\lambda+1)(x_1\!+\!u_1) y_1 \,t}{1 - \cfrac{(x_2\!+\!u_2\!+\!\lambda\we)(y_2\!+\!v_2\!+\!\lambda\wo)  \,t}{1 - \cfrac{(\lambda+2)(x_1\!+\!2u_1) y_1 \,t}{1 - \cfrac{(x_2\!+\!2u_2\!+\!\lambda\we)(y_2\!+\!2v_2\!+\!\lambda\wo)  \,t}{1 - \cdots}}}}}}}
       \nonumber \\[1mm]
   \label{eq.conj.DS}
\end{eqnarray}
with coefficients
\begin{subeqnarray}
   \alpha_{2k-1} & = & (\lambda + k-1) \: [x_1 + (k-1) u_1] \: y_1
        \\[1mm]
   \alpha_{2k}   & = & [x_2 + (k-1) u_2 + \lambda\we] \: [y_2 + (k-1)v_2 + \lambda\wo]
        \\[1mm]
   \delta_1  & = &   \lambda^2 \ze \zo   \\[1mm]
   \delta_n  & = &   0    \qquad\hbox{for $n \ge 2$}
   \label{eq.conj.DS.weights}
\end{subeqnarray}
\label{conj.DS}
\end{conj}

If this continued fraction is further
specialised to $v_2 = y_2$,
the resulting continued fraction is the same as the 
second T-fraction for D-permutations of Deb and Sokal \cite[Theorem~4.2]{Deb-Sokal}
under the specialisation $\widehat{v}_2 = \widehat{y}_2$ 
and then identifying $\widehat{y}_2$ with $y_2$.
Deb and Sokal's second T-fraction was proved using a Biane-like bijection.
Here we will prove the full conjecture using a Foata--Zeilberger-like bijection,
suitably reinterpreted.

In Section~\ref{sec.dperm},
we will see that both of these conjectures 
are special cases of
general T-fractions
involving six families of infinitely many indeterminates
and one additional variable.
We will prove these results in Section~\ref{sec.dperm.proofs}.

\subsection{Overview of proof for results on permutations}
\label{subsec.intro.proof}

We now summarise our proof for permutations (described in Section~\ref{sec.permutations.proofs}).
Our results for D-permutations are also obtained by using 
very similar ideas (described in Section~\ref{sec.dperm.proofs}).
We first provide an overview of the Foata--Zeilberger bijection,
and then briefly mention how we reinterpet it 
to obtain the count of cycles in a permutation.

Let $\sigma \in\Sym_n$ be a permutation on $n$ letters.
This permutation $\sigma$ partitions the set $[n]$
into excedance indices ($F=\{i\in [n]: \sigma(i)>i\}$),
anti-excedance indices ($G=\{i\in [n]: \sigma(i)<i \}$), and fixed points ($H$).
Similarly, $\sigma$ also partitions $[n]$ into
excedance values ($F'=\{i\in [n]: i>\sinv(i)\}$),
anti-excedance values ($G' = \{i\in [n]: i<\sinv(i)\}$),
and fixed points.
Clearly, $\sigma \restrict F \colon\: F \to F'$,
$\sigma \restrict G \colon\: G \to G'$,
and $\sigma \restrict H \colon\: H \to H$
are bijections, 
and the permutation $\sigma$ can be obtained from the following data:
\begin{itemize}
\item Two partitions of the set $[n] \;=\; F \,\cup\, G \,\cup\, H \;=\; F' \,\cup\, G' \,\cup\, H$.
\item The two subwords of $\sigma$: $\sigma(x_1)\sigma(x_2) \ldots \sigma(x_m)$
	and $\sigma(y_1)\sigma(y_2) \ldots \sigma(y_l)$, where $G = \{x_1<x_2<\ldots<x_m\}$
		and $F = \{y_1<y_2<\ldots<y_l\}$.
\end{itemize}

In their construction, Foata and Zeilberger \cite{Foata_90}
use these data to describe a bijection from $\Sym_n$
to a set of labelled Motzkin paths of length $n$ 
(these will be completely defined in Section~\ref{sec.prelimproofs}).
One then uses Flajolet's theorem \cite{Flajolet_80} 
to obtain continued fractions from this bijection
while keeping track of various simultaneous permutation statistics.

The Foata--Zeilberger bijection consists of the following steps 
(following \cite{Sokal-Zeng_masterpoly}):
\begin{itemize}
	\item Step 1: A 3-coloured Motzkin path $\overline{\omega}$ 
		is constructed from $\sigma$
		(each level step is one of three colours).
		The path $\overline{\omega}$ is fully determined by the sets 
		$F,F',G,G',H$.
		%The description of $\omega$ completely depends on the sets $F,F',G,G',H$.

	\item Step 2: The labels $\xi$ associated to $\overline{\omega}$ 
		are constructed from $\sigma$.
		It turns out that the labels depend on 
		the maps
		$\sigma \restrict F \colon\: F \to F'$ 
		and $\sigma \restrict G \colon\: G \to G'$
		and the set $H$, separately.

%	Steps 1 and 2 describe the forward bijection. 
	
\item Step 3: This step describes the construction of the inverse map
	$(\overline{\omega}, \xi)\mapsto \sigma$. 
	This step is broken down as follows:
	\begin{itemize}
		\item Step 3(a): The sets $F,F',G,G',H$ are read off from the 3-coloured 
			Motzkin path $\overline{\omega}$.

		\item Step 3(b): This description is the crucial part of the construction (at least for our purposes).
		We use {\em inversion tables} to 
		construct the words $\sigma(x_1)\sigma(x_2) \ldots \sigma(x_m)$ 
		and $\sigma(y_1)\sigma(y_2) \ldots \sigma(y_l)$;
		the former is constructed using a ``right-to-left'' inversion table
		and the latter is constructed using a ``left-to-right'' inversion table.
	\end{itemize}
\end{itemize}

It is, a priori, unclear how one might be able to track
the number of cycles of $\sigma$ in this construction.
We resolve this issue by reinterpreting Step~3(b).
We describe a ``history'' of this construction using Laguerre digraphs.

Recall that a Laguerre digraph of size $n$ is a directed graph
where each vertex has a distinct label from the label set $[n]$
and has indegree $0$ or $1$ and outdegree $0$ or $1$.
It follows that the connected components in a Laguerre digraph are 
either directed paths or directed cycles.
Clearly, any subgraph of a Laguerre digraph is also a Laguerre digraph.
A permutation $\sigma$ in cycle notation 
is equivalent to a Laguerre digraph $L$ with no paths
(\cite[pp.~22-23]{Stanley_12}).
The directed edges of $L$ are precisely $u\to \sigma(u)$. 
We will interpret Step~3(b) of the Foata--Zeilberger construction
as building up a permutation as a sequence of Laguerre digraphs,
starting from the empty digraph in which all vertices are isolated
(i.e., have no adjacent edges), 
and ending with the digraph of the permutation $\sigma$ 
in which there are no paths.

For a subset $S\subseteq [n]$,
we let $\left. L \right|_{S}$ denote the subgraph of $L$,
containing the same set of vertices $[n]$, but only the edges $u\to \sigma(u)$,
with $u\in S$ (we are allowed to have $\sigma(u)\not\in S$).
Let $u_1,\ldots, u_n$ be a rewriting of $[n]$.
We consider the ``history''
$\laguerre{\emptyset}\subset \laguerre{\{u_1\}} \subset \laguerre{\{u_1, u_2\}}
\subset \ldots \subset \laguerre{\{u_1,\ldots, u_n\}} = L$
as a process of building up the permutation $\sigma$
by successively considering the status of vertices $u_1, u_2, \ldots, u_n$.
At step $u$, we construct the edge $u\to \sigma(u)$.
Thus, at each step we insert a new edge into the digraph,
and at the end of this process,
the resulting digraph obtained is the digraph of $\sigma$.

The crucial part of our construction is that 
we use a very special order $u_1,\ldots, u_n$: 
we first go through $H$ in increasing order (we call this stage (a)),
we then go through $G$ in increasing order (stage (b)), 
finally we go through $F$ but in decreasing order (stage (c)).
This total order is suggested by the inversion tables.
On building up the permutation $\sigma$ using this history, 
we will see that the cycles can only be obtained during 
stage (c) and we can now count the number of cycles.

Our total order on $[n]$ only depends on the sets $F,G,H$,
and hence, only on the \mbox{3-coloured} Motzkin path $\overline{\omega}$ 
and not on the full description of the labels $\xi$.
This is crucial for our proof to work.

\subsection{Outline of paper}
\label{subsec.intro.outline}

The plan of this paper is as follows:
In Section~\ref{subsec.statistics.2}
we introduce some more permutation statistics
as they play a central role in our results.
We then state our results for permutations
in Section~\ref{sec.permutations};
this will include 
the continued fraction
\cite[Conjecture~2.3]{Sokal-Zeng_masterpoly}
along with its generalisations.
Next, we state our results for D-permutations
in Section~\ref{sec.dperm};
this will include 
the continued fractions
\cite[Conjecture~12]{Randrianarivony_96b}
and 
\cite[Conjecture~4.1]{Deb-Sokal}.
In Section~\ref{sec.prelimproofs} we recall how continued fractions
can be proven by bijection to labelled Dyck, Motzkin or Schr\"oder paths.
In Section~\ref{sec.permutations.proofs} we prove our continued fractions
for permutations
by reinterpreting Sokal and Zeng's 
variant of the Foata--Zeilberger bijection \cite[Section~6.1]{Sokal-Zeng_masterpoly}
using Laguerre digraphs.
In Section~\ref{sec.dperm.proofs} we prove our continued fractions
for D-permutations
by reinterpreting two bijections of Deb and Sokal
\cite[Sections~6.1-6.3 and~6.5]{Deb-Sokal}
using Laguerre digraphs.
We conclude (Section~\ref{sec.final})
with some brief remarks on our work.

\section{Permutation statistics: Crossings and nestings}
     \label{subsec.statistics.2}

We now define (following \cite{Sokal-Zeng_masterpoly, Deb-Sokal})
some permutation statistics that count
\textbfit{crossings} and \textbfit{nestings}. 
%We follow the conventions in \cite{Deb-Sokal}.

First we associate to each permutation $\sigma \in \fS_N$
a pictorial representation 
by placing vertices $1,2,\ldots,N$ along a horizontal axis
and then drawing an arc from $i$ to $\sigma(i)$
above (resp.\ below) the horizontal axis
in case $\sigma(i) > i$ [resp.\ $\sigma(i) < i$];
if $\sigma(i) = i$ we do not draw any arc.
This idea was first introduced by Corteel
in \cite{Corteel_07}.
See Figures~\ref{fig.pictorial} and \ref{fig.pictorial.2}
for our two running examples.

\begin{figure}[t]
\centering
\vspace*{4cm}
\begin{picture}(60,20)(120, -65)
\setlength{\unitlength}{2mm}
\linethickness{.5mm}
\put(-2,0){\line(1,0){54}}
\put(0,0){\circle*{1,3}}\put(0,0){\makebox(0,-6)[c]{\small 1}}
\put(5,0){\circle*{1,3}}\put(5,0){\makebox(0,-6)[c]{\small 2}}
\put(10,0){\circle*{1,3}}\put(10,0){\makebox(0,-6)[c]{\small 3}}
\put(15,0){\circle*{1,3}}\put(15,0){\makebox(0,-6)[c]{\small 4}}
\put(20,0){\circle*{1,3}}\put(20,0){\makebox(0,-6)[c]{\small 5}}
\put(25,0){\circle*{1,3}}\put(25,0){\makebox(0,-6)[c]{\small 6}}
\put(30,0){\circle*{1,3}}\put(30,0){\makebox(0,-6)[c]{\small 7}}
\put(35,0){ \circle*{1,3}}\put(36,0){\makebox(0,-6)[c]{\small 8}}
\put(40,0){\circle*{1,3}}\put(40,0){\makebox(0,-6)[c]{\small 9}}
\put(45,0){\circle*{1,3}}\put(45,0){\makebox(0,-6)[c]{\small 10}}
\put(50,0){\circle*{1,3}}\put(50,0){\makebox(0,-6)[c]{\small 11}}
\red{\qbezier(0,0)(20,14)(40,0)
\qbezier(40,0)(42.5,6)(45,0)}
\blue{\qbezier(4.5,0)(7,5)(9.5,0)
\qbezier(9.5,0)(18.5,10)(29.5,0)
}
\blue{\qbezier(18.5,0)(21,5)(23.8,0)
\qbezier(23.8,0)(36.5,12)(48.8,0)
\qbezier(18.5,0)(24.5,-10)(28.8,0)
}
%%%%%%%%%%%%%%%%%%%%%%%%%%
\blue{
\qbezier(3,0)(25,-20)(48,0)
}
\red{\qbezier(-2.5,0)(22,-20)(42.5,0)}
\end{picture}
\caption{
   Pictorial representation of the permutation
   $\sigma = 9\,\, 3\,\, 7\,\, 4\,\, 6\,\, 11\,\, 5\,\, 8\,\, 10\,\, 1\,\, 2
	   = (1,9,10)\,(2,3,7,5,6,11)\,(4)\,(8) \in \Sym_{11}$.
%  {\bf NEED TO CHANGE THIS PICTURE!!!!!!!!!!}
 \label{fig.pictorial}
 \vspace*{8mm}
}
\end{figure}

\begin{figure}[t]
\centering
\vspace*{4cm}
\begin{picture}(100,0)(140, -45)
\setlength{\unitlength}{2mm}
\linethickness{.5mm}
\put(-2,0){\line(1,0){69}}
\put(0,0){\circle*{1,3}}\put(0,0){\makebox(0,-6)[c]{\small 1}}
\put(5,0){\circle*{1,3}}\put(5,0){\makebox(0,-6)[c]{\small 2}}
\put(10,0){\circle*{1,3}}\put(10,0){\makebox(0,-6)[c]{\small 3}}
\put(15,0){\circle*{1,3}}\put(15,0){\makebox(0,-6)[c]{\small 4}}
\put(20,0){\circle*{1,3}}\put(20,0){\makebox(0,-6)[c]{\small 5}}
\put(25,0){\circle*{1,3}}\put(25,0){\makebox(0,-6)[c]{\small 6}}
\put(30,0){\circle*{1,3}}\put(30,0){\makebox(0,-6)[c]{\small 7}}
\put(35,0){ \circle*{1,3}}\put(35,0){\makebox(0,-6)[c]{\small 8}}
\put(40,0){\circle*{1,3}}\put(40,0){\makebox(0,-6)[c]{\small 9}}
\put(45,0){\circle*{1,3}}\put(45,0){\makebox(0,-6)[c]{\small 10}}
\put(50,0){\circle*{1,3}}\put(50,0){\makebox(0,-6)[c]{\small 11}}
\put(55,0){\circle*{1,3}}\put(55,0){\makebox(0,-6)[c]{\small 12}}
\put(60,0){\circle*{1,3}}\put(60,0){\makebox(0,-6)[c]{\small 13}}
\put(65,0){\circle*{1,3}}\put(65,0){\makebox(0,-6)[c]{\small 14}}
\green{\qbezier(0,0)(15,16)(30,0)
\qbezier(30,0)(33,6)(36,0)
\qbezier(36,0)(30.5,-8)(25,0)
\qbezier(25,0)(20,-8)(15,0)
\qbezier(15,0)(10,-8)(5,0)
\qbezier(5,0)(2.5,-6)(0,0)
}
\red{\qbezier(9,0)(24,16)(39,0)
\qbezier(39,0)(41.5,6)(44,0)
\qbezier(44,0)(26.5,-18)(9,0)
}
\blue{\qbezier(58.5,0)(61,6)(63.5,0)
\qbezier(63.5,0)(61,-6)(58.5,0)
}
\end{picture}
\caption{
   Pictorial representation of the permutation
   $\sigma = 7\, 1\, 9\, 2\, 5\, 4\, 8\, 6\, 10\, 3\, 11\, 12\, 14\, 13\,
           = (1,7,8,6,4,2)\,(3,9,10)\,(5)\,(11)\,(12)\,(13,14) \in \Sym_{14}$.
   This $\sigma$ is a D-permutation.
 \label{fig.pictorial.2}
 \vspace*{7mm}
}
\end{figure}

Each vertex thus has either
out-degree = in-degree = 1 (if it is not a fixed point) or
out-degree = in-degree = 0 (if it is a fixed point).
Of course, the arrows on the arcs are redundant,
because the arrow on an arc above (resp.\ below) the axis
always points to the right (resp.\ left);
we therefore omit the arrows for simplicity.

We then say that a quadruplet $i < j < k < l$ forms an
\begin{itemize}
   \item {\em upper crossing}\/ (ucross) if $k = \sigma(i)$ and $l = \sigma(j)$;
   \item {\em lower crossing}\/ (lcross) if $i = \sigma(k)$ and $j = \sigma(l)$;
   \item {\em upper nesting}\/  (unest)  if $l = \sigma(i)$ and $k = \sigma(j)$;
   \item {\em lower nesting}\/  (lnest)  if $i = \sigma(l)$ and $j = \sigma(k)$.
\end{itemize}
We also consider some ``degenerate'' cases with $j=k$,
by saying that a triplet $i < j < l$ forms an
\begin{itemize}
   \item {\em upper joining}\/ (ujoin) if $j = \sigma(i)$ and $l = \sigma(j)$
      [i.e.\ the index $j$ is a cycle double rise];
   \item {\em lower joining}\/ (ljoin) if $i = \sigma(j)$ and $j = \sigma(l)$
      [i.e.\ the index $j$ is a cycle double fall];
   \item {\em upper pseudo-nesting}\/ (upsnest)
      if $l = \sigma(i)$ and $j = \sigma(j)$;
   \item {\em lower pseudo-nesting}\/ (lpsnest)
      if $i = \sigma(l)$ and $j = \sigma(j)$.
\end{itemize}
These are clearly degenerate cases of crossings and nestings, respectively.
%See Figure~\ref{fig.crossnest}.
Note that $\upsnest(\sigma) = \lpsnest(\sigma)$ for all $\sigma$,
since for each fixed point~$j$,
the number of pairs $(i,l)$ with $i < j < l$ such that $l = \sigma(i)$
has to equal the number of such pairs with $i = \sigma(l)$;
we therefore write these two statistics simply as
\be
   \psnest(\sigma) \;\eqdef\; \upsnest(\sigma) \;=\;  \lpsnest(\sigma)
   \;.
\ee
And of course $\ujoin = \cdrise$ and $\ljoin = \cdfall$.
It is clear that
\be
\psnest(\sigma) \;=\; \sum_{i\in \Fix} \psnest(i,\sigma)
\ee
where $\psnest$ was defined in \eqref{def.level}
and Fix denotes the set of all fixed points.

If $\sigma$ is a D-permutation, then its diagram has a special property:
all arrows emanating from odd (resp.~even) vertices
are upper (resp.~lower) arrows.
Otherwise put,
the leftmost (resp.~rightmost) vertex of an upper (resp.~lower) arc
is always odd (resp.~even).
It follows that in an upper crossing or nesting $i < j < k < l$,
the indices $i$ and $j$ must be odd;
and in a lower crossing or nesting $i < j < k < l$,
the indices $k$ and $l$ must be even.
Similar comments apply to upper and lower joinings and pseudo-nestings.

We can further refine the four crossing/nesting categories
by examining more closely the status of the inner index ($j$ or $k$)
whose {\em outgoing}\/ arc belongs to the crossing or nesting:
%we say that a quadruplet $i < j < k < l$ forms an
%\begin{itemize}
%   \item {\em upper crossing of type cval}\/ (ucrosscval)
%       if $k = \sigma(i)$ and $l = \sigma(j)$ and ${\sigma^{-1}(j) > j}$;
%   \item {\em upper crossing of type cdrise}\/ (ucrosscdrise)
%       \hbox{if $k = \sigma(i)$ and $l = \sigma(j)$ and ${\sigma^{-1}(j) < j}$;}
%   \item {\em lower crossing of type cpeak}\/ (lcrosscpeak)
%       if $i = \sigma(k)$ and $j = \sigma(l)$ and ${\sigma^{-1}(k) < k}$;
%   \item {\em lower crossing of type cdfall}\/ (lcrosscdfall)
%       if $i = \sigma(k)$ and $j = \sigma(l)$ and ${\sigma^{-1}(k) > k}$;
%   \item {\em upper nesting of type cval}\/  (unestcval)
%       if $l = \sigma(i)$ and $k = \sigma(j)$ and ${\sigma^{-1}(j) > j}$;
%   \item {\em upper nesting of type cdrise}\/  (unestcdrise)
%       if $l = \sigma(i)$ and $k = \sigma(j)$ and ${\sigma^{-1}(j) < j}$;
%   \item {\em lower nesting of type cpeak}\/  (lnestcpeak)
%       if $i = \sigma(l)$ and $j = \sigma(k)$ and ${\sigma^{-1}(k) < k}$;
%   \item {\em lower nesting of type cdfall}\/  (lnestcdfall)
%       if $i = \sigma(l)$ and $j = \sigma(k)$ and ${\sigma^{-1}(k) > k}$.
%\end{itemize}
\medskip
\begin{center}
\begin{tabular}{c|c|c|c|c|}
 & ucross & unest & lcross & lnest \\
\hline
$j\in \Cval$ & ucrosscval & unestcval & &\\
$j\in \Cdrise$ & ucrosscdrise  & unestcdrise & &\\
$k\in \Cpeak$ & & & lcrosscpeak & lnestcpeak\\
$k\in \Cdfall$ & & & lcrosscdfall & lnestcdfall\\
\hline
\end{tabular}
\end{center}
\medskip
%
%See Figure~\ref{fig.refined_crossnest}.
Please note that for the ``upper'' quantities
the distinguished index
(i.e.\ the one for which we examine both $\sigma$ and $\sigma^{-1}$)
is in second position ($j$),
while for the ``lower'' quantities
the distinguished index is in third position ($k$).

In fact, a central role in our work will be played
(just as in \cite{Sokal-Zeng_masterpoly, Deb-Sokal})
by a refinement of these statistics:
rather than counting the {\em total}\/ numbers of quadruplets
$i < j < k < l$ that form upper (resp.~lower) crossings or nestings,
we will count the number of upper (resp.~lower) crossings or nestings
that use a particular vertex $j$ (resp.~$k$)
in second (resp.~third) position.
More precisely, we define the
\textbfit{index-refined crossing and nesting statistics}
\begin{subeqnarray}
   \ucross(j,\sigma)
   & = &
   \#\{ i<j<k<l \colon\: k = \sigma(i) \hbox{ and } l = \sigma(j) \}
         \\[2mm]
   \unest(j,\sigma)
   & = &
   \#\{ i<j<k<l \colon\: k = \sigma(j) \hbox{ and } l = \sigma(i) \}
      \\[2mm]
   \lcross(k,\sigma)
   & = &
   \#\{ i<j<k<l \colon\: i = \sigma(k) \hbox{ and } j = \sigma(l) \}
         \\[2mm]
   \lnest(k,\sigma)
   & = &
   \#\{ i<j<k<l \colon\: i = \sigma(l) \hbox{ and } j = \sigma(k) \}
 \label{def.ucrossnestjk}
\end{subeqnarray}

Note that $\ucross(j,\sigma)$ and $\unest(j,\sigma)$ can be nonzero
only when $j$ is an excedance
(that is, a cycle valley or a cycle double rise),
while $\lcross(k,\sigma)$ and $\lnest(k,\sigma)$ can be nonzero
only when $k$ is an anti-excedance
(that is, a cycle peak or a cycle double fall).
In a D-permutation, this means that
$\ucross(j,\sigma)$ and $\unest(j,\sigma)$ can be nonzero
only when $j$ is odd and not a fixed point,
while $\lcross(k,\sigma)$ and $\lnest(k,\sigma)$ can be nonzero
only when $k$ is even and not a fixed point.

%When $j$ is a fixed point, we also define the analogous quantity
%for pseudo-nestings:
%\be
%   \psnest(j,\sigma)
%   \;\eqdef\;
%   \# \{i < j \colon\:  \sigma(i) > j \}
%   \;=\;
%   \# \{i > j \colon\:  \sigma(i) < j \}
%   \;.
% \label{def.psnestj}
%\ee
%(Here the two expressions are equal because $\sigma$ is a bijection
% from $[1,j) \cup (j,n]$ to itself.)
%In \cite[eq.~(2.20)]{Sokal-Zeng_masterpoly}
%this quantity was called the {\em level}\/ of the fixed point $j$
%and was denoted $\lev(j,\sigma)$.

We also use 
%% (in Sections~\ref{subsec.second.master} and \ref{sec.proofs.2})
a variant of \eqref{def.ucrossnestjk}
in which the roles of second and third position are interchanged:
\begin{subeqnarray}
   \ucross'(k,\sigma)
   & = &
   \#\{ i<j<k<l \colon\: k = \sigma(i) \hbox{ and } l = \sigma(j) \}
         \\[2mm]
   \unest'(k,\sigma)
   & = &
   \#\{ i<j<k<l \colon\: k = \sigma(j) \hbox{ and } l = \sigma(i) \}
      \\[2mm]
   \lcross'(j,\sigma)
   & = &
   \#\{ i<j<k<l \colon\: i = \sigma(k) \hbox{ and } j = \sigma(l) \}
         \\[2mm]
   \lnest'(j,\sigma)
   & = &
   \#\{ i<j<k<l \colon\: i = \sigma(l) \hbox{ and } j = \sigma(k) \}
 \label{def.ucrossnestjk.prime}
\end{subeqnarray}
We remark that since nestings join the vertices
in second and third positions, we have
\begin{subeqnarray}
   \unest'(k,\sigma) & = & \unest(\sinv(k),\sigma)
      \\[2mm]
   \lnest'(j,\sigma) & = & \lnest(\sinv(j),\sigma)
 \label{eq.nestprime}
\end{subeqnarray}
Note that $\ucross'(k,\sigma)$ and $\unest'(k,\sigma)$ can be nonzero
only when $\sinv(k)$ is an excedance
(that is, when $k$ is a cycle peak or a cycle double rise),
while $\lcross'(j,\sigma)$ and $\lnest'(j,\sigma)$ can be nonzero
only when $\sinv(j)$ is an anti-excedance
(that is, $j$ is a cycle valley or a cycle double fall).
In a D-permutation, this means that
$\ucross'(k,\sigma)$ and $\unest'(k,\sigma)$ can be nonzero
only when $\sinv(k)$ is odd and not a fixed point,
while $\lcross'(j,\sigma)$ and $\lnest'(j,\sigma)$ can be nonzero
only when $\sinv(j)$ is even and not a fixed point.
We call \eqref{def.ucrossnestjk.prime} the
\textbfit{variant index-refined crossing and nesting statistics}.

We can also analogously define the statistics
$\ucrosscpeak'$, $\unestcpeak'$, $\lcrosscval'$, $\lnestcval'$,
$\lcrosscdfall'$, $\lnestcdfall'$, $\ucrosscdrise'$, $\ucrosscdrise'$.
We leave the details to the reader.

\section{Permutations: Statements of results}
\label{sec.permutations}

In this section, we state our continued fractions for permutations,
in three increasingly more general versions.
The first and most basic version (Theorem~\ref{thm.conj.SZ}
is a J-fraction in 8 variables and another family of infinitely many variables
that enumerates permutations with respect to the record-and-cycle classification
except for the segregation of cycle valleys;
it resolves \cite[Conjecture~2.3]{Sokal-Zeng_masterpoly}.
The second version (Theorem~\ref{thm.pqgen}) is a $(p,q)$-generalisation
of the first one: it is a J-fraction with 16 variables 
along with one family of infinitely many variables
that enumerates permutations with respect to the record-and-cycle classification
(introduced in Section~\ref{subsec.intro.stats})
together with three pairs of $(p,q)$-variables counting
the refined categories of crossing and nesting
except for cycle valleys,
and one variable corresponding to pseudo-nestings of fixed points.
Finally, our third version (Theorem~\ref{thm.master}) --- is a J-fraction
in five infinite families of indeterminates along with 
one additional variable that keeps track of the number of cycles;
this generalises the previous two 
by employing the index-refined crossing 
and nesting statistics~\eqref{def.ucrossnestjk}.
All these results will be proved in Section~\ref{sec.permutations.proofs}.

\subsection{J-fraction (Sokal--Zeng conjecture)}

Recall the polynomial $\widehat{Q}_n$ defined in 
Equation~\eqref{def.poly.conjecture}/\cite[Equation~(2.29)]{Sokal-Zeng_masterpoly}
\begin{eqnarray}
        & &
        \widehat{Q}_n(x_1,x_2, y_1, y_2,u_1, u_2, v_1, v_2, \mathbf{w}, \lambda)
        \;=\;
        \nonumber\\[4mm]
         & &\qquad\qquad 
         \sum_{\sigma \in \fS_n}
        x_1^{\eareccpeak(\sigma)} x_2^{\eareccdfall(\sigma)}
        y_1^{\ereccval(\sigma)} y_2^{\ereccdrise(\sigma)}
        \:\times
        \qquad\qquad
        \nonumber\\[-1mm]
   & & \qquad\qquad\qquad\:
   u_1^{\nrcpeak(\sigma)} u_2^{\nrcdfall(\sigma)}
   v_1^{\nrcval(\sigma)} v_2^{\nrcdrise(\sigma)}
        \mathbf{w}^{\boldsymbol{\fix}(\sigma)} \lambda^{\cyc(\sigma)}
\nonumber
\end{eqnarray}
where $\mathbf{w}^{\boldsymbol{\fix}(\sigma)}$ is
\begin{equation}
\mathbf{w}^{\boldsymbol{\fix}(\sigma)} \;=\;  \prod_{i\in \Fix} w_{\psnest(i,\sigma)}\;.
\nonumber
\end{equation}

Our first main result for permutations is 
\cite[Conjecture~2.3]{Sokal-Zeng_masterpoly}.

\begin{thm}[{{\cite[Conjecture~2.3]{Sokal-Zeng_masterpoly}}}, J-fraction for permutations]
The ordinary generating function of the polynomials $\widehat{Q}_n$
specialised to $v_1 = y_1$ has the J-type continued fraction
\begin{eqnarray}
	\sum_{n=0}^\infty \widehat{Q}_n(x_1,x_2,y_1,y_2,u_1,u_2,y_1,v_2,\mathbf{w},\lambda)
        \: t^n
	\;=\; \qquad\qquad\qquad\qquad\qquad\qquad\qquad\qquad
        \nonumber \\
	\Scale[0.9]{\cfrac{1}{1 -  \lambda w_0 t - \cfrac{\lambda x_1 y_1 t^2}{1 - (x_2 \!+\! y_2 \!+\! \lambda w_1)t -  \cfrac{(\lambda \!+\! 1)(x_1 \!+\! u_1) y_1 t^2 }{1 -(x_2 \!+\! y_2 \!+\! u_2 \!+\! v_2 \!+\! \lambda w_2)t - \cfrac{(\lambda \!+\! 2)(x_1 \!+\! 2u_1)y_1 t^2}{1 - \cdots}}}}}
        \nonumber \\[1mm]
    \label{eq.thm.conj.SZ}
	\end{eqnarray}
with coefficients
\begin{subeqnarray}
    \gamma_{0} & = & \lambda w_0
         \\[1mm]
    \gamma_{n}   & = & [x_2 \!+\! (n-1)u_2 ] + [y_2 \!+\! (n-1)v_2 ] + \lambda w_n   \qquad\hbox{for $n \ge 1$}
         \\[1mm]
	\beta_n  & = &   (\lambda+n-1)[x_1 \!+\! (n-1)u_1 ] y_1
  \label{def.weights.thm.conj.SZ}
 \end{subeqnarray}
\label{thm.conj.SZ}
\end{thm}

The continued fraction \eqref{eq.thm.conj.SZ}/\eqref{def.weights.thm.conj.SZ}
requires only one specialisation, namely $v_1 = y_1$.
This clearly generalises the second J-fraction for permutations of Sokal and Zeng
\cite[Theorem~2.4]{Sokal-Zeng_masterpoly}
which also requires the specialisation $v_2 = y_2$.

We will prove Theorem~\ref{thm.conj.SZ} in Section~\ref{sec.permutations.proofs}.

\subsection{$p,q$-generalisation}\label{subsec.pq.conj.SZ}

We now state a $p,q$-generalisation
for Theorem~\ref{thm.conj.SZ}
which also generalises \cite[Theorem~2.12]{Sokal-Zeng_masterpoly}.
Let us first recall the polynomial $\widehat{Q}_n$ 
defined in~\cite[Equation~(2.92)]{Sokal-Zeng_masterpoly}
\begin{eqnarray}
   & &
   \widehat{Q}_n(x_1,x_2,y_1,y_2,u_1,u_2,v_1,v_2,\mathbf{w},p_{+1},p_{+2},p_{-1},p_{-2},q_{+1},q_{+2},q_{-1},q_{-2},s,\lambda)
   \;=\;
       \nonumber \\[4mm]
   & & \qquad
   \sum_{\sigma \in \Sym_n}
   x_1^{\eareccpeak(\sigma)} x_2^{\eareccdfall(\sigma)}
   y_1^{\ereccval(\sigma)} y_2^{\ereccdrise(\sigma)}
   \:\times
       \qquad\qquad
       \nonumber \\[-1mm]
   & & \qquad\qquad\:
   u_1^{\nrcpeak(\sigma)} u_2^{\nrcdfall(\sigma)}
   v_1^{\nrcval(\sigma)} v_2^{\nrcdrise(\sigma)}
	\mathbf{w}^{\boldsymbol{\fix}(\sigma)}  \:\times
       \qquad\qquad
       \nonumber \\[3mm]
   & & \qquad\qquad\:
   p_{+1}^{\ucrosscval(\sigma)}
   p_{+2}^{\ucrosscdrise(\sigma)}
   p_{-1}^{\lcrosscpeak(\sigma)}
   p_{-2}^{\lcrosscdfall(\sigma)}
          \:\times
       \qquad\qquad
       \nonumber \\[3mm]
   & & \qquad\qquad\:
   q_{+1}^{\unestcval(\sigma)}
   q_{+2}^{\unestcdrise(\sigma)}
   q_{-1}^{\lnestcpeak(\sigma)}
   q_{-2}^{\lnestcdfall(\sigma)}
   s^{\psnest(\sigma)}
   \lambda^{\cyc(\sigma)}
   \;.
 \label{def.poly.pqgen}
\end{eqnarray}
%%\begin{eqnarray}
%%        \widehat{Q}_n(x_1,x_2, y_1, y_2,u_1, u_2, v_1, v_2, \mathbf{w}, 
%%	p_{+1}, p_{+2}, p_{-1}, p_{-2}, q_{+1}, q_{+2}, q_{-1}, q_{-2}, 
%%	s,\lambda)
%%        \;=\;\nonumber\\[1em]
%%         \sum_{\sigma \in \fS_n}
%%        x_1^{\eareccpeak(\sigma)} x_2^{\eareccdfall(\sigma)}
%%        y_1^{\ereccval(\sigma)} y_2^{\ereccdrise(\sigma)}\:\times
%%	\qquad\qquad\qquad\qquad \nonumber\\[1em]
%%   u_1^{\nrcpeak(\sigma)} u_2^{\nrcdfall(\sigma)}
%%   v_1^{\nrcval(\sigma)} v_2^{\nrcdrise(\sigma)}
%%	\mathbf{w}^{\boldsymbol{\fix}(\sigma)}\:\times
%%	\qquad\qquad\qquad\qquad \nonumber\\[1em]
%%   p_{+1}^{\ucrosscval(\sigma)} p_{+2}^{\ucrosscdrise(\sigma)}
%%	 p_{-1}^{\lcrosscpeak(\sigma)} p_{-2}^{\lcrosscdfall(\sigma)}\:\times 
%%	 \qquad\qquad\quad\;\;\;\;\nonumber\\[1em]
%%   q_{+1}^{\unestcval(\sigma)} q_{+2}^{\unestcdrise(\sigma)}
%%   q_{-1}^{\lnestcpeak(\sigma)} q_{-2}^{\lnestcdfall(\sigma)}
%%	s^{\psnest(\sigma)}\lambda^{\cyc(\sigma)}.\qquad
%%\label{def.poly.pqgen}
%%\end{eqnarray}
For the $p,q$-generalisation of their second J-fraction, 
involving the polynomials $\widehat{Q}_n$, 
Sokal and Zeng needed the specialisations $v_1 = y_1$, $v_2 = y_2$, 
$q_{+1} = p_{+1}$, and $q_{+2} = p_{+2}$.
However, we now state a J-fraction that only requires the specialisations
$v_1=y_1$ and $q_{+1} = p_{+1}$.

\begin{thm}[J-fraction with $p,q$-generalisation for permutations] 
The ordinary generating function of the polynomials $\widehat{Q}_n$
specialised to $v_1 = y_1$ and $q_{+1} = p_{+1}$ 
has the J-type continued fraction
%\begin{eqnarray}
%	& & \Scale[0.70]{\sum_{n=0}^\infty \widehat{Q}_n(x_1,x_2, y_1, y_2,u_1, u_2, y_1, v_2, \mathbf{w}, 
%        p_{+1}, p_{+2}, p_{-1}, p_{-2}, p_{+1}, q_{+2}, q_{-1}, q_{-2}, 
%        s,\lambda)
%        \: t^n
%	\;=\;} \qquad\qquad\qquad\qquad\qquad\qquad\qquad\qquad\qquad
%	\nonumber\\
%	& &
%\Scale[0.70]{
% \cfrac{1}{1 -  \lambda w_0 t - \cfrac{\lambda x_1 y_1 t^2}{1 - (x_2 \!+\! y_2 \!+\! \lambda s w_1)t -  \cfrac{(\lambda \!+\! 1)(p_{-1}x_1 \!+\! q_{-1}u_1) p_{+1}y_1 t^2 }{1 -(p_{-2}x_2 \!+\! q_{-2}u_2 \!+\! p_{+2}y_2 \!+\! q_{+2}v_2 \!+\! \lambda s^2 w_2)t - \cfrac{(\lambda \!+\! 2)(p_{-1}^2 x_1 \!+\! [q_{-1}p_{-1} + q_{-1}^2] u_1)p_{+1}^2 y_1 t^2}{1 - \cdots}}}}
%}
%\label{eq.thm.pqgen}
%\end{eqnarray}
%
\begin{eqnarray}
   & & \hspace*{-7mm}
\Scale[1]{
   \sum\limits_{n=0}^\infty \widehat{Q}_n(x_1,x_2, y_1, y_2,u_1, u_2, y_1, v_2, \mathbf{w}, p_{+1}, p_{+2}, p_{-1}, p_{-2}, p_{+1}, q_{+2}, q_{-1}, q_{-2}, s,\lambda)\:t^n\;=\;
}
       \nonumber \\[2mm]
   & & \hspace*{-3mm}
\Scale[0.70]{
	\cfrac{1}{1 -  \lambda w_0 t - \cfrac{\lambda x_1 y_1 t^2}{1 - (x_2 \!+\! y_2 \!+\! \lambda s w_1)t -  \cfrac{(\lambda \!+\! 1)(p_{-1}x_1 \!+\! q_{-1}u_1) p_{+1}y_1 t^2 }{1 -(p_{-2}x_2 \!+\! q_{-2}u_2 \!+\! p_{+2}y_2 \!+\! q_{+2}v_2 \!+\! \lambda s^2 w_2)t - \cfrac{(\lambda \!+\! 2)(p_{-1}^2 x_1 \!+\! [q_{-1}p_{-1} + q_{-1}^2] u_1)p_{+1}^2 y_1 t^2}{1 - \cdots}}}}
}
       \nonumber \\[1mm]
   \label{eq.thm.pqgen}
\end{eqnarray}
with coefficients
\begin{subeqnarray}
    \gamma_{0} & = & \lambda w_0
         \\[1mm]
	 \gamma_{n}   & = & (p_{-2}^{n-1} x_2 \!+\! q_{-2} [n-1]_{p_{-2},q_{-2}}u_2 ) + (p_{+2}^{n-1}y_2 \!+\! q_{+2}[n-1]_{p_{+2},q_{+2}}v_2 ) + \lambda s^n w_n \qquad\nonumber\\   
	&& \qquad\qquad\qquad\qquad\qquad\qquad\qquad\qquad\qquad\qquad\qquad\qquad\hbox{for $n \ge 1$}
         \\[1mm]
	 \beta_n  & = &   (\lambda+n-1)(p_{-1}^{n-1}x_1 \!+\! q_{-1}[n-1]_{p_{-1},q_{-1}}u_1 ) p_{+1}^{n-1} y_1
  \label{def.weights.thm.pqgen}
 \end{subeqnarray}
\label{thm.pqgen}
\end{thm}

We will prove this theorem in Section~\ref{sec.permutations.proofs}, as a special case of a more general result.

\subsection{Master J-fraction}\label{subsec.master.conj.SZ}

We can go much farther and obtain a more general J-fraction
generalising Theorems~\ref{thm.conj.SZ} and~\ref{thm.pqgen}.
We obtain a J-fraction in the following 
five families of
infinitely many indeterminates:
$\bsfa = (\sfa_{\ell})_{\ell \ge 0}$,
$\bsfb = (\sfb_{\ell,\ell'})_{\ell,\ell' \ge 0}$,
$\bsfc = (\sfc_{\ell,\ell'})_{\ell,\ell' \ge 0}$,
$\bsfd = (\sfd_{\ell,\ell'})_{\ell,\ell' \ge 0}$,
$\bsfe = (\sfe_{\ell})_{\ell \ge 0}$;
please note that $\bsfa$ and $\bsfe$ have one index while 
$\bsfb, \bsfc$ and $\bsfd$ have two indices.
Using the index-refined crossing and nesting statistics
defined in \eqref{def.ucrossnestjk},
we define the polynomial 
$\widehat{Q}_n(\bsfa,\bsfb,\bsfc,\bsfd,\bsfe,\lambda)$ by
\begin{eqnarray}
   & & \hspace*{-10mm}
	\widehat{Q}_n(\bsfa,\bsfb,\bsfc,\bsfd,\bsfe,\lambda)
   \;=\;
       \nonumber \\[4mm]
   & &
   \sum_{\sigma \in \fS_{n}}
   \;\:
   \lambda^{\cyc(\sigma)}\;
   \prod\limits_{i \in \Cval(\sigma)}  \! \sfa_{\ucross(i,\sigma) + \unest(i,\sigma)}
   \prod\limits_{i \in \Cpeak(\sigma)} \!\!  \sfb_{\lcross(i,\sigma),\,\lnest(i,\sigma)}
       \:\times
       \qquad\qquad
       \nonumber \\[1mm]
   & & \;
	\prod\limits_{i \in \Cdfall(\sigma)} \!\!  \sfc_{\lcross(i,\sigma),\,\lnest(i,\sigma)}
   \;
   \prod\limits_{i \in \Cdrise(\sigma)} \!\!  \sfd_{\ucross(i,\sigma),\,\unest(i,\sigma)}
   \;
   \prod\limits_{i \in \Fix(\sigma)} \!\!  \sfe_{\psnest(i,\sigma)} 
 \label{def.poly.master}
\end{eqnarray}
where recall that $\Cval(\sigma) = \{ i\colon \sinv(i) > i < \sigma(i) \}$
and likewise for the others.

The polynomials \eqref{def.poly.master} have a beautiful J-fraction:
\begin{thm}[Master J-fraction for permutations]
   \label{thm.master}
The ordinary generating function of the polynomials
$\widehat{Q}_n(\bsfa,\bsfb,\bsfc,\bsfd,\bsfe,\lambda)$
has the J-type continued fraction
\begin{eqnarray}
   \sum_{n=0}^\infty \widehat{Q}_n(\bsfa,\bsfb,\bsfc,\bsfd,\bsfe,\lambda) \: t^n
   \;=\; \qquad\qquad\qquad\qquad\qquad\qquad\qquad\qquad\qquad\qquad\qquad\qquad
	\nonumber\\
\Scale[0.8]{\qquad
	\cfrac{1}{1 - \lambda\sfe_0 t - \cfrac{\lambda\sfa_{0} \sfb_{00} t^2}{1 - (\sfc_{00}+\sfd_{00} + \lambda \sfe_1)t - \cfrac{(\lambda + 1)\sfa_1(\sfb_{01} + \sfb_{10}) t^2}{1 -(\sfc_{01} + \sfc_{10} + \sfd_{01} + \sfd_{10}  + \lambda \sfe_2)t -  \cfrac{(\lambda + 2)\sfa_2(\sfb_{02} + \sfb_{11} + \sfb_{20}) t^2}{1 - \cdots}}}}
	}
	\nonumber\\
   \label{eq.thm.master}
\end{eqnarray}
with coefficients
\begin{subeqnarray}
   \gamma_0
   & = &
   \lambda \sfe_0
        \\[2mm]
   \gamma_n
   & = &
      \left( \sum_{\xi=0}^{n-1} \sfc_{n- 1 -\xi, \xi} \right) \!+\!
      \left( \sum_{\xi=0}^{n-1} \sfd_{n- 1 -\xi, \xi} \right) \!+\!
      \lambda \sfe_n 
       \qquad\hbox{for $n \ge 1$}\\[2mm]
   \beta_{n}
   & = &
   (\lambda + n-1)\; \sfa_{n-1}\;
   \left( \sum_{\xi=0}^{n-1} \sfb_{n- 1 -\xi, \xi} \right)
        \\[2mm]
 \label{def.weights.master}
\end{subeqnarray}
\end{thm}

We will prove this theorem in Section~\ref{sec.permutations.proofs}.
It implies Theorems~\ref{thm.conj.SZ} and~\ref{thm.pqgen} by straightforward specialisations.

\begin{rem}
%{\bf Remarks.} 
1. We remark that \eqref{def.poly.master} is {\em almost} the same as the polynomial
introduced in \cite[eq.~(2.77)]{Sokal-Zeng_masterpoly},
except for the extra factor $\lambda^{\cyc(\sigma)}$
and the index of $\sfa$ depends on the sum $\ucross(i,\sigma) + \unest(i,\sigma)$.
This is the price we have to pay in order to include the statistic $\cyc$.
See \cite[p.~13]{Sokal-Zeng_masterpoly}.

2. We also note that \eqref{def.poly.master} is {\em almost} the same as the polynomial
\cite[eq.~(2.100)]{Sokal-Zeng_masterpoly} as well,
except our treatment of $\sfd$ is nicer as
we are able to recover both $\ucross(i,\sigma)$, $\unest(i,\sigma)$,
and not just their sum.
In fact, this separation is what allows us
to prove \cite[Conjecture~2.3]{Sokal-Zeng_masterpoly}
by using \cite[Lemma~2.10]{Sokal-Zeng_masterpoly}.

3. The continued fraction \eqref{eq.thm.master}/\eqref{def.weights.master}
is the same as \cite[eq.~(2.101),(2.102),(2.103)]{Sokal-Zeng_masterpoly}
except for the indexing of $\sfd$.
\myendremark
\end{rem}

\section{D-permutations: Statements of results}
\label{sec.dperm}

%Analogous to \cite[Sections~3.2-3.4 and Sections~3.5]{Deb-Sokal}
%our continued fractions for D-permutations also
%have two versions:
%the first involve the record classification
%and the second involve the variant record classification.
%
%
%We then rephrase some of our results using cycle valley minima
%similar to what was done in \cite[Section~4.1.3]{Deb-Sokal}
%in Section~{\bf REF!!!!}.
%This allows us to resolve \cite[Conjecture~12]{Randrianarivony_96b}.
%
%{\bf SAY MORE!!! WHERE WILL WE PROVE OUR RESULTS?????????}

In this section, we state our continued fractions for D-permutations.
Analogous to \cite[Sections~3.2--3.4, and~3.5]{Deb-Sokal}
our continued fractions have two variants:
the first involve the record classification
and the second involve the variant record classification,
both have been introduced in Section~\ref{subsec.intro.stats}.
The most basic versions in each variant
are a T-fraction (Theorems~\ref{thm.conj.DS} and~\ref{thm.dperm.prime})
in 12 variables that enumerates
D-permutations with respect to the
parity-refined record-and-cycle classification
and the variant parity-refined record-and-cycle classification respectively;
except for the segregation of cycle valleys.
Theorem~\ref{thm.conj.DS} resolves
\cite[Conjecture~4.1]{Deb-Sokal}.
The second versions (Theorems~\ref{thm.DS.pqgen} and~\ref{thm.DS.pqgen.prime})
are respective $(p,q)$-generalisations
of the first versions: they are a T-fraction with 21 variables
that enumerates D-permutations with respect to the
parity-refined record-and-cycle classification
and variant parity-refined record-and-cycle classification respectively,
together with three pairs of $(p,q)$-variables counting
the refined categories of crossings and nestings
except for cycle valleys.
Finally, our third versions
(Theorems~\ref{thm.DS.master} and~\ref{thm.DS.master.prime})
--- is a T-fraction
in six infinite families of indeterminates and 
one additional variable;
this generalises the previous versions
by employing the index-refined crossing and nesting statistics~\eqref{def.ucrossnestjk}
and the variant index-refined crossing and nesting statistics~\eqref{def.ucrossnestjk.prime}.
One of the variables ($\lambda$) in each version counts the number of cycles.

The first variants will be stated in Section~\ref{subsec.dperm.results}
and will be proved in Section~\ref{subsec.dperm.proofs}.
The second variants will be stated in Section~\ref{subsec.dperm.results.prime}
and will be proved in Section~\ref{subsec.dperm.proofs.prime}.

We then rephrase Theorems~\ref{thm.conj.DS} and~\ref{thm.dperm.prime}
using cycle valley minima in Section~\ref{subsec.dperm.minval};
our approach will be similar to that in \cite[Section~4.1.3]{Deb-Sokal}.
This allows us to resolve \cite[Conjecture~12]{Randrianarivony_96b}.

\subsection{Continued fractions using record classification}
\label{subsec.dperm.results}

\subsubsection{T-fraction (Deb--Sokal conjecture)}
\label{subsec.dperm.conj}

Recall the polynomial $\widehat{P}_n$
defined in \eqref{def.Pnhat.dperm}/\cite[Equation~(4.2)]{Deb-Sokal}
\begin{eqnarray}
   & &
   \Scale[0.97]{
   \widehat{P}_n(x_1,x_2,y_1,y_2,u_1,u_2,v_1,v_2,\we,\wo,\ze,\zo,\lambda)}
   \;=\;
       \nonumber \\[4mm]
   & & \qquad\qquad
   \sum_{\sigma \in \dperm_{2n}}
   x_1^{\eareccpeak(\sigma)} x_2^{\eareccdfall(\sigma)}
   y_1^{\ereccval(\sigma)} y_2^{\ereccdrise(\sigma)}
   \:\times
       \qquad\qquad
       \nonumber \\[-1mm]
   & & \qquad\qquad\qquad\:
   u_1^{\nrcpeak(\sigma)} u_2^{\nrcdfall(\sigma)}
   v_1^{\nrcval(\sigma)} v_2^{\nrcdrise(\sigma)}
   \:\times
       \qquad\qquad
       \nonumber \\[3mm]
   & & \qquad\qquad\qquad\:
   \we^{\evennrfix(\sigma)} \wo^{\oddnrfix(\sigma)}
   \ze^{\evenrar(\sigma)} \zo^{\oddrar(\sigma)}
   \, \lambda^{\cyc(\sigma)}.
 \nonumber
\end{eqnarray}

Our first main result for D-permutations is \cite[Conjecture~4.1]{Deb-Sokal}.

\begin{thm}[{{\cite[Conjecture~4.1]{Deb-Sokal}}}]
The ordinary generating function of the polynomials \eqref{def.Pnhat.dperm}
specialised to $v_1 = y_1$ has the T-type continued fraction
\begin{eqnarray}
   & & \hspace*{-12mm}
   \sum_{n=0}^\infty
   \widehat{P}_n(x_1,x_2,y_1,y_2,u_1,u_2,y_1,v_2,\we,\wo,\ze,\zo,\lambda) \: t^n
   \;=\;
       \nonumber \\
   & &
   \cfrac{1}{1 - \lambda^2 \ze \zo  \,t - \cfrac{\lambda x_1 y_1  \,t}{1 -  \cfrac{(x_2\!+\!\lambda\we)(y_2\!+\!\lambda\wo) \,t}{1 - \cfrac{(\lambda+1)(x_1\!+\!u_1) y_1 \,t}{1 - \cfrac{(x_2\!+\!u_2\!+\!\lambda\we)(y_2\!+\!v_2\!+\!\lambda\wo)  \,t}{1 - \cfrac{(\lambda+2)(x_1\!+\!2u_1) y_1 \,t}{1 - \cfrac{(x_2\!+\!2u_2\!+\!\lambda\we)(y_2\!+\!2v_2\!+\!\lambda\wo)  \,t}{1 - \cdots}}}}}}}
       \nonumber \\[1mm]
   \label{eq.thm.DS}
\end{eqnarray}
with coefficients
\begin{subeqnarray}
   \alpha_{2k-1} & = & (\lambda + k-1) \: [x_1 + (k-1) u_1] \: y_1
        \\[1mm]
   \alpha_{2k}   & = & [x_2 + (k-1) u_2 + \lambda\we] \: [y_2 + (k-1)v_2 + \lambda\wo]
        \\[1mm]
   \delta_1  & = &   \lambda^2 \ze \zo   \\[1mm]
   \delta_n  & = &   0    \qquad\hbox{for $n \ge 2$}
   \label{eq.thm.conj.DS.weights}
\end{subeqnarray}
\label{thm.conj.DS}
\end{thm}
\noindent We will prove Theorem~\ref{thm.conj.DS} in Section~\ref{subsec.dperm.proofs}.

\begin{rem}
%{\bf Remarks.}
1. The continued fraction \eqref{eq.thm.DS}/\eqref{eq.thm.conj.DS.weights}
is almost similar to \cite[eq.~(3.3),~(3.4)]{Deb-Sokal}
except for the extra factor $\lambda$ and the specialisation $v_1 = y_1$.
This continued fraction is also
%\eqref{eq.thm.DS}/\eqref{eq.thm.conj.DS.weights} is 
the same as \cite[eq.~(4.7),~(4.8)]{Deb-Sokal}
which proves the equidistribution of statistics for D-permutations in
\cite[Conjecture~4.1$'$]{Deb-Sokal}.

%{\bf SAY THAT A SPECIALISATION WAS FIRST DONE BY PAN--ZENG IN \cite{Pan_23}!!!!!!!!!!!}
2. The continued fraction of Theorem~\ref{thm.conj.DS}
specialised to $x_1 = x_2 = y_1 = y_2 = u_1 = u_2 = v_1 = v_2 = 1$,
$\we = \ze$, $\wo = \zo$, enumerates D-permutations of $[2n]$ 
with a weight $\lambda$ for each cycle, $\ze$ for each even fixed point
and $\zo$ for each odd fixed point.
By applying  
\cite[Proposition~2.1]{Deb-Sokal},
to the resulting T-fraction,
we obtain a J-fraction with
coefficients $\gamma_n = (\lambda\ze +n)(\lambda\zo+n) + (\lambda+n)(n+1)$
and $\beta_n  = n (\lambda+n-1)(\lambda \ze +n)(\lambda \zo + n)$.
This J-fraction was first proved by Pan and Zeng \cite[eq.~(4.11)]{Pan_23}
and they used this continued fraction to establish~\cite[Conjecture~6.5]{Lazar_22}.
\myendremark
\end{rem}

%The continued fraction \eqref{eq.thm.DS}/\eqref{eq.thm.conj.DS.weights}
%is almost similar to \cite[eq.~(3.3),~(3.4)]{Deb-Sokal}
%except for the extra factor $\lambda$ and the specialisation $v_1 = y_1$.
%This continued fraction is also
%%\eqref{eq.thm.DS}/\eqref{eq.thm.conj.DS.weights} is 
%the same as \cite[eq.~(4.7),~(4.8)]{Deb-Sokal}
%which proves the equidistribution of statistics for D-permutations in
%\cite[Conjecture~4.1$'$]{Deb-Sokal}.

We can also enumerate D-cycles by extracting the coefficient of
$\lambda^1$ in Theorem~\ref{thm.conj.DS}
to prove \cite[Conjecture~4.5]{Deb-Sokal}.
The analogous polynomials $P^{\dcycle}_n$
defined in \cite[eq.~(4.17)]{Deb-Sokal} are
\begin{eqnarray}
   & &
   P^{\dcycle}_n(x_1,x_2,y_1,y_2,u_1,u_2,v_1,v_2)
   \;=\;
       \nonumber \\[4mm]
   & & \qquad\qquad
   \sum_{\sigma \in \dcycle_{2n}}
   x_1^{\eareccpeak(\sigma)} x_2^{\eareccdfall(\sigma)}
   y_1^{\ereccval(\sigma)} y_2^{\ereccdrise(\sigma)}
   \:\times
       \qquad\qquad
       \nonumber \\[-1mm]
   & & \qquad\qquad\qquad\quad\!
   u_1^{\nrcpeak(\sigma)} u_2^{\nrcdfall(\sigma)}
   v_1^{\nrcval(\sigma)} v_2^{\nrcdrise(\sigma)}
   \;.
 \label{def.Pn.dcycle}
\end{eqnarray}

\begin{cor}[{{\cite[Conjecture~4.5]{Deb-Sokal}}}]
   \label{conj.Tfrac.second.dcycle}
The ordinary generating function of the polynomials \eqref{def.Pn.dcycle}
specialised to $v_1 = y_1$ has the S-type continued fraction
\be
   \!\!\!
   \sum_{n=0}^\infty
   P^{\dcycle}_{n+1}(x_1,x_2,y_1,y_2,u_1,u_2,y_1,v_2) \: t^n
   \;=\;
   \cfrac{x_1 y_1}{1 - \cfrac{x_2 y_2  \,t}{1 - \cfrac{(x_1\!+\!u_1) y_1 \,t}{1 - \cfrac{(x_2\!+\!u_2\!)(y_2\!+\!v_2)  \,t}{1 - \cfrac{(x_1\!+\!2u_1) 2y_1 \,t}{1 - \cfrac{(x_2\!+\!2u_2)(y_2\!+\!2v_2)  \,t}{1 - \cdots}}}}}}
  \label{eq.conj.Tfrac.second.dcycle}
\ee
with coefficients
\begin{subeqnarray}
   \alpha_{2k-1}   & = & [x_2 + (k-1) u_2] \: [y_2 + (k-1)v_2]
        \\[1mm]
   \alpha_{2k}     & = & [x_1 + k u_1] \: k y_1
   \label{eq.conj.Tfrac.second.weights.dcycle}
\end{subeqnarray}
\end{cor}

\subsubsection{{$p,q$}-generalisation}
\label{subsec.dperm.pqgen}

In this subsection, we shall provide a $p,q$-generalisation
for Theorem~\ref{thm.conj.DS} 
by including four pairs of $(p,q)$--variables corresponding to the four
refined types of crossings and nestings, 
as well as  two variables corresponding to pseudo-nestings for fixed points:
\begin{eqnarray}
   & & \hspace*{-8mm}
   \Scale[0.95]{
   \widehat{P}_n(x_1,x_2,y_1,y_2,u_1,u_2,v_1,v_2,\we,\wo,\ze,\zo,p_{-1},p_{-2},p_{+1},p_{+2},q_{-1},q_{-2},q_{+1},q_{+2},\se,\so,\lambda)
   \;=\; }
       \nonumber \\[4mm]
   & & \qquad\qquad
   \sum_{\sigma \in \dperm_{2n}}
   x_1^{\eareccpeak(\sigma)} x_2^{\eareccdfall(\sigma)}
   y_1^{\ereccval(\sigma)} y_2^{\ereccdrise(\sigma)}
   \:\times
       \qquad\qquad\:
       \nonumber \\[-1mm]
   & & \qquad\qquad\qquad\;\,
   u_1^{\nrcpeak(\sigma)} u_2^{\nrcdfall(\sigma)}
   v_1^{\nrcval(\sigma)} v_2^{\nrcdrise(\sigma)}
   \:\times
       \qquad\qquad\;
       \nonumber \\[3mm]
   & & \qquad\qquad\qquad\;\,
   \we^{\evennrfix(\sigma)} \wo^{\oddnrfix(\sigma)}
   \ze^{\evenrar(\sigma)} \zo^{\oddrar(\sigma)}
   \:\times
       \qquad\qquad\;
       \nonumber \\[3mm]
   & & \qquad\qquad\qquad\;\,
   p_{-1}^{\lcrosscpeak(\sigma)}
   p_{-2}^{\lcrosscdfall(\sigma)}
   p_{+1}^{\ucrosscval(\sigma)}
   p_{+2}^{\ucrosscdrise(\sigma)}
          \:\times
       \qquad\qquad\;
       \nonumber \\[3mm]
   & & \qquad\qquad\qquad\;\,
   q_{-1}^{\lnestcpeak(\sigma)}
   q_{-2}^{\lnestcdfall(\sigma)}
   q_{+1}^{\unestcval(\sigma)}
   q_{+2}^{\unestcdrise(\sigma)}
          \:\times
       \qquad\qquad
       \nonumber \\[3mm]
   & & \qquad\qquad\qquad\;\,
   \se^{\epsnest(\sigma)}
   \so^{\opsnest(\sigma)}
   \, \lambda^{\cyc(\sigma)}
   \;.
 \label{def.dperm.poly.pqgen}
\end{eqnarray}

This is the same as \cite[Equation~(3.21)]{Deb-Sokal} except for the
extra factor of $\lambda^{\cyc(\sigma)}$.
We now state a J-fraction under the specialisations
$v_1=y_1$ and $q_{+1} = p_{+1}$:
\begin{thm}[T-fraction for D-permutations, $p,q$-generalisation]
   \label{thm.DS.pqgen}
The ordinary generating function of the polynomials \eqref{def.dperm.poly.pqgen} specialised to $v_1=y_1$ and $q_{+1} = p_{+1}$ has the
T-type continued fraction
\begin{eqnarray}
   & & \hspace*{-7mm}
\Scale[0.89]{
   \sum\limits_{n=0}^\infty
   \widehat{P}_n(x_1,x_2,y_1,y_2,u_1,u_2,y_1,v_2,\we,\wo,\ze,\zo,p_{-1},p_{-2},p_{+1},p_{+2},q_{-1},q_{-2},p_{+1},q_{+2},\se,\so,\lambda) \: t^n
   \;=\;
}
       \nonumber \\[2mm]
   & & \hspace*{-3mm}
\Scale[0.82]{
   \cfrac{1}{1 - \lambda^2 \ze \zo  \,t - \cfrac{\lambda x_1 y_1  \,t}{1 -  \cfrac{(x_2\!+\!\lambda\se\we)(y_2\!+\!\lambda\so\wo) \,t}{1 - \cfrac{(\lambda + 1) p_{+1} y_1 (p_{-1}x_1\!+\!q_{-1}u_1)  \,t}{1 - \cfrac{(p_{-2}x_2\!+\!q_{-2}u_2\!+\!\lambda\se^2\we)(p_{+2}y_2\!+\!q_{+2}v_2\!+\!\lambda\so^2\wo)  \,t}{1 - \cfrac{(\lambda + 2)p_{+1}^2 y_1(p_{-1}^2 x_1\!+\! q_{-1} [2]_{p_{-1},q_{-1}}u_1)  \,t}{1 - \cfrac{(p_{-2}^2 x_2\!+\! q_{-2} [2]_{p_{-2},q_{-2}} u_2\!+\!\lambda\se^3\we)(p_{+2}^2 y_2\!+\! q_{+2} [2]_{p_{+2},q_{+2}}v_2\!+\!\lambda\so^3\wo)  \,t}{1 - \cdots}}}}}}}
}
       \nonumber \\[1mm]
   \label{eq.thm.dperm.pqgen}
\end{eqnarray}
with coefficients
\begin{subeqnarray}
   \alpha_{2k-1} & = &
   (\lambda + k-1) \: p_{+1}^{k-1} y_1 \:
      \left( p_{-1}^{k-1} x_1 + q_{-1} [k-1]_{p_{-1},q_{-1}} u_1 \right) \:
        \\[1mm]
   \alpha_{2k}   & = &
      \left( p_{-2}^{k-1} x_2 + q_{-2} [k-1]_{p_{-2},q_{-2}} u_2 + \lambda\se^k\we \right) \: \times 
	\nonumber\\
    &&   \left( p_{+2}^{k-1} y_2 + q_{+2} [k-1]_{p_{+2},q_{+2}} v_2 + \lambda\so^k\wo \right) \\[1mm]
        %\nonumber \\ \\
   \delta_1  & = &   \lambda^2\ze \zo   \\[1mm]
   \delta_n  & = &   0    \qquad\hbox{for $n \ge 2$}
   \label{eq.weights.dperm.thm.pqgen}
\end{subeqnarray}
\end{thm}
\noindent We will prove this theorem in Section~\ref{subsec.dperm.proofs}, as a special case of a more general result.

\subsubsection{Master T-fraction}\label{subsec.dperm.master}

In this subsection, we shall provide a master T-fraction
generalising Theorems~\ref{thm.conj.DS} and~\ref{thm.DS.pqgen}.
Let us introduce a polynomial in six infinite families of indeterminates
$\bsfa = (\sfa_{\ell})_{\ell \ge 0}$,
$\bsfb = (\sfb_{\ell,\ell'})_{\ell,\ell' \ge 0}$,
$\bsfc = (\sfc_{\ell,\ell'})_{\ell,\ell' \ge 0}$,
$\bsfd = (\sfd_{\ell,\ell'})_{\ell,\ell' \ge 0}$,
$\bsfe = (\sfe_\ell)_{\ell \ge 0}$,
$\bsff = (\sff_\ell)_{\ell \ge 0}$
that will have a nice T-fraction
and that will include the polynomials 
\eqref{def.Pnhat.dperm} and \eqref{def.dperm.poly.pqgen}
as specialisations.
Please note that $\bsfa, \bsfe$ and $\bsff$ have one index while 
$\bsfb, \bsfc$ and $\bsfd$ have two indices.
Using the index-refined crossing and nesting statistics
defined in \eqref{def.ucrossnestjk},
we define the polynomial $\widehat{Q}_n(\bsfa,\bsfb,\bsfc,\bsfd,\bsfe,\bsff,\lambda)$
by
\begin{eqnarray}
   & & \hspace*{-10mm}
   \widehat{Q}_n(\bsfa,\bsfb,\bsfc,\bsfd,\bsfe,\bsff,\lambda)
   \;=\;
       \nonumber \\[4mm]
   & &
   \sum_{\sigma \in \dperm_{2n}}
   \;\:
   \lambda^{\cyc(\sigma)}
   \prod\limits_{i \in \Cval(\sigma)}  \! \sfa_{\ucross(i,\sigma)\,+\,\unest(i,\sigma)}
   \prod\limits_{i \in \Cpeak(\sigma)} \!\!  \sfb_{\lcross(i,\sigma),\,\lnest(i,\sigma)}
       \:\times
       \qquad\qquad
       \nonumber \\[1mm]
   & & \qquad\;\;\,
   \prod\limits_{i \in \Cdfall(\sigma)} \!\!  \sfc_{\lcross(i,\sigma),\,\lnest(i,\sigma)}
   \;
   \prod\limits_{i \in \Cdrise(\sigma)} \!\!  \sfd_{\ucross(i,\sigma),\,\unest(i,\sigma)}
       \:\times
       \qquad\qquad
       \nonumber \\[1mm]
   & & \qquad\;\:
   \prod\limits_{i \in \Evenfix(\sigma)} \!\!\! \sfe_{\psnest(i,\sigma)}
   \;
   \prod\limits_{i \in \Oddfix(\sigma)}  \!\!\! \sff_{\psnest(i,\sigma)}
   \;.
   \quad
 \label{def.dperm.master}
\end{eqnarray}
where recall that $\Cval(\sigma) = \{ i\colon \sinv(i) > i < \sigma(i) \}$
and likewise for the others.

We remark that \eqref{def.dperm.master} is {\em almost} the same as the polynomial
introduced in \cite[eq.~(3.30)]{Deb-Sokal},
except for the extra factor $\lambda^{\cyc(\sigma)}$
and the index of $\sfa$ depends on the sum $\ucross(i,\sigma) + \unest(i,\sigma)$.
That is the price we have to pay in order to include the statistic $\cyc$. See \cite[Appendix~A]{Deb-Sokal}.
We also note that \eqref{def.dperm.master} is {\em almost} the same as the polynomial
\eqref{def.poly.master} as well,
but restricted to D-permutations and refined to record the parity of
fixed points.

The polynomials \eqref{def.dperm.master} have a beautiful T-fraction:
\begin{thm}[Master T-fraction for D-permutations]
   \label{thm.DS.master}
The ordinary generating function of the polynomials
$\widehat{Q}_n(\bsfa,\bsfb,\bsfc,\bsfd,\bsfe,\bsff,\lambda)$
has the T-type continued fraction
\be
   \sum_{n=0}^\infty \widehat{Q}_n(\bsfa,\bsfb,\bsfc,\bsfd,\bsfe,\bsff,\lambda) \: t^n
   \;=\;
\Scale[0.95]{
	\cfrac{1}{1 - \lambda^2 \sfe_0 \sff_0 t - \cfrac{\lambda \sfa_{0} \sfb_{00} t}{1 - \cfrac{(\sfc_{00} + \lambda\sfe_1)(\sfd_{00} + \lambda\sff_1) t}{1 - \cfrac{(\lambda+1)\sfa_{1}(\sfb_{01} + \sfb_{10}) t}{1 - \cfrac{(\sfc_{01} + \sfc_{10} + \lambda\sfe_2)(\sfd_{01} + \sfd_{10} + \lambda\sff_2)t}{1 - \cdots}}}}}
}
   \label{eq.thm.DS.master}
\ee
with coefficients
\begin{subeqnarray}
   \alpha_{2k-1}
   & = &
	(\lambda+k-1) \sfa_{k-1}
   \left( \sum_{\xi=0}^{k-1} \sfb_{k-1-\xi,\xi} \right)
        \\[2mm]
   \alpha_{2k}
   & = &
   \left( \lambda\sfe_k \,+\, \sum_{\xi=0}^{k-1} \sfc_{k-1-\xi,\xi} \right)
   \left( \lambda\sff_k \,+\, \sum_{\xi=0}^{k-1} \sfd_{k-1-\xi,\xi} \right)
        \\[2mm]
   \delta_1  & = &  \lambda^2\sfe_0 \sff_0 \\[2mm]
   \delta_n  & = &   0    \qquad\hbox{for $n \ge 2$}
 \label{def.weights.DS.master}
\end{subeqnarray}
\end{thm}

We will prove this theorem in Section~\ref{subsec.dperm.proofs}.
It implies Theorems~\ref{thm.conj.DS} and~\ref{thm.DS.pqgen} by straightforward specialisations.

\begin{rem}
%{\bf Remark.} 
We remark that \eqref{def.dperm.master} is {\em almost} the same as the polynomial
\cite[eq.~(4.34)]{Deb-Sokal} as well,
except our treatment of $\sfd$ is different:
we are able to recover the statistics $\ucross(i,\sigma)$, $\unest(i,\sigma)$,
and not the statistics $\ucross'(i,\sigma)$, $\unest'(i,\sigma)$.
In fact, this separation is what allows us
to prove \cite[Conjecture~4.1]{Deb-Sokal}.
\myendremark
\end{rem}

\subsection{Continued fractions using variant record classification}
\label{subsec.dperm.results.prime}

Similar to \cite[Section~3.5]{Deb-Sokal},
our T-fractions for D-permutations 
(Theorems~\ref{thm.conj.DS}, \ref{thm.DS.pqgen}, \ref{thm.DS.master})
also have variant forms in which
we use the variant index-refined crossing and nesting statistics 
\eqref{def.ucrossnestjk.prime}.
We shall state these variants in this subsection.

\subsubsection{T-fraction}

Let $\widehat{P}_n'$ be the polynomials defined as follows:
\begin{eqnarray}
   & &
   \Scale[0.97]{
	   \widehat{P}_n'(x_1,x_2,y_1,y_2,u_1,u_2,v_1,v_2,\we,\wo,\ze,\zo,\lambda)}
   \;=\;
       \nonumber \\[4mm]
   & & \qquad\qquad
   \sum_{\sigma \in \dperm_{2n}}
   x_1^{\ereccpeak'(\sigma)} x_2^{\eareccdfall'(\sigma)}
   y_1^{\eareccval'(\sigma)} y_2^{\ereccdrise'(\sigma)}
   \:\times
       \qquad\qquad
       \nonumber \\[-1mm]
   & & \qquad\qquad\qquad\:
   u_1^{\nrcpeak'(\sigma)} u_2^{\nrcdfall'(\sigma)}
   v_1^{\nrcval'(\sigma)} v_2^{\nrcdrise'(\sigma)}
   \:\times
       \qquad\qquad
       \nonumber \\[3mm]
   & & \qquad\qquad\qquad\:
   \we^{\evennrfix(\sigma)} \wo^{\oddnrfix(\sigma)}
   \ze^{\evenrar(\sigma)} \zo^{\oddrar(\sigma)}
   \, \lambda^{\cyc(\sigma)}.
 \label{def.Pprime.dperm}
\end{eqnarray}

We have the following variant of Theorem~\ref{thm.conj.DS}.

\begin{thm}
The ordinary generating function of the polynomials $\widehat{P}_n'$
defined in \eqref{def.Pprime.dperm}
specialised to $v_1 = y_1$ has the same T-type continued fraction 
\eqref{eq.thm.DS}/\eqref{eq.thm.conj.DS.weights}
as the polynomials $\widehat{P}_n$ defined in \eqref{def.Pnhat.dperm}.
Therefore,
\be
	\Scale[0.9]{
\widehat{P}'_n(x_1,x_2,y_1,y_2,u_1,u_2,y_1,v_2,\we,\wo,\ze,\zo,\lambda) \;=\; 
	\widehat{P}_n(x_1,x_2,y_1,y_2,u_1,u_2,y_1,v_2,\we,\wo,\ze,\zo,\lambda)}
\;.
\ee
\label{thm.dperm.prime}
\end{thm}

We will prove Theorem~\ref{thm.dperm.prime} in Section~\ref{subsec.dperm.proofs.prime}.

\subsubsection{$p,q$-generalisation}

We shall now provide a $p,q$-generalisation
for Theorem~\ref{thm.dperm.prime}
by including four pairs of $(p,q)$--variables corresponding to the four
variants of the refined types of crossings and nestings, 
as well as  two variables corresponding to pseudo-nestings for fixed points:
\begin{eqnarray}
   & & \hspace*{-8mm}
   \Scale[0.95]{
   \widehat{P}_n'(x_1,x_2,y_1,y_2,u_1,u_2,v_1,v_2,\we,\wo,\ze,\zo,p_{-1},p_{-2},p_{+1},p_{+2},q_{-1},q_{-2},q_{+1},q_{+2},\se,\so,\lambda)
   \;=\; }
       \nonumber \\[4mm]
   & & \qquad\qquad
   \sum_{\sigma \in \dperm_{2n}}
   x_1^{\ereccpeak'(\sigma)} x_2^{\eareccdfall'(\sigma)}
   y_1^{\eareccval'(\sigma)} y_2^{\ereccdrise'(\sigma)}
   \:\times
       \qquad\qquad\:
       \nonumber \\[-1mm]
   & & \qquad\qquad\qquad\;\,
   u_1^{\nrcpeak'(\sigma)} u_2^{\nrcdfall'(\sigma)}
   v_1^{\nrcval'(\sigma)} v_2^{\nrcdrise'(\sigma)}
   \:\times
       \qquad\qquad\;
       \nonumber \\[3mm]
   & & \qquad\qquad\qquad\;\,
   \we^{\evennrfix(\sigma)} \wo^{\oddnrfix(\sigma)}
   \ze^{\evenrar(\sigma)} \zo^{\oddrar(\sigma)}
   \:\times
       \qquad\qquad\;
       \nonumber \\[3mm]
   & & \qquad\qquad\qquad\;\,
   p_{-1}^{\ucrosscpeak'(\sigma)}
   p_{-2}^{\lcrosscdfall'(\sigma)}
   p_{+1}^{\lcrosscval'(\sigma)}
   p_{+2}^{\ucrosscdrise'(\sigma)}
          \:\times
       \qquad\qquad\;
       \nonumber \\[3mm]
   & & \qquad\qquad\qquad\;\,
   q_{-1}^{\unestcpeak'(\sigma)}
   q_{-2}^{\lnestcdfall'(\sigma)}
   q_{+1}^{\lnestcval'(\sigma)}
   q_{+2}^{\unestcdrise'(\sigma)}
          \:\times
       \qquad\qquad
       \nonumber \\[3mm]
   & & \qquad\qquad\qquad\;\,
   \se^{\epsnest(\sigma)}
   \so^{\opsnest(\sigma)}
   \, \lambda^{\cyc(\sigma)}
   \;.
 \label{def.dperm.poly.pqgen.prime}
\end{eqnarray}

This is the same as \cite[Equation~(3.35)]{Deb-Sokal} except for the
extra factor of $\lambda^{\cyc(\sigma)}$.
We now state a J-fraction under the specialisations
$v_1=y_1$ and $q_{+1} = p_{+1}$:
\begin{thm}
   \label{thm.DS.pqgen.prime}
The ordinary generating function of the polynomials $\widehat{P}_n'$
defined in \eqref{def.dperm.poly.pqgen.prime} 
specialised to $v_1=y_1$ and $q_{+1} = p_{+1}$ has the
same T-type continued fraction 
\eqref{eq.thm.dperm.pqgen}/\eqref{eq.weights.dperm.thm.pqgen}
as the polynomials $\widehat{P}_n$
defined in \eqref{def.dperm.poly.pqgen}.
Therefore
\begin{eqnarray}
   & &  \hspace*{-7mm}
\Scale[0.96]{
   \widehat{P}_n'(x_1,x_2,y_1,y_2,u_1,u_2,y_1,v_2,\we,\wo,\ze,\zo,p_{-1},p_{-2},p_{+1},p_{+2},q_{-1},q_{-2},p_{+1},q_{+2},\se,\so,\lambda)
}
            \nonumber \\
   & &  \hspace*{-7mm}
\Scale[0.96]{
   =\;
   \widehat{P}_n(x_1,x_2,y_1,y_2,u_1,u_2,y_1,v_2,\we,\wo,\ze,\zo,p_{-1},p_{-2},p_{+1},p_{+2},q_{-1},q_{-2},p_{+1},q_{+2},\se,\so,\lambda)
   \;.
}
\nonumber\\
\end{eqnarray}
\end{thm}

We will prove this theorem in Section~\ref{subsec.dperm.proofs.prime},
as a special case of a more general result.

\subsubsection{Master T-fraction}

We introduce a polynomial in six infinite families of indeterminates
$\bsfa,\bsfb,\bsfc,\bsfd,\bsfe,\bsff$ as before but by
using the variant index-refined crossing and nesting statistics
\eqref{def.ucrossnestjk.prime}:
\begin{eqnarray}
   & & \hspace*{-10mm}
   \widehat{Q}_n'(\bsfa,\bsfb,\bsfc,\bsfd,\bsfe,\bsff,\lambda)
   \;=\;
       \nonumber \\[4mm]
   & &
   \sum_{\sigma \in \dperm_{2n}}
   \;\:
   \lambda^{\cyc(\sigma)}
   \prod\limits_{i \in \Cval(\sigma)}  \! \sfa_{\lcross'(i,\sigma)\,+\,\lnest'(i,\sigma)}
   \prod\limits_{i \in \Cpeak(\sigma)} \!\!  \sfb_{\ucross'(i,\sigma),\,\unest'(i,\sigma)}
       \:\times
       \qquad\qquad
       \nonumber \\[1mm]
   & & \qquad\;\;\,
   \prod\limits_{i \in \Cdfall(\sigma)} \!\!  \sfc_{\lcross'(i,\sigma),\,\lnest'(i,\sigma)}
   \;
   \prod\limits_{i \in \Cdrise(\sigma)} \!\!  \sfd_{\ucross'(i,\sigma),\,\unest'(i,\sigma)}
       \:\times
       \qquad\qquad
       \nonumber \\[1mm]
   & & \qquad\;\:
   \prod\limits_{i \in \Evenfix(\sigma)} \!\!\! \sfe_{\psnest(i,\sigma)}
   \;
   \prod\limits_{i \in \Oddfix(\sigma)}  \!\!\! \sff_{\psnest(i,\sigma)}
   \;.
   \quad
 \label{def.dperm.master.prime}
\end{eqnarray}

We remark that \eqref{def.dperm.master.prime}
is {\em almost} the same as the polynomial
introduced in \cite[eq.~(3.33)]{Deb-Sokal},
except for the extra factor $\lambda^{\cyc(\sigma)}$
and the index of $\sfa$ depends on the sum 
$\lcross'(i,\sigma) + \lnest'(i,\sigma)$.

We have the following variant of Theorem~\ref{thm.DS.master}
\begin{thm}
   \label{thm.DS.master.prime}
The ordinary generating function of the polynomials
$\widehat{Q}_n'$ defined in \eqref{def.dperm.master.prime}
has the same T-type continued fraction 
\eqref{eq.thm.DS.master}/\eqref{def.weights.DS.master}
as the polynomials $\widehat{Q}_n$
defined in \eqref{def.dperm.master}.
Therefore
\be
\widehat{Q}_n'(\bsfa,\bsfb,\bsfc,\bsfd,\bsfe,\bsff,\lambda)
\;=\;
\widehat{Q}_n(\bsfa,\bsfb,\bsfc,\bsfd,\bsfe,\bsff,\lambda)
\;.
\ee
\end{thm}

Theorem~\ref{thm.DS.master.prime} will be proved in Section~\ref{subsec.dperm.proofs.prime}.
It implies Theorems~\ref{thm.dperm.prime} and~\ref{thm.DS.pqgen.prime}
by straightforward specialisations.

\begin{prob}
Is there a simple bijective proof of Theorem~\ref{thm.DS.master.prime}?
\end{prob}

\begin{sloppy}
\subsection[Reformulation of results using cycle valley minima\\and the Randrianarivony--Zeng conjecture ]{Reformulation of results using cycle valley minima and the Randrianarivony--Zeng conjecture}
\label{subsec.dperm.minval}
\end{sloppy}

We will now rephrase our results in this section using some new statistics
which will also help us to prove Conjecture~\ref{conj.RZ}${}^{\bf\prime}$
as a corollary. 
Our approach here will be the same as that of \cite[Section~4.1.3]{Deb-Sokal}.

We notice that the number of cycles in a permutation can be recovered
if we know the number of  (even and odd) fixed points
and the number of cycle valley minima (or the number of cycle peak maxima).
We will rephrase our results by distributing the weight $\lambda$,
that we had been using for the number of cycles,
among fixed points and cycle valley minima.
In \cite{Deb-Sokal},
this was done by introducing the polynomial
$\widetilde{P}_n$ in \cite[eq.~(4.14)]{Deb-Sokal}
which we recall here:
\begin{eqnarray}
   & &
   \widetilde{P}_n(x_1,x_2,\ytilde_1,y_2,u_1,u_2,\vtilde_1,v_2,\we,\wo,\ze,\zo)
   \;=\;
       \nonumber \\[4mm]
   & & \qquad\qquad
   \sum_{\sigma \in \dperm_{2n}}
   x_1^{\eareccpeak(\sigma)} x_2^{\eareccdfall(\sigma)}
   \ytilde_1^{\minval(\sigma)} y_2^{\ereccdrise(\sigma)}
   \:\times
       \qquad\qquad
       \nonumber \\[-1mm]
   & & \qquad\qquad\qquad\:
   u_1^{\nrcpeak(\sigma)} u_2^{\nrcdfall(\sigma)}
   \vtilde_1^{\nminval(\sigma)} v_2^{\nrcdrise(\sigma)}
   \:\times
       \qquad\qquad
       \nonumber \\[3mm]
   & & \qquad\qquad\qquad\:
   \we^{\evennrfix(\sigma)} \wo^{\oddnrfix(\sigma)}
   \ze^{\evenrar(\sigma)} \zo^{\oddrar(\sigma)}
   \;.
 \label{def.Pntilde}
\end{eqnarray}
We can rephrase Theorem~\ref{thm.conj.DS} by replacing
the factor $(\lambda + k-1)y_1$ with $\ytilde + (k-1) \vtilde$
and removing the factors of $\lambda$ multiplying $\we,\wo,\ze,\zo$.
\addtocounter{thm}{-7}
\begin{thmdot}\hspace{-2mm}${{}^{\bf \prime}}${\em ({\cite[Conjecture~4.2${}^{\bf\prime\prime}$]{Deb-Sokal}}){\bf .} }
The ordinary generating function of the polynomials \eqref{def.Pntilde}
has the T-type continued fraction
\begin{eqnarray}
   & & \hspace*{-12mm}
   \sum_{n=0}^\infty
   \widetilde{P}_n(x_1,x_2,\ytilde_1,y_2,u_1,u_2,\vtilde_1,v_2,\we,\wo,\ze,\zo) \: t^n
   \;=\;
       \nonumber \\ 
   & &
   \cfrac{1}{1 - \ze \zo  \,t - \cfrac{x_1 \ytilde_1  \,t}{1 -  \cfrac{(x_2\!+\!\we)(y_2\!+\!\wo) \,t}{1 - \cfrac{(x_1\!+\!u_1) (\ytilde_1 + \vtilde_1) \,t}{1 - \cfrac{(x_2\!+\!u_2\!+\!\we)(y_2\!+\!v_2\!+\!\wo)  \,t}{1 - \cfrac{(x_1\!+\!2u_1) (\ytilde_1 + 2\vtilde_1) \,t}{1 - \cfrac{(x_2\!+\!2u_2\!+\!\we)(y_2\!+\!2v_2\!+\!\wo)  \,t}{1 - \cdots}}}}}}}
       \nonumber \\[1mm]
   \label{eq.thm.Tfrac.second.bis}
\end{eqnarray}
with coefficients
\begin{subeqnarray}
   \alpha_{2k-1} & = & [x_1 + (k-1) u_1] \: [\ytilde_1 + (k-1) \vtilde_1]
        \\[1mm]
   \alpha_{2k}   & = & [x_2 + (k-1) u_2 + \we] \: [y_2 + (k-1)v_2 + \wo]
        \\[1mm]
   \delta_1  & = &   \ze \zo   \\[1mm]
   \delta_n  & = &   0    \qquad\hbox{for $n \ge 2$}
   \label{eq.thm.Tfrac.second.weights.bis}
\end{subeqnarray}
\label{thm.conjDS.primeprime}
\end{thmdot}
\addtocounter{thm}{6}

We leave the rephrasings of Theorems~\ref{thm.DS.pqgen}
and~\ref{thm.DS.master} to the reader
and directly proceed to
rephrase Theorem~\ref{thm.dperm.prime}.
To do this, we introduce the following polynomial:
\begin{eqnarray}
   & &
   \widetilde{P}_n'(x_1,x_2,\ytilde_1,y_2,u_1,u_2,\vtilde_1,v_2,\we,\wo,\ze,\zo)
   \;=\;
       \nonumber \\[4mm]
   & & \qquad\qquad
   \sum_{\sigma \in \dperm_{2n}}
   x_1^{\ereccpeak'(\sigma)} x_2^{\eareccdfall'(\sigma)}
   \ytilde_1^{\minval(\sigma)} y_2^{\ereccdrise'(\sigma)}
   \:\times
       \qquad\qquad
       \nonumber \\[-1mm]
   & & \qquad\qquad\qquad\:
   u_1^{\nrcpeak'(\sigma)} u_2^{\nrcdfall'(\sigma)}
   \vtilde_1^{\nminval(\sigma)} v_2^{\nrcdrise'(\sigma)}
   \:\times
       \qquad\qquad
       \nonumber \\[3mm]
   & & \qquad\qquad\qquad\:
   \we^{\evennrfix(\sigma)} \wo^{\oddnrfix(\sigma)}
   \ze^{\evenrar(\sigma)} \zo^{\oddrar(\sigma)}
   \;.
 \label{def.Pntilde.prime}
\end{eqnarray}

Theorem~\ref{thm.dperm.prime} can now simply be restated as follows:
\addtocounter{thm}{-3}
\begin{thm}\hspace{-3mm}${}^{\bf\prime}$
The ordinary generating function of the polynomials $\widetilde{P}_n'$
defined in \eqref{def.Pntilde.prime}
has the same T-type continued fraction
\eqref{eq.thm.Tfrac.second.bis}/\eqref{eq.thm.Tfrac.second.weights.bis}
as the polynomials $\widetilde{P}_n$ defined in \eqref{def.Pntilde}.
Therefore,
\be
        \Scale[0.9]{
\widetilde{P}_n'(x_1,x_2,\ytilde_1,y_2,u_1,u_2,\vtilde_1,v_2,\we,\wo,\ze,\zo) \;=\;
\widetilde{P}_n(x_1,x_2,\ytilde_1,y_2,u_1,u_2,\vtilde_1,v_2,\we,\wo,\ze,\zo)\;.
}
\ee
\end{thm}
\addtocounter{thm}{2}

Recall that we observed
in Section~\ref{subsec.intro.Dperm.conj} 
the equivalence between
{\cite[Conjecture~12]{Randrianarivony_96b}}
and
Conjecture~\ref{conj.RZ}${}^{\bf\prime}$.
We now obtain Conjecture~\ref{conj.RZ}${}^{\bf\prime}$
as a corollary of Theorem~\ref{thm.dperm.prime}${}^{\bf\prime}$.

\begin{cor}[{Conjecture~\ref{conj.RZ}${}^{\bf\prime}$}]
Recall the polynomials $G_n$,
defined in \eqref{eq.def.Gn.RZ}/\eqref{eq.Gn.stats},
\begin{eqnarray}
	G_n(x,y,\bar{x},\bar{y})  & = & \sum_{\sigma \in \dperm^{\rm o}_{2n}}
      x^{{\rm comi}(\sigma)}
      y^{{\rm lema}(\sigma)}
      \bar{x}^{{\rm cemi}(\sigma)}
      \bar{y}^{{\rm remi}(\sigma)}\nonumber\\
	&=&
	\sum_{\sigma \in \dperm^{\rm o}_{2n}}
      x^{\minval(\sigma)}
      y^{{\ereccpeak'}(\sigma)}
      \bar{x}^{{\evennrfix}(\sigma)}
      \bar{y}^{{\eareccdfall'}(\sigma)}\nonumber\;.
\end{eqnarray}
%and let $G_0(x,y,\bar{x},\bar{y})  \;=\; 1$.
The ordinary generating functions of $G_n$
has the S-type continued fraction
\begin{equation}
        1 + \sum_{n=1}^{\infty} G_n(x,y,\bar{x},\bar{y}) t^n \; = \; 
        \cfrac{1}{1-\cfrac{xy t}{1-\cfrac{1(\bar{x} +\bar{y})t}{1-\cfrac{(x+1)(y+1)t}{1-\cfrac{2(\bar{x} + \bar{y} + 1)}{1-\cfrac{(x+2)(y+2)t}{1-\cfrac{3(\bar{x}+\bar{y}+2)t}{\cdots}}}}}}}
\end{equation}

with coefficients
\begin{subeqnarray}
	\alpha_{2k-1} & = &  (x \,+\, k-1)\: (y\,+\,k-1)
        \\[1mm]
	\alpha_{2k}   & = & k\: (\bar{x} \,+\, \bar{y} \,+\, k-1)
\end{subeqnarray}
\end{cor}

\begin{proof} It is evident from 
\eqref{eq.Gn.stats}/\eqref{def.Pntilde.prime}
that $G_n$ can be obtained from  $\widetilde{P}_n'$
by specialising $\ytilde_1=x, x_1 = y, \we = \bar{x}, x_2 = \bar{y}$
and $\wo=\zo =0$, and setting all other variables to $1$.
This along with Theorems~\ref{thm.conjDS.primeprime}${}^{\bf\prime}$/\ref{thm.dperm.prime}${}^{\bf\prime}$
proves the result.
\end{proof}

%We also mention that Theorems~\ref{blah} can be rewritten in a similar way,
%replacing $y_1,v_1$ by $\ytilde_1,\vtilde_1$ as in \eqref{def.Pntilde}
%{\em and}\/ replacing $y_2,v_2$ by $\yhat_2,\vhat_2$
%as in \eqref{def.Pnhathat};
%for brevity we leave this reformulation to the reader.

\section{Proof preliminaries}
\label{sec.prelimproofs}

A powerful way to prove continued fraction results is Flajolet's 
\cite{Flajolet_80}
combinatorial interpretation of continued fractions
which interprets Jacobi and Stieltjes-type continued fractions 
in terms of Motzkin and Dyck paths
and its generalisation
\cite{Fusy_15,Oste_15,Josuat-Verges_18,Sokal_totalpos,Elvey-Price-Sokal_wardpoly}
to Schr\"oder paths.
One then uses a bijection from
the combinatorial object of study to labelled paths. 
In our situtation,
the later bijections are
the variant Foata--Zeilberger bijection used in 
\cite[Section~6.1]{Sokal-Zeng_masterpoly} (for permutations),
and the two bijections in \cite[Section~6]{Deb-Sokal} (for D-permutations).
We begin by reviewing briefly these two ingredients.

Our exposition in this section will closely follow 
\cite{Sokal-Zeng_masterpoly, Deb-Sokal}. 

\subsection{Combinatorial interpretation of continued fractions}
   \label{subsec.prelimproofs.1}

Recall that a \textbfit{Motzkin path} of length $n \ge 0$
is a path $\omega = (\omega_0,\ldots,\omega_n)$
in the right quadrant $\N \times \N$,
starting at $\omega_0 = (0,0)$ and ending at $\omega_n = (n,0)$,
whose steps $s_j = \omega_j - \omega_{j-1}$
are $(1,1)$ [``rise'' or ``up step''], $(1,-1)$ [``fall'' or ``down step'']
or $(1,0)$ [``level step''].
We write $h_j$ for the \textbfit{height} of the Motzkin path at abscissa~$j$,
i.e.\ $\omega_j = (j,h_j)$;
note in particular that $h_0 = h_n = 0$.
We write $\scrm_n$ for the set of Motzkin paths of length~$n$,
and $\scrm = \bigcup_{n=0}^\infty \scrm_n$.
A Motzkin path is called a \textbfit{Dyck path} if it has no level steps.
A Dyck path always has even length;
we write $\scrd_{2n}$ for the set of Dyck paths of length~$2n$,
and $\scrd = \bigcup_{n=0}^\infty \scrd_{2n}$.

Let ${\bf a} = (a_i)_{i \ge 0}$, ${\bf b} = (b_i)_{i \ge 1}$
and ${\bf c} = (c_i)_{i \ge 0}$ be indeterminates;
we will work in the ring $\Z[[{\bf a},{\bf b},{\bf c}]]$
of formal power series in these indeterminates.
To each Motzkin path $\omega$ we assign a weight
$W(\omega) \in \Z[{\bf a},{\bf b},{\bf c}]$
that is the product of the weights for the individual steps,
where a rise starting at height~$i$ gets weight~$a_i$,
a~fall starting at height~$i$ gets weight~$b_i$,
and a level step at height~$i$ gets weight~$c_i$.
Flajolet \cite{Flajolet_80} showed that
the generating function of Motzkin paths
can be expressed as a continued fraction:

\begin{thm}[Flajolet's master theorem]
   \label{thm.flajolet}
We have
\be
   \sum_{\omega \in \scrm}  W(\omega)
   \;=\;
   \cfrac{1}{1 - c_0 - \cfrac{a_0 b_1}{1 - c_1 - \cfrac{a_1 b_2}{1- c_2 - \cfrac{a_2 b_3}{1- \cdots}}}}
 \label{eq.thm.flajolet}
\ee
as an identity in $\Z[[{\bf a},{\bf b},{\bf c}]]$.
\end{thm}

In particular, if $a_{i-1} b_i = \beta_i t^2$ and $c_i = \gamma_i t$
(note that the parameter $t$ is conjugate to the length of the Motzkin path),
we have
\be
   \sum_{n=0}^\infty t^n \sum_{\omega \in \scrm_n}  W(\omega)
   \;=\;
   \cfrac{1}{1 - \gamma_0 t - \cfrac{\beta_1 t^2}{1 - \gamma_1 t - \cfrac{\beta_2 t^2}{1 - \cdots}}}
   \;\,,
 \label{eq.flajolet.motzkin}
\ee
so that the generating function of Motzkin paths with height-dependent weights
is given by the J-type continued fraction \eqref{def.Jtype}.
Similarly, if $a_{i-1} b_i = \alpha_i t$ and $c_i = 0$
(note that $t$ is now conjugate to the semi-length of the Dyck path), we have
\be
   \sum_{n=0}^\infty t^n \sum_{\omega \in \scrd_{2n}}  W(\omega)
   \;=\;
   \cfrac{1}{1 - \cfrac{\alpha_1 t}{1 - \cfrac{\alpha_2 t}{1 - \cdots}}}
   \;\,,
 \label{eq.flajolet.dyck}
\ee
so that the generating function of Dyck paths with height-dependent weights
is given by the S-type continued fraction \eqref{def.Stype}.

Let us now show how to handle Schr\"oder paths within this framework.
A \textbfit{Schr\"oder path} of length $2n$ ($n \ge 0$)
is a path $\omega = (\omega_0,\ldots,\omega_{2n})$
in the right quadrant $\N \times \N$,
starting at $\omega_0 = (0,0)$ and ending at $\omega_{2n} = (2n,0)$,
whose steps are $(1,1)$ [``rise'' or ``up step''],
$(1,-1)$ [``fall'' or ``down step'']
or $(2,0)$ [``long level step''].
We write $s_j$ for the step starting at abscissa $j-1$.
If the step $s_j$ is a rise or a fall,
we set $s_j = \omega_j - \omega_{j-1}$ as before.
If the step $s_j$ is a long level step,
we set $s_j = \omega_{j+1} - \omega_{j-1}$ and leave $\omega_j$ undefined;
furthermore, in this case there is no step $s_{j+1}$.
We write $h_j$ for the height of the Schr\"oder path at abscissa~$j$
whenever this is defined, i.e.\ $\omega_j = (j,h_j)$.
Please note that $\omega_{2n} = (2n,0)$ and $h_{2n} = 0$
are always well-defined,
because there cannot be a long level step starting at abscissa $2n-1$.
Note also that a long level step at even (resp.~odd) height
can occur only at an odd-numbered (resp.~even-numbered) step.
We write $\scrs_{2n}$ for the set of Schr\"oder paths of length~$2n$,
and $\scrs = \bigcup_{n=0}^\infty \scrs_{2n}$.

There is an obvious bijection between Schr\"oder paths and Motzkin paths:
namely, every long level step is mapped onto a level step.
If we apply Flajolet's master theorem with
$a_{i-1} b_i = \alpha_i t$ and $c_i = \delta_{i+1} t$
to the resulting Motzkin path
(note that $t$ is now conjugate to the semi-length
 of the underlying Schr\"oder path),
we obtain
\be
   \sum_{n=0}^\infty t^n \sum_{\omega \in \scrs_{2n}}  W(\omega)
   \;=\;
   \cfrac{1}{1 - \delta_1 t - \cfrac{\alpha_1 t}{1 - \delta_2 t - \cfrac{\alpha_2 t}{1 - \cdots}}}
   \;\,,
 \label{eq.flajolet.schroder}
\ee
so that the generating function of Schr\"oder paths
with height-dependent weights
is given by the T-type continued fraction \eqref{def.Ttype}.
More precisely, every rise gets a weight~1,
every fall starting at height~$i$ gets a weight $\alpha_i$,
and every long level step at height~$i$ gets a weight $\delta_{i+1}$.
This combinatorial interpretation of T-fractions in terms of Schr\"oder paths
was found recently by several authors
\cite{Fusy_15,Oste_15,Josuat-Verges_18,Sokal_totalpos}.

\subsection{Labelled Dyck, Motzkin and Schr\"oder paths}

Let $\bfscra = (\scra_h)_{h \ge 0}$, $\bfscrb = (\scrb_h)_{h \ge 1}$
and $\bfscrc = (\scrc_h)_{h \ge 0}$ be sequences of finite sets.
An
\textbfit{$(\bfscra,\bfscrb,\bfscrc)$-labelled Motzkin path of length $\bm{n}$}
is a pair $(\omega,\xi)$
where $\omega = (\omega_0,\ldots,\omega_n)$
is a Motzkin path of length $n$,
%% from $\omega_0 = (0,0)$ to $\omega_n = (n,0)$,
and $\xi = (\xi_1,\ldots,\xi_n)$ is a sequence satisfying
\be
   \xi_i  \:\in\:
   \begin{cases}
       \scra(h_{i-1})  & \textrm{if step $i$ is a rise (i.e.\ $h_i = h_{i-1} + 1$)}
              \\[1mm]
       \scrb(h_{i-1})  & \textrm{if step $i$ is a fall (i.e.\ $h_i = h_{i-1} - 1$)}
              \\[1mm]
       \scrc(h_{i-1})  & \textrm{if step $i$ is a level step (i.e.\ $h_i = h_{i-1}$)}
   \end{cases}
 \label{eq.xi.ineq}
\ee
where $h_{i-1}$ (resp.~$h_i$) is the height of the Motzkin path
before (resp.~after) step $i$.
[For typographical clarity
 we have here written $\scra(h)$ as a synonym for $\scra_h$, etc.]
We call $\xi_i$ the \textbfit{label} associated to step $i$.
We call the pair $(\omega,\xi)$
an \textbfit{$(\bfscra,\bfscrb)$-labelled Dyck path}
if $\omega$ is a Dyck path (in this case $\bfscrc$ plays no role).
We denote by $\scrm_n(\bfscra,\bfscrb,\bfscrc)$
the set of $(\bfscra,\bfscrb,\bfscrc)$-labelled Motzkin paths of length $n$,
and by $\scrd_{2n}(\bfscra,\bfscrb)$
the set of $(\bfscra,\bfscrb)$-labelled Dyck paths of length $2n$.

We define a \textbfit{$(\bfscra,\bfscrb,\bfscrc)$-labelled
Schr\"oder path}
in an analogous way;
now the sets $\scrc_h$ refer to long level steps.
We denote by $\scrs_{2n}(\bfscra,\bfscrb,\bfscrc)$
the set of $(\bfscra,\bfscrb,\bfscrc)$-labelled Schr\"oder paths
of length $2n$.

Let us stress that the sets $\scra_h$, $\scrb_h$ and $\scrc_h$ are allowed
to be empty.
Whenever this happens, the path $\omega$ is forbidden to take a step
of the specified kind starting at the specified height.

%% We shall also make use of multicolored Motzkin paths.
%% An {\em $\ell$-colored Motzkin path}\/ is simply a Motzkin path
%% in which each level step has been given a ``color''
%% from the set $\{1,2,\ldots,\ell\}$.
%% In other words, we distinguish $\ell$ different types of level steps.
%% An {\em $(\bfscra,\bfscrb,\bfscrc^{(1)},\ldots,\bfscrc^{(\ell)})$-labelled
%%  $\ell$-colored Motzkin path of length $n$}\/
%% is then defined in the obvious way,
%% where we use the sequence $\bfscrc^{(j)}$ to bound
%% the label $\xi_i$ when step $i$ is a level step of type $j$.
%% We denote by $\scrm_n(\bfscra,\bfscrb,\bfscrc^{(1)},\ldots,\bfscrc^{(\ell)})$
%% the set of $(\bfscra,\bfscrb,\bfscrc^{(1)},\ldots,\bfscrc^{(\ell)})$-labelled
%% $\ell$-colored Motzkin paths of length $n$.

\bigskip

\begin{sloppypar}
Following Flajolet \cite[Proposition~7A]{Flajolet_80},
we can state a ``master J-fraction'' for
$(\bfscra,\bfscrb,\bfscrc)$-labelled Motzkin paths.
Let ${\bf a} = (a_{h,\xi})_{h \ge 0 ,\; \xi \in \scra(h)}$,
${\bf b} = (b_{h,\xi})_{h \ge 1 ,\; \xi \in \scrb(h)}$
and \mbox{${\bf c} = (c_{h,\xi})_{h \ge 0 ,\; \xi \in \scrc(h)}$}
be indeterminates;
we give an $(\bfscra,\bfscrb,\bfscrc)$-labelled Motzkin path $(\omega,\xi)$
a weight $W(\omega,\xi)$
that is the product of the weights for the individual steps,
where a rise starting at height~$h$ with label $\xi$ gets weight~$a_{h,\xi}$,
a~fall starting at height~$h$ with label $\xi$ gets weight~$b_{h,\xi}$,
and a level step at height~$h$ with label $\xi$ gets weight~$c_{h,\xi}$.
Then:
\end{sloppypar}
\begin{thm}[Flajolet's master theorem for labelled Motzkin paths]
   \label{thm.flajolet_master_labeled_Motzkin}
We have
\be
   \sum_{n=0}^\infty t^n
   \!\!
   \sum_{(\omega,\xi) \in \scrm_n(\bfscra,\bfscrb,\bfscrc)} \!\!\!  W(\omega,\xi)
   \;=\;
   \cfrac{1}{1 - c_0 t - \cfrac{a_0 b_1 t^2}{1 - c_1 t - \cfrac{a_1 b_2 t^2}{1- c_2 t - \cfrac{a_2 b_3 t^2}{1- \cdots}}}}
\ee
as an identity in $\Z[{\bf a},{\bf b},{\bf c}][[t]]$, where
\be
   a_h  \;=\;  \sum_{\xi \in \scra(h)} a_{h,\xi}
   \;,\qquad
   b_h  \;=\;  \sum_{\xi \in \scrb(h)} b_{h,\xi}
   \;,\qquad
   c_h  \;=\;  \sum_{\xi \in \scrc(h)} c_{h,\xi}
   \;.
 \label{def.weights.akbkck}
\ee
\end{thm}

\noindent
This is an immediate consequence of Theorem~\ref{thm.flajolet}
together with the definitions.
% There is obviously also a similar theorem for
% $(\bfscra,\bfscrb,\bfscrc^{(1)},\ldots,\bfscrc^{(\ell)})$-labeled
% $\ell$-colored Motzkin paths,
% in which $c_k$ involves a sum over the colors of the level steps.

By specialising to ${\bf c} = \bzero$ and replacing $t^2$ by $t$,
we obtain the corresponding theorem
for $(\bfscra,\bfscrb)$-labelled Dyck paths:

\begin{cor}[Flajolet's master theorem for labelled Dyck paths]
   \label{cor.flajolet_master_labeled_Dyck}
We have
\be
   \sum_{n=0}^\infty t^n
   \!\!
   \sum_{(\omega,\xi) \in \scrd_{2n}(\bfscra,\bfscrb)} \!\!\!  W(\omega,\xi)
   \;=\;
   \cfrac{1}{1 - \cfrac{a_0 b_1 t}{1 - \cfrac{a_1 b_2 t}{1- \cfrac{a_2 b_3 t}{1- \cdots}}}}
\ee
as an identity in $\Z[{\bf a},{\bf b}][[t]]$, where
$a_h$ and $b_h$ are defined by \eqref{def.weights.akbkck}.
\end{cor}

Similarly, for labelled Schr\"oder paths we have:

\begin{thm}[Flajolet's master theorem for labelled Schr\"oder paths]
   \label{thm.flajolet_master_labeled_Schroder}
We have
\be
   \sum_{n=0}^\infty t^n
   \!\!
   \sum_{(\omega,\xi) \in \scrs_{2n}(\bfscra,\bfscrb,\bfscrc)} \!\!\!  W(\omega,\xi)
   \;=\;
   \cfrac{1}{1 - c_0 t - \cfrac{a_0 b_1 t}{1 - c_1 t - \cfrac{a_1 b_2 t}{1- c_2 t - \cfrac{a_2 b_3 t}{1- \cdots}}}}
\ee
as an identity in $\Z[{\bf a},{\bf b},{\bf c}][[t]]$, where
$a_h, b_h, c_h$ are defined by \eqref{def.weights.akbkck},
with $c_{h,\xi}$ now referring to long level steps.
\end{thm}

\section{Permutations: Proof of Theorems~\ref{thm.conj.SZ}, \ref{thm.pqgen}, \ref{thm.master}}
\label{sec.permutations.proofs}

Sokal and Zeng \cite[Section~6.1]{Sokal-Zeng_masterpoly}
used a variant of the Foata--Zeilberger bijection \cite{Foata_90}
to prove \cite[Theorems~2.1(a), 2.2, 2.5, 2.7 and 2.9]{Sokal-Zeng_masterpoly}
i.e., their ``first theorems'' for permutations.
We will provide a new interpretation to this bijection
in terms of Laguerre digraphs
and then use this interpretation
to prove our theorems for permutations.

We first recall Sokal and Zeng's bijection 
in Subsection~\ref{subsec.FZrecall.perm},
and then introduce our interpretation in 
Subsection~\ref{subsec.laguerre}.
In Subsection~\ref{subsec.FZ.laguerre.eg}, 
we write out our construction explicitly for our two running examples.
Finally, we complete our proofs in Subsection~\ref{subsec.computation}.

\subsection{Sokal--Zeng variant of the Foata--Zeilberger bijection}
\label{subsec.FZrecall.perm}

Sokal and Zeng employed a variant of the Foata--Zeilberger bijection
to prove \cite[Theorems~2.1(a), 2.2, 2.5, 2.7, and 2.9]{Sokal-Zeng_masterpoly},
i.e., their ``first theorems'' for permutations.
We begin by recalling this bijection
which is a correspondence between $\Sym_n$ and the set of
$(\bfscra, \bfscrb,\bfscrc)$-labelled
Motzkin paths of length $n$, where the labels $\xi_i$ lie in the sets
\begin{subeqnarray}
	\scra_h        & = & \{0,\ldots, h\}     \,\:\;\;\qquad\qquad\qquad\quad \qquad \qquad\qquad\qquad \qquad\hbox{for $h \ge 0$}  \\
	\scrb_h        & = &  \{0,\ldots,h-1\}   \:\qquad\qquad\qquad\quad  \qquad\qquad\qquad \qquad\quad\hbox{for $h \ge 1$}  \\
	\scrc_h  & = &  
	\left(\{1\}\times C_h^{(1)}\right)
	\,\cup\, 
	\left(\{2\}\times C_h^{(2)}\right)
	\,\cup\,
	\left(\{3\}\times C_h^{(3)}\right)
	\qquad\hbox{for $h \ge 0$}  
%   C_k^{(2)}  & = &  k         \qquad\;\;\,\hbox{for $k \ge 0$}  \\
%   C_k^{(3)}  & = &  1         \qquad\;\;\:\hbox{for $k \ge 0$}
 \label{def.abc}
\end{subeqnarray}
where
\begin{subeqnarray}
	C_h^{(1)} & = & \{0,\ldots,h-1\} \qquad\hbox{for $k \ge 0$} \\
	C_h^{(2)} & = & \{0,\ldots,h-1\} \qquad\hbox{for $k \ge 0$}  \\
	C_h^{(3)} & = & \{0\}  \qquad\qquad\qquad\;\:\hbox{for $k \ge 0$} 
\end{subeqnarray}

Notice that our convention for labels in this paper are slightly different from
\cite{Sokal-Zeng_masterpoly} and 
we instead follow those in 
\cite[Section~5]{Deb-Sokal}. A key difference 
is that our label set starts at $0$ and not at $1$. 
A level step that has label 
$\xi_h\in\{i\}\times C_h^{(i)}$
will be called a level step of type $i$ ($i=1,2,3$).
By a 3-coloured Motzkin path, we will refer to a Motzkin path $\omega$ 
where each of its level steps is assigned a type.
We will use $\overline{\omega}$ to denote a 3-coloured Motzkin path.

We will begin by recalling how the Motzkin path $\omega$ is defined;
then we will explain how the labels $\xi$ are defined;
next we sketch the proof that the mapping is indeed a bijection.
There were two more steps in \cite[Section~6.1]{Sokal-Zeng_masterpoly}
which we will not recall;
these were the translation of the various statistics from
$\Sym_n$ to labelled Motzkin paths;
and summing over labels $\xi$ to obtain the weight $W(\omega)$.
%associated to a Motzkin path $\omega$,
%which upon applying \eqref{eq.flajolet.motzkin}
%will yield Theorem~\ref{thm.permutations.Jtype.final1}.

\bigskip

{\bf Step 1: Definition of the Motzkin path.}
Given a permutation $\sigma \in \Sym_n$,
we classify the indices $i \in [n]$
according to the cycle classification.
We then define a path $\omega = (\omega_0,\ldots,\omega_n)$
starting at $\omega_0 = (0,0)$ and ending at $\omega_n = (n,0)$,
with steps $s_1,\ldots,s_n$, as follows:
\begin{itemize}
   \item If $i$ is a cycle valley, then $s_i$ is a rise.
   \item If $i$ is a cycle peak, then $s_i$ is a fall.
   \item If $i$ is a cycle double fall, then $s_i$ is a level step of type 1.
   \item If $i$ is a cycle double rise, then $s_i$ is a level step of type 2.
   \item If $i$ is a fixed point, then $s_i$ is a level step of type 3.
\end{itemize}
%Of course we need to prove that this is indeed a Motzkin path,
%i.e.\ that all the heights $h_i$ are nonnegative and that $h_n = 0$.
%We do this by obtaining a precise interpretation of the height $h_i$:

The fact that the resulting path is indeed a Motzkin path was proved by
providing an interpretation of the height $h_i$ which we recall here:

\begin{lem}[{{\cite[Lemma~6.1]{Sokal-Zeng_masterpoly}}}]
   \label{lemma.heights}
For $i \in [n+1]$ we have
\begin{subeqnarray}
   h_{i-1}  & = &  \# \{j < i \colon\:  \sigma(j) \ge i \}
       \slabel{eq.lemma.heights.a}  \\[2mm]
            & = &  \# \{j < i \colon\:  \sigma^{-1}(j) \ge i \}.
       \slabel{eq.lemma.heights.b}
 \label{eq.lemma.heights}
\end{subeqnarray}
\end{lem}

We also recall (\cite[eq.~(6.4)]{Sokal-Zeng_masterpoly})
which is an equivalent formulation of Equation~\eqref{eq.lemma.heights}: 
\begin{subeqnarray}
   h_{i}  & = &  \# \{j \leq i \colon\:  \sigma(j) > i \}
       \slabel{eq.lemma.heights.equiv.a}  \\[2mm]
            & = &  \# \{j \leq i \colon\:  \sinv(j) > i \}.
       \slabel{eq.lemma.heights.equiv.b}
 \label{eq.lemma.heights.equiv}
\end{subeqnarray}

Notice that if $i$ is a fixed point,
then by comparing \eqref{eq.lemma.heights.a}/\eqref{eq.lemma.heights.equiv.a}
with \eqref{def.level} 
we see that the height of the Motzkin path after (or before) step $i$
equals the number of pseudo-nestings of the fixed point:
\be
   h_{i-1}  \;=\;  h_i  \;=\;   \psnest(i,\sigma)
   \;.
 \label{eq.height.fix}
\ee

\bigskip

{\bf Step 2: Definition of the labels $\bm{\xi_i}$.}
We now recall the definition of the labels
\be
   \xi_i
   \;=\;
   \begin{cases}
        \#\{j \colon\: j < i \hbox{ and } \sigma(j) > \sigma(i) \}
           & \textrm{if $\sigma(i) > i$ \quad i.e. $i\in \Cval\cup \Cdrise$}
              \\[1mm]
        \#\{j \colon\: j > i \hbox{ and }  \sigma(j) < \sigma(i) \}
	& \textrm{if $\sigma(i) < i$ \quad i.e. $i\in \Cpeak\cup \Cdfall$}
              \\[1mm]
      0   & \textrm{if $\sigma(i) = i$ \quad i.e. $i\in\Fix$}
   \end{cases}
 \label{def.xi}
\ee
where recall that $\Cval$ is the set of all cycle valleys of $\sigma$
and likewise for the others.
For sake of brevity, we are abusing notation
for fixed points, cycle double rises and cycle double falls 
by dropping the first index of $\xi_i$.

\begin{sloppy}
Compare our definition~\eqref{def.xi} with \cite[(6.5)]{Sokal-Zeng_masterpoly}
and notice the shift from $1$ to $0$.
These definitions have a simple interpretation in terms of the
nesting statistics defined in (\ref{def.ucrossnestjk}b,d):
\end{sloppy}
\be
   \xi_i 
   \;=\;
   \begin{cases}
       \unest(i,\sigma) & \textrm{if $\sigma(i) > i$ \quad i.e. $i\in \Cval\cup \Cdrise$}
              \\[1mm]
	      \lnest(i,\sigma) & \textrm{if $\sigma(i) < i$ \quad i.e. $i\in \Cpeak\cup \Cdfall$}
              \\[1mm]
       0   & \textrm{if $\sigma(i) = i$ \quad i.e. $i\in\Fix$}
   \end{cases}
 \label{def.xi.bis}
\ee
To verify that the inequalities \eqref{eq.xi.ineq}/\eqref{def.abc}
are satisfied; to do this, we interpret $h_{i-1} - \xi_i$
in terms of the crossing statistics defined in (\ref{def.ucrossnestjk}a,c):

\begin{lem}[Crossing statistics]
   \label{lemma.crossing}
We have
\begin{eqnarray}
   h_{i-1}  - \xi_i   & = &  \ucross(i,\sigma)
         \quad\textrm{\rm if $i \in \Cval$}
     \label{eq.lemma.ucrosscval} \\[1mm]
   h_{i-1}-1 - \xi_i   & = &  \ucross(i,\sigma)
         \quad\textrm{\rm if $i \in \Cdrise$}
     \label{eq.lemma.ucrosscdrise} \\[1mm]
   h_{i-1}-1 - \xi_i   & = &  \lcross(i,\sigma)
         \quad\:\textrm{\rm if $i \in \Cpeak \,\cup\, \Cdfall$}
     \label{eq.lemma.lcrosscpeak}
\end{eqnarray}
\end{lem}
\noindent Again compare Lemma~\ref{lemma.crossing} with 
\cite[Lemma~6.2]{Sokal-Zeng_masterpoly} to see how the shift of $1$
affects these quantities.

Since the quantities \eqref{eq.lemma.ucrosscval}--\eqref{eq.lemma.lcrosscpeak}
are manifestly nonnegative,
it follows immediately 
that the inequalities \eqref{eq.xi.ineq}/\eqref{def.abc} are satisfied.

%We now recall
%(partially) the interpretion of the labels $\xi_i$
%in terms of the bipartite digraphs in \cite[p.~96]{Sokal-Zeng_masterpoly}
%Lemma~\ref{lemma.heights} (Figure~\ref{fig.FZ}).
%First recall that $h_{i-1}$ equals
%the number of unconnected dots in the top row after $i-1$ steps,
%and also equals the number of unconnected dots in the bottom row
%after $i-1$ steps.
%Now, if $i$ is a cycle double fall, then at stage $i$ we add an arrow
%from $i$ on the top row to an unconnected dot $j'$ on the bottom row,
%where $j = \sigma(i) < i$;
%and if $i$ is a cycle peak, then we add the just-mentioned arrow
%and also add an arrow from an unconnected dot $k$ on the top row to $i'$,
%where $k = \sigma^{-1}(i) < i$.
%We now claim that, in these two cases, $\xi_i$ is the index
%of the unconnected dot $j'$ among all the unconnected dots on the bottom row:
%that is, $\xi_i = r$ if and only if $j'$
%is the $r$th unconnected dot on the bottom row, reading from left to right.
%Indeed, by definition $\xi_i$
%equals $\#\{k \colon\: k > i \hbox{ and }  \sigma(k) < \sigma(i) \}$,
%which is precisely the number of unconnected dots on the bottom row
%to the left of $j' = \sigma(i)'$.
%
%For cycle double rises and cycle valleys, by contrast,
%the labels $\xi_i$ do not have any simple interpretation
%in terms of the bipartite digraph {\em when read from left to right}\/,
%as they depend on the value of $\sigma(i)$,
%which is $> i$ and hence unknown at time $i$.

%(See also the Remark after Step~2 in Section~\ref{subsec.permutations.J.v2}.)

\bigskip

{\bf Step 3: Proof of bijection.}
We recall the description of the inverse map for the mapping 
$\sigma \mapsto (\omega,\xi)$.

First, some preliminaries:
Given a permutation $\sigma \in \Sym_n$,
we define five subsets of~$[n]$:
\begin{subeqnarray}
   F & = & \{i \colon\: \sigma(i) > i\}  \;=\; \hbox{positions of excedances}
       \\[1mm]
   F' & = & \{i \colon\: i > \sigma^{-1}(i)\}  \;=\; \hbox{values of excedances}
       \\[1mm]
   G & = & \{i \colon\: \sigma(i) < i\}  \;=\; \hbox{positions of anti-excedances}
       \\[1mm]
   G' & = & \{i \colon\: i < \sigma^{-1}(i)\}  \;=\; \hbox{values of anti-excedances}
       \\[1mm]
   H  & = & \{i \colon\: \sigma(i) = i\}  \;=\;  \hbox{fixed points}
\label{eq.steps2excedanceclassification}
\end{subeqnarray}
Let us observe that
\begin{subeqnarray}
   F \cap F'  & = &  \hbox{cycle double rises}  \\[1mm]
   G \cap G'  & = &  \hbox{cycle double falls}  \\[1mm]
   F \cap G'  & = &  \hbox{cycle valleys}  \\[1mm]
   F' \cap G  & = &  \hbox{cycle peaks}  \\[1mm]
   F \cap G   & = &  \emptyset \\[1mm]
   F' \cap G'   & = &  \emptyset
\label{eq.steps2cycleclassification}
\end{subeqnarray}
and of course $H$ is disjoint from $F,F',G,G'$.

Let us also recall the notion of an {\em inversion table}\/:
Let $S$ be a totally ordered set of cardinality $k$,
and let $\bm{x} = (x_1,\ldots,x_k)$ be an enumeration of $S$;
then the (left-to-right) inversion table corresponding to $\bm{x}$
is the sequence $\bm{p} = (p_1,\ldots,p_k)$ of nonnegative integers
defined by $p_\alpha = \#\{\beta < \alpha \colon\: x_\beta > x_\alpha \}$.
Note that $0 \le p_\alpha \le \alpha-1$ for all $\alpha \in [k]$,
so there are exactly $k!$ possible inversion tables.
Given the inversion table $\bm{p}$,
we can reconstruct the sequence $\bm{x}$ 
by working from right to left, as follows:
There are $p_k$ elements of $S$ larger than $x_k$,
so $x_k$ must be the $(p_k+1)$th largest element of $S$.
Then there are $p_{k-1}$ elements of $S \setminus \{x_k\}$
larger than $x_{k-1}$,
so $x_{k-1}$ must be the $(p_{k-1}+1)$th largest element
of $S \setminus \{x_k\}$.
And so forth.
[Analogously, the right-to-left inversion table corresponding to $\bm{x}$
is the sequence $\bm{p} = (p_1,\ldots,p_k)$ of nonnegative integers
defined by $p_\alpha = \#\{\beta > \alpha \colon\: x_\beta < x_\alpha \}$.]

With these preliminaries out of the way,
we can now describe the map $(\omega,\xi) \mapsto \sigma$.
Given a Motzkin path along with an assignment of types for its level steps $\overline{\omega}$,
we read off which indices $i$ correspond to cycle valleys,
cycle peaks, cycle double falls, cycle double rises, and fixed points;
this allows us to reconstruct the sets $F,F',G,G',H$.
We now use the labels $\xi$ to reconstruct the maps
$\sigma \restrict F \colon\: F \to F'$ and
$\sigma \restrict G \colon\: G \to G'$, as follows:
Let $i_1,\ldots,i_k$ be the elements of $F$ written in increasing order;
then the sequence $j_1,\ldots,j_k$ defined by $j_\alpha = \sigma(i_\alpha)$
is a listing of $F'$ whose inversion table is given by
$p_\alpha = \xi_{i_\alpha}$:
this is the content of \eqref{def.xi} in the case $\sigma(i) > i$.
So we can use $\xi \restrict F$ to reconstruct $\sigma \restrict F$.
In a similar way we can use $\xi \restrict G$
to reconstruct $\sigma \restrict G$,
but now we must use the right-to-left inversion table
because of how \eqref{def.xi} is written in the case $\sigma(i) < i$.

\bigskip

\begin{rem}
\label{rem.original.FZ}
In the original version of the bijection by Foata and Zeilberger
\cite{Foata_90}, 
fixed points were treated differently and were clubbed together with 
anti-excedances (thus forming the set of non-excedances).
Also, the notion of inversion table was different;
unlike the ``position-based'' inversion table by 
Sokal and Zeng \cite{Sokal-Zeng_masterpoly},
Foata and Zeilberger used a ``value-based'' inversion table.
We will employ both versions of inversion tables in Section~\ref{sec.dperm.proofs}.
\myendremark
\end{rem}

We will now look at the Foata--Zeilberger bijection for our two running examples in 
Sections~\ref{subsubsec.FZbij.eg.1} and~\ref{subsubsec.FZbij.eg.2}.

\subsubsection{Running example 1}
\label{subsubsec.FZbij.eg.1}

First let us take
$\sigma = 9\,\, 3\,\, 7\,\, 4\,\, 6\,\, 11\,\, 5\,\, 8\,\, 10\,\, 1\,\, 2
           = (1,9,10)\,(2,3,7,5,6,11)\,(4)\,(8) \in \Sym_{11}$
which was depicted in Figure~\ref{fig.pictorial}.
From the cycle classification of $\sigma$,
which was recorded in Equation~\eqref{eq.example.1.cycle.classification},
we immediately obtain the Motzkin path $\omega$ corresponding to $\sigma$ 
which has been drawn in Figure~\ref{fig.motzkin.running.example.1}.

\begin{figure}[!h]
\centering
\begin{tikzpicture}[scale = 1]
\draw[help lines, color=gray!30, dashed] (-0.5,-0.5) grid (11.5,4.5);
\draw[->] (0,0)--(11.5,0) node[right]{$x$};
\draw[->] (0,0)--(0,4.5) node[above]{$y$};
%\draw[gray,dotted] (-1,-1) grid (6,2);
\draw[rounded corners=1, color=red, line width=1] (0,0)-- (1,1);
\draw[rounded corners=1, color=red, line width=1] (1,1)-- (2,2);
\draw[rounded corners=1, color=red, line width=1] (2,2)-- (3,2);
\draw[rounded corners=1, color=red, line width=1] (3,2)-- (4,2);
\draw[rounded corners=1, color=red, line width=1] (4,2)-- (5,3);
\draw[rounded corners=1, color=red, line width=1] (5,3)-- (6,3);
\draw[rounded corners=1, color=red, line width=1] (6,3)-- (7,2);
\draw[rounded corners=1, color=red, line width=1] (7,2)-- (8,2);
\draw[rounded corners=1, color=red, line width=1] (8,2)-- (9,2);
\draw[rounded corners=1, color=red, line width=1] (9,2)-- (10,1);
\draw[rounded corners=1, color=red, line width=1] (10,1)-- (11,0);

\node[circle, fill=red, minimum size=5pt, inner sep=1pt] at (0,0) {};
\node[circle, fill=red, minimum size=5pt, inner sep=1pt] at (1,1) {};
\node[circle, fill=red, minimum size=5pt, inner sep=1pt] at (2,2) {};
\node[circle, fill=red, minimum size=5pt, inner sep=1pt] at (3,2) {};
\node[circle, fill=red, minimum size=5pt, inner sep=1pt] at (4,2) {};
\node[circle, fill=red, minimum size=5pt, inner sep=1pt] at (5,3) {};
\node[circle, fill=red, minimum size=5pt, inner sep=1pt] at (6,3) {};
\node[circle, fill=red, minimum size=5pt, inner sep=1pt] at (7,2) {};
\node[circle, fill=red, minimum size=5pt, inner sep=1pt] at (8,2) {};
\node[circle, fill=red, minimum size=5pt, inner sep=1pt] at (9,2) {};
\node[circle, fill=red, minimum size=5pt, inner sep=1pt] at (10,1) {};
\node[circle, fill=red, minimum size=5pt, inner sep=1pt] at (11,0) {};

\node[circle, fill=black, minimum size=5pt, inner sep=1pt] at (1,0) {};
\node[circle, fill=black, minimum size=5pt, inner sep=1pt] at (2,0) {};
\node[circle, fill=black, minimum size=5pt, inner sep=1pt] at (3,0) {};
\node[circle, fill=black, minimum size=5pt, inner sep=1pt] at (4,0) {};
\node[circle, fill=black, minimum size=5pt, inner sep=1pt] at (5,0) {};
\node[circle, fill=black, minimum size=5pt, inner sep=1pt] at (6,0) {};
\node[circle, fill=black, minimum size=5pt, inner sep=1pt] at (7,0) {};
\node[circle, fill=black, minimum size=5pt, inner sep=1pt] at (8,0) {};
\node[circle, fill=black, minimum size=5pt, inner sep=1pt] at (9,0) {};
\node[circle, fill=black, minimum size=5pt, inner sep=1pt] at (10,0) {};

\node[circle, fill=black, minimum size=5pt, inner sep=1pt] at (0,1) {};
\node[circle, fill=black, minimum size=5pt, inner sep=1pt] at (0,2) {};
\node[circle, fill=black, minimum size=5pt, inner sep=1pt] at (0,3) {};
\node[circle, fill=black, minimum size=5pt, inner sep=1pt] at (0,4) {};

\node[] at (0,-0.5) {$0$};
\node[] at (1,-0.5) {$1$};
\node[] at (2,-0.5) {$2$};
\node[] at (3,-0.5) {$3$};
\node[] at (4,-0.5) {$4$};
\node[] at (5,-0.5) {$5$};
\node[] at (6,-0.5) {$6$};
\node[] at (7,-0.5) {$7$};
\node[] at (8,-0.5) {$8$};
\node[] at (9,-0.5) {$9$};
\node[] at (10,-0.5) {$10$};
\node[] at (11,-0.5) {$11$};

\node[] at (-0.5,0) {$0$};
\node[] at (-0.5,1) {$1$};
\node[] at (-0.5,2) {$2$};
\node[] at (-0.5,3) {$3$};
\node[] at (-0.5,4) {$4$};

%
%\draw[rounded corners=1, color=blue, line width=1] (0,0)--(1,0);
%\draw[rounded corners=1, color=black, line width=1] (1,0)--(2,1);
%\draw[rounded corners=1, color=green, line width=1] (2,1)--(3,1);
%\draw[rounded corners=1, color=red, line width=1] (3,1)--(4,1);
%\draw[rounded corners=1, color=blue, line width=1] (4,1)--(5,1);
%\draw[rounded corners=1, color=black, line width=1] (5,1)--(6,0);
%\draw[rounded corners=1, color=blue, line width=1] (6,0)--(7,0);
%\draw[rounded corners=1, color=black, line width=1] (7,0)--(8,-1);
\end{tikzpicture}
\caption{Motzkin path $\omega$ corresponding to the permutation
$\sigma = 9\,\, 3\,\, 7\,\, 4\,\, 6\,\, 11\,\, 5\,\, 8\,\, 10\,\, 1\,\, 2
           = (1,9,10)\,(2,3,7,5,6,11)\,(4)\,(8)$.
}
	\label{fig.motzkin.running.example.1}
\end{figure}

The sets $F,F',G,G',H$ are 
\begin{subeqnarray}
	F & = & \{1, 2, 3, 5, 6, 9\} \\ 
	F' & = & \{3, 6, 7, 9, 10, 11\} \\ 
	G & = & \{7, 10, 11\} \\ 
	G' & = & \{1, 2, 5\} \\ 
	H & = & \{4, 8\} 
\label{eq.FGH.running.example.1}
\end{subeqnarray}

and the labels $\xi_i$ are given in Equation~\eqref{eq.labels.eg.1}. 

\begin{subeqnarray}
\begin{array}{r}
i\in F = \{1, 2, 3, 5, 6, 9\}  \\[3mm]
F' = \{\sigma(i)\; | \: i\in F\}   \\[3mm]
\text{Left-to-right inversion table:} \; \xi_{i}   \\
\end{array}
	&&\hspace*{-6mm}
\begin{pmatrix}
\; 1 & 2 & 3 & 5 & 6 & 9 \; \\[3mm]
\; 9 & 3 & 7 & 6 & 11 & 10 \; \\[3mm]
\; 0 & 1 & (2,1) & 2 & (2,0) & (2,1)  \; 
\end{pmatrix}
\\[5mm]
\begin{array}{r}
i\in G = \{7, 10, 11\}  \\[3mm]
G' = \{\sigma(i)\; | \: i\in G\}   \\[3mm]
\text{Right-to-left inversion table:} \; \xi_{i}   \\
\end{array}
	&&\hspace*{-6mm}
\begin{pmatrix}
\; 7 & 10 & 11 \; \\[3mm]
\; 5 & 1 & 2 \; \\[3mm]
\; 2 & 0 & 0
\end{pmatrix}
\label{eq.labels.eg.1}
\end{subeqnarray}

%{\bf DO THE FIRST INDICES LOOK UGLY???????????  WHAT DO I DO ABOUT EQUATION NUMBERING VS FIGURE ENVIRONMENT!!!!!!!!!}

%{\bf THEN DEMONSTRATE INVERSE BIJECTION!!!!!!!}

\subsubsection{Running example 2}
\label{subsubsec.FZbij.eg.2}

Next let us take
\begin{eqnarray}
	\sigma &=& 7\, 1\, 9\, 2\, 5\, 4\, 8\, 6\, 10\, 3\, 11\, 12\, 14\, 13\, \nonumber\\
	  &=& (1,7,8,6,4,2)\,(3,9,10)\,(5)\,(11)\,(12)\,(13,14) \in \Sym_{14}
\end{eqnarray}
which was depicted in Figure~\ref{fig.pictorial.2}.
From the cycle classification of $\sigma$,
which was recorded in Equation~\eqref{eq.example.2.cycle.classification},
we immediately obtain the Motzkin path $\omega$ corresponding to $\sigma$
which has been drawn in Figure~\ref{fig.motzkin.running.example.2}.

\begin{figure}[!h]
\centering
\begin{tikzpicture}[scale = 0.8]
\draw[help lines, color=gray!30, dashed] (-0.5,-0.5) grid (14.5,4.5);
\draw[->] (0,0)--(14.5,0) node[right]{$x$};
\draw[->] (0,0)--(0,4.5) node[above]{$y$};
%\draw[gray,dotted] (-1,-1) grid (6,2);

\node[circle, fill=black, minimum size=5pt, inner sep=1pt] at (1,0) {};
\node[circle, fill=black, minimum size=5pt, inner sep=1pt] at (2,0) {};
\node[circle, fill=black, minimum size=5pt, inner sep=1pt] at (3,0) {};
\node[circle, fill=black, minimum size=5pt, inner sep=1pt] at (4,0) {};
\node[circle, fill=black, minimum size=5pt, inner sep=1pt] at (5,0) {};
\node[circle, fill=black, minimum size=5pt, inner sep=1pt] at (6,0) {};
\node[circle, fill=black, minimum size=5pt, inner sep=1pt] at (7,0) {};
\node[circle, fill=black, minimum size=5pt, inner sep=1pt] at (8,0) {};
\node[circle, fill=black, minimum size=5pt, inner sep=1pt] at (9,0) {};
\node[circle, fill=black, minimum size=5pt, inner sep=1pt] at (10,0) {};
\node[circle, fill=black, minimum size=5pt, inner sep=1pt] at (11,0) {};
\node[circle, fill=black, minimum size=5pt, inner sep=1pt] at (12,0) {};
\node[circle, fill=black, minimum size=5pt, inner sep=1pt] at (13,0) {};
\node[circle, fill=black, minimum size=5pt, inner sep=1pt] at (14,0) {};

\node[circle, fill=black, minimum size=5pt, inner sep=1pt] at (0,1) {};
\node[circle, fill=black, minimum size=5pt, inner sep=1pt] at (0,2) {};
\node[circle, fill=black, minimum size=5pt, inner sep=1pt] at (0,3) {};
\node[circle, fill=black, minimum size=5pt, inner sep=1pt] at (0,4) {};

\node[] at (0,-0.5) {$0$};
\node[] at (1,-0.5) {$1$};
\node[] at (2,-0.5) {$2$};
\node[] at (3,-0.5) {$3$};
\node[] at (4,-0.5) {$4$};
\node[] at (5,-0.5) {$5$};
\node[] at (6,-0.5) {$6$};
\node[] at (7,-0.5) {$7$};
\node[] at (8,-0.5) {$8$};
\node[] at (9,-0.5) {$9$};
\node[] at (10,-0.5) {$10$};
\node[] at (11,-0.5) {$11$};
\node[] at (12,-0.5) {$12$};
\node[] at (13,-0.5) {$13$};
\node[] at (14,-0.5) {$14$};

\node[] at (-0.5,0) {$0$};
\node[] at (-0.5,1) {$1$};
\node[] at (-0.5,2) {$2$};
\node[] at (-0.5,3) {$3$};
\node[] at (-0.5,4) {$4$};

\draw[rounded corners=1, color=red, line width=1] (0,0)-- (1,1);
\draw[rounded corners=1, color=red, line width=1] (1,1)-- (2,1);
\draw[rounded corners=1, color=red, line width=1] (2,1)-- (3,2);
\draw[rounded corners=1, color=red, line width=1] (3,2)-- (4,2);
\draw[rounded corners=1, color=red, line width=1] (4,2)-- (5,2);
\draw[rounded corners=1, color=red, line width=1] (5,2)-- (6,2);
\draw[rounded corners=1, color=red, line width=1] (6,2)-- (7,2);
\draw[rounded corners=1, color=red, line width=1] (7,2)-- (8,1);
\draw[rounded corners=1, color=red, line width=1] (8,1)-- (9,1);
\draw[rounded corners=1, color=red, line width=1] (9,1)-- (10,0);
\draw[rounded corners=1, color=red, line width=1] (10,0)-- (11,0);
\draw[rounded corners=1, color=red, line width=1] (11,0)-- (12,0);
\draw[rounded corners=1, color=red, line width=1] (12,0)-- (13,1);
\draw[rounded corners=1, color=red, line width=1] (13,1)-- (14,0);

\node[circle, fill=red, minimum size=5pt, inner sep=1pt] at (0,0) {};
\node[circle, fill=red, minimum size=5pt, inner sep=1pt] at (1,1) {};
\node[circle, fill=red, minimum size=5pt, inner sep=1pt] at (2,1) {};
\node[circle, fill=red, minimum size=5pt, inner sep=1pt] at (3,2) {};
\node[circle, fill=red, minimum size=5pt, inner sep=1pt] at (4,2) {};
\node[circle, fill=red, minimum size=5pt, inner sep=1pt] at (5,2) {};
\node[circle, fill=red, minimum size=5pt, inner sep=1pt] at (6,2) {};
\node[circle, fill=red, minimum size=5pt, inner sep=1pt] at (7,2) {};
\node[circle, fill=red, minimum size=5pt, inner sep=1pt] at (8,1) {};
\node[circle, fill=red, minimum size=5pt, inner sep=1pt] at (9,1) {};
\node[circle, fill=red, minimum size=5pt, inner sep=1pt] at (10,0) {};
\node[circle, fill=red, minimum size=5pt, inner sep=1pt] at (11,0) {};
\node[circle, fill=red, minimum size=5pt, inner sep=1pt] at (12,0) {};
\node[circle, fill=red, minimum size=5pt, inner sep=1pt] at (13,1) {};
\node[circle, fill=red, minimum size=5pt, inner sep=1pt] at (14,0) {};

%
%\draw[rounded corners=1, color=blue, line width=1] (0,0)--(1,0);
%\draw[rounded corners=1, color=black, line width=1] (1,0)--(2,1);
%\draw[rounded corners=1, color=green, line width=1] (2,1)--(3,1);
%\draw[rounded corners=1, color=red, line width=1] (3,1)--(4,1);
%\draw[rounded corners=1, color=blue, line width=1] (4,1)--(5,1);
%\draw[rounded corners=1, color=black, line width=1] (5,1)--(6,0);
%\draw[rounded corners=1, color=blue, line width=1] (6,0)--(7,0);
%\draw[rounded corners=1, color=black, line width=1] (7,0)--(8,-1);
\end{tikzpicture}
\caption{Motzkin path $\omega$ corresponding to the permutation\\
$\sigma = 7\, 1\, 9\, 2\, 5\, 4\, 8\, 6\, 10\, 3\, 11\, 12\, 14\, 13\,
           = (1,7,8,6,4,2)\,(3,9,10)\,(5)\,(11)\,(12)\,(13,14) \in \Sym_{14}$.
}
	\label{fig.motzkin.running.example.2}
\end{figure}

The sets $F,F',G,G',H$ are
\begin{subeqnarray}
        F & = & \{1, 3, 7, 9, 13\} \\
        F' & = & \{7, 8, 9, 10, 14\} \\
        G & = & \{2, 4, 6, 8, 10, 14 \} \\
        G' & = & \{1, 2, 3, 4, 6, 13\} \\
        H & = & \{5,11,12\}
\label{eq.FGH.running.example.2}
\end{subeqnarray}

and the labels $\xi_i$ are given in Equation~\eqref{eq.labels.eg.2}.

\begin{subeqnarray}
\begin{array}{r}
i\in F = \{1, 3, 7, 9, 13\} \\[3mm]
F' = \{\sigma(i)\; | \: i\in F\}   \\[3mm]
\text{Left-to-right inversion table:} \; \xi_{i}   \\
\end{array}
&&\hspace*{-6mm}
\begin{pmatrix}
\; 1 &  3 & 7 & 9 & 13 \; \\[3mm]
\; 7 & 9 & 8 & 10 & 14 \; \\[3mm]
\; 0 & 0 & (2,2) & (2,0) & 0  \; 
\end{pmatrix} 
	\\[5mm]
\begin{array}{r}
i\in G = \{2, 4, 6, 8, 10, 14\}  \\[3mm]
G' = \{\sigma(i)\; | \: i\in G\}   \\[3mm]
\text{Right-to-left inversion table:} \; \xi_{i}   \\
\end{array}
	&&\hspace*{-6mm}
\begin{pmatrix}
	\; 2 & 4 & 6 & 8 &  10 & 14 \; \\[3mm]
	\; 1 & 2 & 4 & 6 & 3 & 13 \; \\[3mm]
	\; (1,0) & (1,0) & (1,1) & 1 & 0  & 0
\end{pmatrix}
\label{eq.labels.eg.2}
\end{subeqnarray}

%{\bf DO THE FIRST INDICES LOOK UGLY???????????  WHAT DO I DO ABOUT EQUATION NUMBERING VS FIGURE ENVIRONMENT!!!!!!!!!}

%{\bf THEN DEMONSTRATE INVERSE BIJECTION???????}

\subsection{Combinatorial interpretation using Laguerre digraphs}
\label{subsec.laguerre}

We begin with a Motzkin path $\omega$ 
and an assignment of labels $\xi$ satisfying 
\eqref{eq.xi.ineq}/\eqref{def.abc}.
The inverse bijection 
%\cite[Section~6.1, Step~3]{Sokal-Zeng_masterpoly},
%which has been recalled in 
(Section~\ref{subsec.FZrecall.perm} Step~3),
gives us a permutation $\sigma$.
We will now break this process into several intermediate steps
and reinterpret it using Laguerre digraphs.
Recall that a Laguerre digraph of size $n$ is a directed graph 
where each vertex has a distinct label from the label set $[n]$
and has indegree $0$ or $1$ and outdegree $0$ or $1$.
Clearly, any subgraph of a Laguerre digraph is also a Laguerre digraph.
The connected components in a Laguerre digraph are
either directed paths or directed cycles,
where a path with one vertex is called an isolated vertex
and a cycle with one vertex is called a loop.

Notice that a Laguerre digraph with no paths 
is a diagrammatic representation of a permutation in cycle notation
(see \cite[pp.~22-23]{Stanley_12}).
Let $L^{\sigma}$ denote this Laguerre digraph
corresponding to a permutation $\sigma\in\fS_n$. The directed edges of $L^{\sigma}$
are precisely $u\to \sigma(u)$. Also for $S\subseteq [n]$,
we let $\left. L^{\sigma} \right|_{S}$ denote the subgraph of $L^{\sigma}$,
containing the same set of vertices~$[n]$, but only the edges $u\to \sigma(u)$,
with $u\in S$ (we are allowed to have $\sigma(u)\not\in S$). 
Thus, $\left. L^{\sigma}\right|_{[n]} = L^{\sigma}$, and
$\left. L^{\sigma}\right|_{\emptyset}$ is the digraph containing $n$ vertices
and no edges. Whenever the permutation $\sigma$ is understood,
we shall drop the superscript and 
use $\laguerre{S}$ to denote this digraph.
%only write this subgraph as $\laguerre{S}$.

Recall that the inverse bijection in 
Step~3 Section~\ref{subsec.FZrecall.perm}, 
begins by obtaining the sets $F,F',G,G',H$
from the 3-coloured Motzkin path $\overline{\omega}$.
We then construct $\sigma\restrict F\colon \: F \to F'$ 
and $\sigma\restrict G\colon \: G\to G'$ separately 
by using the labels $\xi\restrict F$ and $\xi\restrict G$ respectively.
(Note that just knowing the set $H$ suffices 
to reconstruct $\sigma\restrict H$.)

%Our proof of the main theorem will use the same bijective procedure
%but we interpret it using a different combinatorial interpretation
%in terms of {\em Laguerre digraphs}.
%A Laguerre digraph of size $n$ is a directed graph 
%where each vertex has a distinct label from the label set $[n]$ 
%and has indegree $0$ or $1$ and outdegree $0$ or $1$.\footnote{Laguerre digraphs
%were introduced by Foata and Strehl in \cite{Foata_84} 
%to combinatorially interpret the Laguerre polynomials.} 
%Thus, the connected components in a Laguerre digraph are
%either directed paths, directed cycles, 
%isolated vertices (with no edges), and loops.
%
%Notice that a Laguerre digraph with no directed paths and isolated vertices
%is a diagrammatic representation of a permutation in cycle notation 
%(see \cite[pp.22-23]{Stanley_12}).
%Let $L^{\sigma}$ denote this Laguerre digraph 
%corresponding to a permutation $\sigma\in\fS$. The directed edges of $L^{\sigma}$
%are precisely $i\to \sigma(i)$. Also for $S\subseteq [n]$, 
%we let $\left. L^{\sigma} \right|_{S}$ denote the subgraph of $L^{\sigma}$,
%containing the same set of vertices $[n]$, but only the edges $i\to \sigma(i)$,
%whenever $i\in S$. Thus, $\left. L^{\sigma}\right|_{[n]} = L^{\sigma}$, and
%$\left. L^{\sigma}\right|_{\emptyset}$ is the digraph containing $n$ vertices
%and no edges. Whenever the permutation $\sigma$ is understood, 
%we shall drop the superscript and only write it as $\laguerre{S}$.

In this interpretation, we start with the digraph $\laguerre{\emptyset}$.
We then go through the set $[n]$.
However, the crucial part of our approach is that
we use the following unusual total order on $[n]$ 
(notice that $[n]=F\cup G\cup H$):
\begin{description}
\item[Stage (a):] We first go through the set $H$ in increasing order.
\item[Stage (b):] We then go through the set $G$ in increasing order.
\item[Stage (c):] Finally, we go through the set $F$ 
		but in \emph{decreasing} order.
\end{description}
As $F,G$ and $H$ are entirely determined by the 3-coloured path 
$\overline{\omega}$,
we call the above order the FZ order on $[n]$ with respect to
the 3-coloured Motzkin path $\overline{\omega}$.
Thus, the FZ order corresponding to two different permutations $\sigma$
and $\sigma'$ coming from the same 3-coloured Motzkin path $\overline{\omega}$ are the same.

Let $u_1,\ldots, u_n$ be a rewriting of $[n]$ as per the FZ order.
We now consider the ``FZ history'' 
$\laguerre{\emptyset}\subset \laguerre{\{u_1\}} \subset \laguerre{\{u_1, u_2\}}
\subset \ldots \subset \laguerre{\{u_1,\ldots, u_n\}} = L$
as a process of building up the permutation $\sigma$
by successively considering the status of vertices $u_1, u_2, \ldots, u_n$.
Thus, at step $u$ 
(where the step number is given by the vertex $u\in [n]$)
we use the inversion tables to construct the edge $u\to \sigma(u)$.
Thus, at each step we insert a new edge into the digraph,
and at the end of this process,
the resulting digraph obtained is the permutation $\sigma$ in cycle notation.

Let us now look at the intermediate Laguerre digraphs obtained 
during stages (a), (b) and (c) more closely.

\medskip

{\bf Stage (a): Going through $\boldsymbol{H}$:}\\
For each vertex $u\in H$, we introduce a loop edge $u\to u$
thus creating a new loop at the end of each step.
After all steps $u\in H$ have been carried out,
the resulting Laguerre digraph $\laguerre{H}$ consists of loops 
at all vertices in $H$.
All other vertices are isolated vertices, i.e., have no adjacent edges.

\medskip

{\bf Stage (b): Going through $\boldsymbol{G}$:}\\
From \eqref{eq.steps2cycleclassification}, we know that $G = \Cdfall(\sigma) \cup \Cpeak(\sigma)$ where $\sigma$ is the resulting permutation obtained at the end
of the inverse bijection.
Let us recall the construction for this case.
We construct 
$\sigma\restrict G\colon\: G\to G'$ using
the right-to-left inversion table.
Let $G = \{x_1<x_2<\ldots < x_k\}$ 
be the elements of $G$ arranged in increasing order,
and similarly let $G' = \{x_1'< x_2' < \ldots <x_k'\}$.
Then the $(p_{j}+1)$th smallest element of
$G'\backslash\{\sigma(x_1),\ldots,\sigma(x_{j-1})\}$ 
is chosen to be $\sigma(x_j)$ where
$p_{j} = \xi_{x_j}$.

Let us reinterpret this in terms of Laguerre digraphs.
At this stage, the vertices in $G$ (resp.~$G'$) 
are our designated starting vertices (ending vertices)
arranged in increasing order.
We then look through the starting vertices
in increasing order.
At the end of step $x_{j-1}$,
directed edges $x_{1}\to \sigma(x_1), \ldots, x_{j-1}\to\sigma(x_{j-1})$
have been inserted, 
the available starting vertices are $x_j,\ldots, x_k$
and the available ending vertices belong to
$G'\backslash \{\sigma(x_1),\ldots, \sigma(x_{j-1})\}$.
We then pick the smallest available starting vertex,
which is $x_j$,
and connect it to the
$(p_{j}+1)$th smallest ending vertex available,
i.e., the $(p_{j}+1)$th smallest element of the set
$G'\backslash \{\sigma(x_1),\ldots, \sigma(x_{j-1})\}$.
(This is analogous to the partial interpretation
of the labels in terms of bipartite digraph
for cycle double falls and cycle peaks on
\cite[p.~96]{Sokal-Zeng_masterpoly}.)

The vertices in $\laguerre{H\,\cup\, \{x_1,\ldots, x_j\}}$
have out-edges in the following situations:
\begin{itemize}
	\item if $u\in H$, $u$ is a loop.

	\item if $u$ is one of the $j$ smallest
		elements of $G$,
		it has an edge going to some vertex $v\neq u$ and $v\in G'$.
	
	\item All other vertices have no out-edges.

\end{itemize}
		
Also, from definition of $G$ in (\ref{eq.steps2excedanceclassification}c),
it follows that for any non-loop edge $u\to v$,
we must have $u>v$.
Thus, the Laguerre digraph $\laguerre{H\,\cup\, \{x_1,\ldots, x_j\}}$
only consists of decreasing directed paths (including isolated vertices)
and loops; there are no non-loop cycles.

\begin{lem} The Laguerre digraph $\laguerre{H\,\cup\, G}$
consists of the following connected components:
\begin{itemize} 
        \item loops on vertices $u\in H$,

        %\item no directed cycles

        \item directed paths with at least two vertices,
		in which the initial vertex of the path
		is a cycle peak in $\sigma$
		(i.e. contained in the set $F'\,\cap\, G$),
		the final vertex is a cycle valley in $\sigma$
		(i.e. contained in the set $F\,\cap\, G'$),
		and the intermediate vertices (if any)
		are cycle double falls
		(i.e. contained in the set $G\,\cap\, G'$).

		%This is true because all edges
                %corresponding to excedances
                %have been inserted
                %and no edge corresponding
                %anti-excedances have been inserted
                %at this stage.

	\item  isolated vertices at $u\in F\,\cap\, F' = \Cdrise(\sigma)$.
\end{itemize} 
Furthermore, it contains no directed cycles.
\label{lem.afterHG}
\end{lem}

\begin{proof} It suffices to prove that for a directed path with 
at least two vertices, the initial vertex is a cycle peak in $\sigma$
(i.e. contained in the set $F'\,\cap\, G$)
and the final vertex is a cycle valley in $\sigma$
(i.e. contained in the set $F\,\cap\, G'$).

Notice that the initial vertex $u$ of such a path must already have an out-neighbour
and hence must belong to the set $G$ as only the elements of the sets $G$ and $H$
have been assigned out-neighbours.
As all vertices in $G'$ have already been assigned in-neighbours,
$u\not\in G'$ and using the fact that 
$G = \Cdfall\, \cup\, \Cpeak$ (from \eqref{eq.steps2cycleclassification}), 
we get that $u\in \Cpeak = F'\,\cap\, G$.
The proof for the final vertex is similar and we omit it.
\end{proof}

\medskip

{\bf Stage (c): Going through $\boldsymbol{F}$:}\\
Similar to both the previous cases,
at each step we introduce edges $u\to \sigma(u)$.
However, now we go through elements of $F$ in decreasing order.

From \eqref{eq.steps2cycleclassification}, we know that $F = \Cdrise \cup \Cval$.
Let us now recall the construction for this case.
We construct 
$\sigma\restrict F\colon\: F\to F'$ using
the left-to-right inversion table.
Let $F = \{y_1<y_2<\ldots < y_l\}$
be the elements of $F$ arranged in increasing order.
Then the $(p_{j}+1)$th largest element of
$F'\backslash\{\sigma(y_l),\ldots,\sigma(y_{j+1})\}$
is chosen to be $\sigma(y_j)$ where
$p_{j} = \xi_{y_j}$.

We now reinterpret this in terms of Laguerre digraphs.
At this stage, the vertices in $F$ ($F'$)
are our designated starting vertices (ending vertices).
We look through the starting vertices
in decreasing order.
At the end of step $y_{j+1}$,
directed edges $y_{l}\to \sigma(y_l), \ldots, y_{j+1}\to\sigma(y_{j+1})$
have been inserted,
the available starting vertices are $y_j,\ldots, y_1$
and the available ending vertices belong to
$F'\backslash \{\sigma(y_l),\ldots, \sigma(y_{j+1})\}$.
We then pick the largest available starting vertex,
which is $y_j$ and connect it to the
$(p_{j}+1)$th largest ending vertex available,
i.e., the $(p_{j}+1)$th largest element of the set
$F'\backslash \{\sigma(y_l),\ldots, \sigma(y_{j+1})\}$.

\begin{lem} Let $u$ be the final vertex of a path
with at least two vertices in\\
$\laguerre{H\,\cup\, G \,\cup\, \{y_l,\ldots, y_j\}}$.
Then $u\in \Cval = F\,\cap \, G'$.
\label{lem.lastvertexF}
\end{lem}

\begin{proof} 
	Since all vertices in $H\,\cup\, G$ were already assigned out-neighbours
	during stages (a) and (b), it must be that $u\in F$.
	As $F = \Cdrise\cup \Cval$, $u$ 
	is either a cycle valley or a cycle double rise in $\sigma$.
	
	Let us assume that $u$ is a cycle double rise in $\sigma$, i.e.,
	$\sinv(u) <u< \sigma(u)$.
	We know that $u$ must have an in-neighbour $v$
	as its path has atleast two vertices.
	Thus, $v = \sinv(u) < u$, which implies $v\in F$
	(from definition of $F$ in~(\ref{eq.steps2excedanceclassification}a)).
	However, this is a situation where we have two vertices $u,v\in F$
	with $u>v$ such that the smaller vertex has an out-neighbour
	even though the larger vertex does not.
	This clearly cannot happen as we assign out-neighbours
	to vertices in $F$ in descending order.
	This is a contradiction and thus $u\in \Cval= F\,\cap \, G'$.	
\end{proof}

Let $u_1,\ldots,u_n$ be the elements of $[n]$ rearranged as per the FZ order
with respect to the 3-coloured Motzkin path $\overline{\omega}$.
We say that $u_j\in [n]\backslash H$ is a {\em cycle closer}
if the edge $u_j\to \sigma(u_j)$ 
is introduced in $\left.L\right|_{\{u_1,\ldots, u_{j-1}\}}$
as an edge between the two ends of a path
turning the path into a cycle.
The following lemma classifies all cycle closers:

\begin{lem}(Classifying cycle closers) Given a permutation $\sigma$,
	an element $u\in [n]$ is a cycle closer
	if and only if it is a cycle valley minimum,
	i.e., 
	it is the smallest element in its cycle.
\label{lem.classifyingCycleF}
\end{lem}

\begin{proof} Let $u_j$ be a cycle closer.
Notice that $u_j$ must have been the final vertex of a path
in $\left.L\right|_{\{u_1,\ldots, u_{j-1}\}}$ 
that contains at least two vertices.
From Lemma~\ref{lem.lastvertexF},
$u_j\in F\cap G' = \Cval$.

Any other cycle valley $v\neq u_j$ in the resulting cycle
must have already been present in this path
and hence must have an out-neighbour.
Hence, it must be that $v>u_j$ as $\Cval \subseteq F$ 
and we assign out-neighbours 
to vertices belonging to $F$ in descending order.
Thus, the cycle closer $u_j$ is the smallest cycle valley 
in its cycle in $\sigma$.
\end{proof}

As each non-singleton cycle has exactly one cycle closer,
counting cycle closers will give us the number of
non-singleton cycles.
This is what we do next.
%We will count the number of cycle closers
%as each non-singleton cycle has exactly one cycle closer.
But before doing that, we require a technical lemma.
However, before going into the lemma, recall that if 
$y_j \in F\,\cap\,G' = \Cval(\sigma)$,
step $s_{y_{j}}$ must be a rise from height $h_{y_j-1}$ to height $h_{y_j}$ 
and hence, $h_{y_j-1}+1 = h_{y_j}$.

\begin{lem}
Given a permutation $\sigma$ and associated sets $F,F',G,G',H$
with $F = \{y_1<y_2<\ldots < y_l\}$, 
and an index $j$ ($1\leq j\leq l$) such that $y_j \in F\,\cap\,G'$.
Then the following is true:
\be
\left|\{u \in  F'\backslash \{\sigma(y_l),\ldots, \sigma(y_{j+1})\} \colon \: u>y_j\}\right| 
	= h_{y_j-1}+1 = h_{y_j}
\label{eq.lem.cycle.closer.technical}
\ee
where $h_i$ denotes the height at position $i$ of the Motzkin path $\omega$
associated to $\sigma$
in Step~1.
\label{lem.cycle.closer.technical}
\end{lem}

\begin{proof} We first establish the following equality of sets:
\be
\{u > y_j \colon \: \sinv(u) \leq  y_j\} \; = \;  \{u \in   F'\backslash \{\sigma(y_l),\ldots, \sigma(y_{j+1})\} \colon\: u > y_j \}.
\label{eq.technical.equation.a}
\ee

Whenever $u\in F'$, we have that $\sinv(u)\in F$
(by description of $F$, $F'$ in (\ref{eq.steps2excedanceclassification}a,b)).
Additionally, if $u\not\in \{\sigma(y_l),\ldots, \sigma(y_{j+1})\}$
then it must be that $\sinv(u)\leq y_j$.
This establishes the containment $\{u > y_j \colon \: \sinv(u) \leq  y_j\} \; \supseteq \; \{u \in   F'\backslash \{\sigma(y_l),\ldots, \sigma(y_{j+1})\} \colon\: u > y_j \}$.

On the other hand, if $u>y_j$ and $\sinv(u)\leq y_j$,
then $u> \sinv(u)$ and hence $u\in F'$.
As $\sinv(u)\leq y_j< y_{j+1}<\ldots< y_l$,
$u$ cannot be one of $\sigma(y_{j+1}), \ldots, \sigma(y_l)$.
Therefore, $u\in F'\backslash \{\sigma(y_l),\ldots, \sigma(y_{j+1})\} $.
This establishes \eqref{eq.technical.equation.a}.

To obtain Equation~\eqref{eq.lem.cycle.closer.technical},
it suffices to show that the cardinality of the set
$\{u > y_j \colon \: \sinv(u) \leq  y_j\}$
is $h_{y_j}$.
To do this, recall the interpretation of heights
in Equation~\eqref{eq.lemma.heights.equiv.a} and observe that
\begin{eqnarray}
    h_{y_j}  & = &  \# \{u \leq y_j \colon\:  \sigma(u) > y_j\} \nonumber \\
            & = &  \# \{u > y_j \colon\:  \sinv(u) \leq y_j \}
\end{eqnarray}
where the second equality is obtained by replacing $u$ with $\sinv(u)$.
\end{proof}

We are now ready to count the number of cycle closers.

\begin{lem}[Counting of cycle closers for permutations]
	Fix a 3-coloured Motzkin path~$\overline{\omega}$ of length $n$ 
        and construct the sets $F,F',G,G',H$ 
	(these are totally
	determined by $\overline{\omega}$).  
        Let $F = \{y_1<y_2<\ldots < y_l\}$ and
	fix an index $j$ ($1\leq j\leq l$) such that $y_j \in F\,\cap\,G'$.
	Also fix labels $\xi_u$ for vertices
        $u\in H\,\cup\,G\,\cup\,\{y_l,y_{l-1},\ldots,y_{j+1}\}$
	satisfying~\eqref{eq.xi.ineq}/\eqref{def.abc}.
	Then
	\begin{itemize}
		\item[(a)] The value of $\xi_{y_j}$ completely determines 
			if $y_j$ is a cycle closer or not.
		\item[(b)] There is exactly one value $\xi_{y_j}\in \{0,1,\ldots,h_{y_j-1}\}$
			that makes $y_j$ a cycle closer, and conversely.
	\end{itemize}
\label{lem.cycle.closer}
\end{lem}

%\begin{lem}[Cycle closer counting lemma] 
%	Fix a Motzkin path $\omega$ of length $n$ 
%	and construct the sets $F,F',G,G',H$ (these are totally
%	determined by $\omega$).  
%	Let $F = \{f_1<f_2<\ldots < f_l\}$ and
%	fix index $1\leq j\leq l$ such that $f_j \in F\,\cap\,G'$.
%	Also fix labels $\xi_i$ for vertices
%	$i\in H\,\cup\,G\,\cup\,\{f_l,f_{l-1},\ldots,f_{i+1}\}$.
%	Consider all permutations $\sigma\in \fS_n$
%	which have step $\omega$ and these fixed labels
%	corresponding to the vertices in 
%	$H\,\cup\,G\,\cup\,\{f_l,f_{l-1},\ldots,f_{i+1}\}$.
%	Then, there is exactly one permutation among these
%	in which $f_j$ is a cycle closer
%	and, it is completely determined 
%	by the value of $\xi_{f_j}\in \{0,1,\ldots, h_{f_j-1}\}$.
%\label{lem.cycle.closer}
%\end{lem}

\begin{proof}
	As $y_j \in F\,\cap\,G'$, it must be that $y_j$
        is a cycle valley in $\sigma$.
	As $y_j$ does not have an out-neighbour in the Laguerre digraph
        $\laguerre{H\,\cup\, G \,\cup\, \{y_l,\ldots,y_{j+1}\}}$,
        it must be the final vertex of a path.
	Let $\rbf$ be the initial vertex of this path.
	During step $y_j$, we choose one of the available ending vertices
        from the set $F'\backslash\{\sigma(y_l),\ldots,\sigma(y_{j+1})\}$
	to be $\sigma(y_j)$.

	For $y_j$ to be a cycle closer, it must be that $\sigma(y_j)=\rbf$
	(so that inserting the edge $y_j \to \rbf$ turns its path into a cycle).
	As each value for label $\xi_{y_j}\in \{0,1,\ldots, h_{y_j-1}\}$
        assigns a different vertex to be $\sigma(y_j)$,
	there is at most one label for which $y_j$ is a cycle closer.

	Next we observe that $\rbf\not\in G'$.
	%We begin by observing that $\rbf\not\in G'$.
	This is because $\rbf$ does not have an in-neighbour
        in $\laguerre{H\,\cup\, G \,\cup\, \{y_l,\ldots,y_{j+1}\}}$
	whereas all $v\in G'$ had in-neighbours at the end of stage (b).
	Therefore, $\rbf\in F'$ and $y_j \neq\rbf$ as $y_j\in G'$.
	To show that there is a $\xi_{y_j} \in \{0,1,\ldots, h_{y_j-1}\}$
        which can connect $y_j$ to $\rbf$,
	we must show that $\rbf$ is among the largest $h_{y_j-1}+1$ elements
	of $F'\backslash \{\sigma(y_l),\ldots, \sigma(y_{j+1})\}$.
	However, from Lemma~\ref{lem.cycle.closer.technical},
	we only need to show that $\rbf> y_j$. 
	The remainder of the proof does this.

	Let $\rbf=\rbf_0,\rbf_1,\ldots, \rbf_{\alpha} = y_j$ be the vertices
        of the path containing $y_j$ and $\rbf$
	with edges {\hbox{} $\rbf_{i} \to \rbf_{i+1}$}.
        Let $\beta$ be the smallest index such that $\rbf_{\beta}\in F$.
        Using the description of $G$ in 
	(\ref{eq.steps2excedanceclassification}c),
        we get $\rbf=\rbf_0>\ldots > \rbf_\beta$ when $\beta>0$. Thus,
        \be
        \rbf = \rbf_0 \leq \rbf_\beta
        \label{eq.rbf0.rbfbeta}
        \ee
        with equality if and only if $\beta = 0$.
        On the other hand, if $\beta <\alpha$,
        then $\rbf_{\beta}$ already has an out-neighbour
        and thus $\rbf_{\beta} > y_j$
        (as we are going through elements of $F$ in descending order).
        Thus,
        \be
        \rbf_\beta \leq \rbf_\alpha = y_j
        \label{eq.rbfbeta.rbalpha}
        \ee
        with equality if and only if $\beta = \alpha$.
        Using \eqref{eq.rbf0.rbfbeta}, \eqref{eq.rbfbeta.rbalpha}
        and the fact that $\rbf\neq y_j$, we obtain
        $\rbf>y_j$.
        This completes the proof.
\end{proof}

\begin{rem}
%\noindent {\bf Remarks.} 
1. Notice that one can also construct a variant of this
interpretation where stage~(c) occurs before stage~(b). The role of cycle closer
will then be played by cycle peak maxima. 

2. In fact, for our results to hold, we could carry out stages~(a) and~(b) in any order
as we only require the digraph $\laguerre{H\cup G}$ and not the previous ones 
for our Lemmas~\ref{lem.afterHG}-\ref{lem.cycle.closer} to hold.
\myendremark
\end{rem}

\subsection{Examples}
\label{subsec.FZ.laguerre.eg}

%{\bf DONT FORGET TO MENTION THIS SUBSECTION IN THE SECTION PLAN!!!!!!}

\subsubsection{Running example 1}

We will now draw the FZ history of our first example
$\sigma = 9\,\, 3\,\, 7\,\, 4\,\, 6\,\, 11\,\, 5\,\, 8\,\, 10\,\, 1\,\, 2\\
           = (1,9,10)\,(2,3,7,5,6,11)\,(4)\,(8) \in \Sym_{11}$.

The sets $F,G$ and $H$ were already recorded in~\eqref{eq.FGH.running.example.1}
and we recall that\\
$F = \{1,2,3,5,6,9\}$, $G = \{7,10,11\}$ and $H=\{4,8\}$.
Thus, the FZ order consists of the following stages:
\begin{itemize}
\item Stage (a): 4, 8
\item Stage (b): 7, 10, 11
\item Stage (c): 9, 6, 5, 3, 2, 1
\end{itemize}
Stages~(a) and~(b) of the FZ history of $\sigma$ has been drawn in 
Figure~\ref{fig.running.example.1.FZ.history.ab},
and Stage~(c) has been drawn in Figure~\ref{fig.running.example.1.FZ.history.c}.
Non-singleton cycles are formed in Stage~(c) when the edges $2\to 3$ and $1\to 9$ are inserted.

%{\bf WHAT ELSE SHOULD I SAY????}

\begin{figure}[!pt]
\vspace*{-8mm}
\centering
\begin{tabular}{l}
%Begin with $\laguerre{\emptyset}$ with vertex set $\{1,2,3,4,5,6,7,8,9,10,11\}$ and no edges \\
%\hline\\
\begin{tikzpicture}[scale = 1]
%\node at (0,0) {Begin with $\laguerre{\emptyset}$ with vertex set $\{1,2,3,4,5,6,7,8,9,10,11\}$ and no edges};
{
\pgfmathsetmacro{\yshift}{-2.5};

\pgfmathsetmacro{\ay}{1+\yshift};
\pgfmathsetmacro{\by}{1+\yshift};
\pgfmathsetmacro{\cy}{0+\yshift};
\pgfmathsetmacro{\dy}{0+\yshift};
\pgfmathsetmacro{\ey}{0+\yshift};
\pgfmathsetmacro{\fy}{1+\yshift};
\pgfmathsetmacro{\gy}{0+\yshift};
\pgfmathsetmacro{\hy}{0+\yshift};
\pgfmathsetmacro{\iy}{0+\yshift};
\pgfmathsetmacro{\jy}{0+\yshift};
\pgfmathsetmacro{\ky}{1+\yshift};

\node[right] at (-4,\yshift) {$\laguerre{\emptyset}$};

	\node[circle,fill=black,inner sep=1pt,minimum size=5pt] (a) at (0,\ay) {};
	\node[circle,fill=black,inner sep=1pt,minimum size=5pt] (i) at (0,\iy) {};
	\node[circle,fill=black,inner sep=1pt,minimum size=5pt] (j) at (1,\jy) {};

	\node[circle,fill=black,inner sep=1pt,minimum size=5pt] (b) at (3,\by) {};
	\node[circle,fill=black,inner sep=1pt,minimum size=5pt] (c) at (3,\cy) {};
	\node[circle,fill=black,inner sep=1pt,minimum size=5pt] (g) at (4,\gy) {};
	\node[circle,fill=black,inner sep=1pt,minimum size=5pt] (e) at (5,\ey) {};
	\node[circle,fill=black,inner sep=1pt,minimum size=5pt] (f) at (5,\fy) {};
	\node[circle,fill=black,inner sep=1pt,minimum size=5pt] (k) at (4,\ky) {};

	\node[circle,fill=black,inner sep=1pt,minimum size=5pt] (d) at (7,\dy) {};
	\node[circle,fill=black,inner sep=1pt,minimum size=5pt] (h) at (8,\hy) {};

	\node[above = 1pt of a] {{$1$}};
	\node[above = 1pt of b] {{$2$}};
	\node[below = 1pt of c] {{$3$}};
	\node[below = 1pt of d] {{$4$}};
	\node[below = 1pt of e] {{$5$}};
	\node[above = 1pt of f] {{$6$}};
	\node[below = 1pt of g] {{$7$}};
	\node[below = 1pt of h] {{$8$}};
	\node[below = 1pt of i] {{$9$}};
	\node[below = 1pt of j] {{$10$}};
	\node[above = 1pt of k] {{$11$}};
}
\end{tikzpicture}\\
\hline\\[-4mm]
\hline\\
After Stage (a): $H = \{4,8\}$\\[2mm]
\hline\\[2mm]
\begin{tikzpicture}[scale = 1]
{
\pgfmathsetmacro{\yshift}{-6.5};

\pgfmathsetmacro{\ay}{1+\yshift};
\pgfmathsetmacro{\by}{1+\yshift};
\pgfmathsetmacro{\cy}{0+\yshift};
\pgfmathsetmacro{\dy}{0+\yshift};
\pgfmathsetmacro{\ey}{0+\yshift};
\pgfmathsetmacro{\fy}{1+\yshift};
\pgfmathsetmacro{\gy}{0+\yshift};
\pgfmathsetmacro{\hy}{0+\yshift};
\pgfmathsetmacro{\iy}{0+\yshift};
\pgfmathsetmacro{\jy}{0+\yshift};
\pgfmathsetmacro{\ky}{1+\yshift};

\node[right] at (-4,\yshift) {$\laguerre{\{4,8\}}$};

	\node[circle,fill=black,inner sep=1pt,minimum size=5pt] (a) at (0,\ay) {};
	\node[circle,fill=black,inner sep=1pt,minimum size=5pt] (i) at (0,\iy) {};
	\node[circle,fill=black,inner sep=1pt,minimum size=5pt] (j) at (1,\jy) {};

	\node[circle,fill=black,inner sep=1pt,minimum size=5pt] (b) at (3,\by) {};
	\node[circle,fill=black,inner sep=1pt,minimum size=5pt] (c) at (3,\cy) {};
	\node[circle,fill=black,inner sep=1pt,minimum size=5pt] (g) at (4,\gy) {};
	\node[circle,fill=black,inner sep=1pt,minimum size=5pt] (e) at (5,\ey) {};
	\node[circle,fill=black,inner sep=1pt,minimum size=5pt] (f) at (5,\fy) {};
	\node[circle,fill=black,inner sep=1pt,minimum size=5pt] (k) at (4,\ky) {};

	\node[circle,fill=black,inner sep=1pt,minimum size=5pt] (d) at (7,\dy) {};
	\node[circle,fill=black,inner sep=1pt,minimum size=5pt] (h) at (8,\hy) {};

	\node[above = 1pt of a] {{$1$}};
	\node[above = 1pt of b] {{$2$}};
	\node[below = 1pt of c] {{$3$}};
	\node[below = 1pt of d] {{$4$}};
	\node[below = 1pt of e] {{$5$}};
	\node[above = 1pt of f] {{$6$}};
	\node[below = 1pt of g] {{$7$}};
	\node[below = 1pt of h] {{$8$}};
	\node[below = 1pt of i] {{$9$}};
	\node[below = 1pt of j] {{$10$}};
	\node[above = 1pt of k] {{$11$}};

\node[circle,fill=black,inner sep=1pt,minimum size=5pt] (d) at (7,\dy) {} edge [in=45,out=135, loop above, thick, color=red] node {} ();
\node[circle,fill=black,inner sep=1pt,minimum size=5pt] (d) at (8,\hy) {} edge [in=45,out=135, loop above, thick, color=red] node {} ();
}
\end{tikzpicture}\\
\hline\\[-4mm]
\hline\\
Stage (b): $G = \{7, 10, 11\}$ in increasing order\\[2mm]
\hline\\[2mm] 

\begin{tikzpicture}[scale = 1]
{
\pgfmathsetmacro{\yshift}{-10.5};

\pgfmathsetmacro{\ay}{1+\yshift};
\pgfmathsetmacro{\by}{1+\yshift};
\pgfmathsetmacro{\cy}{0+\yshift};
\pgfmathsetmacro{\dy}{0+\yshift};
\pgfmathsetmacro{\ey}{0+\yshift};
\pgfmathsetmacro{\fy}{1+\yshift};
\pgfmathsetmacro{\gy}{0+\yshift};
\pgfmathsetmacro{\hy}{0+\yshift};
\pgfmathsetmacro{\iy}{0+\yshift};
\pgfmathsetmacro{\jy}{0+\yshift};
\pgfmathsetmacro{\ky}{1+\yshift};

\node[right] at (-4,\yshift) {$\laguerre{\{4,8,7\}}$};

	\node[circle,fill=black,inner sep=1pt,minimum size=5pt] (a) at (0,\ay) {};
	\node[circle,fill=black,inner sep=1pt,minimum size=5pt] (i) at (0,\iy) {};
	\node[circle,fill=black,inner sep=1pt,minimum size=5pt] (j) at (1,\jy) {};

	\node[circle,fill=black,inner sep=1pt,minimum size=5pt] (b) at (3,\by) {};
	\node[circle,fill=black,inner sep=1pt,minimum size=5pt] (c) at (3,\cy) {};
	\node[circle,fill=black,inner sep=1pt,minimum size=5pt] (g) at (4,\gy) {};
	\node[circle,fill=black,inner sep=1pt,minimum size=5pt] (e) at (5,\ey) {};
	\node[circle,fill=black,inner sep=1pt,minimum size=5pt] (f) at (5,\fy) {};
	\node[circle,fill=black,inner sep=1pt,minimum size=5pt] (k) at (4,\ky) {};

	\node[circle,fill=black,inner sep=1pt,minimum size=5pt] (d) at (7,\dy) {};
	\node[circle,fill=black,inner sep=1pt,minimum size=5pt] (h) at (8,\hy) {};

	\node[above = 1pt of a] {{$1$}};
	\node[above = 1pt of b] {{$2$}};
	\node[below = 1pt of c] {{$3$}};
	\node[below = 1pt of d] {{$4$}};
	\node[below = 1pt of e] {{$5$}};
	\node[above = 1pt of f] {{$6$}};
	\node[below = 1pt of g] {{$7$}};
	\node[below = 1pt of h] {{$8$}};
	\node[below = 1pt of i] {{$9$}};
	\node[below = 1pt of j] {{$10$}};
	\node[above = 1pt of k] {{$11$}};

\node[circle,fill=black,inner sep=1pt,minimum size=5pt] (d) at (7,\dy) {} edge [in=45,out=135, loop above, thick] node {} ();
\node[circle,fill=black,inner sep=1pt,minimum size=5pt] (d) at (8,\hy) {} edge [in=45,out=135, loop above, thick] node {} ();

\graph [multi, edges = {thick,red}] {(g) -> (e);};

}
\end{tikzpicture}\\
\hdashline\\
\begin{tikzpicture}
{
\pgfmathsetmacro{\yshift}{-13.5};

\pgfmathsetmacro{\ay}{1+\yshift};
\pgfmathsetmacro{\by}{1+\yshift};
\pgfmathsetmacro{\cy}{0+\yshift};
\pgfmathsetmacro{\dy}{0+\yshift};
\pgfmathsetmacro{\ey}{0+\yshift};
\pgfmathsetmacro{\fy}{1+\yshift};
\pgfmathsetmacro{\gy}{0+\yshift};
\pgfmathsetmacro{\hy}{0+\yshift};
\pgfmathsetmacro{\iy}{0+\yshift};
\pgfmathsetmacro{\jy}{0+\yshift};
\pgfmathsetmacro{\ky}{1+\yshift};

\node[right] at (-4,\yshift) {$\laguerre{\{4,8,7,10\}}$};

	\node[circle,fill=black,inner sep=1pt,minimum size=5pt] (a) at (0,\ay) {};
	\node[circle,fill=black,inner sep=1pt,minimum size=5pt] (i) at (0,\iy) {};
	\node[circle,fill=black,inner sep=1pt,minimum size=5pt] (j) at (1,\jy) {};

	\node[circle,fill=black,inner sep=1pt,minimum size=5pt] (b) at (3,\by) {};
	\node[circle,fill=black,inner sep=1pt,minimum size=5pt] (c) at (3,\cy) {};
	\node[circle,fill=black,inner sep=1pt,minimum size=5pt] (g) at (4,\gy) {};
	\node[circle,fill=black,inner sep=1pt,minimum size=5pt] (e) at (5,\ey) {};
	\node[circle,fill=black,inner sep=1pt,minimum size=5pt] (f) at (5,\fy) {};
	\node[circle,fill=black,inner sep=1pt,minimum size=5pt] (k) at (4,\ky) {};

	\node[circle,fill=black,inner sep=1pt,minimum size=5pt] (d) at (7,\dy) {};
	\node[circle,fill=black,inner sep=1pt,minimum size=5pt] (h) at (8,\hy) {};

	\node[above = 1pt of a] {{$1$}};
	\node[above = 1pt of b] {{$2$}};
	\node[below = 1pt of c] {{$3$}};
	\node[below = 1pt of d] {{$4$}};
	\node[below = 1pt of e] {{$5$}};
	\node[above = 1pt of f] {{$6$}};
	\node[below = 1pt of g] {{$7$}};
	\node[below = 1pt of h] {{$8$}};
	\node[below = 1pt of i] {{$9$}};
	\node[below = 1pt of j] {{$10$}};
	\node[above = 1pt of k] {{$11$}};

\node[circle,fill=black,inner sep=1pt,minimum size=5pt] (d) at (7,\dy) {} edge [in=45,out=135, thick, loop above] node {} ();
\node[circle,fill=black,inner sep=1pt,minimum size=5pt] (d) at (8,\hy) {} edge [in=45,out=135, thick, loop above] node {} ();

\graph [multi, edges = {thick}] {(g) -> (e); };

\graph [multi, edges = {thick, red}] {(j) -> (a)  };

}
\end{tikzpicture}\\
\hdashline\\
\begin{tikzpicture}
{
\pgfmathsetmacro{\yshift}{-16.5};

\pgfmathsetmacro{\ay}{1+\yshift};
\pgfmathsetmacro{\by}{1+\yshift};
\pgfmathsetmacro{\cy}{0+\yshift};
\pgfmathsetmacro{\dy}{0+\yshift};
\pgfmathsetmacro{\ey}{0+\yshift};
\pgfmathsetmacro{\fy}{1+\yshift};
\pgfmathsetmacro{\gy}{0+\yshift};
\pgfmathsetmacro{\hy}{0+\yshift};
\pgfmathsetmacro{\iy}{0+\yshift};
\pgfmathsetmacro{\jy}{0+\yshift};
\pgfmathsetmacro{\ky}{1+\yshift};

\node[right] at (-4,\yshift) {$\laguerre{\{4,8,7,10,11\}}$};

	\node[circle,fill=black,inner sep=1pt,minimum size=5pt] (a) at (0,\ay) {};
	\node[circle,fill=black,inner sep=1pt,minimum size=5pt] (i) at (0,\iy) {};
	\node[circle,fill=black,inner sep=1pt,minimum size=5pt] (j) at (1,\jy) {};

	\node[circle,fill=black,inner sep=1pt,minimum size=5pt] (b) at (3,\by) {};
	\node[circle,fill=black,inner sep=1pt,minimum size=5pt] (c) at (3,\cy) {};
	\node[circle,fill=black,inner sep=1pt,minimum size=5pt] (g) at (4,\gy) {};
	\node[circle,fill=black,inner sep=1pt,minimum size=5pt] (e) at (5,\ey) {};
	\node[circle,fill=black,inner sep=1pt,minimum size=5pt] (f) at (5,\fy) {};
	\node[circle,fill=black,inner sep=1pt,minimum size=5pt] (k) at (4,\ky) {};

	\node[circle,fill=black,inner sep=1pt,minimum size=5pt] (d) at (7,\dy) {};
	\node[circle,fill=black,inner sep=1pt,minimum size=5pt] (h) at (8,\hy) {};

	\node[above = 1pt of a] {{$1$}};
	\node[above = 1pt of b] {{$2$}};
	\node[below = 1pt of c] {{$3$}};
	\node[below = 1pt of d] {{$4$}};
	\node[below = 1pt of e] {{$5$}};
	\node[above = 1pt of f] {{$6$}};
	\node[below = 1pt of g] {{$7$}};
	\node[below = 1pt of h] {{$8$}};
	\node[below = 1pt of i] {{$9$}};
	\node[below = 1pt of j] {{$10$}};
	\node[above = 1pt of k] {{$11$}};

\node[circle,fill=black,inner sep=1pt,minimum size=5pt] (d) at (7,\dy) {} edge [in=45,out=135, thick, loop above] node {} ();
\node[circle,fill=black,inner sep=1pt,minimum size=5pt] (d) at (8,\hy) {} edge [in=45,out=135, thick, loop above] node {} ();

\graph [multi, edges = {thick}] {(g) -> (e); (j) -> (a); };
\graph [multi, edges = {thick,red}] {(k) -> (b); };

}
\end{tikzpicture}
\end{tabular}
\caption{Stages (a) and (b) of the FZ history for the permutation\\
$\sigma = 9\,\, 3\,\, 7\,\, 4\,\, 6\,\, 11\,\, 5\,\, 8\,\, 10\,\, 1\,\, 2
= (1,9,10)\,(2,3,7,5,6,11)\,(4)\,(8) \in \Sym_{11}$.}
\label{fig.running.example.1.FZ.history.ab}
\end{figure}

\begin{figure}
\centering
\begin{tabular}{l}
\hline\\[-4mm]
\hline\\
Stage (c): $F = \{9, 6, 5, 3, 2, 1\}$ in decreasing order\\[2mm]
\hline\\[2mm]
\begin{tikzpicture}[scale = 1]
{
\pgfmathsetmacro{\yshift}{-20.5};

\pgfmathsetmacro{\ay}{1+\yshift};
\pgfmathsetmacro{\by}{1+\yshift};
\pgfmathsetmacro{\cy}{0+\yshift};
\pgfmathsetmacro{\dy}{0+\yshift};
\pgfmathsetmacro{\ey}{0+\yshift};
\pgfmathsetmacro{\fy}{1+\yshift};
\pgfmathsetmacro{\gy}{0+\yshift};
\pgfmathsetmacro{\hy}{0+\yshift};
\pgfmathsetmacro{\iy}{0+\yshift};
\pgfmathsetmacro{\jy}{0+\yshift};
\pgfmathsetmacro{\ky}{1+\yshift};

\node[right] at (-4,\yshift) {$\laguerre{\{4,8,7,10,11,9\}}$};

	\node[circle,fill=black,inner sep=1pt,minimum size=5pt] (a) at (0,\ay) {};
	\node[circle,fill=black,inner sep=1pt,minimum size=5pt] (i) at (0,\iy) {};
	\node[circle,fill=black,inner sep=1pt,minimum size=5pt] (j) at (1,\jy) {};

	\node[circle,fill=black,inner sep=1pt,minimum size=5pt] (b) at (3,\by) {};
	\node[circle,fill=black,inner sep=1pt,minimum size=5pt] (c) at (3,\cy) {};
	\node[circle,fill=black,inner sep=1pt,minimum size=5pt] (g) at (4,\gy) {};
	\node[circle,fill=black,inner sep=1pt,minimum size=5pt] (e) at (5,\ey) {};
	\node[circle,fill=black,inner sep=1pt,minimum size=5pt] (f) at (5,\fy) {};
	\node[circle,fill=black,inner sep=1pt,minimum size=5pt] (k) at (4,\ky) {};

	\node[circle,fill=black,inner sep=1pt,minimum size=5pt] (d) at (7,\dy) {};
	\node[circle,fill=black,inner sep=1pt,minimum size=5pt] (h) at (8,\hy) {};

	\node[above = 1pt of a] {{$1$}};
	\node[above = 1pt of b] {{$2$}};
	\node[below = 1pt of c] {{$3$}};
	\node[below = 1pt of d] {{$4$}};
	\node[below = 1pt of e] {{$5$}};
	\node[above = 1pt of f] {{$6$}};
	\node[below = 1pt of g] {{$7$}};
	\node[below = 1pt of h] {{$8$}};
	\node[below = 1pt of i] {{$9$}};
	\node[below = 1pt of j] {{$10$}};
	\node[above = 1pt of k] {{$11$}};

\node[circle,fill=black,inner sep=1pt,minimum size=5pt] (d) at (7,\dy) {} edge [in=45,out=135, thick, loop above] node {} ();
\node[circle,fill=black,inner sep=1pt,minimum size=5pt] (d) at (8,\hy) {} edge [in=45,out=135, thick, loop above] node {} ();

\graph [multi, edges = {thick}] {(g) -> (e); (j) -> (a); (k) -> (b); };

\graph [multi, edges = {thick,red}] {(i) -> (j);};
}
\end{tikzpicture}\\
\hdashline\\
\begin{tikzpicture}
{
\pgfmathsetmacro{\yshift}{-23.5};

\pgfmathsetmacro{\ay}{1+\yshift};
\pgfmathsetmacro{\by}{1+\yshift};
\pgfmathsetmacro{\cy}{0+\yshift};
\pgfmathsetmacro{\dy}{0+\yshift};
\pgfmathsetmacro{\ey}{0+\yshift};
\pgfmathsetmacro{\fy}{1+\yshift};
\pgfmathsetmacro{\gy}{0+\yshift};
\pgfmathsetmacro{\hy}{0+\yshift};
\pgfmathsetmacro{\iy}{0+\yshift};
\pgfmathsetmacro{\jy}{0+\yshift};
\pgfmathsetmacro{\ky}{1+\yshift};

\node[right] at (-4,\yshift) {$\laguerre{\{4,8,7,10,11,9,6\}}$};

	\node[circle,fill=black,inner sep=1pt,minimum size=5pt] (a) at (0,\ay) {};
	\node[circle,fill=black,inner sep=1pt,minimum size=5pt] (i) at (0,\iy) {};
	\node[circle,fill=black,inner sep=1pt,minimum size=5pt] (j) at (1,\jy) {};

	\node[circle,fill=black,inner sep=1pt,minimum size=5pt] (b) at (3,\by) {};
	\node[circle,fill=black,inner sep=1pt,minimum size=5pt] (c) at (3,\cy) {};
	\node[circle,fill=black,inner sep=1pt,minimum size=5pt] (g) at (4,\gy) {};
	\node[circle,fill=black,inner sep=1pt,minimum size=5pt] (e) at (5,\ey) {};
	\node[circle,fill=black,inner sep=1pt,minimum size=5pt] (f) at (5,\fy) {};
	\node[circle,fill=black,inner sep=1pt,minimum size=5pt] (k) at (4,\ky) {};

	\node[circle,fill=black,inner sep=1pt,minimum size=5pt] (d) at (7,\dy) {};
	\node[circle,fill=black,inner sep=1pt,minimum size=5pt] (h) at (8,\hy) {};

	\node[above = 1pt of a] {{$1$}};
	\node[above = 1pt of b] {{$2$}};
	\node[below = 1pt of c] {{$3$}};
	\node[below = 1pt of d] {{$4$}};
	\node[below = 1pt of e] {{$5$}};
	\node[above = 1pt of f] {{$6$}};
	\node[below = 1pt of g] {{$7$}};
	\node[below = 1pt of h] {{$8$}};
	\node[below = 1pt of i] {{$9$}};
	\node[below = 1pt of j] {{$10$}};
	\node[above = 1pt of k] {{$11$}};

\node[circle,fill=black,inner sep=1pt,minimum size=5pt] (d) at (7,\dy) {} edge [in=45,out=135, thick, loop above] node {} ();
\node[circle,fill=black,inner sep=1pt,minimum size=5pt] (d) at (8,\hy) {} edge [in=45,out=135, thick, loop above] node {} ();

\graph [multi, edges = {thick}] {(g) -> (e); (j) -> (a); (k) -> (b); };

\graph [multi, edges = {thick}] {(i) -> (j);};
\graph [multi, edges = {thick,red}] {(f) -> (k);};

}
\end{tikzpicture}\\
\hdashline\\
\begin{tikzpicture}

{
\pgfmathsetmacro{\yshift}{-26.5};

\pgfmathsetmacro{\ay}{1+\yshift};
\pgfmathsetmacro{\by}{1+\yshift};
\pgfmathsetmacro{\cy}{0+\yshift};
\pgfmathsetmacro{\dy}{0+\yshift};
\pgfmathsetmacro{\ey}{0+\yshift};
\pgfmathsetmacro{\fy}{1+\yshift};
\pgfmathsetmacro{\gy}{0+\yshift};
\pgfmathsetmacro{\hy}{0+\yshift};
\pgfmathsetmacro{\iy}{0+\yshift};
\pgfmathsetmacro{\jy}{0+\yshift};
\pgfmathsetmacro{\ky}{1+\yshift};

\node[right] at (-4,\yshift) {$\laguerre{\{4,8,7,10,11,9,6,5\}}$};

	\node[circle,fill=black,inner sep=1pt,minimum size=5pt] (a) at (0,\ay) {};
	\node[circle,fill=black,inner sep=1pt,minimum size=5pt] (i) at (0,\iy) {};
	\node[circle,fill=black,inner sep=1pt,minimum size=5pt] (j) at (1,\jy) {};

	\node[circle,fill=black,inner sep=1pt,minimum size=5pt] (b) at (3,\by) {};
	\node[circle,fill=black,inner sep=1pt,minimum size=5pt] (c) at (3,\cy) {};
	\node[circle,fill=black,inner sep=1pt,minimum size=5pt] (g) at (4,\gy) {};
	\node[circle,fill=black,inner sep=1pt,minimum size=5pt] (e) at (5,\ey) {};
	\node[circle,fill=black,inner sep=1pt,minimum size=5pt] (f) at (5,\fy) {};
	\node[circle,fill=black,inner sep=1pt,minimum size=5pt] (k) at (4,\ky) {};

	\node[circle,fill=black,inner sep=1pt,minimum size=5pt] (d) at (7,\dy) {};
	\node[circle,fill=black,inner sep=1pt,minimum size=5pt] (h) at (8,\hy) {};

	\node[above = 1pt of a] {{$1$}};
	\node[above = 1pt of b] {{$2$}};
	\node[below = 1pt of c] {{$3$}};
	\node[below = 1pt of d] {{$4$}};
	\node[below = 1pt of e] {{$5$}};
	\node[above = 1pt of f] {{$6$}};
	\node[below = 1pt of g] {{$7$}};
	\node[below = 1pt of h] {{$8$}};
	\node[below = 1pt of i] {{$9$}};
	\node[below = 1pt of j] {{$10$}};
	\node[above = 1pt of k] {{$11$}};

\node[circle,fill=black,inner sep=1pt,minimum size=5pt] (d) at (7,\dy) {} edge [in=45,out=135, thick, loop above] node {} ();
\node[circle,fill=black,inner sep=1pt,minimum size=5pt] (d) at (8,\hy) {} edge [in=45,out=135, thick, loop above] node {} ();

\graph [multi, edges = {thick}] {(g) -> (e); (j) -> (a); (k) -> (b); };

\graph [multi, edges = {thick}] {(i) -> (j); (f) -> (k);};
\graph [multi, edges = {thick,red}] {(e) -> (f);};
}
\end{tikzpicture}\\
\hdashline\\
\begin{tikzpicture}

{
\pgfmathsetmacro{\yshift}{-29.5};

\pgfmathsetmacro{\ay}{1+\yshift};
\pgfmathsetmacro{\by}{1+\yshift};
\pgfmathsetmacro{\cy}{0+\yshift};
\pgfmathsetmacro{\dy}{0+\yshift};
\pgfmathsetmacro{\ey}{0+\yshift};
\pgfmathsetmacro{\fy}{1+\yshift};
\pgfmathsetmacro{\gy}{0+\yshift};
\pgfmathsetmacro{\hy}{0+\yshift};
\pgfmathsetmacro{\iy}{0+\yshift};
\pgfmathsetmacro{\jy}{0+\yshift};
\pgfmathsetmacro{\ky}{1+\yshift};

\node[right] at (-4,\yshift) {$\laguerre{\{4,8,7,10,11,9,6,5,3\}}$};

	\node[circle,fill=black,inner sep=1pt,minimum size=5pt] (a) at (0,\ay) {};
	\node[circle,fill=black,inner sep=1pt,minimum size=5pt] (i) at (0,\iy) {};
	\node[circle,fill=black,inner sep=1pt,minimum size=5pt] (j) at (1,\jy) {};

	\node[circle,fill=black,inner sep=1pt,minimum size=5pt] (b) at (3,\by) {};
	\node[circle,fill=black,inner sep=1pt,minimum size=5pt] (c) at (3,\cy) {};
	\node[circle,fill=black,inner sep=1pt,minimum size=5pt] (g) at (4,\gy) {};
	\node[circle,fill=black,inner sep=1pt,minimum size=5pt] (e) at (5,\ey) {};
	\node[circle,fill=black,inner sep=1pt,minimum size=5pt] (f) at (5,\fy) {};
	\node[circle,fill=black,inner sep=1pt,minimum size=5pt] (k) at (4,\ky) {};

	\node[circle,fill=black,inner sep=1pt,minimum size=5pt] (d) at (7,\dy) {};
	\node[circle,fill=black,inner sep=1pt,minimum size=5pt] (h) at (8,\hy) {};

	\node[above = 1pt of a] {{$1$}};
	\node[above = 1pt of b] {{$2$}};
	\node[below = 1pt of c] {{$3$}};
	\node[below = 1pt of d] {{$4$}};
	\node[below = 1pt of e] {{$5$}};
	\node[above = 1pt of f] {{$6$}};
	\node[below = 1pt of g] {{$7$}};
	\node[below = 1pt of h] {{$8$}};
	\node[below = 1pt of i] {{$9$}};
	\node[below = 1pt of j] {{$10$}};
	\node[above = 1pt of k] {{$11$}};

\node[circle,fill=black,inner sep=1pt,minimum size=5pt] (d) at (7,\dy) {} edge [in=45,out=135, thick, loop above] node {} ();
\node[circle,fill=black,inner sep=1pt,minimum size=5pt] (d) at (8,\hy) {} edge [in=45,out=135, thick, loop above] node {} ();

\graph [multi, edges = {thick}] {(g) -> (e); (j) -> (a); (k) -> (b); };

\graph [multi, edges = {thick}] {(i) -> (j); (f) -> (k); (e) -> (f); };
\graph [multi, edges = {thick,red}] {(c) -> (g);};
}
\end{tikzpicture}\\
\hdashline\\
\begin{tikzpicture}

{
\pgfmathsetmacro{\yshift}{-32.5};

\pgfmathsetmacro{\ay}{1+\yshift};
\pgfmathsetmacro{\by}{1+\yshift};
\pgfmathsetmacro{\cy}{0+\yshift};
\pgfmathsetmacro{\dy}{0+\yshift};
\pgfmathsetmacro{\ey}{0+\yshift};
\pgfmathsetmacro{\fy}{1+\yshift};
\pgfmathsetmacro{\gy}{0+\yshift};
\pgfmathsetmacro{\hy}{0+\yshift};
\pgfmathsetmacro{\iy}{0+\yshift};
\pgfmathsetmacro{\jy}{0+\yshift};
\pgfmathsetmacro{\ky}{1+\yshift};

\node[right] at (-4,\yshift) {$\laguerre{\{4,8,7,10,11,9,6,5,3,2\}}$};

	\node[circle,fill=black,inner sep=1pt,minimum size=5pt] (a) at (0,\ay) {};
	\node[circle,fill=black,inner sep=1pt,minimum size=5pt] (i) at (0,\iy) {};
	\node[circle,fill=black,inner sep=1pt,minimum size=5pt] (j) at (1,\jy) {};

	\node[circle,fill=black,inner sep=1pt,minimum size=5pt] (b) at (3,\by) {};
	\node[circle,fill=black,inner sep=1pt,minimum size=5pt] (c) at (3,\cy) {};
	\node[circle,fill=black,inner sep=1pt,minimum size=5pt] (g) at (4,\gy) {};
	\node[circle,fill=black,inner sep=1pt,minimum size=5pt] (e) at (5,\ey) {};
	\node[circle,fill=black,inner sep=1pt,minimum size=5pt] (f) at (5,\fy) {};
	\node[circle,fill=black,inner sep=1pt,minimum size=5pt] (k) at (4,\ky) {};

	\node[circle,fill=black,inner sep=1pt,minimum size=5pt] (d) at (7,\dy) {};
	\node[circle,fill=black,inner sep=1pt,minimum size=5pt] (h) at (8,\hy) {};

	\node[above = 1pt of a] {{$1$}};
	\node[above = 1pt of b] {{$2$}};
	\node[below = 1pt of c] {{$3$}};
	\node[below = 1pt of d] {{$4$}};
	\node[below = 1pt of e] {{$5$}};
	\node[above = 1pt of f] {{$6$}};
	\node[below = 1pt of g] {{$7$}};
	\node[below = 1pt of h] {{$8$}};
	\node[below = 1pt of i] {{$9$}};
	\node[below = 1pt of j] {{$10$}};
	\node[above = 1pt of k] {{$11$}};

\node[circle,fill=black,inner sep=1pt,minimum size=5pt] (d) at (7,\dy) {} edge [in=45,out=135, thick, loop above] node {} ();
\node[circle,fill=black,inner sep=1pt,minimum size=5pt] (d) at (8,\hy) {} edge [in=45,out=135, thick, loop above] node {} ();

\graph [multi, edges = {thick}] {(g) -> (e); (j) -> (a); (k) -> (b); };

\graph [multi, edges = {thick}] {(i) -> (j); (f) -> (k); (e) -> (f); (c) -> (g);};
\graph [multi, edges = {thick,red}] {(b) -> (c);};
}
\end{tikzpicture}\\
\hdashline\\
\begin{tikzpicture}

{
\pgfmathsetmacro{\yshift}{-35.5};

\pgfmathsetmacro{\ay}{1+\yshift};
\pgfmathsetmacro{\by}{1+\yshift};
\pgfmathsetmacro{\cy}{0+\yshift};
\pgfmathsetmacro{\dy}{0+\yshift};
\pgfmathsetmacro{\ey}{0+\yshift};
\pgfmathsetmacro{\fy}{1+\yshift};
\pgfmathsetmacro{\gy}{0+\yshift};
\pgfmathsetmacro{\hy}{0+\yshift};
\pgfmathsetmacro{\iy}{0+\yshift};
\pgfmathsetmacro{\jy}{0+\yshift};
\pgfmathsetmacro{\ky}{1+\yshift};

\node[right] at (-4,\yshift) {$\laguerre{\{4,8,7,10,11,9, 6, 5, 3, 2, 1\}}$};

	\node[circle,fill=black,inner sep=1pt,minimum size=5pt] (a) at (0,\ay) {};
	\node[circle,fill=black,inner sep=1pt,minimum size=5pt] (i) at (0,\iy) {};
	\node[circle,fill=black,inner sep=1pt,minimum size=5pt] (j) at (1,\jy) {};

	\node[circle,fill=black,inner sep=1pt,minimum size=5pt] (b) at (3,\by) {};
	\node[circle,fill=black,inner sep=1pt,minimum size=5pt] (c) at (3,\cy) {};
	\node[circle,fill=black,inner sep=1pt,minimum size=5pt] (g) at (4,\gy) {};
	\node[circle,fill=black,inner sep=1pt,minimum size=5pt] (e) at (5,\ey) {};
	\node[circle,fill=black,inner sep=1pt,minimum size=5pt] (f) at (5,\fy) {};
	\node[circle,fill=black,inner sep=1pt,minimum size=5pt] (k) at (4,\ky) {};

	\node[circle,fill=black,inner sep=1pt,minimum size=5pt] (d) at (7,\dy) {};
	\node[circle,fill=black,inner sep=1pt,minimum size=5pt] (h) at (8,\hy) {};

	\node[above = 1pt of a] {{$1$}};
	\node[above = 1pt of b] {{$2$}};
	\node[below = 1pt of c] {{$3$}};
	\node[below = 1pt of d] {{$4$}};
	\node[below = 1pt of e] {{$5$}};
	\node[above = 1pt of f] {{$6$}};
	\node[below = 1pt of g] {{$7$}};
	\node[below = 1pt of h] {{$8$}};
	\node[below = 1pt of i] {{$9$}};
	\node[below = 1pt of j] {{$10$}};
	\node[above = 1pt of k] {{$11$}};

\node[circle,fill=black,inner sep=1pt,minimum size=5pt] (d) at (7,\dy) {} edge [in=45,out=135, thick, loop above] node {} ();
\node[circle,fill=black,inner sep=1pt,minimum size=5pt] (d) at (8,\hy) {} edge [in=45,out=135, thick, loop above] node {} ();

\graph [multi, edges = {thick}] {(g) -> (e); (j) -> (a); (k) -> (b); };

\graph [multi, edges = {thick}] {(i) -> (j); (f) -> (k); (e) -> (f); (c) -> (g); (b) -> (c);};
\graph [multi, edges = {thick,red}] {(a) -> (i);};
}
\end{tikzpicture}
\end{tabular}
\caption{Stage (c) of the FZ history for the permutation
$\sigma = 9\,\, 3\,\, 7\,\, 4\,\, 6\,\, 11\,\, 5\,\, 8\,\, 10\,\, 1\,\, 2
	   = (1,9,10)\,(2,3,7,5,6,11)\,(4)\,(8) \in \Sym_{11}$.} 
\label{fig.running.example.1.FZ.history.c}
\end{figure}

\subsubsection{Running example 2}

Let us now look at the FZ history of our second example
$\sigma  =  7\, 1\, 9\, 2\, 5\, 4\, 8\, 6\, 10\, 3\, 11\, 12\, 14\, 13\,\\
=      (1,7,8,6,4,2)\,(3,9,10)\,(5)\,(11)\,(12)\,(13,14) \in \Sym_{14}.$
The sets $F,G$ and $H$ were already recorded in~\eqref{eq.FGH.running.example.2}
and we recall that $F = \{1, 3, 7, 9, 13\}$, $G = \{2, 4, 6, 8, 10, 14\}$ and \\ $H=\{5, 11, 12\}$.
Thus, the FZ order consists of the following stages:
\begin{itemize}
\item Stage (a): 5, 11, 12
\item Stage (b): 2, 4, 6, 8, 10, 14
\item Stage (c): 13, 9, 7, 3, 1
\end{itemize}
Stages~(a) and~(b) of the FZ history of $\sigma$ has been drawn in
Figure~\ref{fig.running.example.2.FZ.history.ab},
and Stage~(c) has been drawn in Figure~\ref{fig.running.example.2.FZ.history.c}.
Non-singleton cycles are formed in Stage~(c) when the edges $13\to 14$, $3\to 9$ 
and $1\to 7$ are inserted.

%{\bf WHAT ELSE???????}

\newcommand{\emptysigmatwo}[1]{%
\pgfmathsetmacro{\yshift}{#1};

\pgfmathsetmacro{\ay}{1+\yshift};
\pgfmathsetmacro{\by}{1+\yshift};
\pgfmathsetmacro{\cy}{1+\yshift};
\pgfmathsetmacro{\dy}{1+\yshift};
\pgfmathsetmacro{\ey}{0+\yshift};
\pgfmathsetmacro{\fy}{0+\yshift};
\pgfmathsetmacro{\gy}{0+\yshift};
\pgfmathsetmacro{\hy}{0+\yshift};
\pgfmathsetmacro{\iy}{0+\yshift};
\pgfmathsetmacro{\jy}{0+\yshift};
\pgfmathsetmacro{\ky}{0+\yshift};
\pgfmathsetmacro{\ly}{0+\yshift};
\pgfmathsetmacro{\my}{1+\yshift};
\pgfmathsetmacro{\ny}{0+\yshift};

\node[circle,fill=black,inner sep=1pt,minimum size=5pt] (a) at (0,\ay) {};
\node[circle,fill=black,inner sep=1pt,minimum size=5pt] (g) at (0,\gy) {};
\node[circle,fill=black,inner sep=1pt,minimum size=5pt] (h) at (1,\hy) {};
\node[circle,fill=black,inner sep=1pt,minimum size=5pt] (f) at (2,\fy) {};
\node[circle,fill=black,inner sep=1pt,minimum size=5pt] (d) at (2,\dy) {};
\node[circle,fill=black,inner sep=1pt,minimum size=5pt] (b) at (1,\by) {};

\node[circle,fill=black,inner sep=1pt,minimum size=5pt] (c) at (4,\cy) {};
\node[circle,fill=black,inner sep=1pt,minimum size=5pt] (i) at (4,\iy) {};
\node[circle,fill=black,inner sep=1pt,minimum size=5pt] (j) at (5,\jy) {};

\node[circle,fill=black,inner sep=1pt,minimum size=5pt] (m) at (7,\my) {};
\node[circle,fill=black,inner sep=1pt,minimum size=5pt] (n) at (7,\ny) {};

\node[circle,fill=black,inner sep=1pt,minimum size=5pt] (e) at (9,\ey) {};
\node[circle,fill=black,inner sep=1pt,minimum size=5pt] (k) at (10,\ky) {};
\node[circle,fill=black,inner sep=1pt,minimum size=5pt] (l) at (11,\ly) {};

\node[above = 1pt of a] {{$1$}};
\node[above = 1pt of b] {{$2$}};
\node[above = 1pt of c] {{$3$}};
\node[above = 1pt of d] {{$4$}};
\node[below = 1pt of e] {{$5$}};
\node[below = 1pt of f] {{$6$}};
\node[below = 1pt of g] {{$7$}};
\node[below = 1pt of h] {{$8$}};
\node[below = 1pt of i] {{$9$}};
\node[below = 1pt of j] {{$10$}};
\node[below = 1pt of k] {{$11$}};
\node[below = 1pt of l] {{$12$}};
\node[above = 1pt of m] {{$13$}};
\node[below = 1pt of n] {{$14$}};
}

\begin{figure}[p]
\vspace*{-12mm}
\centering
\begin{tabular}{l}
%\begin{tikzpicture}[scale=0.6][scale = 1]
%\node at (0,0) {Begin with $\laguerre{\emptyset}$ with vertex set $\{1,2,3,4,5,6,7,8,9,10,11\}$ and no edges};
%\node[right] at (-4,-2.5) {$\laguerre{\emptyset}$};
%\emptysigmatwo{-2.5}
%\node[right] at (-4,-5) {$\laguerre{\emptyset}$};
%\emptysigmatwo{-5}
%\end{tikzpicture}\\
%
\begin{tikzpicture}[scale=0.6]
%\node[right] at (-7,0) {$\laguerre{\emptyset}$};

\node[right] at (-7,0) {$\laguerre{\emptyset}$};
\emptysigmatwo{0}%

\end{tikzpicture}\\
\hline\\[-4mm]
\hline\\[-3mm]
After Stage (a): $H = \{5, 11, 12\}$\\[2mm]
\hline\\[-3mm]
\begin{tikzpicture}[scale=0.6]
\node[right] at (-7,0) {$\laguerre{\{5,11,12\}}$};
\emptysigmatwo{0}
\node[circle,fill=black,inner sep=1pt,minimum size=5pt] (e) at (9,\ey) {} edge [in=45,out=135, thick, loop above, color=red] node {} ();
\node[circle,fill=black,inner sep=1pt,minimum size=5pt] (k) at (10,\ky) {} edge [in=45,out=135, thick, loop above,color=red] node {} ();
\node[circle,fill=black,inner sep=1pt,minimum size=5pt] (l) at (11,\ly) {} edge [in=45,out=135, thick, loop above,color=red] node {} ();
\end{tikzpicture}\\
\hline\\[-4mm]
\hline\\[-3mm]
Stage (b): $G = \{2, 4, 6, 8, 10, 14\}$ in increasing order\\[2mm]
\hline\\[-3mm]
\begin{tikzpicture}[scale=0.6]
\node[right] at (-7,0) {$\laguerre{\{5,11,12,2\}}$};
\emptysigmatwo{0}
\node[circle,fill=black,inner sep=1pt,minimum size=5pt] (e) at (9,\ey) {} edge [in=45,out=135, thick, loop above] node {} ();
\node[circle,fill=black,inner sep=1pt,minimum size=5pt] (k) at (10,\ky) {} edge [in=45,out=135, thick, loop above] node {} ();
\node[circle,fill=black,inner sep=1pt,minimum size=5pt] (l) at (11,\ly) {} edge [in=45,out=135, thick, loop above] node {} ();
\graph [multi, edges = {thick,red}] {(b) -> (a); };
\end{tikzpicture}\\
\hdashline\\[-3mm]
\begin{tikzpicture}[scale=0.6]
\node[right] at (-7,0) {$\laguerre{\{5,11,12,2,4\}}$};
\emptysigmatwo{0}
\node[circle,fill=black,inner sep=1pt,minimum size=5pt] (e) at (9,\ey) {} edge [in=45,out=135, thick, loop above] node {} ();
\node[circle,fill=black,inner sep=1pt,minimum size=5pt] (k) at (10,\ky) {} edge [in=45,out=135, thick, loop above] node {} ();
\node[circle,fill=black,inner sep=1pt,minimum size=5pt] (l) at (11,\ly) {} edge [in=45,out=135, thick, loop above] node {} ();
	\graph [multi, edges = {thick}] {(b) -> (a); };
	\graph [multi, edges = {thick,red}] {(d) -> (b); };
\end{tikzpicture}\\
\hdashline\\[-3mm]
\begin{tikzpicture}[scale=0.6]
\node[right] at (-7,0) {$\laguerre{\{5,11,12,2,4,6\}}$};
\emptysigmatwo{0}
\node[circle,fill=black,inner sep=1pt,minimum size=5pt] (e) at (9,\ey) {} edge [in=45,out=135, thick, loop above] node {} ();
\node[circle,fill=black,inner sep=1pt,minimum size=5pt] (k) at (10,\ky) {} edge [in=45,out=135, thick, loop above] node {} ();
\node[circle,fill=black,inner sep=1pt,minimum size=5pt] (l) at (11,\ly) {} edge [in=45,out=135, thick, loop above] node {} ();
\graph [multi, edges = {thick}] {(b) -> (a); (d) -> (b); };
\graph [multi, edges = {thick,red}] {(f) -> (d); };
\end{tikzpicture}\\
\hdashline\\[-3mm]
\begin{tikzpicture}[scale=0.6]
\node[right] at (-7,0) {$\laguerre{\{5,11,12,2,4,6,8\}}$};
\emptysigmatwo{0}
\node[circle,fill=black,inner sep=1pt,minimum size=5pt] (e) at (9,\ey) {} edge [in=45,out=135, thick, loop above] node {} ();
\node[circle,fill=black,inner sep=1pt,minimum size=5pt] (k) at (10,\ky) {} edge [in=45,out=135, thick, loop above] node {} ();
\node[circle,fill=black,inner sep=1pt,minimum size=5pt] (l) at (11,\ly) {} edge [in=45,out=135, thick, loop above] node {} ();
\graph [multi, edges = {thick}] {(b) -> (a); (d) -> (b); (f) -> (d); };
\graph [multi, edges = {thick,red}] {(h) -> (f); };
\end{tikzpicture}\\
\hdashline\\[-3mm]
\begin{tikzpicture}[scale=0.6]
\node[right] at (-7,0) {$\laguerre{\{5,11,12,2,4,6,8,10\}}$};
\emptysigmatwo{0}
\node[circle,fill=black,inner sep=1pt,minimum size=5pt] (e) at (9,\ey) {} edge [in=45,out=135, thick, loop above] node {} ();
\node[circle,fill=black,inner sep=1pt,minimum size=5pt] (k) at (10,\ky) {} edge [in=45,out=135, thick, loop above] node {} ();
\node[circle,fill=black,inner sep=1pt,minimum size=5pt] (l) at (11,\ly) {} edge [in=45,out=135, thick, loop above] node {} ();
\graph [multi, edges = {thick}] {(b) -> (a); (d) -> (b); (f) -> (d), (h) -> (f); }; 
\graph [multi, edges = {thick,red}] {(j) -> (c); }; 
\end{tikzpicture}\\
\hdashline\\[-3mm]
\begin{tikzpicture}[scale=0.6]
\node[right] at (-7,0) {$\laguerre{\{5,11,12,2,4,6,8,10,14\}}$};
\emptysigmatwo{0}
\node[circle,fill=black,inner sep=1pt,minimum size=5pt] (e) at (9,\ey) {} edge [in=45,out=135, thick, loop above] node {} ();
\node[circle,fill=black,inner sep=1pt,minimum size=5pt] (k) at (10,\ky) {} edge [in=45,out=135, thick, loop above] node {} ();
\node[circle,fill=black,inner sep=1pt,minimum size=5pt] (l) at (11,\ly) {} edge [in=45,out=135, thick, loop above] node {} ();
\graph [multi, edges = {thick}] {(b) -> (a); (d) -> (b); (f) -> (d), (h) -> (f); 
	(j) -> (c); }; 

\graph [multi, edges = {thick,red}] {(n) -> [bend right] (m)};

\end{tikzpicture}
\end{tabular}
\caption{Stages (a) and (b) of the FZ history for the permutation\\
$\sigma = 7\, 1\, 9\, 2\, 5\, 4\, 8\, 6\, 10\, 3\, 11\, 12\, 14\, 13\,
         =  (1,7,8,6,4,2)\,(3,9,10)\,(5)\,(11)\,(12)\,(13,14) \in \Sym_{14}.$
%	   {\bf ANYTHING ELSE???????}
}
\label{fig.running.example.2.FZ.history.ab}
\end{figure}

\begin{figure}[p]
\centering
\begin{tabular}{l}
\hline\\[-4mm]
\hline\\
Stage (c): $F = \{13,9,7,3,1\}$ in decreasing order\\[2mm]
\hline\\[2mm]
\begin{tikzpicture}[scale=0.7]
\node[right] at (-7,0) {$\laguerre{\{5,11,12,2,4,6,8,10,14,13\}}$};
\emptysigmatwo{0}
\node[circle,fill=black,inner sep=1pt,minimum size=5pt] (e) at (9,\ey) {} edge [in=45,out=135, thick, loop above] node {} ();
\node[circle,fill=black,inner sep=1pt,minimum size=5pt] (k) at (10,\ky) {} edge [in=45,out=135, thick, loop above] node {} ();
\node[circle,fill=black,inner sep=1pt,minimum size=5pt] (l) at (11,\ly) {} edge [in=45,out=135, thick, loop above] node {} ();
\graph [multi, edges = {thick}] {(b) -> (a); (d) -> (b); (f) -> (d), (h) -> (f);
        (j) -> (c); };

\graph [multi, edges = {thick}] {(n) -> [bend right] (m);};
\graph [multi, edges = {thick,red}] {(m) -> [bend right] (n);};

\end{tikzpicture}\\
\hdashline\\
\begin{tikzpicture}[scale=0.7]
\node[right] at (-7,0) {$\laguerre{\{5,11,12,2,4,6,8,10,14,13,9\}}$};
\emptysigmatwo{0}
\node[circle,fill=black,inner sep=1pt,minimum size=5pt] (e) at (9,\ey) {} edge [in=45,out=135, thick, loop above] node {} ();
\node[circle,fill=black,inner sep=1pt,minimum size=5pt] (k) at (10,\ky) {} edge [in=45,out=135, thick, loop above] node {} ();
\node[circle,fill=black,inner sep=1pt,minimum size=5pt] (l) at (11,\ly) {} edge [in=45,out=135, thick, loop above] node {} ();
\graph [multi, edges = {thick}] {(b) -> (a); (d) -> (b); (f) -> (d), (h) -> (f);
        (j) -> (c); };

\graph [multi, edges = {thick}] {(n) -> [bend right] (m);};

\graph [multi, edges = {thick}] {(m) -> [bend right] (n);};

\graph [multi, edges = {thick,red}] {(i) -> (j);};
\end{tikzpicture}\\
\hdashline\\
\begin{tikzpicture}[scale=0.7]
\node[right] at (-7,0) {$\laguerre{\{5,11,12,2,4,6,8,10,14,13,9,7\}}$};
\emptysigmatwo{0}
\node[circle,fill=black,inner sep=1pt,minimum size=5pt] (e) at (9,\ey) {} edge [in=45,out=135, thick, loop above] node {} ();
\node[circle,fill=black,inner sep=1pt,minimum size=5pt] (k) at (10,\ky) {} edge [in=45,out=135, thick, loop above] node {} ();
\node[circle,fill=black,inner sep=1pt,minimum size=5pt] (l) at (11,\ly) {} edge [in=45,out=135, thick, loop above] node {} ();
\graph [multi, edges = {thick}] {(b) -> (a); (d) -> (b); (f) -> (d), (h) -> (f);
        (j) -> (c); };

\graph [multi, edges = {thick}] {(n) -> [bend right] (m);};

\graph [multi, edges = {thick}] {(m) -> [bend right] (n);};

\graph [multi, edges = {thick}] {(i) -> (j); };
\graph [multi, edges = {thick,red}] {(g) -> (h); };
\end{tikzpicture}\\
\hdashline\\
\begin{tikzpicture}[scale=0.7]
\node[right] at (-7,0) {$\laguerre{\{5,11,12,2,4,6,8,10,14,13,9,7,3\}}$};
\emptysigmatwo{0}
\node[circle,fill=black,inner sep=1pt,minimum size=5pt] (e) at (9,\ey) {} edge [in=45,out=135, thick, loop above] node {} ();
\node[circle,fill=black,inner sep=1pt,minimum size=5pt] (k) at (10,\ky) {} edge [in=45,out=135, thick, loop above] node {} ();
\node[circle,fill=black,inner sep=1pt,minimum size=5pt] (l) at (11,\ly) {} edge [in=45,out=135, thick, loop above] node {} ();
\graph [multi, edges = {thick}] {(b) -> (a); (d) -> (b); (f) -> (d), (h) -> (f);
        (j) -> (c); };

\graph [multi, edges = {thick}] {(n) -> [bend right] (m);};

\graph [multi, edges = {thick}] {(m) -> [bend right] (n);};

\graph [multi, edges = {thick}] {(i) -> (j); (g) -> (h); };
\graph [multi, edges = {thick,red}] {(c) -> (i); };
\end{tikzpicture}\\
\hdashline\\
\begin{tikzpicture}[scale=0.7]
\node[right] at (-7.5,0) {$\laguerre{\{5,11,12,2,4,6,8,10,14,13,9,7,3,1\}}$};
\emptysigmatwo{0}
\node[circle,fill=black,inner sep=1pt,minimum size=5pt] (e) at (9,\ey) {} edge [in=45,out=135, thick, loop above] node {} ();
\node[circle,fill=black,inner sep=1pt,minimum size=5pt] (k) at (10,\ky) {} edge [in=45,out=135, thick, loop above] node {} ();
\node[circle,fill=black,inner sep=1pt,minimum size=5pt] (l) at (11,\ly) {} edge [in=45,out=135, thick, loop above] node {} ();
\graph [multi, edges = {thick}] {(b) -> (a); (d) -> (b); (f) -> (d), (h) -> (f);
        (j) -> (c); };

\graph [multi, edges = {thick}] {(n) -> [bend right] (m);};

\graph [multi, edges = {thick}] {(m) -> [bend right] (n);};

\graph [multi, edges = {thick}] {(i) -> (j); (g) -> (h); (c) -> (i);};
\graph [multi, edges = {thick,red}] {(a) -> (g)};
\end{tikzpicture}
\end{tabular}
\caption{Stage (c) of the FZ history for the permutation
$\sigma = 7\, 1\, 9\, 2\, 5\, 4\, 8\, 6\, 10\, 3\, 11\, 12\, 14\, 13\,\\
	 =  (1,7,8,6,4,2)\,(3,9,10)\,(5)\,(11)\,(12)\,(13,14) \in \Sym_{14}.$}
%	   {\bf ANYTHING ELSE???????}}
\label{fig.running.example.2.FZ.history.c}
\end{figure}

\subsection{Computation of weights}
\label{subsec.computation}

To each index $i\in[n]$ of a given permutation $\sigma\in \Sym_n$,
we assign a weight $\wt(i)$ as follows:
\begin{itemize}
	\item if $i$ is a cycle valley minimum, we set 
		$\wt(i) = \lambda\,\sfa_{\ucross(i,\sigma) + \unest(i,\sigma)}$.
		For all other cycle valleys, we set 
		$\wt(i) = \sfa_{\ucross(i,\sigma) + \unest(i,\sigma)}$,

	\item if $i$ is a cycle peak, we set
		$\wt(i) = \sfb_{\lcross(i,\sigma),\,\lnest(i,\sigma)}$,
	
	\item if $i$ is a cycle double fall, we set
		$\wt(i) = \sfc_{\lcross(i,\sigma),\,\lnest(i,\sigma)}$,

	\item if $i$ is a cycle double rise, we set
		$\wt(i) = \sfd_{\ucross(i,\sigma),\,\unest(i,\sigma)}$,

	\item and finally, if $i$ is a fixed point, we set
		$\wt(i) = \lambda\, \sfe_{\psnest(i,\sigma)}$.
\end{itemize}
We then set the weight of the permutation $\sigma$ to be 
$\wt(\sigma)  = \prod_{i=1}^n \wt(i)$.
It is clear that our polynomial $\widehat{Q}_n$ defined in \eqref{def.poly.master}
is simply 
$\widehat{Q}_n(\bsfa, \bsfb, \bsfc, \bsfd, \bsfe, \lambda) = \sum_{\sigma\in\Sym_n} \wt(\sigma)$.

We now use the bijection $\sigma \mapsto (\omega, \xi)$
to transfer the weights $\wt(i)$ onto the pair $(\omega, \xi)$.
Let $k = h_{i-1}$ the height at which step $s_i$ starts.
From~\ref{eq.height.fix},~\ref{def.xi.bis}, and 
Lemma~\ref{lemma.crossing}, we immediately see that
\begin{itemize}
	\item if $s_i$ is a rise, then 
		$\wt(i) = 
		\bigl(1 + {\rm I}[\text{$i$ is a cycle valley minimum}](\lambda -1) \bigr)\, 
		\sfa_{k}$.\\
		(Here ${\rm I}[\hbox{\sl proposition}] = 1$ if {\sl proposition} is true,
and 0 if it is false.)
	
	\item if $s_i$ is a fall, then $\wt(i) = \sfb_{k-1-\xi_i,\, \xi_i}$,

	\item if $s_i$ is a level step of type 1 (here $i$ is a cycle double fall),
		$\wt(i) = \sfc_{k-1-\xi_i,\, \xi_i}$,
	
	\item if $s_i$ is a level step of type 2 (here $i$ is a cycle double rise), 
		$\wt(i) = \sfd_{k-1-\xi_i,\, \xi_i}$,
	
	\item if $s_i$ is a level step of type 3 (here $i$ is a fixed point), 
		$\wt(i) = \lambda\sfe_{k}$,
	
\end{itemize}

Next, we fix a Motzkin path $\omega$ and sum over all labels $\xi$ 
to obtain the weight $W(\omega)$. 
This weight is then factorised over the individual steps $s_i$
to obtain the weight contributed by each step $s_i$.
As all the variables involved commute, 
we can consider the weights for the steps $s_i$ in any order.
We do this in the FZ order as follows:
\begin{itemize}
	\item if $s_i$ is a rise, then $i$ is a cycle valley.
		Among the possible choices of labels $\xi\in [0,k]$
		there is exactly one which closes a cycle and the others don't
		(Lemma~\ref{lem.cycle.closer}).
		Therefore, we obtain
		\be
		a_k \;\eqdef\; 
		\sum_{\xi_i}  
		\bigl(1 + {\rm I}[\text{$i$ is a cycle valley minimum}](\lambda -1) \bigr)
		\, 
                \sfa_{k}
		\;=\;  (\lambda+k) \sfa_k.
		\ee

	\item if $s_i$ is a fall, then the weight is
		\be
		b_k \;=\; \sum_{\xi_i} \sfb_{k-\xi_i-1,\xi_i} 
		\ee
	\item if $s_i$ is a level step of type 1, then the weight is
                \be
                c_k \;=\; \sum_{\xi_i} \sfc_{k-\xi_i-1,\xi_i}
                \ee

	\item if $s_i$ is a level step of type 2, then the weight is
                \be
                d_k \;=\; \sum_{\xi_i} \sfd_{k-\xi_i-1,\xi_i}
                \ee

	\item if $s_i$ is a level step of type 3, then the weight is
                \be
                e_k \;=\; \lambda \sfe_{k}
                \ee

\end{itemize}

Setting $\beta_k = \sfa_{k-1}\sfb_k$ and $\gamma_k = \sfc_k+\sfd_k+\sfe_k$
as instructed in Theorem~\ref{thm.flajolet_master_labeled_Motzkin},
we obtain the weights~\ref{def.weights.master}.
This completes the proof of Theorem~\ref{thm.master}. \qed

\proofof{Theorem~\ref{thm.pqgen}}
We recall 
\cite[Lemma~2.10]{Sokal-Zeng_masterpoly}
which was used to separate records and antirecords:
Let $\sigma \in \Sym_n$ and $i \in [n]$.
\begin{itemize}
   \item[(a)]  If $i$ is a cycle valley or cycle double rise,
       then $i$ is a record if and only if $\unest(i,\sigma) = 0$;
       and in this case it is an exclusive record.
   \item[(b)]  If $i$ is a cycle peak or cycle double fall,
       then $i$ is an antirecord if and only if $\lnest(i,\sigma) = 0$;
       and in this case it is an exclusive antirecord.
\end{itemize}

We then specialise Theorem~\ref{thm.master} to 
\begin{subeqnarray}
   \sfa_{\ell} & = & p_{+1}^{\ell} y_1\\[2mm]
   \sfb_{\ell,\ell'}
	& = & p_{-1}^{\ell}q_{-1}^{\ell'} \times
   \begin{cases}
        x_1  &  \textrm{if $\ell' = 0$}  \\
        u_1  &  \textrm{if $\ell' \ge 1$}
   \end{cases}
       \\[2mm]
   \sfc_{\ell,\ell'}
        & = & p_{-2}^{\ell}q_{-2}^{\ell'} \times
   \begin{cases}
        x_2  &  \textrm{if $\ell' = 0$}  \\
        u_2  &  \textrm{if $\ell' \ge 1$}
   \end{cases}
       \\[2mm]
   \sfd_{\ell,\ell'}
        & = & p_{+2}^{\ell}q_{+2}^{\ell'} \times
   \begin{cases}
        y_2  &  \textrm{if $\ell' = 0$}  \\
        v_2  &  \textrm{if $\ell' \ge 1$}
   \end{cases}
	\\[2mm]
   \sfe_{\ell} & = & s^{\ell} w_\ell
\end{subeqnarray}

\qed

\begin{rem}
%\noindent {\bf Remark.} 
Notice that the specialisation in the above proof is almost the same
as \cite[eq.~(2.81)]{Sokal-Zeng_masterpoly}, except for the treatment of $\sfa$. 
\myendremark
\end{rem}

\proofof{Theorem~\ref{thm.conj.SZ}} Specialise Theorem~\ref{thm.pqgen} to $p_{+1}=p_{-1}=p_{+2}=p_{-2}=q_{-1}=q_{+2}=q_{-2}=s=1$.
\qed

%%{\color{red} 
%%
%%\section{Generalisation of Theorem~\ref{thm.conj.SZ}}
%%
%%Sokal and Zeng provide $p,q$-generalisation for their first J-fraction
%%for permutations~(\cite[Theorem~2.2]{Sokal-Zeng_masterpoly})
%%in \cite[Theorem~2.7]{Sokal-Zeng_masterpoly}. These involve a family of {\bf HOW MANY VARIABLES??????}
%%Similarly, they also provide $p,q$-generalisation for their second J-fraction
%%for permutations~(\cite[Theorem~2.4]{Sokal-Zeng_masterpoly})
%%in \cite[Theorem~2.12]{Sokal-Zeng_masterpoly}.
%%In subsection~\ref{subsec.pq.conj.SZ}, 
%%we provide a similar $p,q$-generalisation for Theorem~\ref{thm.conj.SZ}
%%which generalises \cite[Theorem~2.12]{Sokal-Zeng_masterpoly}.
%%
%%Sokal and Zeng also provide master continued fractions
%%for their first and second J-fractions in
%%\cite[Theorem~2.9 and~2.14]{Sokal-Zeng_masterpoly}, respectively.
%%In subsection~\ref{subsec.master.conj.SZ},
%%we provide a master J-fraction generalising Theorem~\ref{thm.conj.SZ}
%%and Theorem~\ref{thm.conj.SZ.pq}.
%%
%%
%%
%%{\bf COUNT THE NUMBER OF VARIABLES AND MENTION!!!!!!!}
%%
%%{\bf MOVE THESE PARAGRAPHS TO THE RIGHT PLACE!!!!}
%%
%%
%%}

\section{D-permutations: Proofs}
\label{sec.dperm.proofs}

Motivated by Randrianarivony's \cite{Randrianarivony_97} bijection
for D-o-semiderangements,
Deb and Sokal came up with two different bijections
for D-permutations.
One of them \cite[Sections~6.1-6.3]{Deb-Sokal} was used
to prove \cite[Theorems~3.3, 3.9, 3.11]{Deb-Sokal},
i.e., their ``first T-fractions'' for D-permutations.
The second bijection \cite[Section~6.5]{Deb-Sokal} was used to prove 
\cite[Theorems~3.12, 3.13]{Deb-Sokal},
i.e. their ``variant forms of the first T-fractions''
for D-permutations.
We will provide new interpretations to both of these bijections
and use them to prove our theorems for D-permutations.
As both bijections have the same path but different description of labels,
we state the description of the paths for both bijections here.

%It is worth noting that both bijections have the same path 
%but different description of labels.
%Thus, we will first state the description of the paths for both bijections here.
%Also, for their first variant, Deb and Sokal noted that the labels
%were exactly the same except for the treatment of fixed points;
%they said ``The surprise
%was that the definitions of the labels 
%in our constructions turned out to be almost
%identical to those employed in \cite{Sokal-Zeng_masterpoly}
%(where “almost” means that fixed points are treated
%differently.''
%We now understand the reason for this surprise
%and will explain that the description of the path
%can be obtained from the bijection of Sokal and Zeng \cite{Sokal-Zeng_masterpoly}, which we recalled in Section~\ref{subsec.FZrecall.perm}.\footnote{We thank the anonymous reviewer for their suggestion using which we were able to formalise why the first bijection in this section seems to be a ``special case'' of the bijection in Section~\ref{subsec.FZrecall.perm}.}
%
%{\bf NOW EXPLAIN!!!}

An \textbfit{almost-Dyck path} of length $2n$ is defined
to be a path $\omega = (\omega_0,\ldots,\omega_{2n})$
in the right half-plane $\N \times \Z$,
starting at $\omega_0 = (0,0)$ and ending at $\omega_{2n} = (2n,0)$,
using the steps $(1,1)$ and $(1,-1)$,
that stays always at height $\ge -1$.
Thus, an almost-Dyck path is like a Dyck path {\em except that}\/
a down step from height 0 to height $-1$ is allowed;
note, however, that it must be immediately followed by
an up step back to height 0.
Each non-Dyck part of the path is therefore of the form
$(h_{2i-2},h_{2i-1},h_{2i}) = (0,-1,0)$.
We write $\scrd^\sharp_{2n}$ for the set of almost-Dyck paths of length $2n$.

A \textbfit{0-Schr\"oder path} is defined
to be a Schr\"oder path in which long level steps, if any,
occur only at height 0.
We write $\scrs^0_{2n}$ for the set of 0-Schr\"oder paths of length $2n$.
There is an obvious bijection $\psi\colon \scrd^\sharp_{2n} \to \scrs^0_{2n}$
from almost-Dyck paths to 0-Schr\"oder paths:
namely, we replace each down-up pair starting and ending at height 0
with a long level step at height 0.

Both bijections
(\cite[Sections~6.1-6.3]{Deb-Sokal} 
and 
\cite[Section~6.5]{Deb-Sokal})
are correspondences between $\dperm_{2n}$ and the set of
$(\bfscra, \bfscrb,\bfscrc)$-labelled
0-Schr\"oder paths of length $2n$, where
the labels $\xi_i$ lie in the sets
\cite[eq.~(6.1)]{Deb-Sokal}
\begin{subeqnarray}
   \scra_h        & = &  \{0,\ldots, \lceil h/2 \rceil \}  \qquad\quad\;\;\hbox{for $h \ge 0$}  \\
   \scrb_h        & = &  \{0,\ldots, \lceil (h-1)/2 \rceil \}    \quad\hbox{for $h \ge 1$}  \\
   \scrc_0        & = &  \{0\}       \\
   \scrc_h        & = &  \emptyset   \qquad\qquad\qquad\qquad\quad\hbox{for $h \ge 1$}
 \label{def.abc.dperm}
\end{subeqnarray}

The Schr\"oder path was obtained by
first constructing an almost-Dyck path $\omega$
and then transforming it into a 0-Schr\"oder path $\omegahat = \psi(\omega)$.
The almost-Dyck path $\omega$ is described as follows:

\medskip

{\bf Step 1: Definition of the almost-Dyck path.}
Given a D-permutation $\sigma \in \dperm_{2n}$,
we define a path $\omega = (\omega_0,\ldots,\omega_{2n})$
starting at $\omega_0 = (0,0)$,
%% and ending at $\omega_{2n} = (2n,0)$,
with steps $s_1,\ldots,s_{2n}$ as follows:
\begin{itemize}
   \item If $\sinv(i)$ is even, then $s_i$ is a rise.
      (Note that in this case we must have $\sinv(i) \ge i$,
       by definition of D-permutation.)
   \item If $\sinv(i)$ is odd, then $s_i$ is a fall.
      (Note that in this case we must have $\sinv(i) \le i$,
       by definition of D-permutation.)
\end{itemize}
An alternative way of saying this is:
\begin{itemize}
   \item If $i$ is a cycle valley, cycle double fall or even fixed point,
      then $s_i$ is a rise.
   \item If $i$ is a cycle peak, cycle double rise or odd fixed point,
      then $s_i$ is a fall.
\end{itemize}

The resulting path is indeed an almost-Dyck path;
this was proved by obtaining a precise interpretation of the height $h_i$,
which then implied for every $i\in [2n]$, $h_i\ge -1$, and also $h_{2n} = 0$.
We recall this interpretation here:

%The fact that the resulting path is indeed an almost-Dyck path can be done
%by showing that all the heights $h_i$ are $\ge -1$ and that $h_{2n} = 0$.
%This was done by obtaining a precise interpretation
%of the height $h_i$ which we recall here:

\begin{lem}[{{\cite[Lemma~6.1]{Deb-Sokal}}}]
   \label{lemma.heights.dperm}
For $k \in [2n]$ we have
\be
   h_k
   \;=\;
   \begin{cases}
       2f_k - 1   & \textrm{if $k$ is odd}  \\
       2f_k       & \textrm{if $k$ is even}
   \end{cases}
 \label{eq.lemma.heights.dperm}
\ee
	where $f_k$ is defined as \cite[eq.~(6.2)/(6.3)]{Deb-Sokal}
\be
f_k \;=\; \#\{ i \le k \colon\: \sigma(i) > k \} \;=\; \#\{ i \le k \colon\: \sinv(i) > k \} \;. 
\label{eq.def.fk.dperm}
\ee
\end{lem}

By \cite[Lemma~3.2]{Deb-Sokal}, 
we have that $2i-1$ and $2i$ are record-antirecords
if and only if $\sigma$ maps
$\{1,\ldots,2i-1\}$ onto itself
and this is the only situation when $h_k = -1$ where $k=2i-1$ and $f_k = 0$.
In this situation, $f_{2i-2} = f_{2i-1} = f_{2i} = 0$,
so that $(h_{2i-2}, h_{2i-1}, h_{2i}) = (0, -1, 0)$.

\medskip

\begin{rem}
\label{rem.almostDyck.FZ}
We will now show that the almost-Dyck path $\omega$ described here 
can be obtained from the Motzkin path of length $2n$ 
used in the Foata--Zeilberger bijection. 
%(see Step~1 in Section~\ref{subsec.FZrecall.perm}).
For a given D-permutation $\sigma\in \dperm_{2n}$,
let $\widetilde{\omega}$ be the Motzkin path of length $2n$
constructed in Step~1 of Section~\ref{subsec.FZrecall.perm}.
We now construct a path $\omega'$ of length $2n$
which will consist only of rises and falls (with no level steps) as follows:
	\begin{itemize}
		\item if step $i$ in 
			$\widetilde{\omega}$ is a rise (fall resp.)
			then step $i$ in
			$\omega'$ is also a rise (fall resp.).
		\item if step $i$ in
                        $\widetilde{\omega}$ is a level step,
			then step $i$ in
                        $\omega'$ is a rise if $i$ is even,
			else it is a fall when $i$ is odd.
	\end{itemize}

We now note down the steps of $\omega'$ according to the different
cycle types of indices in $\sigma$:
	\begin{itemize}
		\item if $i$ is a cycle valley (cycle peak resp.) in $\sigma$
			then step $i$ in both 
			$\widetilde{\omega}$ and $\omega'$ are 
			rises (falls resp.).
		\item if $i$ is a cycle double fall 
			or an even fixed point in $\sigma$,
			$i$ must be even and
			step $i$ in
                        $\widetilde{\omega}$ must be a level step.
			Thus, step $i$ in 
			$\omega'$ is a rise.
		\item similarly, if $i$ is a cycle double rise 
                        or an odd fixed point in $\sigma$,
                        $i$ must be odd and 
                        step $i$ in
                        $\widetilde{\omega}$ must be a level step.
                        Thus, step $i$ in
                        $\omega'$ is a fall.
	\end{itemize}
From this description, it is clear that 
the path $\omega'$ is identical to the almost-Dyck path $\omega$
which we just constructed.
Furthermore, comparing $f_k$ in Lemma~\ref{lemma.heights.dperm} 
with Equation~\ref{eq.lemma.heights.equiv},
we get that Lemma~\ref{lemma.heights.dperm} can be thought of as
the relation between heights of steps in $\widetilde{\omega}$ 
(given by $f_i$)
and the corresponding heights in $\omega = \omega'$
(given by $h_i$).

For their first bijection, Deb and Sokal noted that the labels
were exactly the same except for the treatment of fixed points;
they said ``The surprise
was that the definitions of the labels 
in our constructions turned out to be almost
identical to those employed in \cite{Sokal-Zeng_masterpoly}
(where “almost” means that fixed points are treated
differently)''.
We have shown here that even the path, and not just the labels,
can be obtained from the bijection of Sokal and Zeng 
\cite{Sokal-Zeng_masterpoly} 
(which we recalled in Section~\ref{subsec.FZrecall.perm}.)
\myendremark
\end{rem}

\bigskip

\subsection[Continued fractions using record classification:\\ Proof of Theorems~\ref{thm.conj.DS},~\ref{thm.DS.pqgen},~\ref{thm.DS.master}]{Continued fractions using record classification:\\ Proof of Theorems~\ref{thm.conj.DS},~\ref{thm.DS.pqgen},~\ref{thm.DS.master}}
\label{subsec.dperm.proofs}

We first recall the description of labels and the inverse bijection
from \cite[Sections~6.2, 6.3]{Deb-Sokal} \
in Subsection~\ref{subsec.DSrecall.dperm}.
We then reinterpret the inverse bijection
in Subsection~\ref{subsec.laguerre.dperm}.

\subsubsection{Description of labels and inverse bijection}
\label{subsec.DSrecall.dperm}

\indent {\bf Step 2: Definition of the labels $\bm{\xi_i}$.}
We now recall the definition of the labels \cite[eq.~(6.9)]{Deb-Sokal}
\be
   \xi_i
   \;=\;
   \begin{cases}
       \#\{j \colon\: \sigma(j) < \sigma(i) \leq i <j\}
           & \textrm{if $i$ is even}
              \\[2mm]
     \#\{j \colon\: j<i \leq \sigma(i)<\sigma(j) \}
           & \textrm{if $i$ is odd}
   \end{cases}
 \label{def2.xi.dperm}
\ee
The definition \eqref{def2.xi.dperm} can be written equivalently as
\cite[eq.~(6.10)]{Deb-Sokal}
\be
   \xi_i
   \;=\;
   \begin{cases}
       \#\{2l > 2k \colon\: \sigma(2l) < \sigma(2k) \}
           & \textrm{if $i=2k$}
              \\[2mm]
     \#\{2l-1 < 2k-1 \colon\: \sigma(2l-1) > \sigma(2k-1) \}
           & \textrm{if $i=2k-1$}
   \end{cases}
 \label{def2.xi.bis.dperm}
\ee
since $\sigma(j) < j$ implies that $j$ is even,
and $j < \sigma(j)$ implies that $j$ is odd.

These definitions have a simple interpretation in terms of the
nesting statistics defined in \eqref{def.level}/(\ref{def.ucrossnestjk}b,d):
\be
   \xi_i 
   \;=\;
   \begin{cases}
        \lnest(i,\sigma)
           & \hbox{\rm if $i$ is even and $\neq \sigma(i)$}
             \;\, \hbox{\rm [equivalently, $i > \sigma(i)$]}
              \\[0.5mm]
        \unest(i,\sigma)
           & \hbox{\rm if $i$ is odd and $\neq \sigma(i)$}
             \;\, \hbox{\rm [equivalently, $i < \sigma(i)$]}
              \\[0.5mm]
        \psnest(i,\sigma)
           & \hbox{\rm if $i = \sigma(i)$  (that is, $i$ is a fixed point)}
   \end{cases}
 \label{def.xi.bis.dperm}
\ee

We state the inequalities \eqref{eq.xi.ineq}/\eqref{def.abc.dperm}
satisfied by the labels:
\begin{lem}[{\cite[Lemma~6.3]{Deb-Sokal}}]
   \label{lemma.ineqs.labels.dperm}
We have
\begin{subeqnarray}
   & 0 \;\le\; \xi_i  \;\le\; \left\lceil \dfrac{h_i-1}{2} \right\rceil \;=\;
                      \left\lceil\dfrac{h_{i-1}}{2} \right\rceil
       & \textrm{if $\sinv(i) $ is even (i.e., $s_i$ is a rise)}
    \qquad
	\slabel{eq.xi.ineqs.dperm.a}
      \\[3mm]
   & 0 \;\le\; \xi_i  \;\le\; \left\lceil \dfrac{h_i}{2} \right\rceil  \;=\;
                      \left\lceil\dfrac{h_{i-1}-1}{2} \right\rceil
       & \textrm{if $\sinv(i) $ is odd (i.e., $s_i$ is a fall)}
 \slabel{eq.xi.ineqs.dperm.b}
 \label{eq.xi.ineqs.dperm}
\end{subeqnarray}
\end{lem}

To verify that the inequalities \eqref{eq.xi.ineqs.dperm}
are satisfied, we interpret 
$\left\lceil \dfrac{h_i-1}{2} \right\rceil - \xi_i$ when $s_i$ is a rise,
and $\left\lceil \dfrac{h_i}{2} \right\rceil - \xi_i$ when $s_i$ is a fall,
in terms of the crossing statistics defined in (\ref{def.ucrossnestjk}a,c):

\begin{lem}[{\cite[Lemma~6.4]{Deb-Sokal}}]
   \label{lemma.xii.crossing.dperm}
\nopagebreak
\quad\hfill
\vspace*{-1mm}
\begin{itemize}
   \item[(a)]  If $s_i$ a rise and $i$ (hence also $h_i$) is odd, then
\be
   \left\lceil \dfrac{h_i - 1}{2} \right\rceil - \xi_i
   \;=\;
   \ucross(i,\sigma)
   \;.
 \label{eq.lemma.xii.crossing.dperm.a}
\ee
   \item[(b)]  If $s_i$ a rise and $i$ (hence also $h_i$) is even, then
\begin{subeqnarray}
   \left\lceil \dfrac{h_i - 1}{2} \right\rceil - \xi_i
   & = &
   \lcross(i,\sigma) \,+\, {\rm I}[\sigma(i) \neq i]
       \\[-1mm]
   & = &
   \lcross(i,\sigma) \,+\, {\rm I}[\hbox{\rm $i$ is a cycle double fall}]
   \;.
 \label{eq.lemma.xii.crossing.dperm.b}
\end{subeqnarray}
   \item[(c)]  If $s_i$ a fall and $i$ (hence also $h_i$) is odd, then
\begin{subeqnarray}
   \left\lceil \dfrac{h_i}{2} \right\rceil - \xi_i
   & = &
   \ucross(i,\sigma) \,+\, {\rm I}[\sigma(i) \neq i]
       \\[-1mm]
   & = &
   \ucross(i,\sigma) \,+\, {\rm I}[\hbox{\rm $i$ is a cycle double rise}]
   \;.
 \label{eq.lemma.xii.crossing.dperm.c}
\end{subeqnarray}
   \item[(d)]  If $s_i$ a fall and $i$ (hence also $h_i$) is even, then
\be
   \left\lceil \dfrac{h_i}{2} \right\rceil - \xi_i
   \;=\;
   \lcross(i,\sigma)
   \;.
 \label{eq.lemma.xii.crossing.dperm.d}
\ee
\end{itemize}
(Here ${\rm I}[\hbox{\sl proposition}] = 1$ if {\sl proposition} is true,
and 0 if it is false.)
\end{lem}

Since the quantities \eqref{eq.lemma.xii.crossing.dperm.a}--\eqref{eq.lemma.xii.crossing.dperm.d}
are manifestly nonnegative, the inequalities~\eqref{eq.xi.ineqs.dperm}
are satisfied.

\bigskip

{\bf Step 3: Proof of bijection.}
We recall the description of the inverse map for the mapping 
$\sigma \mapsto (\omega,\xi)$.

First, some preliminaries: Given a D-permutation $\sigma\in \dperm_{2n}$
we can define four subsets of $[2n]$:
\begin{subeqnarray}
F & = & \{2,4,\ldots,2n\} \; = \;  \hbox{even positions} \\
F' & = & \{i \colon \sinv(i) \text{ is even} \} \; = \; \{\sigma(2),\sigma(4),\ldots, \sigma(2n)\}\\
G & = & \{1,3,\ldots, 2n-1\} \; = \; \hbox{odd positions}  \\
G' & = & \{i \colon \sinv(i) \text{ is odd} \} \; = \; \{\sigma(1),\sigma(3),\ldots, \sigma(2n-1)\}
\label{eq.steps2excedanceclassification.dperm}
\end{subeqnarray}
Note that $F'$ (resp. $G'$) are the positions of the rises (resp. falls) in the
almost-Dyck path~$\omega$.

Let us observe that
\begin{subeqnarray}
F\cap F'    &=& \hbox{cycle double falls and even fixed points}\\
G\cap G'    &=& \hbox{cycle double rises and odd fixed points}\\
F\cap G'    &=& \hbox{cycle peaks}\\
F'\cap G    &=& \hbox{cycle valleys} \\
F \cap G    &=& \emptyset \\
F' \cap G'  &=& \emptyset
\label{eq.steps2cycleclassification.dperm}
\end{subeqnarray}

We can now describe the map $(\omega,\xi)\mapsto \sigma$. 
Given the almost-Dyck path $\omega$, 
we can immediately reconstruct the sets $F, F',G, G'$. 
We now use the labels $\xi$ to reconstruct the maps
$\sigma \restrict F \colon\: F \to F'$ and
$\sigma \restrict G \colon\: G \to G'$ as follows: 
The even subword $\sigma(2)\sigma(4)\cdots\sigma(2n)$ is a listing of $F'$
whose right-to-left inversion table is given by $q_\alpha = \xi_{2\alpha}$;
this is the content of~(\ref{def2.xi.bis.dperm}a).
Similarly, the odd subword $\sigma(1)\sigma(3)\cdots\sigma(2n-1)$ 
is a listing of $G'$
whose left-to-right inversion table is given by $p_\alpha = \xi_{2\alpha-1}$;
this is the content of~(\ref{def2.xi.bis.dperm}b).

\medskip

\begin{rem}
\label{rem.DS.FZ}
Using Remark~\ref{rem.almostDyck.FZ} and 
the fact that the description of labels~\eqref{def2.xi.dperm} other than fixed points is 
exactly the same as that in~\eqref{def.xi},
we can think of this bijection of Deb and Sokal \cite[Sections~6.1-6.3]{Deb-Sokal}
as almost a special case of the 
variant Foata--Zeilberger bijection of Sokal and Zeng \cite[Section~6.1]{Sokal-Zeng_masterpoly}
when restricted to D-permutations,
except for the treatment of fixed points; 
they are distinguished by parity.
The odd fixed points are clubbed together with excedances 
and the even fixed points are clubbed together with anti-excedances.\footnote{We
thank the anonymous referee for their suggestion
using which we were able to formalise why this bijection
seems to be almost a special case of the bijection 
in Section~\ref{subsec.FZrecall.perm}.}
\myendremark
\end{rem}

\medskip

\subsubsection{Running example 2}
\label{subsubsec.DSbij.eg.2}

Now let us take our second running example
\begin{eqnarray}
	\sigma &=& 7\, 1\, 9\, 2\, 5\, 4\, 8\, 6\, 10\, 3\, 11\, 12\, 14\, 13\, \nonumber\\
	  &=& (1,7,8,6,4,2)\,(3,9,10)\,(5)\,(11)\,(12)\,(13,14) \in \dperm_{14}
\end{eqnarray}
which was depicted in Figure~\ref{fig.pictorial.2}.
From the cycle classification of $\sigma$,
which was recorded in Equation~\eqref{eq.example.2.cycle.classification},
we immediately obtain the sets $F', G'$:
\begin{subeqnarray}
        F' & = & \{1,2,3,4,6,12,13\}\;, \\
        G' & = & \{5,7,8,9,10,11,14\}\;. 
	\label{eq.DSbij.eg.2.FGprime}
\end{subeqnarray}
This also gives us the positions of rises and falls in the almost-Dyck path 
$\omega$ corresponding to $\sigma$
which has been drawn in Figure~\ref{fig.almostdyck.running.example.2}.

\begin{figure}[!h]
\centering
\begin{tikzpicture}[scale = 0.8]
\draw[help lines, color=gray!30, dashed] (-0.5,-1.5) grid (14.5,5.5);
\draw[->] (0,0)--(14.5,0) node[right]{$x$};
\draw[->] (0,0)--(0,4.5) node[above]{$y$};
\draw[gray,dotted] (-1,-1) grid (6,2);

\node[circle, fill=black, minimum size=5pt, inner sep=1pt] at (1,0) {};
\node[circle, fill=black, minimum size=5pt, inner sep=1pt] at (2,0) {};
\node[circle, fill=black, minimum size=5pt, inner sep=1pt] at (3,0) {};
\node[circle, fill=black, minimum size=5pt, inner sep=1pt] at (4,0) {};
\node[circle, fill=black, minimum size=5pt, inner sep=1pt] at (5,0) {};
\node[circle, fill=black, minimum size=5pt, inner sep=1pt] at (6,0) {};
\node[circle, fill=black, minimum size=5pt, inner sep=1pt] at (7,0) {};
\node[circle, fill=black, minimum size=5pt, inner sep=1pt] at (8,0) {};
\node[circle, fill=black, minimum size=5pt, inner sep=1pt] at (9,0) {};
\node[circle, fill=black, minimum size=5pt, inner sep=1pt] at (10,0) {};
\node[circle, fill=black, minimum size=5pt, inner sep=1pt] at (11,0) {};
\node[circle, fill=black, minimum size=5pt, inner sep=1pt] at (12,0) {};
\node[circle, fill=black, minimum size=5pt, inner sep=1pt] at (13,0) {};
\node[circle, fill=black, minimum size=5pt, inner sep=1pt] at (14,0) {};

\node[circle, fill=black, minimum size=5pt, inner sep=1pt] at (0,1) {};
\node[circle, fill=black, minimum size=5pt, inner sep=1pt] at (0,2) {};
\node[circle, fill=black, minimum size=5pt, inner sep=1pt] at (0,3) {};
\node[circle, fill=black, minimum size=5pt, inner sep=1pt] at (0,4) {};

\node[] at (0,-0.5) {$0$};
\node[] at (1,-0.5) {$1$};
\node[] at (2,-0.5) {$2$};
\node[] at (3,-0.5) {$3$};
\node[] at (4,-0.5) {$4$};
\node[] at (5,-0.5) {$5$};
\node[] at (6,-0.5) {$6$};
\node[] at (7,-0.5) {$7$};
\node[] at (8,-0.5) {$8$};
\node[] at (9,-0.5) {$9$};
\node[] at (10,-0.5) {$10$};
\node[] at (11,-0.5) {$11$};
\node[] at (12,-0.5) {$12$};
\node[] at (13,-0.5) {$13$};
\node[] at (14,-0.5) {$14$};

\node[] at (-0.5,0) {$0$};
\node[] at (-0.5,1) {$1$};
\node[] at (-0.5,2) {$2$};
\node[] at (-0.5,3) {$3$};
\node[] at (-0.5,4) {$4$};

\draw[rounded corners=1, color=red, line width=1] (0,0)-- (1,1);
\draw[rounded corners=1, color=red, line width=1] (1,1)-- (2,2);
\draw[rounded corners=1, color=red, line width=1] (2,2)-- (3,3);
\draw[rounded corners=1, color=red, line width=1] (3,3)-- (4,4);
\draw[rounded corners=1, color=red, line width=1] (4,4)-- (5,3);
\draw[rounded corners=1, color=red, line width=1] (5,3)-- (6,4);
\draw[rounded corners=1, color=red, line width=1] (6,4)-- (7,3);
\draw[rounded corners=1, color=red, line width=1] (7,3)-- (8,2);
\draw[rounded corners=1, color=red, line width=1] (8,2)-- (9,1);
\draw[rounded corners=1, color=red, line width=1] (9,1)-- (10,0);
\draw[rounded corners=1, color=red, line width=1] (10,0)-- (11,-1);
\draw[rounded corners=1, color=red, line width=1] (11,-1)-- (12,0);
\draw[rounded corners=1, color=red, line width=1] (12,0)-- (13,1);
\draw[rounded corners=1, color=red, line width=1] (13,1)-- (14,0);

\node[circle, fill=red, minimum size=5pt, inner sep=1pt] at (0,0) {};
\node[circle, fill=red, minimum size=5pt, inner sep=1pt] at (1,1) {};
\node[circle, fill=red, minimum size=5pt, inner sep=1pt] at (2,2) {};
\node[circle, fill=red, minimum size=5pt, inner sep=1pt] at (3,3) {};
\node[circle, fill=red, minimum size=5pt, inner sep=1pt] at (4,4) {};
\node[circle, fill=red, minimum size=5pt, inner sep=1pt] at (5,3) {};
\node[circle, fill=red, minimum size=5pt, inner sep=1pt] at (6,4) {};
\node[circle, fill=red, minimum size=5pt, inner sep=1pt] at (7,3) {};
\node[circle, fill=red, minimum size=5pt, inner sep=1pt] at (8,2) {};
\node[circle, fill=red, minimum size=5pt, inner sep=1pt] at (9,1) {};
\node[circle, fill=red, minimum size=5pt, inner sep=1pt] at (10,0) {};
\node[circle, fill=red, minimum size=5pt, inner sep=1pt] at (11,-1) {};
\node[circle, fill=red, minimum size=5pt, inner sep=1pt] at (12,0) {};
\node[circle, fill=red, minimum size=5pt, inner sep=1pt] at (13,1) {};
\node[circle, fill=red, minimum size=5pt, inner sep=1pt] at (14,0) {};

\end{tikzpicture}
\caption{Almost Dyck path $\omega$ corresponding to the permutation\\
$\sigma = 7\, 1\, 9\, 2\, 5\, 4\, 8\, 6\, 10\, 3\, 11\, 12\, 14\, 13\,
           = (1,7,8,6,4,2)\,(3,9,10)\,(5)\,(11)\,(12)\,(13,14) \in \Sym_{14}$.}
\label{fig.almostdyck.running.example.2}
\end{figure}

%The sets $F,F',G,G',H$ are
%\begin{subeqnarray}
%        F & = & \{1, 3, 7, 9, 13\} \\
%        F' & = & \{7, 8, 9, 10, 14\} \\
%        G & = & \{2, 4, 6, 8, 10, 14 \} \\
%        G' & = & \{1, 2, 3, 4, 6, 13\} \\
%        H & = & \{5,11,12\}
%\label{eq.FGH.running.example.2}
%\end{subeqnarray}

The labels $\xi_i$ are given in Equation~\eqref{eq.labels.DS.eg.2}.
\begin{subeqnarray}
\begin{array}{r}
2i\in F  \\[3mm]
	F' = \{\sigma(2i)\; | \: i\in [7]\}   \\[3mm]
\text{Right-to-left inversion table:} \; \xi_{2i}   \\
\end{array}
&&\hspace*{-6mm}
\begin{pmatrix}
	\; 2 &  4 & 6 & 8 & 10  & 12 & 14\; \\[3mm]
	\; 1 & 2 & 4 & 6 & 3 & 12 & 13 \; \\[3mm]
	\; 0 & 0 & 1 & 1 & 0 & 0 & 0  \; 
\end{pmatrix} 
	\\[5mm]
\begin{array}{r}
2i-1\in G  \\[3mm]
	G' = \{\sigma(2i-1)\; | \: i\in [7]\}   \\[3mm]
\text{Left-to-right inversion table:} \; \xi_{2i-1}   \\
\end{array}
	&&\hspace*{-6mm}
\begin{pmatrix}
	\; 1 & 3 & 5 & 7 &  9 & 11 & 13 \; \\[3mm]
	\; 7 & 9 & 5 & 8 & 10 & 11 & 14 \; \\[3mm]
	\; 0 & 0 & 2 & 1 & 0 & 0 & 0
\end{pmatrix}
\label{eq.labels.DS.eg.2}
\end{subeqnarray}
%
%
%
%%{\bf DO THE FIRST INDICES LOOK UGLY???????????  WHAT DO I DO ABOUT EQUATION NUMBERING VS FIGURE ENVIRONMENT!!!!!!!!!}
%
%
%{\bf THEN DEMONSTRATE INVERSE BIJECTION???????}
%
\bigskip

\subsubsection{Combinatorial interpretation using Laguerre digraphs}
\label{subsec.laguerre.dperm}

The construction here will almost mirror 
the construction in Section~\ref{subsec.laguerre}.
We will only include the necessary details and state the necessary lemmas
and will omit most of the proofs.
%as they are mostly similar
%to the lemmas in Section~\ref{subsec.laguerre}.

We begin with an almost-Dyck path $\omega$ 
and an assignment of labels $\xi$ satisfying 
\eqref{eq.xi.ineqs.dperm}.
The inverse bijection in 
Section~\ref{subsec.DSrecall.dperm} Step~3,
gives us a D-permutation $\sigma$.
We will again break this process into several intermediate steps
and provide a reinterpretation using Laguerre digraphs.
We will use the same conventions for denoting Laguerre digraphs 
as in Section~\ref{subsec.laguerre}.

Recall that the inverse bijection  
(Section~\ref{subsec.DSrecall.dperm} Step 3)
begins by obtaining the sets $F, G$ (which are fixed for any given $n$)
and $F',G'$ from the almost-Dyck path $\omega$.
We then construct $\sigma\restrict F\colon \: F \to F'$ 
and $\sigma\restrict G\colon \: G\to G'$ separately 
by using the labels $\xi\restrict F$ and $\xi\restrict G$ respectively.

We again start with the digraph $\laguerre{\emptyset}$ and 
then go through the set $[2n]$.
However, we first go through $F$ and then go through $G$,
i.e., the order of our steps is now
$2,4,\ldots, 2n,2n-1,\ldots, 3, 1$.
In this situation, stage (a) will involve going through the even vertices in
increasing order
and then stage (b) will involve going through the odd vertices but
in decreasing order.
Thus, unlike the situation in Section~\ref{subsec.laguerre},
the FZ order corresponding to two different 
D-permutations $\sigma,\sigma' \in \dperm_{2n}$
are the same, 
irrespective of the underlying
almost-Dyck path.

Let us now look at the intermediate Laguerre digraphs obtained 
during stages~(a) and~(b).

\medskip

%\noindent{\bf Stage (a): Going through $\boldsymbol{H}$:}\\
%For each vertex $i\in H$, we introduce a loop edge $i\to i$
%thus creating a new loop at the end of each step.
%
%After all steps $i\in H$ have been carried out,
%the resulting Laguerre digraph $\laguerre{H}$ consists of loops 
%at all vertices in $H$.
%All other vertices are isolated vertices.
%
%
%\medskip

{\bf Stage (a): Going through $\boldsymbol{F = \{2,4,\ldots,2n\}}$:}\\
We go through the even vertices in increasing order.
From \eqref{eq.steps2cycleclassification.dperm}, 
we know that 
$F = \Cdfall(\sigma) \, \cup \, \Evenfix(\sigma) \, \cup \, \Cpeak(\sigma)$ 
and 
$F' = \Cdfall(\sigma) \, \cup \, \Evenfix(\sigma) \, \cup \, \Cval(\sigma)$
where $\sigma$ is the resulting D-permutation obtained at the end
of the inverse bijection.

The construction here is similar to the construction in
stage (b) Section~\ref{subsec.laguerre}.
The connected components at the end of this stage can be described as follows:

\begin{lem} The Laguerre digraph $\laguerre{F}$
consists of the following connected components:
\begin{itemize} 
        \item loops on vertices $u\in \Evenfix$,

        %\item no directed cycles

	\item directed paths with at least two vertices,
                in which the initial vertex of the path
                is a cycle peak in $\sigma$
                (i.e. contained in the set $F\,\cap\, G'$),
                the final vertex is a cycle valley in $\sigma$
                (i.e. contained in the set $F'\,\cap\, G$),
                and the intermediate vertices (if any)
                are cycle double falls
                (which belong to the set $F\,\cap\, F'$).
		%This is true because all edges
                %corresponding to excedances
                %have been inserted
                %and no edge corresponding
                %anti-excedances have been inserted
                %at this stage.

	\item  isolated vertices at $u\in G\,\cap\, G' = \Cdrise(\sigma)\, \cup \, \Oddfix(\sigma)$.
\end{itemize} 
Furthermore, it contains no directed cycles.
\label{lem.afterF.dperm}
\end{lem}

%The proof of this lemma is similar to that of
%Lemma~\ref{lem.afterHG} and we omit it.

\medskip

{\bf Stage (b): Going through $\boldsymbol{G = \{2n-1,\ldots, 3,1\}}$:}\\
We now go through the odd vertices in decreasing order.
From \eqref{eq.steps2cycleclassification.dperm},
we know that
$G = \Cdrise(\sigma) \, \cup \, \Oddfix(\sigma) \, \cup \, \Cval(\sigma)$,
and
$G' = \Cdrise(\sigma) \, \cup \, \Oddfix(\sigma) \, \cup \, \Cpeak(\sigma)$,
where $\sigma$ is the resulting D-permutation obtained at the end
of the inverse bijection.

The construction here is similar to stage (c) 
in Section~\ref{subsec.laguerre}.
The final vertices of a path with at least two vertices
have the following description:
\begin{sloppy}
\begin{lem} Let $u$ be the final vertex of a path
with at least two vertices in $\laguerre{F \,\cup\, \{2n-1,\ldots, 2(n-j)+1\}}$ 
for some index $j$ ($1\leq j \leq n$). 
Then $u\in \Cval$.
\label{lem.lastvertexG.dperm}
\end{lem}
\end{sloppy}

%\begin{proof} 
%	Let $u$ be the last vertex of a path in 
%	$\laguerre{F \,\cup\, \{2n-1,\ldots, 2(n-j)+1\}}$.
%	Since all vertices in $F$ were already assigned out-neighbours
%	during stage (a), it must be that $u\in G$.
%	Thus $u$ is either a cycle valley, a cycle double rise or 
%	an odd fixed point in $\sigma$.
%	However, as odd fixed point can only be an isolated vertex
%	or a loop, it cannot be the last vertex of a path with $2$ 
%	vertices.
%
%	Let us assume that $u$ is a cycle double rise in $\sigma$, i.e.,
%	$\sinv(u) <u< \sigma(u)$.
%	We know that $u$ must have an in-neighbour $v$
%	as its path has at-least two vertices.
%	Thus, $v = \sinv(u)$ and $v<u$, which implies $v\in G$
%	{\bf FROM WHERE??????}
%%	(from definition of $G$ in~(\ref{eq.steps2excedanceclassification.dperm})).
%	However, this is a situation where we have two vertices $u,v\in G$
%	with $u>v$ such that the smaller vertex has an out-neighbour
%	even though the larger vertex does not.
%	This clearly cannot happen as we assign out-neighbours
%	to vertices in $F$ in descending order.
%	This is a contradiction and thus $u\in \Cval$.	
%\end{proof}

Our definition of cycle closers is again the same as in 
Section~\ref{subsec.laguerre}.
The following lemma classifies all cycle closers.

\begin{lem}(Classifying cycle closers) Given a D-permutation $\sigma$,
	an element $u\in [2n]$ is a cycle closer
	if and only if 
	it is a cycle valley minimum,
	i.e., 
	it is the smallest element
	in its cycle.
\label{lem.classifyingCycleG.dperm}
\end{lem}

%\begin{proof} Let $u_j$ be a cycle closer.
%Notice that $u_j$ must have been the last vertex of an unfinished path
%in $\left.L\right|_{\{u_1,\ldots, u_{j-1}\}}$.
%Also, from Lemma~\ref{lem.afterF}, $u_j\not\in F$.
%Thus, $u_j\in G$ and from Lemma~\ref{lem.lastvertexG.dperm},
%$u_j\in G\cap F' = \Cval$.
%
%The out-edge from any other cycle valley $v$ in the cycle of $u_j$
%must have been present in $\left.L\right|_{\{u_1,\ldots, u_{j-1}\}}$
%before $u_j\to \sigma(u_j)$ was inserted to obtain 
%	$\left.L\right|_{\{u_1,\ldots, u_{j}\}}$.
%Hence, it must be that $v>u_j$ as $\Cval \subseteq G$ 
%and we go through $G$ in descending order.
%Thus, the cycle closer $u_j$ is the smallest cycle valley 
%in its cycle in $\sigma$.
%\end{proof}

Next, we will count the number of cycle closers.
But before doing that, we require a technical lemma similar to 
Lemma~\ref{lem.cycle.closer.technical} for the case of permutations.
However, first  notice that if 
$i \in G\,\cap\,F' = \Cval(\sigma)$,
step $s_{i}$ must be a rise from height $h_{i-1}$ to height $h_{i}$ 
and hence, $h_{i-1}+1 = h_{i}$.
Also, from the interpretation of the heights in Lemma~\ref{lemma.heights.dperm}, 
we must have $\lceil h_{i-1}/2\rceil +1 = \lceil (h_{i}+1)/2\rceil = f_i$.

\begin{lem}
Given a D-permutation $\sigma$ and associated sets $F',G'$
and an odd\\ \mbox{$y\in \{2n-1, \ldots, 3,1\}$} such that $y\in G\,\cap\, F' = \Cval$.
Then the following is true:
\begin{eqnarray}
\#\{u \in  G'\backslash \{\sigma(2n-1),\sigma(2n-3),\ldots, \sigma(y+2)\} \colon \: u>y\}
 &=& \lceil h_{y-1}/2\rceil +1 \nonumber\\
	&=& \lceil (h_{y}+1)/2\rceil \nonumber\\
	&=& f_y	
\label{eq.lem.cycle.closer.technical.dperm}
\end{eqnarray}
where $h_i$ denotes the height at position $i$ of the almost-Dyck path $\omega$
associated to $\sigma$
in Step~1,
and $f_i$ is defined in \eqref{eq.def.fk.dperm}.
\label{lem.cycle.closer.technical.dperm}
\end{lem}

\begin{proof} We first establish the following equality of sets:
\be
	\{u > y \colon \: \sinv(u) \leq  y\} \; = \;  \{u \in   G'\backslash \{\sigma(2n-1),\sigma(2n-3),\ldots, \sigma(y+2)\} \colon\: u > y \}.
\label{eq.technical.equation.a.dperm}
\ee

Whenever $u\in G'$, we have that $\sinv(u)\in G$
(by description of $G$, $G'$ in (\ref{eq.steps2excedanceclassification.dperm})).
Additionally, if $u\not\in \{\sigma(2n-1),\sigma(2n-3), \ldots, \sigma(y+2)\}$
then it must be that $\sinv(u)\leq y$.
This establishes the containment
\[\{u > y \colon \: \sinv(u) \leq  y\} \; \supseteq \;  \{u \in   G'\backslash \{\sigma(2n-1),\sigma(2n-3),\ldots, \sigma(y+2)\} \colon\: u > y \}.\]

On the other hand, if $u>y$ and $\sinv(u)\leq y$,
then $u> \sinv(u)$ and as $\sigma$ is a D-permutation, 
$\sinv(u)$ must be odd. Therefore, $u\in G'$.
As $\sinv(u)\leq y$,
$u$ cannot be one of $\sigma(y+2), \ldots,\sigma(2n-3), \sigma(2n-1)$.
Therefore, $u  \in   G'\backslash \{\sigma(2n-1),\sigma(2n-3),\ldots, \sigma(y+2)\}$.
This establishes \eqref{eq.technical.equation.a.dperm}.

To obtain Equation~\eqref{eq.lem.cycle.closer.technical.dperm},
it suffices to show that the cardinality of the set
$\{u> y \colon \: \sinv(u) \leq  y\}$
is $f_{y}$.
To do this, recall the description of $f_y$
in Equation~\eqref{eq.def.fk.dperm} and observe that
\begin{eqnarray}
    f_y  & = &  \#\{ u \le y \colon\: \sigma(u) > y \}  \nonumber \\
            & = &  \# \{u > y \colon\:  \sinv(u) \leq y \}
\end{eqnarray}
where the second equality is obtained by replacing $u$ with $\sinv(u)$.
\end{proof}

We are now ready to state the counting of cycle closers.

\begin{lem}[Counting of cycle closers for D-permutations]
Fix an almost-Dyck path $\omega$ of length $2n$ and
construct $F',G'$ (these are completely determined by $\omega$).
Also fix labels $\xi_u$ for vertices
$u\in\{2,4,\ldots, 2n\} \,\cup\, \{2n-1,\ldots, 2(n-j)+1\}$
satisfying~\eqref{eq.xi.ineqs.dperm}.
Also let $y = 2(n-j)-1\in G\cap F' = \Cval(\sigma)$.
Then
\begin{itemize}
      \item[(a)] The value of $\xi_{y}$ completely determines 
               if $y$ is a cycle closer or not.
      \item[(b)] There is exactly one value $\xi_{y}\in \{0,1,\ldots,\lceil h_{y-1}/2\rceil\}$
               that makes $y$ a cycle closer, and conversely.
\end{itemize}
\label{lem.cycle.closer.dperm}
\end{lem}

\begin{rem}
%\noindent {\bf Remark.} 
Notice that one can also construct a variant of this
interpretation where stage (b) occurs before stage (a). The role of cycle closer
will then be played by cycle peak maximum. 
\myendremark
\end{rem}

\subsubsection{Running example 2}

Let us now look at the DS history of our running example for D-permutation
\begin{eqnarray}
	\sigma  &=&  7\, 1\, 9\, 2\, 5\, 4\, 8\, 6\, 10\, 3\, 11\, 12\, 14\, 13\,
	\nonumber\\
	&=&      (1,7,8,6,4,2)\,(3,9,10)\,(5)\,(11)\,(12)\,(13,14) \in \dperm_{14}.
\end{eqnarray}
%The sets $F,G$ and $H$ were already recorded in~\eqref{eq.FGH.running.example.2}
%and we recall that $F = \{1, 3, 7, 9, 13\}$, $G = \{2, 4, 6, 8, 10, 14\}$ and $H=\{5, 11, 12\}$.
It consists of the following stages:
\begin{itemize}
\item Stage (a): 2, 4, 6, 8, 10, 12, 14
\item Stage (b): 13, 11, 9, 7, 5, 3, 1
\end{itemize}
Stage~(a) of the  history of $\sigma$ has been drawn in
Figure~\ref{fig.running.example.2.DS.history.a}
and Stage~(b) has been drawn in Figure~\ref{fig.running.example.2.DS.history.b}.
Non-singleton cycles are formed in Stage~(b) when the edges $13\to 14$, $3\to 9$
and $1\to 7$ are inserted.

\begin{figure}[p]
\vspace*{-12mm}
\centering
\begin{tabular}{l}
%\begin{tikzpicture}[scale=0.6][scale = 1]
%\node at (0,0) {Begin with $\laguerre{\emptyset}$ with vertex set $\{1,2,3,4,5,6,7,8,9,10,11\}$ and no edges};
%\node[right] at (-4,-2.5) {$\laguerre{\emptyset}$};
%\emptysigmatwo{-2.5}
%\node[right] at (-4,-5) {$\laguerre{\emptyset}$};
%\emptysigmatwo{-5}
%\end{tikzpicture}\\
%
\begin{tikzpicture}[scale=0.6]
%\node[right] at (-7,0) {$\laguerre{\emptyset}$};

\node[right] at (-7,0) {$\laguerre{\emptyset}$};
\emptysigmatwo{0}%

\end{tikzpicture}\\
%\hline\\[-4mm]
%\hline\\[-3mm]
%After Stage (a): $H = \{5, 11, 12\}$\\[2mm]
%\hline\\[-3mm]
%%
%%
%\begin{tikzpicture}[scale=0.6]
%\node[right] at (-7,0) {$\laguerre{\{5,11,12\}}$};
%\emptysigmatwo{0}
%\node[circle,fill=black,inner sep=1pt,minimum size=5pt] (e) at (9,\ey) {} edge [in=45,out=135, thick, loop above, color=red] node {} ();
%\node[circle,fill=black,inner sep=1pt,minimum size=5pt] (k) at (10,\ky) {} edge [in=45,out=135, thick, loop above,color=red] node {} ();
%\node[circle,fill=black,inner sep=1pt,minimum size=5pt] (l) at (11,\ly) {} edge [in=45,out=135, thick, loop above,color=red] node {} ();
%\end{tikzpicture}\\
%
\hline\\[-4mm]
\hline\\[-3mm]
Stage (a): $F = \{2, 4, 6, 8, 10, 12, 14\}$ in increasing order\\[2mm]
\hline\\[-3mm]
\begin{tikzpicture}[scale=0.6]
\node[right] at (-7,0) {$\laguerre{\{2\}}$};
\emptysigmatwo{0}
%
%\node[circle,fill=black,inner sep=1pt,minimum size=5pt] (e) at (9,\ey) {} edge [in=45,out=135, thick, loop above] node {} ();
%\node[circle,fill=black,inner sep=1pt,minimum size=5pt] (k) at (10,\ky) {} edge [in=45,out=135, thick, loop above] node {} ();
%\node[circle,fill=black,inner sep=1pt,minimum size=5pt] (l) at (11,\ly) {} edge [in=45,out=135, thick, loop above] node {} ();
%
\graph [multi, edges = {thick,red}] {(b) -> (a); };
\end{tikzpicture}\\
\hdashline\\[-3mm]
\begin{tikzpicture}[scale=0.6]
\node[right] at (-7,0) {$\laguerre{\{2,4\}}$};
\emptysigmatwo{0}
%
%\node[circle,fill=black,inner sep=1pt,minimum size=5pt] (e) at (9,\ey) {} edge [in=45,out=135, thick, loop above] node {} ();
%\node[circle,fill=black,inner sep=1pt,minimum size=5pt] (k) at (10,\ky) {} edge [in=45,out=135, thick, loop above] node {} ();
%\node[circle,fill=black,inner sep=1pt,minimum size=5pt] (l) at (11,\ly) {} edge [in=45,out=135, thick, loop above] node {} ();
%
	\graph [multi, edges = {thick}] {(b) -> (a); };
	\graph [multi, edges = {thick,red}] {(d) -> (b); };
\end{tikzpicture}\\
\hdashline\\[-3mm]
\begin{tikzpicture}[scale=0.6]
\node[right] at (-7,0) {$\laguerre{\{2,4,6\}}$};
\emptysigmatwo{0}
%
%\node[circle,fill=black,inner sep=1pt,minimum size=5pt] (e) at (9,\ey) {} edge [in=45,out=135, thick, loop above] node {} ();
%\node[circle,fill=black,inner sep=1pt,minimum size=5pt] (k) at (10,\ky) {} edge [in=45,out=135, thick, loop above] node {} ();
%\node[circle,fill=black,inner sep=1pt,minimum size=5pt] (l) at (11,\ly) {} edge [in=45,out=135, thick, loop above] node {} ();
%
\graph [multi, edges = {thick}] {(b) -> (a); (d) -> (b); };
\graph [multi, edges = {thick,red}] {(f) -> (d); };
\end{tikzpicture}\\
\hdashline\\[-3mm]
\begin{tikzpicture}[scale=0.6]
\node[right] at (-7,0) {$\laguerre{\{2,4,6,8\}}$};
\emptysigmatwo{0}
%
%\node[circle,fill=black,inner sep=1pt,minimum size=5pt] (e) at (9,\ey) {} edge [in=45,out=135, thick, loop above] node {} ();
%\node[circle,fill=black,inner sep=1pt,minimum size=5pt] (k) at (10,\ky) {} edge [in=45,out=135, thick, loop above] node {} ();
%\node[circle,fill=black,inner sep=1pt,minimum size=5pt] (l) at (11,\ly) {} edge [in=45,out=135, thick, loop above] node {} ();
%
\graph [multi, edges = {thick}] {(b) -> (a); (d) -> (b); (f) -> (d); };
\graph [multi, edges = {thick,red}] {(h) -> (f); };
\end{tikzpicture}\\
\hdashline\\[-3mm]
\begin{tikzpicture}[scale=0.6]
\node[right] at (-7,0) {$\laguerre{\{2,4,6,8,10\}}$};
\emptysigmatwo{0}
%
%\node[circle,fill=black,inner sep=1pt,minimum size=5pt] (e) at (9,\ey) {} edge [in=45,out=135, thick, loop above] node {} ();
%\node[circle,fill=black,inner sep=1pt,minimum size=5pt] (k) at (10,\ky) {} edge [in=45,out=135, thick, loop above] node {} ();
%\node[circle,fill=black,inner sep=1pt,minimum size=5pt] (l) at (11,\ly) {} edge [in=45,out=135, thick, loop above] node {} ();
%
\graph [multi, edges = {thick}] {(b) -> (a); (d) -> (b); (f) -> (d), (h) -> (f); }; 
\graph [multi, edges = {thick,red}] {(j) -> (c); }; 
\end{tikzpicture}\\
\hdashline\\[-3mm]
\begin{tikzpicture}[scale=0.6]
\node[right] at (-7,0) {$\laguerre{\{2,4,6,8,10,12\}}$};
\emptysigmatwo{0}
%
%\node[circle,fill=black,inner sep=1pt,minimum size=5pt] (e) at (9,\ey) {} edge [in=45,out=135, thick, loop above] node {} ();
%\node[circle,fill=black,inner sep=1pt,minimum size=5pt] (k) at (10,\ky) {} edge [in=45,out=135, thick, loop above] node {} ();
\node[circle,fill=black,inner sep=1pt,minimum size=5pt] (l) at (11,\ly) {} edge [in=45,out=135, thick, loop above, color=red] node {} ();
\graph [multi, edges = {thick}] {(b) -> (a); (d) -> (b); (f) -> (d), (h) -> (f); 
	(j) -> (c); }; 

%\graph [multi, edges = {thick}] {(n) -> [bend right] (m)};

%
\end{tikzpicture}\\
\hdashline\\[-3mm]
\begin{tikzpicture}[scale=0.6]
\node[right] at (-7,0) {$\laguerre{\{2,4,6,8,10,12,14\}}$};
\emptysigmatwo{0}
%
%\node[circle,fill=black,inner sep=1pt,minimum size=5pt] (e) at (9,\ey) {} edge [in=45,out=135, thick, loop above] node {} ();
%\node[circle,fill=black,inner sep=1pt,minimum size=5pt] (k) at (10,\ky) {} edge [in=45,out=135, thick, loop above] node {} ();
\node[circle,fill=black,inner sep=1pt,minimum size=5pt] (l) at (11,\ly) {} edge [in=45,out=135, thick, loop above] node {} ();
\graph [multi, edges = {thick}] {(b) -> (a); (d) -> (b); (f) -> (d), (h) -> (f);
        (j) -> (c); };

\graph [multi, edges = {thick,red}] {(n) -> [bend right] (m)};

\end{tikzpicture}
\end{tabular}
\caption{Stage (a) of the DS history for the D-permutation\\
$\sigma = 7\, 1\, 9\, 2\, 5\, 4\, 8\, 6\, 10\, 3\, 11\, 12\, 14\, 13\,
	 =  (1,7,8,6,4,2)\,(3,9,10)\,(5)\,(11)\,(12)\,(13,14) \in \dperm_{14}.$}
\label{fig.running.example.2.DS.history.a}
\end{figure}

\begin{figure}[p]
\centering
\begin{tabular}{l}
\hline\\[-4mm]
\hline\\
Stage (b): $G = \{13,11, 9, 7, 5, 3,1\}$ in decreasing order\\[2mm]
\hline\\[2mm]
\begin{tikzpicture}[scale=0.7]
\node[right] at (-8,0) {$\laguerre{\{2,4,6,8,10,12,14,13\}}$};
\emptysigmatwo{0}
%
%\node[circle,fill=black,inner sep=1pt,minimum size=5pt] (e) at (9,\ey) {} edge [in=45,out=135, thick, loop above] node {} ();
%\node[circle,fill=black,inner sep=1pt,minimum size=5pt] (k) at (10,\ky) {} edge [in=45,out=135, thick, loop above] node {} ();
\node[circle,fill=black,inner sep=1pt,minimum size=5pt] (l) at (11,\ly) {} edge [in=45,out=135, thick, loop above] node {} ();
\graph [multi, edges = {thick}] {(b) -> (a); (d) -> (b); (f) -> (d), (h) -> (f);
        (j) -> (c); };

\graph [multi, edges = {thick}] {(n) -> [bend right] (m);};
\graph [multi, edges = {thick,red}] {(m) -> [bend right] (n);};

\end{tikzpicture}\\
\hdashline\\[-3mm]
\begin{tikzpicture}[scale=0.7]
\node[right] at (-8,0) {$\laguerre{\{2,4,6,8,10,12,14,13,11\}}$};
\emptysigmatwo{0}
%
%\node[circle,fill=black,inner sep=1pt,minimum size=5pt] (e) at (9,\ey) {} edge [in=45,out=135, thick, loop above] node {} ();
\node[circle,fill=black,inner sep=1pt,minimum size=5pt] (k) at (10,\ky) {} edge [in=45,out=135, thick, loop above, color=red] node {} ();
\node[circle,fill=black,inner sep=1pt,minimum size=5pt] (l) at (11,\ly) {} edge [in=45,out=135, thick, loop above] node {} ();
\graph [multi, edges = {thick}] {(b) -> (a); (d) -> (b); (f) -> (d), (h) -> (f);
        (j) -> (c); };

\graph [multi, edges = {thick}] {(n) -> [bend right] (m);};

\graph [multi, edges = {thick}] {(m) -> [bend right] (n);};

%\graph [multi, edges = {thick,red}] {(i) -> (j);};
%
\end{tikzpicture}\\
\hdashline\\[-3mm]
\begin{tikzpicture}[scale=0.7]
\node[right] at (-8,0) {$\laguerre{\{2,4,6,8,10,12,14,13,11,9\}}$};
\emptysigmatwo{0}
%
%\node[circle,fill=black,inner sep=1pt,minimum size=5pt] (e) at (9,\ey) {} edge [in=45,out=135, thick, loop above] node {} ();
\node[circle,fill=black,inner sep=1pt,minimum size=5pt] (k) at (10,\ky) {} edge [in=45,out=135, thick, loop above] node {} ();
\node[circle,fill=black,inner sep=1pt,minimum size=5pt] (l) at (11,\ly) {} edge [in=45,out=135, thick, loop above] node {} ();
\graph [multi, edges = {thick}] {(b) -> (a); (d) -> (b); (f) -> (d), (h) -> (f);
        (j) -> (c); };

\graph [multi, edges = {thick}] {(n) -> [bend right] (m);};

\graph [multi, edges = {thick}] {(m) -> [bend right] (n);};

\graph [multi, edges = {thick,red}] {(i) -> (j); };
%\graph [multi, edges = {thick,red}] {(g) -> (h); };
%
\end{tikzpicture}\\
\hdashline\\[-3mm]
\begin{tikzpicture}[scale=0.7]
\node[right] at (-8,0) {$\laguerre{\{2,4,6,8,10,12,14,13,11,9,7\}}$};
\emptysigmatwo{0}
%
%\node[circle,fill=black,inner sep=1pt,minimum size=5pt] (e) at (9,\ey) {} edge [in=45,out=135, thick, loop above] node {} ();
\node[circle,fill=black,inner sep=1pt,minimum size=5pt] (k) at (10,\ky) {} edge [in=45,out=135, thick, loop above] node {} ();
\node[circle,fill=black,inner sep=1pt,minimum size=5pt] (l) at (11,\ly) {} edge [in=45,out=135, thick, loop above] node {} ();
\graph [multi, edges = {thick}] {(b) -> (a); (d) -> (b); (f) -> (d), (h) -> (f);
        (j) -> (c); };

\graph [multi, edges = {thick}] {(n) -> [bend right] (m);};

\graph [multi, edges = {thick}] {(m) -> [bend right] (n);};

%\graph [multi, edges = {thick}] {(i) -> (j); (g) -> (h); };
\graph [multi, edges = {thick}] {(i) -> (j); };
\graph [multi, edges = {thick,red}] {(g) -> (h); };

%\graph [multi, edges = {thick,red}] {(c) -> (i); };
%
\end{tikzpicture}\\
\hdashline\\[-3mm]
\begin{tikzpicture}[scale=0.7]
\node[right] at (-8,0) {$\laguerre{\{2,4,6,8,10,12,14,13,11,9,7,5\}}$};
\emptysigmatwo{0}
\node[circle,fill=black,inner sep=1pt,minimum size=5pt] (e) at (9,\ey) {} edge [in=45,out=135, thick, loop above,color=red] node {} ();
\node[circle,fill=black,inner sep=1pt,minimum size=5pt] (k) at (10,\ky) {} edge [in=45,out=135, thick, loop above] node {} ();
\node[circle,fill=black,inner sep=1pt,minimum size=5pt] (l) at (11,\ly) {} edge [in=45,out=135, thick, loop above] node {} ();
\graph [multi, edges = {thick}] {(b) -> (a); (d) -> (b); (f) -> (d), (h) -> (f);
        (j) -> (c); };

\graph [multi, edges = {thick}] {(n) -> [bend right] (m);};

\graph [multi, edges = {thick}] {(m) -> [bend right] (n);};

\graph [multi, edges = {thick}] {(i) -> (j); (g) -> (h);};
%\graph [multi, edges = {thick,red}] {(a) -> (g)};
%
\end{tikzpicture}\\
\hdashline\\[-3mm]
\begin{tikzpicture}[scale=0.7]
\node[right] at (-8,0) {$\laguerre{\{2,4,6,8,10,12,14,13,11,9,7,5,3\}}$};
\emptysigmatwo{0}
\node[circle,fill=black,inner sep=1pt,minimum size=5pt] (e) at (9,\ey) {} edge [in=45,out=135, thick, loop above] node {} ();
\node[circle,fill=black,inner sep=1pt,minimum size=5pt] (k) at (10,\ky) {} edge [in=45,out=135, thick, loop above] node {} ();
\node[circle,fill=black,inner sep=1pt,minimum size=5pt] (l) at (11,\ly) {} edge [in=45,out=135, thick, loop above] node {} ();
\graph [multi, edges = {thick}] {(b) -> (a); (d) -> (b); (f) -> (d), (h) -> (f);
        (j) -> (c); };

\graph [multi, edges = {thick}] {(n) -> [bend right] (m);};

\graph [multi, edges = {thick}] {(m) -> [bend right] (n);};

\graph [multi, edges = {thick}] {(i) -> (j); (g) -> (h);};
\graph [multi, edges = {thick,red}] {(c) -> (i); };
%\graph [multi, edges = {thick,red}] {(a) -> (g)};
%
\end{tikzpicture}\\
\hdashline\\[-3mm]
\begin{tikzpicture}[scale=0.7]
\node[right] at (-8,0) {$\laguerre{\{2,4,6,8,10,12,14,13,11,9,7,5,3,1\}}$};
\emptysigmatwo{0}
\node[circle,fill=black,inner sep=1pt,minimum size=5pt] (e) at (9,\ey) {} edge [in=45,out=135, thick, loop above] node {} ();
\node[circle,fill=black,inner sep=1pt,minimum size=5pt] (k) at (10,\ky) {} edge [in=45,out=135, thick, loop above] node {} ();
\node[circle,fill=black,inner sep=1pt,minimum size=5pt] (l) at (11,\ly) {} edge [in=45,out=135, thick, loop above] node {} ();
\graph [multi, edges = {thick}] {(b) -> (a); (d) -> (b); (f) -> (d), (h) -> (f);
        (j) -> (c); };

\graph [multi, edges = {thick}] {(n) -> [bend right] (m);};

\graph [multi, edges = {thick}] {(m) -> [bend right] (n);};

\graph [multi, edges = {thick}] {(i) -> (j); (g) -> (h); (c) -> (i);};
\graph [multi, edges = {thick,red}] {(a) -> (g)};
\end{tikzpicture}
\end{tabular}
\caption{Stage (b) of the DS history for the D-permutation\\
$\sigma = 7\, 1\, 9\, 2\, 5\, 4\, 8\, 6\, 10\, 3\, 11\, 12\, 14\, 13\,
	 =  (1,7,8,6,4,2)\,(3,9,10)\,(5)\,(11)\,(12)\,(13,14) \in \dperm_{14}.$}
\label{fig.running.example.2.DS.history.b}
\end{figure}

\subsubsection{Computation of weights}
\label{subsec.computation.dperm}

The steps in this subsection will be similar to that in
Section~\ref{subsec.computation}.
We first assign weights to each index
$i\in[2n]$ of a given D-permutation $\sigma\in \dperm_{2n}$,
and then the weight of a permutation is defined
to be the weights of its indices.
We then use the bijection to transfer the weights to 0-Schr\"oder paths
and then factorise the total weight for each of the steps
which then finally leads us to Theorem~\ref{thm.DS.master}.

The index weights are assigned as follows:
\begin{itemize}
        \item if $i$ is a cycle valley minimum, we set
                $\wt(i) = \lambda\,\sfa_{\ucross(i,\sigma) + \unest(i,\sigma)}$.
                For all other cycle valleys, we set
                $\wt(i) = \sfa_{\ucross(i,\sigma) + \unest(i,\sigma)}$,

        \item if $i$ is a cycle peak, we set
                $\wt(i) = \sfb_{\lcross(i,\sigma),\,\lnest(i,\sigma)}$,

        \item if $i$ is a cycle double fall, we set
                $\wt(i) = \sfc_{\lcross(i,\sigma),\,\lnest(i,\sigma)}$,

        \item if $i$ is a cycle double rise, we set
                $\wt(i) = \sfd_{\ucross(i,\sigma),\,\unest(i,\sigma)}$,

        \item and finally, if $i$ is a fixed point, we set
                \[\wt(i)
		\;=\;
		\begin{cases}
		\lambda\, \sfe_{\psnest(i,\sigma)} \;\;\;
		\text{if $i$ is even}\\
		\lambda\, \sff_{\psnest(i,\sigma)} \;\;\;
		\text{if $i$ is odd}
		\end{cases}\]
\end{itemize}
We set the weight of the D-permutation $\sigma$
to be $\wt(\sigma)  = \prod_{i=1}^{2n} \wt(i)$.
It is clear that our polynomial $\widehat{Q}_n$ defined in \eqref{def.dperm.master}
is simply
$\widehat{Q}_n(\bsfa, \bsfb, \bsfc, \bsfd, \bsfe, \bsff, \lambda)
= \sum_{\sigma\in\dperm_{2n}} \wt(\sigma)$.

We now use the bijection $\sigma\mapsto (\omegahat, \xi)$
to transfer the weights $\wt(i)$ onto the pair $(\omegahat, \xi)$.
As our polynomial $\widehat{Q}$
is almost the same as the polynomial introduced in
\cite[eq.~(3.30)]{Deb-Sokal}
except for the extra factor $\lambda^{\cyc(\sigma)}$
and the index of $\sfa$,
the dependence on cycle peaks, cycle double rises, cycle double falls,
are same as in \cite[eq.~(3.30)]{Deb-Sokal}
but the treatment of cycle valleys and also fixed points is different.
As we use the same bijection used to obtain
the continued fraction (\cite[eq.~(3.31)/(3.32)]{Deb-Sokal})
the computation of weights corresponding to
the variables $\sfb,\sfc,\sfd$ are going to be exactly the same.
The weights corresponding to the fixed points are almost same
except $\sfe_k$ and $\sff_k$ are now replaced with
$\lambda \sfe_k$ and $\lambda \sff_k$, respectively.
Hence, we skip the details for all cycle types and the corresponding steps
in the 0-Schr\"oder path $\omegahat$
except for cycle valleys.

Let us now compute the weights contributed by cycle valleys.
%The only thing that remains is to compute the weights for the variables $\sfa$.
These correspond to steps $s_i$ in $\omegahat$ where $s_i$ is a rise
starting at height $h_{i-1} = 2k$.
Then from Equations~\eqref{def.xi.bis.dperm}/\eqref{eq.lemma.xii.crossing.dperm.a}
we get $\lceil (h_{i-1}-1)/2\rceil = \lceil h_i/2\rceil  = \ucross(i,\sigma)+\unest(i,\sigma)$ ($=k$).
Also among the possible choices of labels $\xi\in [0,k]$
there is exactly one which closes a cycle and the others don't
(Lemma~\ref{lem.cycle.closer.dperm}).
Therefore, we obtain
\be
a_{2k} \;\eqdef\; \sum_{\xi} a_{2k,\xi} \;=\;  (\lambda+k) \sfa_{2k}.
\ee

This completes the proof of Theorem~\ref{thm.DS.master}. \qed

%{\bf THIS IS OLD VERSION!!!!!! }
%
%We can now compute the weights associated to the 0-Schr\"oder path $\omegahat$ in
%Step~1.
%As our polynomial $\widehat{Q}$ defined in \eqref{def.dperm.master}
%is almost the same as the polynomial introduced in 
%\cite[eq.~(3.30)]{Deb-Sokal}
%except for the extra factor $\lambda^{\cyc(\sigma)}$
%and the index of $\sfa$,
%the dependence on cycle peaks, cycle double rises, cycle double falls, 
%and even and odd fixed points are same as in \cite[eq.~(3.30)]{Deb-Sokal}
%but the treatment of cycle valleys is different.
%As we use the same bijection used to obtain 
%the continued fraction (\cite[eq.~(3.31)/(3.32)]{Deb-Sokal}) 
%the computation of weights corresponding to 
%the variables $\sfb,\sfc,\sfd,\sfe,\sff$ are going to be exactly the same.
%
%
%The only thing that remains is to compute the weights for the variables $\sfa$.
%These correspond to steps $s_i$ in $\omegahat$ where $s_i$ is a rise 
%starting at height $h_{i-1} = 2k$ (so that $i$ is a cycle valley).
%Then from Equations~\eqref{def.xi.bis.dperm}/\eqref{eq.lemma.xii.crossing.dperm.a}
%we get $\lceil (h_{i-1}-1)/2\rceil = \lceil h_i/2\rceil  = \ucross(i,\sigma)+\unest(i,\sigma)$ ($=k$).
%Also among the possible choices of labels $\xi\in [0,k]$ 
%there is exactly one which closes a cycle and the others don't
%(Lemma~\ref{lem.cycle.closer.dperm}).
%Therefore, we obtain 
%\be
%a_{2k} \;\eqdef\; \sum_{\xi} a_{2k,\xi} \;=\;  (\lambda+k) \sfa_{2k}.
%\ee
%
%This completes the proof of Theorem~\ref{thm.DS.master}. \qed

\proofof{Theorem~\ref{thm.DS.pqgen}} Specialise Theorem~\ref{thm.DS.master} to
\begin{eqnarray}
   \sfa_{k-1}
   & = &
	p_{+1}^{k-1}  \,\times\, y_1
        \label{eq.proof.weights.sfa}  \\[2mm]
   \sfb_{k-1-\xi,\xi}
   & = &
   p_{-1}^{k-1-\xi} q_{-1}^\xi  \,\times\,
   \begin{cases}
      x_1  & \textrm{if $\xi = 0$}   \\
      u_1  & \textrm{if $1 \le \xi \le k-1$}
   \end{cases}
        \\[2mm]
   \sfc_{k-1-\xi,\xi}
   & = &
   p_{-2}^{k-1-\xi} q_{-2}^\xi  \,\times\,
   \begin{cases}
      x_2  & \textrm{if $\xi = 0$}   \\
      u_2  & \textrm{if $1 \le \xi \le k-1$}
   \end{cases}
        \\[2mm]
   \sfd_{k-1-\xi,\xi}
   & = &
   p_{+2}^{k-1-\xi} q_{+2}^\xi  \,\times\,
   \begin{cases}
      y_2  & \textrm{if $\xi = 0$}   \\
      v_2  & \textrm{if $1 \le \xi \le k-1$}
   \end{cases}
        \\[2mm]
   \sfe_k
   & = &
   \begin{cases}
      \ze  & \textrm{if $k = 0$}   \\[1mm]
      \se^k \we  & \textrm{if $k \ge 1$}
   \end{cases}
        \\[2mm]
   \sff_k
   & = &
   \begin{cases}
      \zo  & \textrm{if $k = 0$}   \\[1mm]
      \so^k \wo  & \textrm{if $k \ge 1$}
   \end{cases}
        \label{eq.proof.weights.sff}
\end{eqnarray}
\qed

\begin{rem}
%\noindent {\bf Remark.} 
Notice that the specialisation in the above proof is almost the same
as \cite[eq.~(6.40-45)]{Deb-Sokal}, except for the treatment of $\sfa$. \myendremark
\end{rem}

\proofof{Theorem~\ref{thm.conj.DS}} Specialise Theorem~\ref{thm.DS.pqgen} to 
\be
p_{+1}=p_{-1}=p_{+2}=p_{-2}=q_{-1}=q_{+2}=q_{-2}=\se = \so=1.\label{eq.thm.conj.DS.specialise}\ee\qed

%%{\color{red} 
%%
%%\section{Generalisation of Theorem~\ref{thm.conj.SZ}}
%%
%%Sokal and Zeng provide $p,q$-generalisation for their first J-fraction
%%for permutations~(\cite[Theorem~2.2]{Sokal-Zeng_masterpoly})
%%in \cite[Theorem~2.7]{Sokal-Zeng_masterpoly}. These involve a family of {\bf HOW MANY VARIABLES??????}
%%Similarly, they also provide $p,q$-generalisation for their second J-fraction
%%for permutations~(\cite[Theorem~2.4]{Sokal-Zeng_masterpoly})
%%in \cite[Theorem~2.12]{Sokal-Zeng_masterpoly}.
%%In subsection~\ref{subsec.pq.conj.SZ}, 
%%we provide a similar $p,q$-generalisation for Theorem~\ref{thm.conj.SZ}
%%which generalises \cite[Theorem~2.12]{Sokal-Zeng_masterpoly}.
%%
%%Sokal and Zeng also provide master continued fractions
%%for their first and second J-fractions in
%%\cite[Theorem~2.9 and~2.14]{Sokal-Zeng_masterpoly}, respectively.
%%In subsection~\ref{subsec.master.conj.SZ},
%%we provide a master J-fraction generalising Theorem~\ref{thm.conj.SZ}
%%and Theorem~\ref{thm.conj.SZ.pq}.
%%
%%
%%
%%{\bf COUNT THE NUMBER OF VARIABLES AND MENTION!!!!!!!}
%%
%%{\bf MOVE THESE PARAGRAPHS TO THE RIGHT PLACE!!!!}
%%
%%
%%}

\subsection[Continued fractions using variant record classification:\\ Proof of Theorems~\ref{thm.dperm.prime},~\ref{thm.DS.pqgen.prime},~\ref{thm.DS.master.prime}]{Continued fractions using variant record classification:\\ Proof of Theorems~\ref{thm.dperm.prime},~\ref{thm.DS.pqgen.prime},~\ref{thm.DS.master.prime}}
\label{subsec.dperm.proofs.prime}

We first recall the description of the alternate labels and the inverse bijection
from \cite[Section~6.5]{Deb-Sokal}
in Subsection~\ref{subsec.DSrecall.dperm.prime}.
We then reinterpret the inverse bijection
in Subsection~\ref{subsec.laguerre.dperm.prime}.

\subsubsection{Description of labels and inverse bijection}
\label{subsec.DSrecall.dperm.prime}

{\bf Step 2: Definition of the labels $\bm{\widehat{\xi}_i}$.}

We recall the alternate definition of labels 
introduced in \cite[eq.~(6.46)]{Deb-Sokal}
%\begin{subeqnarray}
%   \widehat{\xi}_i
%   & = &
%   \begin{cases}
%       \#\{j \colon\:  j < i \le \sinv(i) < \sinv(j) \}
%           & \textrm{if $\sinv(i)$ is even}
%              \\[2mm]
%       \#\{j \colon\:  \sinv(j) < \sinv(i) \le i < j \}
%           & \textrm{if $\sinv(i)$ is odd}
%   \end{cases}
% \label{def.xi.hat}
%\end{subeqnarray}
%
\begin{subeqnarray}
   \widehat{\xi}_i
   & = &
   \begin{cases}
       \#\{2l > 2k \colon\: \sigma(2l) < \sigma(2k) \}
           & \textrm{if $i = \sigma(2k)$}
              \\[2mm]
     \#\{2l-1< 2k-1 \colon\: \sigma(2l-1) > \sigma(2k-1) \}
           & \textrm{if $i= \sigma(2k-1) $}
   \end{cases}
   \qquad
 \slabel{def.xi.hat.a}
       \\[3mm]
   & = &
   \begin{cases}
       \#\{j \colon\:  j < i \le \sinv(i) < \sinv(j) \}
           & \textrm{if $\sinv(i)$ is even}
              \\[2mm]
       \#\{j \colon\:  \sinv(j) < \sinv(i) \le i < j \}
           & \textrm{if $\sinv(i)$ is odd}
   \end{cases}
 \slabel{def.xi.hat.b}
 \label{def.xi.hat}
\end{subeqnarray}
These labels have a simple interpretation,
but now we use the variant nesting statistics~(\ref{def.ucrossnestjk.prime}b,d).
\begin{lem}[{\cite[Lemma~6.7]{Deb-Sokal}}]
   \label{lemma.xii.nesting.alternative}
We have
\be
   \widehat{\xi}_i
   \;=\;
   \begin{cases}
        \lnest'(i,\sigma)
           & \hbox{\rm if $\sinv(i) $ is even (i.e., $s_i$ is a rise)
                       and $\neq i$}
              \\[1mm]
        \unest'(i,\sigma)
           & \hbox{\rm if $\sinv(i) $ is odd (i.e., $s_i$ is a fall)
                       and $\neq i$}
              \\[1mm]
        \psnest(i,\sigma)
           & \hbox{\rm if $\sinv(i) = i$ (i.e., $i$ is a fixed point)}
   \end{cases}
 \label{eq.xii.nestings.version1}
\ee
\end{lem}

The labels $\widehat{\xi}_i$ satisfy the following inequalities:
\begin{lem}[{\cite[Lemma~6.8]{Deb-Sokal}}]
   \label{lemma.ineqs.labels.xihat}
We have
\begin{subeqnarray}
   & 0 \;\le\; \widehat{\xi}_i  \;\le\; \left\lceil \dfrac{h_i-1}{2} \right\rceil \;=\;
                      \left\lceil\dfrac{h_{i-1}}{2} \right\rceil
       & \textrm{if $\sinv(i) $ is even (i.e., $s_i$ is a rise)}
    \qquad
 \slabel{eq.xihat.ineqs.a}
      \\[3mm]
   & 0 \;\le\; \widehat{\xi}_i  \;\le\; \left\lceil \dfrac{h_i}{2} \right\rceil  \;=\;
                      \left\lceil\dfrac{h_{i-1}-1}{2} \right\rceil
       & \textrm{if $\sinv(i) $ is odd (i.e., $s_i$ is a fall)}
 \slabel{eq.xihat.ineqs.b}
 \label{eq.xihat.ineqs}
\end{subeqnarray}
\end{lem}

The inequalities \eqref{eq.xihat.ineqs}
are proved by providing an interpretation of the 
variant crossing statistics defined in (\ref{def.ucrossnestjk.prime}a,c)
in terms of these labels.
\begin{lem}[{\cite[Lemma~6.9]{Deb-Sokal}}]
   \label{lemma.xihat.crossing}
\nopagebreak
\quad\hfill
\vspace*{-1mm}
\begin{itemize}
   \item[(a)]  If $s_i$ a rise (i.e.~$\sinv(i)$ is even), then
\begin{subeqnarray}
   \left\lceil \dfrac{h_i - 1}{2} \right\rceil - \widehat{\xi}_i
   & = &
   \lcross'(i,\sigma)
       \,+\, {\rm I}[\hbox{\rm $i$ is even and } \sigma(i) \neq i]
     \\[-1mm]
   & = &
   \lcross'(i,\sigma)
       \,+\, {\rm I}[\hbox{\rm $i$ is a cycle double fall}]
   \;.
 \label{eq.lemma.xihat.crossing.1}
\end{subeqnarray}
   \item[(b)]  If $s_i$ a fall (i.e.~$\sinv(i)$ is odd), then
\begin{subeqnarray}
   \left\lceil \dfrac{h_i}{2} \right\rceil - \widehat{\xi}_i
   & = &
   \ucross'(i,\sigma)
       \,+\, {\rm I}[\hbox{\rm $i$ is odd and } \sigma(i) \neq i]
     \\[-1mm]
   & = &
   \ucross'(i,\sigma)
       \,+\, {\rm I}[\hbox{\rm $i$ is a cycle double rise}]
   \;.
 \label{eq.lemma.xihat.crossing.2}
\end{subeqnarray}
\end{itemize}
\end{lem}

Since the quantities 
\eqref{eq.lemma.xihat.crossing.1}, \eqref{eq.lemma.xihat.crossing.2}
are manifestly nonnegative, 
the inequalities \eqref{eq.xihat.ineqs} are
satisfied.

{\bf Step 3: Proof of bijection.}
The proof is similar to that presented in
Step 3 Section~\ref{subsec.DSrecall.dperm},
but using a value-based rather than position-based notion
of inversion table.
Recall that if $S = \{s_1 < s_2 < \ldots < s_k \}$
is a totally ordered set of cardinality $k$,
and $\bm{x} = (x_1,\ldots,x_k)$ is a permutation of $S$,
then the (left-to-right) (position-based) inversion table
corresponding to $\bm{x}$
is the sequence $\bm{p} = (p_1,\ldots,p_k)$ of nonnegative integers
defined by $p_\alpha = \#\{\beta < \alpha \colon\: x_\beta > x_\alpha \}$.
We now define the (left-to-right) {\em value-based}\/ inversion table
$\bm{p}'$ by $p'_{x_i} = p_i$;
note that $\bm{p}'$ is a map from $S$ to $\{0,\ldots,k-1\}$,
such that $p'_{x_i}$ is the number of entries to the left of $x_i$
(in the word $\bm{x}$) that are larger than $x_i$.
In particular, $0 \le p'_{s_i} \le k-i$.
Given the value-based inversion table $\bm{p}'$,
we can reconstruct the sequence $\bm{x}$
by working from largest to smallest value,
as follows \cite[section~5.1.1]{Knuth_98}:
We start from an empty sequence, and insert $s_k$.
Then we insert $s_{k-1}$ so that the resulting word has $p'_{s_{k-1}}$
entries to its left.
Next we insert $s_{k-2}$ so that the resulting word has $p'_{s_{k-2}}$
entries to its left, and so on.
[The right-to-left value-based inversion table $\bm{q}'$
 is defined analogously, and the reconstruction proceeds from
 smallest to largest.]

We now recall the definitions
\begin{subeqnarray}
F & = & \{2,4,\ldots,2n\} \; = \;  \hbox{even positions} \\
F' & = & \{i \colon \sinv(i) \text{ is even} \} \; = \; \{\sigma(2),\sigma(4),\ldots, \sigma(2n)\}\\
G & = & \{1,3,\ldots, 2n-1\} \; = \; \hbox{odd positions}  \\
G' & = & \{i \colon \sinv(i) \text{ is odd} \} \; = \; \{\sigma(1),\sigma(3),\ldots, \sigma(2n-1)\}
\label{eq.steps2classification.dperm.prime}
\end{subeqnarray}
Note that $F'$ (resp. $G'$) are the positions of the rises (resp. falls)
in the almost-Dyck path $\omega$.

We can now describe the map
$(\omega,\widehat{\xi})\mapsto \sigma$.
Given the almost-Dyck path $\omega$,
we can immediately reconstruct the sets $F, F',G, G'$.
We now use the labels $\widehat{\xi}$ to reconstruct the maps
$\sigma \restrict F \colon\: F \to F'$ and
$\sigma \restrict G \colon\: G \to G'$ as follows:
The even subword $\sigma(2)\sigma(4)\cdots\sigma(2n)$ is a listing of $F'$
whose right-to-left value-based inversion table
is given by $q'_i = \widehat{\xi}_i$ for all $i \in F'$;
this is the content of \eqref{def.xi.hat.a}.
Similarly, the odd subword $\sigma(1)\sigma(3)\cdots\sigma(2n-1)$
is a listing of $G'$
whose left-to-right value-based inversion table
is given by $p'_i = \widehat{\xi}_i$ for all $i \in G'$;
this again is the content of \eqref{def.xi.hat.a}.

Additionally, we state a lemma which gives an alternate way of
constructing a value-based inversion table for D-permutation.
\begin{sloppy}
\begin{lem}
Let $\sigma\in \dperm_{2n}$ and let $\left(\omega,\widehat{\xi}\right)$
be its associated almost-Dyck path with an assignment of labels determined by
\eqref{def.xi.hat}.
Also let $G' = \{x_1<\ldots< x_n\}$ and let $F' =\{y_1<\ldots< y_n\}$.
Then
\begin{itemize}
	\item[(a)] For any index $j$ ($1\leq j\leq n$), $\sinv(x_j)$ is the
                $(\widehat{\xi}_{x_j}+1)$th smallest element
		of \mbox{$G \setminus \{\sinv(x_1), \ldots, \sinv(x_{j-1})\}$}.
        \item[(b)] For any index $j$ ($1\leq j\leq n$), 
		$\sinv(y_j)$ is the
                $(\widehat{\xi}_{y_j}+1)$th largest element
		of \mbox{$F \setminus \{\sinv(y_n),\ldots, \sinv(y_{j+1})\}$}.
\end{itemize}
\label{lem.valuebased.alternate}
\end{lem}
This lemma can be proved using \eqref{def.xi.hat.a}
and we omit the details.
\end{sloppy}

\begin{rem}
\label{rem.DS.FZ.prime}
For this bijection, a comment similar to Remark~\ref{rem.DS.FZ} can be made.
As mentioned earlier in Remark~\ref{rem.original.FZ}, 
Foata and Zeilberger used value-based inversion tables for their presentation of the bijection.
We think that it is possible to write a variant of the Foata--Zeilberger bijection
in which fixed points are treated separately and for which the definition of the Motzkin path is 
identical to the one in Step~1 of Section~\ref{subsec.FZrecall.perm} 
but the description of the labels will be different.
We expect that the inversion tables used in the reverse bijection to be value-based.
We refrain from writing out the details.
With this, the bijection in this section can be seen as almost a special case of this 
variant Foata--Zeilberger bijection except for the treatment of fixed points.
\myendremark
\end{rem}

\subsubsection{Running example 2}
\label{subsubsec.DSbij.eg.2.prime}

Now let us take our second running example
\begin{eqnarray}
	\sigma &=& 7\, 1\, 9\, 2\, 5\, 4\, 8\, 6\, 10\, 3\, 11\, 12\, 14\, 13\, \nonumber\\
	  &=& (1,7,8,6,4,2)\,(3,9,10)\,(5)\,(11)\,(12)\,(13,14) \in \dperm_{14}
\end{eqnarray}
which was depicted in Figure~\ref{fig.pictorial.2}.
The almost-Dyck path is the same as the one in Figure~\ref{fig.almostdyck.running.example.2}.
The sets $F',G'$ were also already recorded in~\eqref{eq.DSbij.eg.2.FGprime}.

The labels $\xi_i$ in \eqref{eq.labels.DS.eg.2.prime} are now assigned to $\sigma(i)$: 
\begin{subeqnarray}
\begin{array}{r}
2i\in F  \\[3mm]
	F' = \{\sigma(2i)\; | \: i\in [7]\}   \\[3mm]
	\text{Right-to-left inversion table:} \; \widehat{\xi}_{\sigma(2i)}   \\
\end{array}
&&\hspace*{-6mm}
\begin{pmatrix}
	\; 2 &  4 & 6 & 8 & 10  & 12 & 14\; \\[3mm]
	\; 1 & 2 & 4 & 6 & 3 & 12 & 13 \; \\[3mm]
	\; 0 & 0 & 1 & 1 & 0 & 0 & 0  \; 
\end{pmatrix} 
	\\[5mm]
\begin{array}{r}
2i-1\in G  \\[3mm]
	G' = \{\sigma(2i-1)\; | \: i\in [7]\}   \\[3mm]
	\text{Left-to-right inversion table:} \; \widehat{\xi}_{\sigma(2i-1)}   \\
\end{array}
	&&\hspace*{-6mm}
\begin{pmatrix}
	\; 1 & 3 & 5 & 7 &  9 & 11 & 13 \; \\[3mm]
	\; 7 & 9 & 5 & 8 & 10 & 11 & 14 \; \\[3mm]
	\; 0 & 0 & 2 & 1 & 0 & 0 & 0
\end{pmatrix}
\label{eq.labels.DS.eg.2.prime}
\end{subeqnarray}
%
%
%
%%{\bf DO THE FIRST INDICES LOOK UGLY???????????  WHAT DO I DO ABOUT EQUATION NUMBERING VS FIGURE ENVIRONMENT!!!!!!!!!}
%
%
%{\bf THEN DEMONSTRATE INVERSE BIJECTION???????}
%
\bigskip

\subsubsection{Combinatorial interpretation using Laguerre digraphs}
\label{subsec.laguerre.dperm.prime}

The construction here will again be similar to those 
in Sections~\ref{subsec.laguerre} and~\ref{subsec.laguerre.dperm}.
We will only include the necessary details 
and state the necessary lemmas,
and will omit most of the proofs.

We begin with an almost-Dyck path $\omega$
and an assignment of labels $\widehat{\xi}$ satisfying
\eqref{eq.xihat.ineqs}.
The inverse bijection in Section~\ref{subsec.DSrecall.dperm.prime}
Step~3, gives us a D-permutation $\sigma$.
We will again break this process into several intermediate steps
and provide a reinterpretation using Laguerre digraphs.

This time however, we will use a different convention 
for denoting Laguerre digraphs
than the one used in 
Sections~\ref{subsec.laguerre} and~\ref{subsec.laguerre.dperm}.
Let $\sigma\in\Sym_n$ be a permutation on $[n]$.
For $S\subseteq [n]$,
we let $\laguerrep{S}$ denote the subgraph of $L^{\sigma}$,
containing the same set of vertices $[n]$, 
but only containing the edges $\sinv(i)\to i$, whenever $i\in S$
(we are allowed to have $\sinv(i)\not\in S$). 
Thus, $\laguerrep{[n]} = L^{\sigma}$, and
$\laguerrep{\emptyset} = \left. L^{\sigma}\right|_{\emptyset} $
is the digraph containing $n$ vertices and no edges. 
Whenever the permutation $\sigma$ is understood,
we shall drop the superscript and denote it as $\laguerrep{S}$.

Now, let $\sigma\in\dperm_{2n}$ be a D-permutation.
Similar to the construction in Section~\ref{subsec.DSrecall.dperm},
recall that the inverse bijection 
begins by obtaining the sets $F,G, F', G'$ 
(the sets $F,G$ are fixed for any given $n$
and the sets $F', G'$ are obtained from the almost-Dyck path $\omega$).
We then construct $\sigma\restrict F\colon \: F \to F'$
and $\sigma\restrict G\colon \: G\to G'$ separately
but by using the labels $\widehat{\xi}\restrict F'$ 
and $\widehat{\xi}\restrict G'$ respectively.

In this interpretation, 
we start with the digraph $\laguerrep{\emptyset}$ and
then go through the set $[2n]$.
This time, we first go through the elements of $G'$ in increasing order (stage~(a))  
and then through the elements of $F'$ in decreasing order (stage~(b)).
We call this the {\em variant DS order} on the set $[2n]$.

Here however, the history that we consider is different
from the one in Section~\ref{subsec.laguerre.dperm}.
Let $u_1,\ldots, u_{2n}$ be a rewriting of $[2n]$ as per the variant DS order.
We now consider the ``variant DS history''
$\laguerrep{\emptyset}\subset \laguerrep{\{u_1\}} \subset \laguerrep{\{u_1, u_2\}}
\subset \ldots \subset \laguerrep{\{u_1,\ldots, u_{2n}\}} = L$.
Thus, at step $u$ as per the variant DS order,
we use the (value-based) inversion tables and Lemma~\ref{lem.valuebased.alternate}
to construct the edge $\sinv(u)\to u$.
Similar to the previous histories, 
at each step we insert a new edge into the digraph,
and at the end of this process,
the resulting digraph obtained is the permutation $\sigma$ in cycle notation. 

Let us now look at the intermediate Laguerre digraphs obtained
during stages~(a) and~(b).

\medskip

{\bf Stage (a): Going through $\boldsymbol{G'}$:}\\
From \eqref{eq.steps2classification.dperm.prime},
we know that 
$G = \Cdrise(\sigma) \, \cup \, \Oddfix(\sigma) \, \cup \, \Cval(\sigma)$
and
$G' = \Cdrise(\sigma) \, \cup \, \Oddfix(\sigma) \, \cup \, \Cpeak(\sigma)$
where $\sigma$ is the resulting D-permutation obtained at the end
of the inverse bijection.

The connected components
at the end of this stage can be described as follows:
\begin{lem} The Laguerre digraph $\laguerrep{G'}$
consists of the following connected components:
\begin{itemize} 
        \item loops on vertices $u\in \Oddfix$,

        %\item no directed cycles

        \item directed paths with at least two vertices,
                in which the initial vertex of the path
                is a cycle valley in $\sigma$
                (i.e. contained in the set $F'\,\cap\, G$)
                and the final vertex is a cycle peak in $\sigma$
                (i.e. contained in the set $F\,\cap\, G'$),
		and the intermediate vertices (if any)
		are cycle double rises 
		(which belong to the set $G\,\cap\, G'$).

                %This is true because all edges
                %corresponding to excedances
                %have been inserted
                %and no edge corresponding
                %anti-excedances have been inserted
                %at this stage.

        \item  isolated vertices at $u\in F\,\cap\, F' = \Cdfall(\sigma)\, \cup \, \Evenfix(\sigma)$.
\end{itemize}
Furthermore, it contains no directed cycles.
\label{lem.afterGp.dperm.prime}
\end{lem}

{\bf Stage (b): Going through $\boldsymbol{F'}$:}\\
We now go through the elements of $F'$ in decreasing order.
From \eqref{eq.steps2classification.dperm.prime},
we know that
$F = \Cdfall(\sigma) \, \cup \, \Evenfix(\sigma) \, \cup \, \Cpeak(\sigma)$
and
$F' = \Cdfall(\sigma) \, \cup \, \Evenfix(\sigma) \, \cup \, \Cval(\sigma)$
where $\sigma$ is the resulting D-permutation obtained at the end
of the inverse bijection.
We let $F' = \{y_1<\ldots<y_n\}$.

This time, we describe the
the {\em initial vertices} of paths with at least two vertices 
(and not the final vertices):

\begin{sloppy}
\begin{lem} Let $u$ be the initial vertex of a path
with at least two vertices in $\laguerrep{G' \,\cup\, \{y_n,\ldots,y_{j+1}\}}$
for some index $j$ ($1\leq j \leq n$). Then $u\in \Cval$.
\label{lem.initialvertexFp.dperm.prime}
\end{lem}
\end{sloppy}

%\begin{proof} 
%       Let $u$ be the initial vertex of a path in 
%       $\laguerre{G' \,\cup\, \{f_n,\ldots,f_{j+1}\}}$.
%       Since all vertices in $G'$ were already assigned in-neighbours
%       during stage (a), it must be that $u\in F'$.
%       Thus $u$ is either a cycle valley, a cycle double fall or 
%       an even fixed point in $\sigma$.
%       However, as even fixed point can only be an isolated vertex
%       or a loop, it cannot be the initial vertex of a path with $2$ 
%       vertices.
%
%       Let us assume that $u$ is a cycle double fall in $\sigma$, i.e.,
%       $\sinv(u) >u> \sigma(u)$.
%       We know that $u$ must have an out-neighbour $v$
%       as its path has at-least two vertices.
%       Thus, $v = \sigma(u)$ and $u>v$, which implies $v\in F'$
%       (from definition of $F$ in~(\ref{eq.steps2excedanceclassification.dperm})).
%       However, this is a situation where we have two vertices $u,v\in F'$
%       with $u>v$ such that the smaller vertex has an in-neighbour
%       even though the larger vertex does not.
%       This clearly cannot happen as we assign in-neighbours
%       to vertices in $F'$ in descending order.
%       This is a contradiction and thus $u\in \Cval$.  
%\end{proof}

Our definition of cycle closers here is different.
Let $u_1,\ldots, u_{2n}$ be the vertices $[2n]$ arranged according to the 
variant DS order.
We say that $u_j\in [2n]$ is a {\em cycle closer}
if the edge $\sinv(u_j)\to u_j$
is introduced in $\laguerrep{\{u_1,\ldots, u_{j-1}\}}$
as an edge between the two ends of a path
turning the path to a cycle.
The following lemma classifies all cycle closers.

\begin{lem}(Classifying cycle closers) Given a D-permutation $\sigma$,
        an element $u\in [2n]$ is a cycle closer
        if and only if it is a cycle valley minimum,
	i.e.,
	it is the smallest element in its cycle.
\label{lem.classifyingCycleG.dperm.variant}
\end{lem}

%\begin{proof} Let $u_j$ be a cycle closer.
%Notice that $u_j$ must have been the initial vertex of an unfinished path
%in $\laguerrep{\{u_1,\ldots, u_{j-1}\}}$. 
%Also, from Lemma~\ref{lem.afterGp.dperm.prime}, $u_j\not\in G'$.
%Thus, $u_j\in F'$ and from Lemma~\ref{lem.initialvertexFp.dperm.prime},
%$u_j\in F'\cap G = \Cval$.
%
%The in-edge from any other cycle valley $v$ in the cycle of $u_j$
%must have been present in $\laguerrep{\{u_1,\ldots, u_{j-1}\}}$
%before $\sinv(u_j)\to u_j$ was inserted to obtain 
%       $\laguerrep{\{u_1,\ldots, u_{j}\}}$.
%Hence, it must be that $v>u_j$ as $\Cval \subseteq F'$ 
%and we go through $F'$ in descending order.
%Thus, the cycle closer $u_j$ is the smallest cycle valley 
%in its cycle in $\sigma$.
%\end{proof}

Next, we will count the number of cycle closers.
But before doing that, we require a technical lemma similar to 
Lemmas~\ref{lem.cycle.closer.technical},~\ref{lem.cycle.closer.technical.dperm}.
(As before, recall that if
$i \in G\,\cap\,F' = \Cval(\sigma)$,
step $s_{i}$ must be a rise from height $h_{i-1}$ to height $h_{i}$
and hence, $h_{i-1}+1 = h_{i}$.
Also, from the interpretation of the heights in Lemma~\ref{lemma.heights.dperm},
we must have $\lceil h_{i-1}/2\rceil +1 = \lceil (h_{i}+1)/2\rceil = f_i$.)

\begin{lem}
Given a D-permutation $\sigma$ and associated sets $F,G,F',G'$
where $F' = \{y_1<\ldots< y_n\}$, 
and an index $j$ ($1\leq j\leq n$) 
such that $y_j\in G\,\cap\, F' = \Cval$.
Then the following is true:
\be
\#\{u \in  F\backslash \{\sinv(y_n),\ldots, \sinv(y_{j+1})\} \colon \: u>y_j\}
        = \lceil h_{y_j-1}/2\rceil +1 = \lceil (h_{y_j}+1)/2\rceil
	= f_{y_j}
\label{eq.lem.cycle.closer.technical.dperm.prime}
\ee
where $h_i$ denotes the height at position $i$ of the almost-Dyck path $\omega$
associated to $\sigma$
in Step~1
and $f_i$ is defined in \eqref{eq.def.fk.dperm}.
\label{lem.cycle.closer.technical.dperm.prime}
\end{lem}

%The proof is similar to that of 
%Lemmas~\ref{lem.cycle.closer.technical},~\ref{lem.cycle.closer.technical.dperm}
%and we omit it.

\begin{proof}
Notice the equality of the following sets:
\be
\{u \in  F\backslash \{\sinv(y_n),\ldots, \sinv(y_{j+1})\} \colon \: u>y_j\}
\;=\;
	\{u>y_j\:\colon \sigma(u)\leq y_j\}.
\label{eq.cycle.closer.technical.dperm.prime}
\ee
Next use \eqref{eq.def.fk.dperm}
to notice that the set on 
the right hand side of \eqref{eq.cycle.closer.technical.dperm.prime}
has cardinality $f_{y_j}$.
\end{proof}

%{\bf THE PROOF BELOW IS THE UNEDITED COPIED PROOF. NEED TO PUT THE ACTUAL PROOF}
%
%\begin{proof} Here we establish the following equality of sets:
%\be
%\{i > u \colon \: \sinv(i) \leq  u\} \; = \;  \{i \in   G'\backslash \{\sigma(2n-1),\ldots, \sigma(u+2)\} \colon\: i > u \}.
%\label{eq.technical.equation.a.dperm}
%\ee
%
%Whenever $i\in G'$, we have that $\sinv(i)\in G$
%(by description of $G$, $G'$ in (\ref{eq.steps2excedanceclassification.dperm})).
%Additionally, if $i\not\in \{\sigma(2n-1), \ldots, \sigma(u+2)\}$
%then it must be that $\sinv(i)\leq u$.
%Thus, if we also have $i> u$, then
%$i\in \{i >u \colon \: \sinv(i) \leq  u\}$.
%
%On the other hand, if $i>u$ and $\sinv(i)\leq u$,
%then $i> \sinv(i)$ and as $\sigma$ is a D-permutation, 
%$\sinv(i)$ must be odd. Therefore, $i\in G'$.
%As $\sinv(i)\leq u< u+2<\ldots< 2n-1$,
%$i$ cannot be one of $\sigma(u+2), \ldots, \sigma(2n-1)$.
%Therefore, $i\in \{i \in   \in   G'\backslash \{\sigma(2n-1),\ldots, \sigma(u+2)\} \colon\: i > u \}$.
%This establishes the equality in \eqref{eq.technical.equation.a.dperm}.
%
%To obtain Equation~\eqref{eq.lem.cycle.closer.technical.dperm},
%it suffices to show that the cardinality of the set
%$\{i > u \colon \: \sinv(i) \leq  u\}$
%is $f_{u}$.
%To do this, recall the description of $f_u$
%in Equation~\eqref{eq.def.fk.dperm} and observe that
%\begin{eqnarray}
%    f_u  & = &  \#\{ i \le u \colon\: \sigma(i) > u \}  \nonumber \\
%            & = &  \# \{i > u \colon\:  \sinv(i) \leq u \}
%\label{eq.technical.equation.b}
%\end{eqnarray}
%where the second equality is obtained by replacing $i$ with $\sinv(i)$.
%\end{proof}

We are now ready to count the number of cycle closers.

\begin{lem}[Counting of cycle closers for D-permutations using variant labels]
Fix an almost-Dyck path $\omega$ of length $2n$ and
construct $F',G'$ (these are completely determined by $\omega$).
Let $y_j \in G\cap F'$.
Also fix labels $\widehat{\xi}_u$ for vertices
$u\in G' \,\cup\, \{y_n,\ldots,y_{j+1}\}$
satisfying~\eqref{eq.xihat.ineqs}.
Then
\begin{itemize}
      \item[(a)] The value of $\widehat{\xi}_{y_j}$ completely determines
               if $y_j$ is a cycle closer or not.
      \item[(b)] There is exactly one value $\widehat{\xi}_{y_j}\in \{0,1,\ldots,\lceil h_{y_j-1}/2\rceil\}$
               that makes $y_j$ a cycle closer, and conversely.
\end{itemize}
\label{lem.cycle.closer.dperm.prime}
\end{lem}

\subsubsection{Running example 2}

Let us now look at the variant DS history of our running example for D-permutation
$\sigma  =  7\, 1\, 9\, 2\, 5\, 4\, 8\, 6\, 10\, 3\, 11\, 12\, 14\, 13\,
=      (1,7,8,6,4,2)\,(3,9,10)\,(5)\,(11)\,(12)\,(13,14) \in \dperm_{14}.$
The sets $G'$ and $F'$ are $G' = \{5, 7, 8, 9, 10, 11, 14 \}$
and $F' =  \{13, 12, 6, 4, 3, 2, 1\}$.
%The sets $F,G$ and $H$ were already recorded in~\eqref{eq.FGH.running.example.2}
%and we recall that $F = \{1, 3, 7, 9, 13\}$, $G = \{2, 4, 6, 8, 10, 14\}$ and $H=\{5, 11, 12\}$.
The variant DS order consists of the following stages:
\begin{itemize}
\item Stage (a): $G' = \{5, 7, 8, 9, 10, 11, 14 \}$ in increasing order
\item Stage (b): $F' = \{13, 12, 6, 4, 3, 2, 1\}$ in decreasing order
\end{itemize}
Stage~(a) of the  history of $\sigma$ has been drawn in
Figure~\ref{fig.running.example.2.DS.variant.history.a}
and Stage~(b) has been drawn in Figure~\ref{fig.running.example.2.DS.variant.history.b}.
Non-singleton cycles are formed in Stage~(b) when the edges $14\to 13$, $10\to 3$
and $2\to 1$ are inserted.

\begin{figure}[p]
\vspace*{-12mm}
\centering
\begin{tabular}{l}
%\begin{tikzpicture}[scale=0.6][scale = 1]
%\node at (0,0) {Begin with $\laguerre{\emptyset}$ with vertex set $\{1,2,3,4,5,6,7,8,9,10,11\}$ and no edges};
%\node[right] at (-4,-2.5) {$\laguerre{\emptyset}$};
%\emptysigmatwo{-2.5}
%\node[right] at (-4,-5) {$\laguerre{\emptyset}$};
%\emptysigmatwo{-5}
%\end{tikzpicture}\\
%
\begin{tikzpicture}[scale=0.6]
%\node[right] at (-7,0) {$\laguerre{\emptyset}$};

\node[right] at (-7,0) {$\laguerre{\emptyset}$};
\emptysigmatwo{0}%

\end{tikzpicture}\\
%\hline\\[-4mm]
%\hline\\[-3mm]
%After Stage (a): $H = \{5, 11, 12\}$\\[2mm]
%\hline\\[-3mm]
%%
%%
%\begin{tikzpicture}[scale=0.6]
%\node[right] at (-7,0) {$\laguerre{\{5,11,12\}}$};
%\emptysigmatwo{0}
%\node[circle,fill=black,inner sep=1pt,minimum size=5pt] (e) at (9,\ey) {} edge [in=45,out=135, thick, loop above, color=red] node {} ();
%\node[circle,fill=black,inner sep=1pt,minimum size=5pt] (k) at (10,\ky) {} edge [in=45,out=135, thick, loop above,color=red] node {} ();
%\node[circle,fill=black,inner sep=1pt,minimum size=5pt] (l) at (11,\ly) {} edge [in=45,out=135, thick, loop above,color=red] node {} ();
%\end{tikzpicture}\\
%
\hline\\[-4mm]
\hline\\[-3mm]
Stage (a): $G' = \{5, 7, 8, 9, 10, 11, 14 \}$ in increasing order\\[2mm]
\hline\\[-3mm]
\begin{tikzpicture}[scale=0.6]
\node[right] at (-7,0) {$\laguerre{\{5\}}$};
\emptysigmatwo{0}
\node[circle,fill=black,inner sep=1pt,minimum size=5pt] (e) at (9,\ey) {} edge [in=45,out=135, thick, loop above, color=red] node {} ();
%\node[circle,fill=black,inner sep=1pt,minimum size=5pt] (k) at (10,\ky) {} edge [in=45,out=135, thick, loop above] node {} ();
%\node[circle,fill=black,inner sep=1pt,minimum size=5pt] (l) at (11,\ly) {} edge [in=45,out=135, thick, loop above] node {} ();
%
%\graph [multi, edges = {thick,red}] {(b) -> (a); };
%
\end{tikzpicture}\\
\hdashline\\[-3mm]
\begin{tikzpicture}[scale=0.6]
\node[right] at (-7,0) {$\laguerre{\{5, 7\}}$};
\emptysigmatwo{0}
\node[circle,fill=black,inner sep=1pt,minimum size=5pt] (e) at (9,\ey) {} edge [in=45,out=135, thick, loop above] node {} ();
%\node[circle,fill=black,inner sep=1pt,minimum size=5pt] (k) at (10,\ky) {} edge [in=45,out=135, thick, loop above] node {} ();
%\node[circle,fill=black,inner sep=1pt,minimum size=5pt] (l) at (11,\ly) {} edge [in=45,out=135, thick, loop above] node {} ();
%
%	\graph [multi, edges = {thick}] {(b) -> (a); };
%	\graph [multi, edges = {thick,red}] {(d) -> (b); };
	\graph [multi, edges = {thick,red}] {(a) -> (g); };
\end{tikzpicture}\\
\hdashline\\[-3mm]
\begin{tikzpicture}[scale=0.6]
\node[right] at (-7,0) {$\laguerre{\{5, 7, 8\}}$};
\emptysigmatwo{0}
\node[circle,fill=black,inner sep=1pt,minimum size=5pt] (e) at (9,\ey) {} edge [in=45,out=135, thick, loop above] node {} ();
%\node[circle,fill=black,inner sep=1pt,minimum size=5pt] (k) at (10,\ky) {} edge [in=45,out=135, thick, loop above] node {} ();
%\node[circle,fill=black,inner sep=1pt,minimum size=5pt] (l) at (11,\ly) {} edge [in=45,out=135, thick, loop above] node {} ();
%
%\graph [multi, edges = {thick}] {(b) -> (a); (d) -> (b); };
%\graph [multi, edges = {thick,red}] {(f) -> (d); };
	\graph [multi, edges = {thick}] {(a) -> (g); };
	\graph [multi, edges = {thick,red}] {(g) -> (h); };
\end{tikzpicture}\\
\hdashline\\[-3mm]
\begin{tikzpicture}[scale=0.6]
\node[right] at (-7,0) {$\laguerre{\{5, 7, 8, 9\}}$};
\emptysigmatwo{0}
\node[circle,fill=black,inner sep=1pt,minimum size=5pt] (e) at (9,\ey) {} edge [in=45,out=135, thick, loop above] node {} ();
%\node[circle,fill=black,inner sep=1pt,minimum size=5pt] (k) at (10,\ky) {} edge [in=45,out=135, thick, loop above] node {} ();
%\node[circle,fill=black,inner sep=1pt,minimum size=5pt] (l) at (11,\ly) {} edge [in=45,out=135, thick, loop above] node {} ();
%
%\graph [multi, edges = {thick}] {(b) -> (a); (d) -> (b); (f) -> (d); };
%\graph [multi, edges = {thick,red}] {(h) -> (f); };
        \graph [multi, edges = {thick}] {(a) -> (g); };
        \graph [multi, edges = {thick}] {(g) -> (h); };
	\graph [multi, edges = {thick,red}] {(c) -> (i); };
\end{tikzpicture}\\
\hdashline\\[-3mm]
\begin{tikzpicture}[scale=0.6]
\node[right] at (-7,0) {$\laguerre{\{5, 7, 8, 9, 10\}}$};
\emptysigmatwo{0}
\node[circle,fill=black,inner sep=1pt,minimum size=5pt] (e) at (9,\ey) {} edge [in=45,out=135, thick, loop above] node {} ();
%\node[circle,fill=black,inner sep=1pt,minimum size=5pt] (k) at (10,\ky) {} edge [in=45,out=135, thick, loop above] node {} ();
%\node[circle,fill=black,inner sep=1pt,minimum size=5pt] (l) at (11,\ly) {} edge [in=45,out=135, thick, loop above] node {} ();
%
%\graph [multi, edges = {thick}] {(b) -> (a); (d) -> (b); (f) -> (d), (h) -> (f); }; 
%\graph [multi, edges = {thick,red}] {(j) -> (c); }; 
        \graph [multi, edges = {thick}] {(a) -> (g); };
        \graph [multi, edges = {thick}] {(g) -> (h); };
        \graph [multi, edges = {thick}] {(c) -> (i); };
	\graph [multi, edges = {thick,red}] {(i) -> (j); };	
\end{tikzpicture}\\
\hdashline\\[-3mm]
\begin{tikzpicture}[scale=0.6]
\node[right] at (-7,0) {$\laguerre{\{5, 7, 8, 9, 10, 11\}}$};
\emptysigmatwo{0}
\node[circle,fill=black,inner sep=1pt,minimum size=5pt] (e) at (9,\ey) {} edge [in=45,out=135, thick, loop above] node {} ();
\node[circle,fill=black,inner sep=1pt,minimum size=5pt] (k) at (10,\ky) {} edge [in=45,out=135, thick, loop above,color=red] node {} ();
%\node[circle,fill=black,inner sep=1pt,minimum size=5pt] (l) at (11,\ly) {} edge [in=45,out=135, thick, loop above, color=red] node {} ();
%
%\graph [multi, edges = {thick}] {(b) -> (a); (d) -> (b); (f) -> (d), (h) -> (f); 
%	(j) -> (c); }; 
%
%\graph [multi, edges = {thick}] {(n) -> [bend right] (m)};
        \graph [multi, edges = {thick}] {(a) -> (g); };
        \graph [multi, edges = {thick}] {(g) -> (h); };
        \graph [multi, edges = {thick}] {(c) -> (i); };
        \graph [multi, edges = {thick}] {(i) -> (j); };
\end{tikzpicture}\\
\hdashline\\[-3mm]
\begin{tikzpicture}[scale=0.6]
\node[right] at (-7,0) {$\laguerre{\{5, 7, 8, 9, 10, 11, 14\}}$};
\emptysigmatwo{0}
\node[circle,fill=black,inner sep=1pt,minimum size=5pt] (e) at (9,\ey) {} edge [in=45,out=135, thick, loop above] node {} ();
\node[circle,fill=black,inner sep=1pt,minimum size=5pt] (k) at (10,\ky) {} edge [in=45,out=135, thick, loop above] node {} ();
%\node[circle,fill=black,inner sep=1pt,minimum size=5pt] (l) at (11,\ly) {} edge [in=45,out=135, thick, loop above] node {} ();
%
%\graph [multi, edges = {thick}] {(b) -> (a); (d) -> (b); (f) -> (d), (h) -> (f);
%        (j) -> (c); };
%
%\graph [multi, edges = {thick,red}] {(n) -> [bend right] (m)};
	\graph [multi, edges = {thick}] {(a) -> (g); };
        \graph [multi, edges = {thick}] {(g) -> (h); };
        \graph [multi, edges = {thick}] {(c) -> (i); };
        \graph [multi, edges = {thick}] {(i) -> (j); };
	\graph [multi, edges = {thick,red}] {(m) -> [bend right] (n)};
\end{tikzpicture}
\end{tabular}
\caption{Stage (a) of the variant DS history for the D-permutation\\
$\sigma = 7\, 1\, 9\, 2\, 5\, 4\, 8\, 6\, 10\, 3\, 11\, 12\, 14\, 13\,
	 =  (1,7,8,6,4,2)\,(3,9,10)\,(5)\,(11)\,(12)\,(13,14) \in \dperm_{14}.$}
\label{fig.running.example.2.DS.variant.history.a}
\end{figure}

\begin{figure}[p]
\centering
\begin{tabular}{l}
\hline\\[-4mm]
\hline\\
Stage (b): $F' = \{13, 12, 6, 4, 3, 2, 1\}$ in decreasing order\\[2mm]
\hline\\[2mm]
\begin{tikzpicture}[scale=0.7]
\node[right] at (-8,0) {$\laguerre{\{5, 7, 8, 9, 10, 11, 14, 13\}}$};
\emptysigmatwo{0}
\node[circle,fill=black,inner sep=1pt,minimum size=5pt] (e) at (9,\ey) {} edge [in=45,out=135, thick, loop above] node {} ();
\node[circle,fill=black,inner sep=1pt,minimum size=5pt] (k) at (10,\ky) {} edge [in=45,out=135, thick, loop above] node {} ();
%\node[circle,fill=black,inner sep=1pt,minimum size=5pt] (l) at (11,\ly) {} edge [in=45,out=135, thick, loop above] node {} ();
        \graph [multi, edges = {thick}] {(a) -> (g); };
        \graph [multi, edges = {thick}] {(g) -> (h); };
        \graph [multi, edges = {thick}] {(c) -> (i); };
        \graph [multi, edges = {thick}] {(i) -> (j); };
        \graph [multi, edges = {thick}] {(m) -> [bend right] (n)};
	\graph [multi, edges = {thick,red}] {(n) -> [bend right] (m)};
\end{tikzpicture}\\
\hdashline\\[-3mm]
\begin{tikzpicture}[scale=0.7]
\node[right] at (-8,0) {$\laguerre{\{5, 7, 8, 9, 10, 11, 14, 13, 12\}}$};
\emptysigmatwo{0}
\node[circle,fill=black,inner sep=1pt,minimum size=5pt] (e) at (9,\ey) {} edge [in=45,out=135, thick, loop above] node {} ();
\node[circle,fill=black,inner sep=1pt,minimum size=5pt] (k) at (10,\ky) {} edge [in=45,out=135, thick, loop above] node {} ();
\node[circle,fill=black,inner sep=1pt,minimum size=5pt] (l) at (11,\ly) {} edge [in=45,out=135, thick, loop above, color=red] node {} ();
        \graph [multi, edges = {thick}] {(a) -> (g); };
        \graph [multi, edges = {thick}] {(g) -> (h); };
        \graph [multi, edges = {thick}] {(c) -> (i); };
        \graph [multi, edges = {thick}] {(i) -> (j); };
        \graph [multi, edges = {thick}] {(m) -> [bend right] (n)};
        \graph [multi, edges = {thick}] {(n) -> [bend right] (m)};

\end{tikzpicture}\\
\hdashline\\[-3mm]
\begin{tikzpicture}[scale=0.7]
\node[right] at (-8,0) {$\laguerre{\{5, 7, 8, 9, 10, 11, 14, 13, 12, 6\}}$};
\emptysigmatwo{0}
\node[circle,fill=black,inner sep=1pt,minimum size=5pt] (e) at (9,\ey) {} edge [in=45,out=135, thick, loop above] node {} ();
\node[circle,fill=black,inner sep=1pt,minimum size=5pt] (k) at (10,\ky) {} edge [in=45,out=135, thick, loop above] node {} ();
\node[circle,fill=black,inner sep=1pt,minimum size=5pt] (l) at (11,\ly) {} edge [in=45,out=135, thick, loop above] node {} ();
        \graph [multi, edges = {thick}] {(a) -> (g); };
        \graph [multi, edges = {thick}] {(g) -> (h); };
        \graph [multi, edges = {thick}] {(c) -> (i); };
        \graph [multi, edges = {thick}] {(i) -> (j); };
        \graph [multi, edges = {thick}] {(m) -> [bend right] (n)};
        \graph [multi, edges = {thick}] {(n) -> [bend right] (m)};
        \graph [multi, edges = {thick, red}] {(h) -> (f); };
\end{tikzpicture}\\
\hdashline\\[-3mm]
\begin{tikzpicture}[scale=0.7]
\node[right] at (-8,0) {$\laguerre{\{5, 7, 8, 9, 10, 11, 14, 13, 12, 6, 4\}}$};
\emptysigmatwo{0}
\node[circle,fill=black,inner sep=1pt,minimum size=5pt] (e) at (9,\ey) {} edge [in=45,out=135, thick, loop above] node {} ();
\node[circle,fill=black,inner sep=1pt,minimum size=5pt] (k) at (10,\ky) {} edge [in=45,out=135, thick, loop above] node {} ();
\node[circle,fill=black,inner sep=1pt,minimum size=5pt] (l) at (11,\ly) {} edge [in=45,out=135, thick, loop above] node {} ();
        \graph [multi, edges = {thick}] {(a) -> (g); };
        \graph [multi, edges = {thick}] {(g) -> (h); };
        \graph [multi, edges = {thick}] {(c) -> (i); };
        \graph [multi, edges = {thick}] {(i) -> (j); };
        \graph [multi, edges = {thick}] {(m) -> [bend right] (n)};
        \graph [multi, edges = {thick}] {(n) -> [bend right] (m)};
        \graph [multi, edges = {thick}] {(h) -> (f); };
	\graph [multi, edges = {thick,red}] {(f) -> (d); };
\end{tikzpicture}\\
\hdashline\\[-3mm]
\begin{tikzpicture}[scale=0.7]
\node[right] at (-8,0) {$\laguerre{\{5, 7, 8, 9, 10, 11, 14, 13, 12, 6, 4, 3\}}$};
\emptysigmatwo{0}
\node[circle,fill=black,inner sep=1pt,minimum size=5pt] (e) at (9,\ey) {} edge [in=45,out=135, thick, loop above] node {} ();
\node[circle,fill=black,inner sep=1pt,minimum size=5pt] (k) at (10,\ky) {} edge [in=45,out=135, thick, loop above] node {} ();
\node[circle,fill=black,inner sep=1pt,minimum size=5pt] (l) at (11,\ly) {} edge [in=45,out=135, thick, loop above] node {} ();
        \graph [multi, edges = {thick}] {(a) -> (g); };
        \graph [multi, edges = {thick}] {(g) -> (h); };
        \graph [multi, edges = {thick}] {(c) -> (i); };
        \graph [multi, edges = {thick}] {(i) -> (j); };
        \graph [multi, edges = {thick}] {(m) -> [bend right] (n)};
        \graph [multi, edges = {thick}] {(n) -> [bend right] (m)};
        \graph [multi, edges = {thick}] {(h) -> (f); };
        \graph [multi, edges = {thick}] {(f) -> (d); };
        \graph [multi, edges = {thick, red}] {(j) -> (c); };
\end{tikzpicture}\\
\hdashline\\[-3mm]
\begin{tikzpicture}[scale=0.7]
\node[right] at (-8,0) {$\laguerre{\{5, 7, 8, 9, 10, 11, 14, 13, 12, 6, 4, 3, 2\}}$};
\emptysigmatwo{0}
\node[circle,fill=black,inner sep=1pt,minimum size=5pt] (e) at (9,\ey) {} edge [in=45,out=135, thick, loop above] node {} ();
\node[circle,fill=black,inner sep=1pt,minimum size=5pt] (k) at (10,\ky) {} edge [in=45,out=135, thick, loop above] node {} ();
\node[circle,fill=black,inner sep=1pt,minimum size=5pt] (l) at (11,\ly) {} edge [in=45,out=135, thick, loop above] node {} ();
        \graph [multi, edges = {thick}] {(a) -> (g); };
        \graph [multi, edges = {thick}] {(g) -> (h); };
        \graph [multi, edges = {thick}] {(c) -> (i); };
        \graph [multi, edges = {thick}] {(i) -> (j); };
        \graph [multi, edges = {thick}] {(m) -> [bend right] (n)};
        \graph [multi, edges = {thick}] {(n) -> [bend right] (m)};
        \graph [multi, edges = {thick}] {(h) -> (f); };
        \graph [multi, edges = {thick}] {(f) -> (d); };
        \graph [multi, edges = {thick}] {(j) -> (c); };
        \graph [multi, edges = {thick, red}] {(d) -> (b); };
\end{tikzpicture}\\
\hdashline\\[-3mm]
\begin{tikzpicture}[scale=0.7]
\node[right] at (-8,0) {$\laguerre{\{5, 7, 8, 9, 10, 11, 14, 13, 12, 6, 4, 3, 2, 1\}}$};
\emptysigmatwo{0}
\node[circle,fill=black,inner sep=1pt,minimum size=5pt] (e) at (9,\ey) {} edge [in=45,out=135, thick, loop above] node {} ();
\node[circle,fill=black,inner sep=1pt,minimum size=5pt] (k) at (10,\ky) {} edge [in=45,out=135, thick, loop above] node {} ();
\node[circle,fill=black,inner sep=1pt,minimum size=5pt] (l) at (11,\ly) {} edge [in=45,out=135, thick, loop above] node {} ();
        \graph [multi, edges = {thick}] {(a) -> (g); };
        \graph [multi, edges = {thick}] {(g) -> (h); };
        \graph [multi, edges = {thick}] {(c) -> (i); };
        \graph [multi, edges = {thick}] {(i) -> (j); };
        \graph [multi, edges = {thick}] {(m) -> [bend right] (n)};
        \graph [multi, edges = {thick}] {(n) -> [bend right] (m)};
        \graph [multi, edges = {thick}] {(h) -> (f); };
        \graph [multi, edges = {thick}] {(f) -> (d); };
        \graph [multi, edges = {thick}] {(j) -> (c); };
        \graph [multi, edges = {thick}] {(d) -> (b); };
        \graph [multi, edges = {thick, red}] {(b) -> (a); };
\end{tikzpicture}
\end{tabular}
\caption{Stage (b) of the variant DS history for the D-permutation\\
$\sigma = 7\, 1\, 9\, 2\, 5\, 4\, 8\, 6\, 10\, 3\, 11\, 12\, 14\, 13\,
	 =  (1,7,8,6,4,2)\,(3,9,10)\,(5)\,(11)\,(12)\,(13,14) \in \dperm_{14}.$}
\label{fig.running.example.2.DS.variant.history.b}
\end{figure}

\subsubsection{Computation of weights}
\label{subsec.computation.dperm.prime}

\proofof{Theorem~\ref{thm.DS.master.prime}}
The computation of weights is completely analogous to what was done in
Section~\ref{subsec.computation.dperm}, 
but using Lemma~\ref{lem.cycle.closer.dperm.prime}
in place of Lemma~\ref{lem.cycle.closer.dperm}.
We leave the details to the reader: the upshot 
(similar to the proof of \cite[Theorem~3.12]{Deb-Sokal})
is that for cycle valleys and cycle peaks, ``u'' 
and ``l'' are interchanged compared to 
Section~\ref{subsec.computation.dperm},
and all the statistics are primed.
It therefore completes the proof of
Theorem~\ref{thm.DS.master.prime}.
\qed

\proofof{Theorems~\ref{thm.DS.pqgen.prime} and ~\ref{thm.dperm.prime}}
Comparing 
\eqref{def.dperm.poly.pqgen.prime} 
with \eqref{def.dperm.master.prime}
we see that the needed specialisation
in \eqref{def.dperm.master.prime} are
the same as given in \eqref{eq.proof.weights.sfa}--\eqref{eq.proof.weights.sff}.
Inserting these into Theorem~\ref{thm.DS.master.prime}
gives Theorem~\ref{thm.DS.pqgen.prime}.

Similarly, the proof of Theorem~\ref{thm.dperm.prime} follows by 
specialising the weights in Theorem~\ref{thm.DS.pqgen.prime} 
to \eqref{eq.thm.conj.DS.specialise}.
\qed

\section{Final remarks}
\label{sec.final}

We began this work only hoping to prove \cite[Conjecture~2.3]{Sokal-Zeng_masterpoly}.
Our initial guess was that this would involve constructing a new bijection
from permutations to labelled Motzkin paths,
possibly by tweaking the Biane bijection \cite{Biane_93}.
However, on discovering our proof, 
we were surprised to see that not only did we not construct any new bijection,
but we used the same variant of the Foata--Zeilberger bijection, 
which Sokal and Zeng use to prove their ``first'' continued fractions for permutations.
As the proofs of the ``first'' theorems in \cite{Deb-Sokal} for D-permutations
are {\em parallel} to the proofs of the ``first'' theorems in \cite{Sokal-Zeng_masterpoly} for permutations
(as mentioned in \cite[Section~8]{Deb-Sokal}),
we managed to then prove \cite[Conjecture~12]{Randrianarivony_96b}
and \cite[Conjecture~4.1]{Deb-Sokal}.

We had to introduce a new total order on $[n]$ ($[2n]$ for D-permutations)
to describe histories for these bijections.
The crucial reason why our proof works is that the
total order is the same for any given Motzkin path (almost-Dyck path for D-permutations),
and also because the associated weights are commutative.
Thus, the order in which we multiply them has no effect on
the product as long as we stick to the same order for any given path.

On the other hand, 
Flajolet \cite{Flajolet_80} provides a more general combinatorial
interpretation with non-commutative weights,
as long as these weights are multiplied using the natural order of the path.
With this in mind,
we think that it will probably not be too difficult to generalize
the ``second'' theorems
for permutations in
\cite{Sokal-Zeng_masterpoly},
and for D-permutations
in \cite{Deb-Sokal}
to obtain continued fractions with non-commutative weights.
We predict that this will also be possible for continued fractions obtained using
the Fran\c{c}on--Viennot bijection \cite{Francon_79}.
However, why such continued fractions 
might be of interest is not immediate to us,
and at the present moment, we refrain from working out the details.

\section*{Acknowledgements}
\addcontentsline{toc}{section}{Acknowledgements}

We wish to thank Alan~Sokal and Jiang~Zeng for helpful conversations
and for their advice on drafts of this work.
We also thank Ashvni Narayanan for her comments on the title and abstract
of this paper.
We thank the anonymous referees whose comments have vastly improved
the exposition of this paper.

This work was supported by a teaching assistantship from
the Department of Mathematics, University College London.
The author is currently supported by the
DIMERS project ANR-18-CE40-0033 funded by
Agence Nationale de la Recherche (ANR, France).

%{\bf THE FINAL VERSION SHOULD ACKNOWLEDGE DIMERS!!!!!}

\end{document}